\documentclass[10pt]{amsart}

\usepackage[hmargin={2.1cm,2.1cm},vmargin={2.5cm,2.5cm}]{geometry}

\usepackage{amsmath, amsthm, amssymb}
\usepackage{mathtools}
\usepackage{graphics}
\usepackage{float}
\usepackage{epsfig}
\usepackage{enumerate}
\usepackage{nicefrac}
\usepackage{cancel}

\DeclareMathAlphabet{\mathpzc}{OT1}{pzc}{m}{it}
\usepackage{mathrsfs}

\usepackage{stmaryrd}

\usepackage{macros}
\usepackage[usenames,dvipsnames]{xcolor}

\begin{document}

\title[Two-phase ferrofluids]{\LARGE{A diffuse interface model for two-phase ferrofluid flows}}

\author[R.H.~Nochetto]{Ricardo H.~Nochetto}
\address[R.H.~Nochetto]{Department of Mathematics and Institute for Physical Science and Technology, University of Maryland, College Park, MD 20742, USA.}
\email[R.H.~Nochetto]{rhn@math.umd.edu}

\author[A.J.~Salgado]{Abner J.~Salgado}
\address[A.J.~Salgado]{Department of Mathematics, University of Tennessee, Knoxville, TN 37996, USA.}
\email[A.J.~Salgado]{asalgad1@utk.edu}

\author[I.~Tomas]{Ignacio Tomas}
\address[I.~Tomas]{Department of Mathematics, University of Maryland, College Park, MD 20742, USA.}
\curraddr{Department of Mathematics, Texas A\&M University, College Station, TX 77843, USA.}
\email{itomas@tamu.edu}

\thanks{The work of RHN and IT is supported by NSF grants DMS-1109325
  and DMS-1411808. AJS is partially supported by NSF grants
  DMS-1109325 and DMS-1418784.}



\date{\textcolor{magenta}{\today.}}

\begin{abstract}
We develop a model describing the behavior of two-phase ferrofluid flows using phase field-techniques and present an energy-stable numerical scheme for
it. For a simplified, yet physically realistic, version of this model and the corresponding numerical scheme we prove, in addition to stability, convergence and as by-product existence of solutions. With a series of numerical experiments we illustrate the potential of these simple models and their ability to capture basic phenomenological features of ferrofluids such as the Rosensweig instability.
\end{abstract}

\maketitle
\section{Introduction}
A ferrofluid is a liquid which becomes strongly magnetized in the presence of applied magnetic fields. It is a colloid made of nanoscale monodomain ferromagnetic particles suspended in a carrier fluid (water, oil, or other organic solvent). These particles are suspended by Brownian motion and will not precipitate nor clump under normal conditions. Ferrofluids are dielectric (non conducting) and paramagnetic (they are attracted by magnetic fields, and do not retain magnetization in the absence of an applied field); see \cite{Behrens}.

Ferrofluids can be controlled by means of external magnetic fields, which gives rise to a wealth of control-based applications. They were developed in the 1960's to pump fuel in spacecrafts without mechanical action \cite{Stephen1995}. Recent interest in ferrofluids is related to technical applications such as instrumentation, vacuum technology, lubrication, vibration damping, radar absorbing materials, and acoustics \cite{Miwa2003,Raj1995,Vinoy2011}; they are used, for instance, as liquid seals for the drive shafts of hard disks, for vibration control and damping in vehicles and enhanced heat transfer of electronics. Other potential applications are in micro/nanoelectromechanical systems: magnetic manipulation of microchannel flows, particle separation, nanomotors, micro electrical generators, and nanopumps \cite{Shib2011,Hart2004,Yama2005,Zahn01,Zeng2013}. One of the most promising applications is in the field of medicine, where targeted (magnetically guided) chemotherapy and radiotherapy, hyperthermia 
treatments, 
and magnetic resonance imaging contrast enhancement are very active areas or research \cite{Rin2009,Pank03,Shap2013}. Yet another potential applications of ferrofluids under current research is the construction of adaptive deformable mirrors \cite{Laird04,Laird06,Brou07}.

At the time of this writing there are two well established PDE models as a mathematical description for the behavior of ferrofluids which we will call by the name of their developers: the Rosensweig model and the Shliomis model (cf.~\cite{Ros97,Shli2002}). Rigorous mathematical work on the mathematical analysis (existence of global weak solutions and local existence of strong solutions) for the Rosensweig and the Shliomis models is very recent (cf. \cite{AmiShliomis2008,Ami2009,Ami2010,Ami2008}).

The applications mentioned above justify the development of tools for the simulation of ferrofluids, but they are not the only reasons. Mathematical models for ferrofluids and their scope of validity have been areas of active research (cf. \cite{Oden2002,Oden2009}). Most ferrofluid flows have so far been studied using exact and approximate analytical solutions of the Rosensweig model (see for instance \cite{Rinal02}) contrasted with experimental data. However, these flows are analytically tractable in a very limited number of cases \cite{Rinal02,Zahn95}, and as shown for instance in \cite{Chav2008}, satisfactory model calibration/validation is beyond the current capabilities of analytic (asymptotic and perturbation) methods. Clearly, there is significant room for interdisciplinary work at the interface between model development, numerical analysis, simulation and experimentation.

Both the Rosensweig and Shliomis models deal with one-phase flows, which is the case of many technological applications. However, some applications arise naturally in the form of a two-phase flow: one of the phases has magnetic properties and the other one does not (e.g. magnetic manipulation of microchannel flows, microvalves, magnetically guided transport, etc.). There has been a major effort in order to develop appropriate interfacial conditions of two-phase flows in the sharp interface regime within the micropolar theory (see \cite{ros2007,Chav2014}), yet we are far from having at our disposal a mathematically and physically sound PDE model for two-phase ferrofluid flows. There are not well established PDE models describing the behavior of two-phase ferrofluid flows. On the other hand, systematic derivation of a two-phase model from first principles, using energy-variational techniques in the spirit of Onsager's principle as in \cite{Liu2009,Huan2009,Liu2010,Garcke2012,ElectroAbner}, would be 
highly desirable, but most probably too premature, given the current state of the art.

In this context, numerical analysis and scientific computation have a lot to offer, since carefully crafted computational experiments can help understand much better the limits of the current models and assist the development of new ones. Ad-hoc development (trial and error) of new models and numerical evaluation does not replace a proper mathematical derivation, but it can clearly help to find a reasonable starting point. In this spirit, the main goal of this work is to present a simple two-phase PDE model for ferrofluids. The model is not derived, but rather assembled using components of already existing models and high-level (as opposite to deep) understanding of the physics of ferrofluids. The model attempts to retain only the essential features and mathematical difficulties that might appear in much more sophisticated models. To the best of our knowledge this contribution is the first modeling/numerical work in the direction of time-dependent behavior of two-phase ferrofluid flows together with energy-stable and/or convergent schemes.

Regarding pre-existing work, closely related to two-phase flows, it is worth mentioning the interdisciplinary (including physical experiments) work of Tobiska and collaborators \cite{Tob2006,Tob2007,Tob2008} in the context of stationary configurations of free surfaces of ferrofluids using a sharp interface approach. Other models for two-phase ferrofluid flows, this time for non-stationary phenomena, are presented in \cite{Trung2011,AFK2008,AFK2010}, using either Level-Set or Volume of Fluid methods, but very little details are given about their actual numerical implementation, stability or convergence properties.

Our presentation is organized as follows: in Section~\ref{Sderivation}
we select the components of our two-phase model and assemble it. In
Section~\ref{Senergyest} we derive formal energy estimates which will
serve as basis for the development of an energy-stable scheme in
Section~\ref{numerSchSection}; in \S\ref{Sexistence} we prove
that the scheme always has a solution. In
Section~\ref{Ssimplification}, we propose a simplification of the
model with a more restrictive scope of physical validity. We present
the corresponding numerical scheme, and prove its stability and
convergence in \S\ref{Sconvscheme} and \S\ref{convschemesec}
respectively. Finally, we show the potential of the model in
Section~\ref{Sexperiment} with a series of insightful numerical
experiments. They include the Rosensweig instability with uniform and
non-uniform applied magnetic fields, the latter leading to an open
pattern of spikes.

\section{Heuristic derivation of a two-phase model}
\label{Sderivation}
We want to develop a simplified model which captures the essence of immiscible, matching density (or almost matching density), two-phase flows, one of them a ferrofluid and the other one a non-magnetic fluid. We aim at a simple mixture like water and an oil based ferrofluid (with, for instance, densities $1000 \ kg/m^3$ and $1050 \  kg/m^3$, respectively), where the dominant body force is the Kelvin force, and the gravitational body force only plays a secondary role, so that we could use the Boussinesq approximation in order to capture gravitational effects.

We will not present a systematic derivation of the model, but rather review existing models and standard assumptions, discard all the non-essential components, and select the right ingredients which could capture the basic phenomenological features of ferrofluids. Our main guidelines are minimalism and symmetry. We want the simplest model, with the smallest number of constitutive parameters and coupled PDEs, that still retains the essential features, and has sufficient symmetries (i.e. cancellations) in order to make possible the development of an energy law.

We consider a two-fluid system (a ferrofluid and a non-ferromagnetic
one), confined in a bounded convex polygon/polyhedron $\Omega \subset
\mathbb{R}^d$ ($d=2$ or $3$), and we denote by $\Gamma$ the boundary of $\Omega$.
Since we implicitly track the position of each fluid with a phase
field variable $\phvar$, our model is of the so-called diffuse-interface type. The evolution of the system is described by its velocity $\bv{u}$ and pressure $p$. As one of the fluid phases is susceptible to magnetic actuation, we need to keep track of the magnetization $\bv{m}$, which is induced by a magnetic field $\bv{h}$. To describe the evolution of these quantities we will consider:

\begin{enumerate}[\itemizebullet]
\item \textbf{Evolution of the phase-field variable $\phvar$:} there are two well-known PDE models for this purpose, the Allen-Cahn and the Cahn-Hilliard models. In particular, we will consider the Cahn-Hilliard equation:
\begin{align}
\left\{
\begin{aligned}
\phvar_t &= - \mobility \Delta \chpot  &&\text{in} \ \ \Omega\\
 \chpot &= \layerthick\Delta\phvar - \tfrac{1}{\layerthick}f(\phvar) &&\text{in} \ \ \Omega \\
\partial_n \phvar &= \partial_n\chpot = 0 
&&\text{on} \ \ \Gamma ,
\end{aligned}
\right. \, 
\end{align}
where $0 < \layerthick \ll 1$ is related to the interface thickness,
$\mobility>0$ is the (constant) mobility, $\psi$ is the chemical
potential, $f(\phvar) = F'(\phvar)$ and $F(\phvar)$ is the truncated double well potential
\begin{align}\label{truncdef}
F(\phvar) = 
\begin{cases}
(\phvar + 1)^2 & \text{if } \phvar \in  (-\infty, -1] \\
\tfrac{1}{4} (\phvar^2 - 1)^2 & \text{if } \phvar \in [-1, 1] \\
(\phvar - 1)^2 & \text{if } \phvar \in  [1,+\infty) \, .
\end{cases}
\end{align}
It is straightforward to check that 
\begin{align}\label{doublewellbound}
|f(\phvar)| = |F '(\phvar)| \leq 2 |\phvar| + 1 \ \ \text{and} \ \
|f'(\phvar)| = |F''(\phvar)|\leq 2 \ \ \forall \phvar \in \mathbb{R}.
\end{align}
The reason to choose the Cahn-Hilliard equation is that it is mass conservative, an easy consequence of the divergence theorem:
\begin{align}
\tfrac{d}{dt} \bulkint{\Omega}{ \phvar } = \bulkint{\Omega}{ \phvar_t } = 
- \mobility \bulkint{\Omega}{  \Delta \chpot } =
- \mobility \bdryint{\bdry}{  \partial_n \chpot } = 0 \, .
\end{align}
\item[\itemizebullet] \textbf{Simplified ferrohydrodynamics:} the
  Shliomis model (see for instance \cite{AmiShliomis2008,Shli2002}) is
  perhaps the simplest well-known PDE model describing the behavior of
  ferrofluids. It couples the Navier-Stokes equations of
  incompressible fluids for the velocity-pressure pair $(\bv{u},p)$
  with an advection-reaction equation for the magnetization $\bv{m}$:
\begin{subequations}\label{Shleq}
\begin{align}
\label{shliNS}
\bv{u}_t + (\bv{u}\cdot\nabla)\bv{u}- \nu \Delta\bv{u} + \nabla p &=  \mu_0 (\bv{m}\cdot\nabla)\heff + \tfrac{\mu_0}{2} \curl{}(\bv{m \times \heff}) \, ,\\
\label{shliMag}
\bv{m}_t + (\bv{u}\cdot\nabla)\bv{m} - \tfrac{1}{2} \curl{u} \times \bv{m} &= - \tfrac{1}{\chartime} (\bv{m} - \permit \heff) - \beta \, \bv{m} \times (\bv{m} \times \heff) \, , 
\end{align}
\end{subequations}
where $\nu$, $\mu_0$, $\chartime$, $\beta$, and $\permit$ are positive constitutive constants. Expression \eqref{shliMag} can be partially understood as the $\ltwod$ gradient flow of the functional
\begin{align}\label{functional}
\mathcal{J}(\bv{m}) = \tfrac{1}{2\chartime} \|\bv{m} - \permit \heff\|_{\ltwods}^2 \, , 
\end{align}
augmented with the corresponding kinematics. In other words
\begin{align}\label{magderiv}
\langle \bv{m}_t , \bv{z} \rangle =  - \langle \tfrac{\delta\mathcal{J}}{\delta\bv{m}},\bv{z}\rangle \ \ \ \Longrightarrow  \ \ \  
\bv{m}_t + \tfrac{1}{\chartime} \bv{m} = \tfrac{\permit}{\chartime} \heff \, , 
\end{align}
where the symbol $\delta$ denotes variational derivative in this context. We observe that if $\chartime <<1$ then the dynamics is fast towards equilibrium $\bv{m} \approx \permit \heff$. After that, we can replace the partial derivative $\bv{m}_t$ in \eqref{magderiv} with the co-rotational derivative $\bv{m}_t + (\bv{u}\cdot\nabla)\bv{m} - \tfrac{1}{2} \curl{u} \times \bv{m}$ (see for instance \cite{Virga2012}) accounting for the appropriate kinematics. On the other hand, the term $\beta \, \bv{m} \times (\bv{m} \times \heff)$ has phenomenological origins which in principle cannot be easily related to kinematic or energy principles (see \cite{Ros2002,Shli2002}).

The Shliomis model can be considered to be a limiting case of the more sophisticated Rosensweig model (see for instance \cite{Ami2008,Wang2010}) which takes into account the angular velocity. The core dynamics of the magnetization equation \eqref{shliMag} is dominated by the reaction terms for most flows of interest (see for instance \cite{rinaldi2002,Rinal02} for the dimensional analysis of the Rosensweig model). Essentially, this is the case because the relaxation time $\chartime$ of commercial grade ferrofluids is in the range of $10^{-5}$ to $10^{-9}$ seconds (see for instance \cite{rinaldi2002,Shli2002}), which makes $\tfrac{1}{\chartime}$ a very large constant. This somehow justifies the straightforward simplification of \eqref{Shleq}:
\begin{subequations}\label{Shleqref}
\begin{align}
\bv{u}_t + (\bv{u}\cdot\nabla)\bv{u} - \nu \Delta\bv{u} + \nabla p &=  \mu_0 (\bv{m}\cdot\nabla)\heff \, , \\
\bv{m}_t + (\bv{u}\cdot\nabla)\bv{m} + \tfrac{1}{\chartime} \bv{m} &= \tfrac{\permit}{\chartime} \heff \, .
\end{align}
\end{subequations}
In fact, in \eqref{Shleqref} we have dropped the terms $\tfrac{\mu_0}{2} \curl{}(\bv{m \times \heff})$ and $\beta \, \bv{m} \times (\bv{m} \times \heff)$ under the assumption that at every moment the behavior of $\bv{m}$ is very close to equilibrium, meaning that $\bv{m} \approx \permit \heff$ and $\bv{m} \times\heff \approx 0$, so that these terms are negligible. However, the convective term $(\bv{u}\cdot\nabla)\bv{m}$ is kept because of symmetry considerations and to be able to develop an energy stable model. On the other hand, the term $- \tfrac{1}{2} \curl{u} \times \bv{m}$ was dropped under the assumption that convection and reaction are the dominant terms. In a somewhat different context, similar ideas where used in order to simplify liquid-crystal models (see for instance \cite{LL1995}).

\smallskip
\item[\itemizebullet] \textbf{Simplified capillary force:} the
  capillary force is given by $\bv{f}_{\text{c}} = - \diver{} \stresstensor_{\text{c}}$, where $\stresstensor_{\text{c}} = \capcoeff \nabla\phvar \otimes \nabla\phvar$ is the so-called capillary stress tensor and $\capcoeff$ is the capillary coefficient (see for instance \cite{Ander1998,Lowen1998}). Manipulating $\bv{f}_{\text{c}}$ we get:
\begin{equation}
\label{capfderiv}
  \begin{aligned}
    \bv{f}_{\text{c}} &= - \diver{} \stresstensor_{\text{c}} =  - \tfrac{\capcoeff}{2} \nabla |\nabla\phvar|^2  - \capcoeff \Delta\phvar \nabla\phvar
    =  - \tfrac{\capcoeff}{2} \nabla |\nabla\phvar|^2 - \tfrac{\capcoeff}{\layerthick^2} f(\phvar) \nabla\phvar - \tfrac{\capcoeff}{\layerthick} \chpot \nabla\phvar \\
    &= - \capcoeff \nabla \alp \tfrac{1}{2}|\nabla\phvar|^2 + \tfrac{1}{\layerthick^2} F(\phvar) \arp - \tfrac{\capcoeff}{\layerthick} \chpot \nabla\phvar 
    = - \capcoeff \nabla \alp \tfrac{1}{2}|\nabla\phvar|^2 + \tfrac{1}{\layerthick^2} F(\phvar) + \tfrac{1}{\layerthick} \chpot \phvar \arp
    + \tfrac{\capcoeff}{\layerthick} \phvar \nabla \chpot \, .
  \end{aligned}
\end{equation}
Therefore the term $- \capcoeff \nabla \alp \tfrac{1}{2}|\nabla\phvar|^2 
+ \tfrac{1}{\layerthick^2} F(\phvar) + \tfrac{1}{\layerthick} \chpot \phvar \arp$ only modifies the pressure in the Navier-Stokes system (see Remark \ref{remredefpress}), so that it can be eliminated at the expense of redefining the pressure. Our capillary force will finally be:
\begin{align}\label{capforcedef}
\bv{f}_{\text{c}} := \tfrac{\capcoeff}{\layerthick} \phvar \nabla \chpot .
\end{align}
This definition of the capillary force traces back to \cite{LL1995,LiuShen2003} and is not a cosmetic manipulation but rather an essential ingredient in order to have sufficient cancellations allowing the development of an energy law.

\smallskip
\item[\itemizebullet] \textbf{Simplified electromagnetism:} the natural choice in this context are the magnetostatics equations:
\begin{align*}
\curl{}\heff = 0 \ , \ \  \diver{b} = 0 \, ,
\end{align*}
where
\begin{align}\label{heffdefS4}
\bv{b} := \mu_0 (\heff + \bv{m}) \ , \ \   \heff := \ha + \hd \, , 
\end{align}
$\ha$ is the (given) smooth harmonic (curl-free and div-free) applied
magnetizing field, and $\hd$ is the so-called demagnetizing field.
We use that $\bv{m}=\bv 0$ in $\mathbb{R}^d \backslash\Omega$ and
  assume that $\hd\approx\bv 0$ in $\mathbb{R}^d \backslash\Omega$ to derive
  the boundary condition $\heff\cdot\bv{n} = (\ha - \bv{m})\cdot\bv{n}$.
A simplified approach involves the scalar potential
$\hdpot$  \cite{Tom14a}, which satisfies
\begin{align}
\label{totalhTwPh}
\heff = \nabla\hdpot \, ,
\end{align}
along with
\begin{align}
\label{phiNeuIIItwo}
-\Delta \hdpot = \diver{}(\bv{m}-\ha) \ \text{in } \Omega,
\ \
\partial_n \hdpot = (\ha - \bv{m})\cdot \normal \ \text{on } \bdry
\, .
\end{align}

\end{enumerate}

Collecting all these simplifications we end up with the
following set of equations in strong form in $\Omega$
\begin{subequations}
\label{Themodel}
\begin{align}
\label{systemCH1} 
\phvar_t + \diver{}(\bv{u} \phvar)+ \mobility \Delta \chpot &= 0 \, ,\\
\label{systemCH2} %
\chpot - \layerthick \Delta\phvar + \tfrac{1}{\layerthick} f(\phvar)   &= 0 \, , \\
\label{systemM} %
\bv{m}_t + (\bv{u}\cdot\nabla)\bv{m} &= 
- \tfrac{1}{\chartime} (\bv{m} - \suscepc \heff) \, ,  \\
\label{systemPhi} %
-\Delta\hdpot &= \diver{}(\bv{m} - \ha) \, , \\
\label{systemNS1} %
\bv{u}_t + (\bv{u}\cdot\nabla)\bv{u}- \diver{}( \nu_{\phvar} \, \bv{T}(\bv{u})) + \nabla p &=  \mu_0 (\bv{m}\cdot\nabla)\heff + \tfrac{\capcoeff}{\layerthick} \phvar\nabla\chpot \, , \\
\label{systemNS2} %
\diver{u} &= 0 \, ,
\end{align}
\end{subequations}
for every $t \in [0,\tf]$, where $\bv{T}(\bv{u}) = \tfrac12 ( \nabla \bv{u} + \nabla \bv{u}^\intercal)$ denotes the symmetric gradient, and $\heff$ is defined in \eqref{totalhTwPh}-\eqref{phiNeuIIItwo}.

We supplement this system with the following boundary conditions
\begin{align}
\partial_n \phvar = \partial_n\chpot = 0 \, , \ \  \bv{u} = 0 \, , \ \
\text{and} \  \ \partial_n \hdpot = (\ha - \bv{m})\cdot \normal \ \
  \text{on} \ \ \Gamma.
\end{align}

Here $\nu_{\phvar}$ and $\suscepc$ are viscosities and susceptibilities depending on the phase-field variable $\phvar$. They are Lipschitz-continuous functions of $\phvar$ satisfying
\begin{align}\label{coeffbounds}
0< \min\{ \nu_{w}, \nu_f \}  \leq \nu_{\phvar} \leq \max\{ \nu_{w}, \nu_f \} \ \ \ \text{and} \ \ \ 
0 \leq \suscepc \leq \suscep \, , 
\end{align}
where $\nu_{w}$ is the viscosity of the non-magnetic phase (e.g. water) and $\nu_f$ is the viscosity of the ferrofluid (e.g. mineral oil). Here $\suscep > 0$ is the magnetic susceptibility of the ferrofluid phase, and we set the non-magnetic phase to have zero magnetic susceptibility. For commercial grade ferrofluids we have that $\suscep$ ranges from $0.5$ to $4.3$ (see for instance \cite{Rinal02}). The choice of functions $\nu_{\phvar}$ and $\suscepc$ is arbitrary, but essentially they involve a regularized approximation of the Heaviside step function, for instance
\begin{align}\label{viscsuscep}
\nu_{\phvar} = \nu_{w} + (\nu_f - \nu_w) \mathcal{H}(\phvar/\layerthick) \ \  \text{and}
\ \ \suscepc = \suscep \mathcal{H}(\phvar/\layerthick) \, ,
\end{align}
where $\mathcal{H}(x)$ could be for instance the sigmoid function
\begin{align}\label{approxstep}
\mathcal{H}(x) = \frac{1}{1+e^{-x}} \, .
\end{align}
Both in theory and practice, the choice of $\mathcal{H}(x)$ and the internal structure of $\nu_{\phvar}$ and $\suscepc$ are of very little importance, provided they are Lipschitz-continuous and satisfy inequalities \eqref{coeffbounds}. Here \eqref{viscsuscep} and \eqref{approxstep} are just provided as a couple of simple choices, but they are not the only ones; other choices are used for instance in \cite{ElectroAbner,LiuShen2003}.

Since this model is not a genuinely variable-density two-phase model,
gravitational effects can only be included approximately. We
will consider supplementing the right hand side of the conservation of
linear momentum \eqref{systemNS1} with a Boussinesq-like approximation
$\bv{f}_g$ in order to include such effects
\begin{align}\label{Bouapprox}
\bv{f}_g = (1 + r \,  \mathcal{H}(\phvar/\layerthick)) \bv{g} \ , \ \ r = \tfrac{|\rho_f - \rho_w|}{\textsl{min}(\rho_f,\rho_w)} \ , 
\end{align}
where $\rho_f$ and $\rho_w$ are the densities of the
ferromagnetic phase and non-ferromagnetic phase respectively, and
$\bv{g}$ is the acceleration of gravity.
Approximation \eqref{Bouapprox} is reasonable provided $r < < 1$.

The development of a complete existence theory for system \eqref{Themodel} seems unlikely, as it has been the history of most systems of PDEs without sufficient regularization mechanisms (e.g. compressible Euler equations of gas dynamics). This is primarily because of the sub-system \eqref{systemM}--\eqref{systemPhi} and the term $\mu_0 (\bv{m}\cdot\nabla)\heff$ on the right hand side of \eqref{systemNS1}. A first approach to solve this problem would be adding a regularization of the form $-\sigma\Delta\bv{m} =  \curl{}(\sigma\curl{}\bv{m}) -\nabla (\sigma \diver{m})$ in the equation \eqref{systemM} (\cf \cite{Tom14a,Ami2008}), or any other second order operator in space. However, most forms of regularization that we could add to this system will introduce new problems, primarily (but not only) related to boundary conditions, and the overall system might not even be formally energy-stable (see for instance \cite{Tom14a}).

These mathematical obstacles will not interfere with our exploration of the model \eqref{Themodel}, which is still a reasonable starting point to understand and develop PDE models for two-phase ferrofluid flows. It is actually possible to develop energy stable numerical methods for system \eqref{Themodel} and prove local-solvability of the scheme for every time step. With the aid of our numerical scheme, we will explore the behavior of this coupled system. Finally, under special circumstances, and with some simplifications, we will obtain a system for which it is possible to prove convergence when the discretization parameters $h$ and $\dt$ go to zero and, as a by product, global existence of weak solutions. 
\begin{remark}[redefinition of the pressure]\label{remredefpress} Given $\varpi \in \{0,1\}$, let $(\bv{u},p) \in \hzerod \times \lzerotwo$ solve
\begin{alignat*}{2}
(\nabla\bv{u},\nabla\bv{v}) - (p,\diver{v}) &= (\bv{f},\bv{v}) + \varpi (g,\diver{v})   && \ \ \forall \bv{v} \in \hzerod \, , \\
(q,\diver{u}) &= 0  &&\ \ \forall q \in \lzerotwo \, , 
\end{alignat*}
Since the pressure can be redefined as $\widehat{p} = p + \varpi g$ for all $\varpi \in \{0,1\}$, we can conclude that the velocity is independent of $g$. This has been used to devise energy stable schemes for phase field models \cite{LiuShen2003} and liquid crystals \cite{LinLiu2007}. In this work, we have already used this property for the derivation of the capillary force when we eliminated the term $- \capcoeff \nabla \alp \tfrac{1}{2}|\nabla\phvar|^2 + \tfrac{1}{\layerthick^2} F(\phvar) + \tfrac{1}{\layerthick} \chpot \phvar \arp$ from \eqref{capfderiv} and defined the capillary force in \eqref{capforcedef}. We will use this idea again in Section \S\ref{addnot} for the discretization of the Kelvin force $\mu_0 (\bv{m}\cdot\nabla)\heff$ to eliminate some terms which have no effect on the velocity field.
\end{remark}
%
%
\section{Formal energy estimates}
\label{Senergyest}
We begin by defining the trilinear form $\tril(\cdot,\cdot,\cdot)$
\begin{align}
\label{eq:defoftrilb}
  \tril(\bv{m},\heff,\bv{u}) = \sum_{i,j = 1}^d \int_{\Omega} \bv{m}^i \heff_{x_i}^j  \bv{u}^j \, dx \, ,
\end{align}
and showing a crucial identity that makes the energy estimate possible.

\begin{lemma}[identities for the magnetization and magnetic field]
\label{lem:ids}
Let $\bv{m}$ and $\bv{h}$ denote the magnetization and effective magnetizing field, respectively, and let $\bv{u}$ be a solenoidal field such that $\bv{u}\cdot n = 0$ on $\bdry$. Then the following identity holds true
\begin{align}
\label{eq:centralident}
\tril(\bv{u},\bv{m},\heff) = - \tril(\bv{m},\heff,\bv{u}) \, .
\end{align}
\end{lemma}
\begin{proof}
Since $\heff$ is curl-free we have that $\heff_{x_i}^j = \heff_{x_j}^i$. Integration by parts then yields
\begin{align}
\label{skewproof}
\begin{aligned}
\tril(\bv{m},\heff,\bv{u}) &= \sum_{i,j = 1}^d \int_{\Omega} \bv{m}^i \heff_{x_i}^j  \bv{u}^j \, dx =
\sum_{i,j = 1}^d \int_{\Omega} \bv{m}^i \heff_{x_j}^i  \bv{u}^j \, dx
= \sum_{i,j = 1}^d \int_{\Omega} \lp ( \bv{m}^i \heff^i  )_{x_j} - \bv{m}_{x_j}^i \heff^i \rp \bv{u}^j \, dx  \\
&= \sum_{i,j = 1}^d \int_{\Omega}  - ( \bv{m}^i \heff^i  ) \bv{u}_{x_j}^j
- \bv{m}_{x_j}^i \heff^i \bv{u}^j \, dx
+ \int_{\bdry} ( \bv{m}^i \heff^i  ) \bv{u}^j \normal^j \, ds \\
&=  - \tril(\bv{u},\bv{m},\heff)
- \bulkint{\Omega}{(\bv{m}\cdot\heff) \, \diver{u}}
+ \int_{\bdry} ( \bv{m} \cdot \heff ) \bv{u}\cdot \normal \, ds.
\end{aligned}
\end{align}
The fact that $\bv{u}$ is solenoidal and has vanishing normal trace on the boundary yields \eqref{eq:centralident}.
\end{proof}

With this identity at hand we readily obtain a formal energy estimate. From now on we use the compact notation $\ltwos = \ltwo$ for scalar-valued functions and $\ltwods = \ltwod$ for vector-valued functions.

\begin{proposition}[energy estimate] \label{prop:formalenergy}
If $\suscep \leq 4$, then the following estimate holds for solutions of the system \eqref{Themodel}
\begin{align}
\label{enerestTwPh}
\begin{split}
\mathscr{E}(\bv{u},\bv{m},\heff, \phvar;\tf) + \int_{0}^{\tf} \mathscr{D}(\bv{u},\bv{m},\heff, \phvar;s) \, ds 
\leq \mathscr{E}(\bv{u},\bv{m},\heff, \phvar;0) \,
+ \int_{0}^{\tf} \mathscr{F}(\ha;s)  \, ds \, ,
\end{split}
\end{align}
where
{\allowdisplaybreaks\begin{subequations}
\begin{align*}
\mathscr{E}(\bv{u},\bv{m},\heff, \phvar;s) &= \tfrac{1}{2} \|\bv{u}(s)\|_{\ltwods}^2 
+ \tfrac{\mu_0}{2} \|\bv{m}(s)\|_{\ltwods}^2
+ \tfrac{\mu_0}{2} \|\heff(s) \|_{\ltwods}^2 
+ \tfrac{\capcoeff}{2} \|\nabla\phvar(s)\|_{\ltwods}^2 
+ \tfrac{\capcoeff}{\layerthick^2}(F(\phvar(s)),1) \, , \\
\mathscr{D}(\bv{u},\bv{m},\heff, \phvar;s) &=
\tfrac{\mu_0}{\chartime} \big( 1 - \tfrac{\suscep }{4} \big) \|\bv{m}(s)\|_{\ltwods}^2 
+ \tfrac{\mu_0}{2 \chartime} \|\heff (s)\|_{\ltwods}^2 
+ \|\sqrt{\nu_{\phvar}} \, \bv{T}(\bv{u})(s) \|_{\ltwods}^2 
+ \tfrac{\capcoeff \mobility}{\layerthick} \|\nabla \chpot(s)\|_{\ltwods}^2 \, , \\
\mathscr{F}(\ha;s) &= \mu_0 \chartime  \|\partial_t \ha (s)\|_{\ltwods}^2
+ \tfrac{\mu_0}{\chartime}\|\ha(s)\|_{\ltwods}^2  \, , 
\end{align*}
\end{subequations}}
\end{proposition}
\begin{proof}
Multiply \eqref{systemCH1} by $\chpot$ and  \eqref{systemCH2} by
$\phvar_t$ respectively and integrate by parts. This yields
\begin{align*}
- (\phvar_t,\chpot) + \mobility \|\nabla \chpot\|_{\ltwods}^2 &= - ( \bv{u} \phvar, \nabla \chpot) \, , \\
\tfrac{d}{dt} \alp \tfrac{\layerthick}{2} \|\nabla\phvar\|_{\ltwods}^2
+ \tfrac{1}{\layerthick} (F(\phvar),1) \arp + (\chpot,\phvar_t)  &= 0 \, ,
\end{align*}
whence, adding both lines and multiplying by $\tfrac{\capcoeff}{\layerthick}$, we get
\begin{align}
\label{withWC}
\tfrac{d}{dt} \alp \tfrac{\capcoeff}{2} \|\nabla\phvar\|_{\ltwods}^2 
+  \tfrac{\capcoeff}{\layerthick^2} (F(\phvar),1) \arp 
+ \tfrac{\capcoeff \mobility}{\layerthick}  \|\nabla \chpot\|_{\ltwods}^2 = 
- \tfrac{\capcoeff}{\layerthick} (\phvar \nabla\chpot,\bv{u}) \, .
\end{align}
Now we multiply \eqref{systemNS1} by $\bv{u}$, integrate by parts and use \eqref{eq:centralident} for the Kelvin force
\begin{align}
\label{withU}
\tfrac{1}{2}\tfrac{d}{dt} \|\bv{u}\|_{\ltwods}^2 + \|\sqrt{\nu_{\phvar}} \, \bv{T}(\bv{u}) \|_{\ltwods}^2 =  - \mu_0 \tril(\bv{u},\bv{m},\heff) 
+ \tfrac{\capcoeff}{\layerthick} (\phvar \nabla \chpot, \bv{u}) \, .
\end{align}
Adding \eqref{withWC} and \eqref{withU} we get
\begin{equation}
\label{NSandCHenergy}
    \tfrac{d}{dt} \big( \tfrac{1}{2} \|\bv{u}\|_{\ltwods}^2 
    + \tfrac{\capcoeff}{2} \|\nabla\phvar\|_{\ltwods}^2 
    + \tfrac{\capcoeff}{\layerthick^2}(F(\phvar),1) \big)
    + \|\sqrt{\nu_{\phvar}} \, \bv{T}(\bv{u}) \|_{\ltwods}^2 
    + \tfrac{\capcoeff \mobility}{\layerthick} \|\nabla \chpot\|_{\ltwods}^2 
    =  - \mu_0 \tril(\bv{u},\bv{m},\heff)
\end{equation}
Finally multiply \eqref{systemM} by $\mu_0 \bv{m}$ and $\mu_0 \heff$ and integrate to obtain
\begin{align}
\label{testM1}
\tfrac{\mu_0}{2} \tfrac{d}{dt} \|\bv{m}\|_{\ltwods}^2 
+ \tfrac{\mu_0}{\chartime} \|\bv{m}\|_{\ltwods}^2 
&= \tfrac{\mu_0}{\chartime} (\suscepc \heff, \bv{m}) \, , \\
\label{testM2}
\tfrac{\mu_0}{\chartime} \|\sqrt{\suscepc} \, \heff \|_{\ltwods}^2 &= \mu_0 (\bv{m}_t,\heff) 
+ \mu_0 \tril(\bv{u},\bv{m},\bv{h}) 
+ \tfrac{\mu_0}{\chartime} (\bv{m},\heff) \, . 
\end{align}
Expression \eqref{testM2} requires further manipulation. For this purpose we will use the following two identities
\begin{align}
\label{phiID00}
  \|\nabla\hdpot\|_{\ltwods}^2 = (\ha - \bv{m}, \nabla\hdpot) \, ,  \\
\label{phiID01}
\tfrac{1}{2}\tfrac{d}{dt} \|\nabla\hdpot\|_{\ltwods}^2 = (\partial_t\ha - \partial_t\bv{m}, \nabla\hdpot).
\end{align}
Identity \eqref{phiID00} is obtained by multiplying \eqref{systemPhi} by $\hdpot$ and integrating by parts. In order to derive \eqref{phiID01} we first differentiate \eqref{systemPhi} with respect to time, then multiply by $\hdpot$ and integrate in space. Using \eqref{phiID00} and \eqref{phiID01} (along with $\nabla\hdpot = \heff$) we can rewrite \eqref{testM2} as follows
\begin{align}
\label{testM2proc}
\begin{split}
\tfrac{\mu_0}{2} \tfrac{d}{dt} \|\heff \|_{\ltwods}^2 
&+ \tfrac{\mu_0}{\chartime} \|\heff \|_{\ltwods}^2
+ \tfrac{\mu_0}{\chartime} \|\sqrt{\suscepc} \, \heff \|_{\ltwods}^2 \\
&= \mu_0 (\partial_t \ha,\heff)
+ \mu_0 \tril(\bv{u},\bv{m},\bv{h})
+ \tfrac{\mu_0}{\chartime} (\ha,\heff) .
\end{split}
\end{align}
Adding \eqref{NSandCHenergy}, \eqref{testM1} and \eqref{testM2proc}, we get
\begin{align}\label{onebutlast}
\begin{split}
\tfrac{d}{dt} \mathscr{E}(\bv{u},\bv{m},\heff, \phvar;t) 
&+ \mathscr{D}(\bv{u},\bv{m},\heff, \phvar;s)
+ \tfrac{\mu_0\suscep}{4\chartime} \|\bv{m}\|_{\ltwods}^2 
+ \tfrac{\mu_0}{2\chartime} \|\heff \|_{\ltwods}^2 \\
&+ \tfrac{\mu_0}{\chartime} \|\sqrt{\suscepc} \, \heff \|_{\ltwods}^2 
= \tfrac{\mu_0}{\chartime} (\suscepc \heff, \bv{m})
+ \mu_0 (\partial_t \ha,\heff)
+ \tfrac{\mu_0}{\chartime} (\ha,\heff) \, ,
\end{split}
\end{align}
where the last two terms on the left hand side are artificially created to control the right hand side. In fact, using the bound \eqref{coeffbounds}, the term $\tfrac{\mu_0}{\chartime} (\suscepc \heff, \bv{m})$ can be estimated as follows:
\begin{align*}
\tfrac{\mu_0}{\chartime} (\suscepc \heff, \bv{m}) \leq 
\tfrac{\mu_0}{\chartime} \|\sqrt{\suscepc} \, \heff\|_{\ltwods} \|\sqrt{\suscepc} \, \bv{m} \|_{\ltwods}
\leq \tfrac{\mu_0}{\chartime} \|\sqrt{\suscepc} \, \heff\|_{\ltwods}^2
+ \tfrac{\mu_0 \suscep }{4 \chartime} \|\bv{m} \|_{\ltwods}^2 \, , 
\end{align*}
so that using this estimate in \eqref{onebutlast} we finally get
\begin{align*}
\tfrac{d}{dt} \mathscr{E}(\bv{u},\bv{m},\heff, \phvar;t) 
&+ \mathscr{D}(\bv{u},\bv{m},\heff, \phvar;s)
+ \tfrac{\mu_0}{2\chartime} \|\heff \|_{\ltwods}^2 \leq \mu_0 (\partial_t \ha,\heff)
+ \tfrac{\mu_0}{\chartime} (\ha,\heff) \, .
\end{align*}
The rest is just a matter of applying Cauchy-Schwarz and Young's inequalities to get the energy estimate \eqref{enerestTwPh}. From the expression for $\mathscr{D}(\bv{u},\bv{m},\heff, \phvar;s)$ we also get the restriction that we can only consider $\suscep \leq 4$. \end{proof}


\begin{remark}[range of susceptibility]
The restriction $\suscep \leq 4$, necessary for Proposition~\ref{prop:formalenergy} to hold, covers almost the complete range of commercial grade ferrofluids; see \cite{Rinal02,Oden2002}.
\end{remark}

\section{An energy stable scheme}
\label{numerSchSection}

In this section we present and analyze a discretization of system \eqref{Themodel}, its stability, and the existence of solutions. This scheme is our workhorse: the numerical simulations of Section~\ref{Sexperiment} use this method, and the existence of solutions for our simplified model are based on a scheme very similar to this one.

\subsection{Notation and assumptions}
We introduce $K > 0$ to denote the number of time steps, define the uniform time step as $\dt := T/K >0$ and set $t^k := k\dt$ for $0\leq k \leq K$. For $\varrho : [0,T] \rightarrow E$, with $E$ a normed space, we set $\varrho^k := \varrho(t^k)$, $\varrho^\dt = \lb \varrho^k\rb_{k=0}^{K}$. We introduce the following discrete norms:
\begin{align*}
\|\varrho^\dt \|_{\ell^\infty(E)} := \max_{0 \leq k \leq K} \|\varrho^k\|_E \, ,
\ \ \ 
\|\varrho^\dt \|_{\ell^2(E)} := \alp \sum_{k=0}^K \dt \|\varrho^k\|_E^2 \arp^{\!\!\nicefrac{1}{2}} \, .
\end{align*}
We also define the backward difference operator $\inc$:
\begin{align}\label{backwarddifferenceop}
  \delta\varrho^k = \varrho^k - \varrho^{k-1} .
\end{align}
The following identity will be used repeatedly
\begin{align}\label{sumid}
2 (a,a-b) = |a|^2 - |b|^2 + |a-b|^2 \, .
\end{align}
Space discretization will be based on Galerkin techniques. To this
effect, we introduce families of finite dimensional subspaces
$\FEspacePhase \subset \hone$, $\FEspaceChemP \subset \hone$,
$\FEspaceM \subset \ltwod$, $\FEspacePhi \subset \hone$, $\FEspaceU
\subset \hzerod$ and $\FEspaceP \subset \ltwo$, where we will
approximate the phase field, chemical potential, magnetization,
magnetic potential, velocity and pressure respectively. These
spaces are parametrized by the discretization parameter $h>0$, which
in the context of finite element methods will later be the meshsize.
Regarding the pair of spaces $(\FEspaceU,\FEspaceP)$, we assume that they satisfy a uniform inf-sup compatibility condition (also called LBB compatibility condition), i.e.
\begin{equation}
\label{discreteinfsup}
 \inf_{0\neq\Ptest \in \FEspaceP} \sup_{0\neq\Utest \in \FEspaceU} 
  \frac{( \diver{}\Utest, \Ptest )}{\|\Ptest\|_{\ltwods} \|\nabla\Utest\|_{\ltwods} } \geq \beta^*,
\end{equation}
with $\beta^*$ independent of the discretization parameter $h$. Specific constructions and examples of finite element spaces satisfying this condition can be found in \cite{Girault,ErnGuermond}. To be able to focus on the fundamental difficulties in the design of an energy stable scheme we will first describe the scheme without being specific about the particular structure of these discrete spaces. As we will see, the choice of discrete spaces shall come naturally from this analysis.

We introduce a discretization of the trilinear form \eqref{eq:defoftrilb} for the convective term in the Navier-Stokes equation:
\[
  \tril_h(\cdot,\cdot, \cdot) : \FEspaceU \times \FEspaceU \times \FEspaceU \to \mathbb{R}.
\]
We assume that $\tril_h$ is skew symmetric with respect to its last two arguments (\cf\cite{Temam}), i.e.
\begin{align}
\label{eq:bhmskewns}
  \tril_h(\bv{U},\bv{V},\bv{W}) = - \tril_h(\bv{U},\bv{W},\bv{V}), \quad
  \forall \, \bv{U},\bv{V},\bv{W} \in \FEspaceU.
\end{align}
Similarly, we also consider a discretization for the trilinear form associated to the convective term of \eqref{systemM} and the Kelvin force in \eqref{systemNS1}:
\begin{align*}
 \trilm(\cdot, \cdot,\cdot) : \FEspaceU \times \FEspaceM \times \FEspaceM \to \mathbb{R} \, , 
\end{align*}
which we also assume to be skew symmetric with respect to its last two arguments
\begin{align}\label{bmskew}
\trilm(\bv{U},\bv{V},\bv{W}) = - \trilm(\bv{U},\bv{W},\bv{V}), \quad
\forall \, \bv{U} \in \FEspaceU, \ \bv{V},\bv{W} \in \FEspaceM \, .
\end{align}
Let $\text{I}_{\FEspacePhase}$, $\text{I}_{\FEspaceChemP}$, $\text{I}_{\FEspaceM}$ and  $\text{I}_{\FEspaceU}$ denote mappings
\begin{align}\label{interpolants} 
\begin{split}
\text{I}_{\FEspacePhase} &:\mathcal{C}^0(\overline{\Omega}) \longrightarrow \FEspacePhase \, ,  \\
\text{I}_{\FEspaceChemP} &:\mathcal{C}^0(\overline{\Omega}) \longrightarrow \FEspaceChemP \, , \\
\text{I}_{\FEspaceM} &:\boldsymbol{\mathcal{C}}^0(\overline{\Omega}) \longrightarrow \FEspaceM \cap \boldsymbol{\mathcal{C}}^0(\overline{\Omega}) \, , \\
\text{I}_{\FEspaceU} &:\boldsymbol{\mathcal{C}}^0(\overline{\Omega}) \longrightarrow \FEspaceU \, ,
\end{split}
\end{align}
with optimal approximation properties
\begin{align}
\label{optestim}
\begin{split}
\|\text{I}_{\mathbb{A}_h}\phasetest - \phasetest\|_{\ltwos}
+ h \|\nabla(\text{I}_{\mathbb{A}_h}\phasetest - \phasetest)\|_{\ltwos} &\leq
c \, h^{\polydegree +1}  |\phasetest|_{H^{\polydegree+1}(\Omega)} \, , \\
\|\text{I}_{\mathbb{A}_h}\phasetest - \phasetest\|_{\linfs}
+ h \|\nabla(\text{I}_{\mathbb{A}_h}\phasetest - \phasetest)\|_{\linfs}
 &\leq 
c \, h^{\polydegree +1} |\phasetest|_{W_{\infty}^{\polydegree+1}(\Omega)} \, ,
\end{split}
\end{align}
where $\mathbb{A}_h$ is either $\FEspacePhase$, $\FEspaceChemP$, $\FEspaceM$ or $\FEspaceU$, and $\ell$ a positive integer. More notation and details about the space discretization will be provided in \S\ref{addnot}. Here we confine ourselves to mention that they can be easily constructed using finite elements (see for instance \cite{Ciar78,ErnGuermond}).

\subsection{Description of the scheme and stability}
\label{sub:descscheme}

For the Cahn-Hilliard equation we shall use the stabilization methodology proposed by Shen and Yang \cite{shenyang2010} in order to eliminate the constraint $\dt \lesssim \layerthick^4$ from the time step. The price paid in this stabilization is the introduction of an error of order $\mathcal{O}(\dt)$ which is consistent with the truncation order of the scheme.

To avoid unnecessary technicalities, we assume that the initial data is smooth and initialize the scheme as follows
\begin{align}\label{initialization}
\phvarh^{0} = \text{I}_{\FEspacePhase}[\phvar(0)]  \, , \
\bv{M}^{0} = \text{I}_{\FEspaceM}[\bv{m}(0)] \, , \
\bv{U}^{0} = \text{I}_{\FEspaceU}[\bv{u}(0)] \, . \
\end{align} 
Then, for every $k\in \{1,\ldots,K\}$ we compute $\{\phvarh^{k},\chpoth^k, \bv{M}^k,\Phi^k, \bv{U}^k,P^k\} \in \FEspacePhase \times \FEspaceChemP \times \FEspaceM \times \FEspacePhi \times 
\FEspaceU \times \FEspaceP$ that solves 
\begin{subequations}
\label{firstscheme}
\begin{align}
\label{Cahn1}
\alp \tfrac{\inc\phvarh^{k}}{\dt},\Phasetest \arp 
- (\bv{U}^{k} \phvarh^{k-1},\nabla\Phasetest)
- \mobility (\nabla\chpoth^k,\nabla\Phasetest) &= 0 \, , \\
\label{Cahn2}
(\chpoth^k, \Chpotest ) 
+ \layerthick (\nabla\phvarh^k,\nabla\Chpotest) 
+ \tfrac{1}{\layerthick} (f(\phvarh^{k-1}),\Chpotest) 
+ \tfrac{1}{\eta} (\inc\phvarh^k,\Chpotest) 
&= 0 \, , \\
\label{Mag1}
\alp \tfrac{\inc\bv{M}^{k}}\dt,\Mtest \arp
- \trilm\lp\bv{U}^{k},\Mtest,\bv{M}^{k}\rp
+ \tfrac{1}{\chartime} \lp\bv{M}^k,\Mtest\rp &=
\tfrac{1}{\chartime} \lp \suscepch \bv{H}^{k}, \Mtest \rp \, , \\
\label{discmagpot}
(\nabla\hdpoth^k,\nabla\Phitest) &= (\ha^k -\bv{M}^k,\nabla\Phitest) \, ,\\
\label{NS1}
\begin{split}
  \alp \tfrac{ \inc\bv{U}^{k}}\dt, \Utest \arp
+ \tril_h\lp\bv{U}^{k-1},\bv{U}^{k},\Utest\rp 
+ \lp \nu_{\phvarh} \bv{T}(\bv{U}^{k}),\bv{T}(\Utest) \rp
- \lp P^{k}, \diver{}\Utest \rp &=
\mu_0 \trilm\lp \Utest, \bv{H}^{k},\bv{M}^{k}\rp
\\
& + \! \tfrac{\capcoeff}{\layerthick} (\phvarh^{k-1} \nabla\chpoth^k,
\Utest)  ,
\end{split} \\
\label{NS2}
\lp \Ptest, \diver{U}^k\rp &= 0 \, ,
\end{align}
\end{subequations}
for all $\lb \Phasetest,\Chpotest,\Mtest,\Phitest,\Utest,\Ptest\rb \in 
\FEspacePhase \times \FEspaceChemP \times \FEspaceM \times \FEspacePhi \times 
\FEspaceU \times \FEspaceP$.
Here the term $\tfrac{1}{\eta} (\inc\phvarh^k,\Chpotest)$ in
\eqref{Cahn2} is a stabilization term with $\eta$ sufficiently small,
we have defined $\bv{H}^\dt = \nabla \hdpoth^\dt$ in analogy to
\eqref{totalhTwPh} and, finally, we have set $\nu_{\phvarh} := \nu( \phvarh^{k-1})$ and 
$\suscepch := \varkappa( \phvarh^{k-1})$.

\begin{remark}[initialization] The initialization proposed in \eqref{initialization} is the simplest choice. From the point of view of convergence to strong solutions (a priori error estimates) it is suboptimal (cf. \cite{Tom13,MR2249024,ErnGuermond,DiPi10}). However, this choice has no effect on the stability of the scheme, it only affects the regularity assumed on the initial data.
\end{remark}

\begin{proposition}[discrete energy stability]\label{discenerglemm00} Let $\lb\phvarh^{\dt}, \chpoth^{\dt}, \bv{M}^{\dt},\hdpoth^{\dt}, \bv{U}^{\dt},P^{\dt} \rb \in \FEspacePhase \times \FEspaceChemP \times \FEspaceM \times \FEspacePhi \times \FEspaceU \times \FEspaceP$ solve \eqref{firstscheme}. If $\nabla \FEspacePhi \subset \FEspaceM$, $\eta \leq \layerthick$ and $\suscep \leq 4$, then we have the following stability estimate
\begin{align}
\label{finaldiscenergy}
\begin{split}
\mathscr{E}(\bv{U}^{\dt} ,\bv{M}^{\dt}, \hdpoth^{\dt}, \phvarh^{\dt};K)
&+ \sum_{k=1}^K \Big( \mathscr{D}_n(\inc\bv{U}^{\dt} ,\inc\bv{M}^{\dt}, \inc\hdpoth^{\dt}, \inc\phvarh^{\dt};k) 
+ \dt \mathscr{D}_p(\bv{U}^{\dt} ,\bv{M}^{\dt}, \hdpoth^{\dt}, \phvarh^{\dt}, \chpoth^{\dt};k) \Big) \\
&\leq  \mathscr{E}(\bv{U}^{\dt} ,\bv{M}^{\dt}, \hdpoth^{\dt},\phvarh^{\dt};0) +
\sum_{k=1}^K \dt \mathscr{F}(\ha;k)  \, ,
\end{split}
\end{align}
where
\begin{align*}
\mathscr{E}(\bv{U}^{\dt} ,\bv{M}^{\dt}, \hdpoth^{\dt}, \phvarh^{\dt};k) &= 
\tfrac{1}{2} \|\bv{U}^{k}\|_{\ltwods}^2
+ \tfrac{\mu_0}{2} \|\bv{M}^{k}\|_{\ltwods}^2
+ \tfrac{\mu_0}{2} \|\nabla\hdpoth^k \|_{\ltwods}^2
+ \tfrac{\capcoeff}{2} \|\nabla\phvarh^k\|_{\ltwods}^2
+ \tfrac{\capcoeff}{\layerthick^2}(F(\phvarh^k),1) \, , \\
\mathscr{D}_n(\inc\bv{U}^{\dt} ,\inc\bv{M}^{\dt}, \inc\hdpoth^{\dt}, \inc\phvarh^{\dt};k) &=
\tfrac{1}{2} \|\inc\bv{U}^{k}\|_{\ltwods}^2
+ \tfrac{\mu_0}{2} \|\inc\bv{M}^{k}\|_{\ltwods}^2
+ \tfrac{\mu_0}{2} \|\inc\nabla\hdpoth^k \|_{\ltwods}^2
+ \tfrac{\capcoeff}{2} \|\inc\nabla\phvarh^k\|_{\ltwods}^2 \, ,  \\
\mathscr{D}_p(\bv{U}^{\dt} ,\bv{M}^{\dt}, \hdpoth^{\dt}, \phvarh^{\dt}, \chpoth^{\dt};k) &=
\tfrac{\mu_0}{\chartime} \big( 1 - \tfrac{\suscep }{4} \big) \|\bv{M}^{k}\|_{\ltwods}^2
+ \tfrac{\mu_0}{2 \chartime} \|\nabla\hdpoth^{k}\|_{\ltwods}^2
+ \|\sqrt{\nu_{\phvar}} \, \bv{T}(\bv{U}^k) \|_{\ltwods}^2
+ \tfrac{\capcoeff \mobility}{\layerthick} \|\nabla \chpoth^k\|_{\ltwods}^2 \, , \\
\mathscr{F}(\ha;k) &= \tfrac{\mu_0}{\chartime}\|\ha^k\|_{\ltwods}^2
+ \tfrac{\mu_0 \chartime}{\dt} \int_{t^{k-1}}^{t^k} \|\partial_t \ha (s)\|_{\ltwods}^2 \, ds \, .
\end{align*}
\end{proposition}

\begin{proof}
We set $\Phasetest = \tfrac{2 \capcoeff \dt}{\layerthick} \chpoth^k $, $\Chpotest = \tfrac{2 \capcoeff}{\layerthick} \inc\phvarh^k$, $\Utest = 2\dt \bv{U}^k$, $\Mtest = 2\dt\mu_0 \bv{M}^k$, $\Mtest = 2\dt\mu_0 \bv{H}^k$ and $\Phitest = \tfrac{2 \mu_0 \dt}{\chartime} \hdpoth^k$ in \eqref{firstscheme}
\begin{subequations}
\label{partial00}
\begin{align}
\label{Cahn101}
\tfrac{2 \capcoeff \mobility \dt}{\layerthick}  \|\nabla\chpoth^k\|_{\ltwods}^2 
- \tfrac{2 \capcoeff}{\layerthick} \alp \inc\phvarh^{k},\chpoth^k \arp &= 
- \tfrac{2 \capcoeff \dt}{\layerthick} (\bv{U}^{k} \phvarh^{k-1},\nabla\chpoth^k) \, , \\
\label{Cahn201}
\tfrac{2 \capcoeff}{\layerthick}  (\chpoth^k, \inc\phvarh^k ) 
+ \tfrac{2 \capcoeff}{\eta\layerthick} \|\inc\phvarh^k\|_{\ltwods}^2 
+ 2 \capcoeff (\inc\nabla\phvarh^k, \nabla\phvarh^k) 
+ \tfrac{2 \capcoeff}{\layerthick^2} (f(\phvarh^{k-1}), \inc\phvarh^k)&= 0 \, , \\
\label{NS101}
\begin{split}
2 \alp \inc\bv{U}^{k}, \bv{U}^{k} \arp
+ 2 \dt \|\sqrt{\nu_{\phvarh}} \, \bv{T}(\bv{U}^{k})\|_{\ltwods}^2
= 2 \mu_0 \dt \trilm\lp \bv{U}^{k}, \bv{H}^{k}, \bv{M}^{k}\rp
&+ \tfrac{2 \capcoeff\dt}{\layerthick} (\phvarh^{k-1}\nabla\chpoth^k, \bv{U}^{k}) 
\end{split} \, , \\
\label{Mag101}
\begin{split}
2 \mu_0 \alp \inc\bv{M}^{k},\bv{M}^{k} \arp
+ \tfrac{2 \mu_0 \dt }{\chartime} \|\bv{M}^k \|_{\ltwods}^2 &= 
\tfrac{2 \mu_0 \dt}{\chartime} \lp \suscepch \bv{H}^{k}, \bv{M}^k \rp  
\end{split} \, , \\
\label{Mag201}
\tfrac{2 \mu_0 \dt}{\chartime} \| \sqrt{\suscepch} \, \bv{H}^{k}\|_{\ltwods}^{2} =
2 \mu_0 \alp \inc\bv{M}^{k},\bv{H}^k \arp 
- 2 \mu_0 \dt \trilm\lp\bv{U}^{k},\bv{H}^k,\bv{M}^{k} \rp
&+ \tfrac{2 \mu_0 \dt }{\chartime} \lp\bv{M}^k,\bv{H}^k\rp \, , \\ 
\tfrac{2 \mu_0 \dt}{\chartime} \|\nabla\hdpoth^k\|_{\ltwods}^2 &= \tfrac{2 \mu_0 \dt}{\chartime} (\ha^k -\bv{M}^k,\nabla\hdpoth^k) \, ;
\end{align}
\end{subequations}
note that $\bv{H}^h = \nabla\Phi^k\in\FEspaceM$ is an admissible
test function. Furthermore, observe
that the convective terms $\tril_h(\bv{U}^{k-1},\bv{U}^k,\bv{U}^k)$
and $\trilm(\bv{U}^{k-1},\bv{M}^k,\bv{M}^k)$ vanish in view of
\eqref{eq:bhmskewns} and \eqref{bmskew} respectively. Adding all the
lines of \eqref{partial00}, and using \eqref{sumid}
for the backward difference operator $\inc$ of
  \eqref{backwarddifferenceop}, we get
\begin{align}\label{discenergy01}
\begin{split}
\delta\|\bv{U}^{k} \|_{\ltwods}^2 
& + \mu_0 \delta \|\bv{M}^{k} \|_{\ltwods}^2 
+ \capcoeff \delta \|\nabla\phvarh^k\|_{\ltwods}^2
+ \|\inc\bv{U}^{k} \|_{\ltwods}^2
+ \mu_0 \|\inc\bv{M}^{k} \|_{\ltwods}^2
+ \capcoeff \|\inc\nabla\phvarh^k\|_{\ltwods}^2
\\
&
+ \tfrac{2 \capcoeff}{\layerthick^2} (f(\phvarh^{k-1}), \inc\phvarh^k)
+ \tfrac{2 \capcoeff}{\eta\layerthick} \|\inc\phvarh^k\|_{\ltwods}^2
+ \tfrac{2 \capcoeff \mobility \dt}{\layerthick} \|\nabla\chpoth^k\|_{\ltwods}^2 \\
& + 2 \dt   \|\sqrt{\nu_{\phvarh}} \, \bv{T}(\bv{U}^{k})\|_{\ltwods}^2 
+ \tfrac{2 \mu_0 \dt }{\chartime} \|\bv{M}^k\|_{\ltwods}^2
+ \tfrac{2 \mu_0 \dt}{\chartime} \| \sqrt{\suscepch} \, \bv{H}^{k}\|_{\ltwods}^{2}
+ \tfrac{2 \mu_0 \dt}{\chartime} \|\nabla\hdpoth^k\|_{\ltwods}^2 \\
& = \tfrac{2 \mu_0 \dt}{\chartime} \lp \suscepch \bv{H}^{k}, \bv{M}^k \rp  
+2 \mu_0 \alp \inc\bv{M}^{k},\nabla\hdpoth^k \arp 
+ \tfrac{2 \mu_0 \dt }{\chartime} (\ha^k,\nabla\hdpoth^k) \, .
\end{split}
\end{align}
Applying the operator $\inc$ to \eqref{discmagpot} and setting $\Phitest = \Phi^k$ yields
\begin{align}
\label{phiID01disc}
  (\inc\nabla\hdpoth^k, \nabla\hdpoth^k) = (\inc\ha^k - \inc\bv{M}^k,
\nabla\hdpoth^k)\, ,
\end{align}
whence \eqref{discenergy01} can be rewritten as
\begin{align}\label{discenergy02}
\begin{split}
\delta\|\bv{U}^{k} \|_{\ltwods}^2  
& + \mu_0 \delta \|\bv{M}^{k} \|_{\ltwods}^2 
+ \mu_0 \delta \|\nabla\hdpoth^k \|_{\ltwods}^2 
+ \capcoeff \delta \|\nabla\phvarh^k\|_{\ltwods}^2 
+ \|\inc\bv{U}^{k} \|_{\ltwods}^2
+ \mu_0 \|\inc\bv{M}^{k} \|_{\ltwods}^2 
\\
&+ \mu_0 \|\inc\nabla\hdpoth^k \|_{\ltwods}^2 
+ \capcoeff \|\inc\nabla\phvarh^k\|_{\ltwods}^2 
+ \tfrac{2 \capcoeff}{\layerthick^2} (f(\phvarh^{k-1}), \inc\phvarh^k)
+ \tfrac{2 \capcoeff}{\eta\layerthick} \|\inc\phvarh^k\|_{\ltwods}^2
+ \tfrac{2 \capcoeff \mobility \dt}{\layerthick} \|\nabla\chpoth^k\|_{\ltwods}^2
\\
&+ 2 \dt   \|\sqrt{\nu_{\phvarh}} \, \bv{T}(\bv{U}^{k})\|_{\ltwods}^2 
+ \tfrac{2 \mu_0 \dt }{\chartime} \|\bv{M}^k\|_{\ltwods}^2
+ \tfrac{2 \mu_0 \dt}{\chartime} \| \sqrt{\suscepch} \, \bv{H}^{k}\|_{\ltwods}^{2}
+ \tfrac{2 \mu_0 \dt}{\chartime} \|\nabla\hdpoth^k\|_{\ltwods}^2 \\
& = \tfrac{2 \mu_0 \dt}{\chartime} \lp \suscepch \bv{H}^{k}, \bv{M}^k \rp
+2 \mu_0 \alp \inc\ha^k,\nabla\hdpoth^k \arp 
+ \tfrac{2 \mu_0 \dt }{\chartime} (\ha^k,\nabla\hdpoth^k) .
\end{split}
\end{align}
In only remains to control the terms $\tfrac{2
  \capcoeff}{\layerthick^2} (f(\phvarh^{k-1}), \inc\phvarh^k)+
\tfrac{2 \capcoeff}{\eta\layerthick}
\|\inc\phvarh^k\|_{\ltwods}^2$. This is a standard argument
\cite{shenyang2010}, which for the sake of completeness we
repeat. In view of Taylor's formula for $F(\phvarh^{k})$ around $F(\phvarh^{k-1})$
\begin{align}\label{stabCH1}
F(\phvarh^{k}) = F(\phvarh^{k-1}) + f(\phvarh^{k-1}) \inc\phvarh^{k} + \tfrac{f'(\xi)}{2} (\inc\phvarh^{k})^2
\end{align}
for some $\xi$, we invoke the bound \eqref{doublewellbound} to get
\begin{align}\label{stabCH2}
(\inc F(\phvarh^{k}),1) \leq (f(\phvarh^{k-1}), \inc\phvarh^{k}) 
+ \|\inc\phvarh^{k} \|_{\ltwods}^2 \, .
\end{align}
Therefore, if we choose $\eta \leq \layerthick$, we can finally estimate \eqref{discenergy02} as follows:
\begin{align}\label{discenergy03}
\begin{split}
\,\inc\mathscr{E}(\bv{U}^{\dt} ,\bv{M}^{\dt}, \hdpoth^{\dt}, \phvarh^{\dt};k)
&+  \mathscr{D}_n(\inc\bv{U}^{\dt} ,\inc\bv{M}^{\dt}, \inc\hdpoth^{\dt}, \inc\phvarh^{\dt};k) 
+  \dt \mathscr{D}_p(\bv{U}^{\dt} ,\bv{M}^{\dt}, \hdpoth^{\dt}, \phvarh^{\dt}, \chpoth^{\dt};k)  \\
&+ \tfrac{\mu_0 \permit \dt}{4\chartime} \|\bv{M}^k\|_{\ltwods}^2
+ \tfrac{\mu_0 \dt}{\chartime} \| \sqrt{\suscepch} \, \bv{H}^{k}\|_{\ltwods}^{2} 
+ \tfrac{\mu_0 \dt}{2\chartime} \|\nabla\hdpoth^k\|_{\ltwods}^2 \\
&\leq \tfrac{ \mu_0 \dt}{\chartime} \lp \suscepch \bv{H}^{k}, \bv{M}^k \rp  
+ \mu_0 \alp \inc\ha^k,\nabla\hdpoth^k \arp 
+ \tfrac{ \mu_0 \dt }{\chartime} (\ha^k,\nabla\hdpoth^k) .
\end{split}
\end{align}
The rest is a matter of bounding the right hand side using
Cauchy-Schwarz and Young's inequalities with appropriate constants. As
in the continuous case, the term $\tfrac{\mu_0 \dt}{\chartime} \lp
\suscepch \bv{H}^{k}, \bv{M}^k \rp$ gives rise to the
  restriction $\suscep \leq 4$. Finally, for the term $\mu_0 \alp
\inc\ha^k,\nabla\hdpoth^k \arp$, we resort to the estimate 
\begin{align*}
\LN\tfrac{\inc\ha^k}{\dt}\RN_{\ltwods}^2 \leq \tfrac{1}{\dt} \int_{t^{k-1}}^{t^k} \|\partial_t \ha(s)\|_{\ltwods}^2 ds .
\end{align*}
This concludes the proof. \end{proof}
Note that \eqref{finaldiscenergy} is consistent with \eqref{enerestTwPh}, except for the term $\mathscr{D}_n$ of time increments (jumps). The latter is a dissipative term characteristic of the implicit Euler scheme.
\subsection{Existence of fixed points: local solvability}
\label{Sexistence}
Let us establish the local solvability of the numerical scheme \eqref{firstscheme}. To do so we will make use of the well-known Leray-Schauder fixed point theorem \cite[p. 280]{GT2001}. As it is usual when invoking this result, the core of the proof is a local (in time) a priori estimate, which very much resembles the arguments of \S\ref{sub:descscheme}. Therefore, a few intermediate steps have been eliminated leaving the details to the reader.

\begin{theorem}[existence]
\label{thm:existence}
Let $h,\dt > 0$, and let the parameters $\suscep$ and $\eta$ satisfy $\suscep \leq 4$ and $\eta \leq \layerthick$. If $\nabla\FEspacePhi \subset \FEspaceM$, then for all $k=1,\ldots,K$ there exists $\lb\phvarh^k,\chpoth^k,\bv{M}^k,\hdpoth^k,\bv{U}^k,P^k \rb$ $\in$ $\FEspacePhase \times \FEspaceChemP \times \FEspaceM \times  \FEspacePhi \times \FEspaceU \times \FEspaceP$
that solves \eqref{firstscheme}. Moreover, any such sequence of solutions satisfies estimate \eqref{finaldiscenergy}.
\end{theorem}
\begin{proof} We define the affine map $\widehat{x} = \mathcal{L} x$,
\begin{align*}
\lb\phvarh^k,\chpoth^k,\bv{M}^k,\hdpoth^k,\bv{U}^k,P^k \rb \xmapsto{\mathcal{L}} 
\lb \widehat{\phvarh}^k,\widehat{\chpoth}^k,\widehat{\bv{M}}^k,\widehat{\hdpoth}^k, \widehat{\bv{U}}^k,\widehat{P}^k \rb \, , 
\end{align*}
where the quantities with hats solve of the following variational problem:
\begin{subequations}
\label{}
\begin{align}
\label{Cahn1schau}
\alp \tfrac{\widehat{\phvarh}^{k} - \phvarh^{k-1}}{\dt},\Phasetest \arp 
- (\bv{U}^{k} \phvarh^{k-1},\nabla\Phasetest)
- \mobility (\nabla\widehat{\chpoth}^k,\nabla\Phasetest) &= 0 \, , \\
\label{Cahn2schau}
(\widehat{\chpoth}^k, \Chpotest ) 
+ \tfrac{1}{\eta} (\widehat{\phvarh}^k - \phvarh^{k-1},\Chpotest) + \layerthick (\nabla\widehat{\phvarh}^k,\nabla\Chpotest) 
+ \tfrac{1}{\layerthick} (f(\phvarh^{k-1}),\Chpotest)&= 0 \, , \\
\label{MagnSchau1}
\begin{split}
\alp \tfrac{\widehat{\bv{M}}^{k} - \bv{M}^{k-1}}\dt,\Mtest \arp 
- \trilm\lp\bv{U}^{k},\Mtest,\widehat{\bv{M}}^{k}\rp
+ \tfrac{1}{\chartime} \lp\widehat{\bv{M}}^k,\Mtest\rp &= 
\tfrac{1}{\chartime} \lp \suscepch \widehat{\bv{H}}^{k}, \Mtest \rp \, ,
\end{split} \\
\label{MagnSchau2}
(\nabla\widehat{\hdpoth}^k,\nabla\Phitest) &= (\ha^k -\widehat{\bv{M}}^k,\nabla\Phitest) \, , \\
\label{NS1schau}
\begin{split}
\alp \tfrac{\widehat{\bv{U}}^{k} - \bv{U}^{k-1} }\dt, \Utest \arp
+ \lp \nu_{\phvarh} \bv{T}(\widehat{\bv{U}}^{k}),\bv{T}(\Utest) \rp + 
\tril_h\lp\bv{U}^{k-1},\widehat{\bv{U}}^{k},\Utest\rp - \lp \widehat{P}^{k}, \diver{}\Utest \rp 
&= \trilm\lp \Utest, \bv{H}^{k},\widehat{\bv{M}}^{k}\rp \\
& + \tfrac{\capcoeff}{\layerthick} (\phvarh^{k-1} \nabla\widehat{\chpoth}^k  , \Utest) \, ,
\end{split}
\\
\label{NS2schau}
  \lp \diver{} \widehat{\bv{U}}^k, \Ptest \rp & = 0 \, ,
\end{align}
\end{subequations}
for all $\lb \Phasetest,\Chpotest,\Mtest,\Phitest,\Utest,\Ptest\rb \in 
\FEspacePhase \times \FEspaceChemP \times \FEspaceM \times \FEspacePhi \times 
\FEspaceU \times \FEspaceP$. 
As always, we set $\widehat{\bv{H}}^\dt = \nabla
  \widehat{\Phi}^\dt$. Notice that a fixed point of the map
$\mathcal L$ is a solution to \eqref{firstscheme}. Let us then show
that $ \mathcal{L} $ verifies the assumptions of the Leray-Schauder
theorem, namely well-posedness, boundedness, and compactness:
\begin{enumerate}[\itemizebullet]
\setlength{\itemsep}{2pt} 
\setlength{\parskip}{0cm}
\item
\textbf{Well-posedness.} The operator $\mathcal{L}$ is clearly well 
defined. The information follows a top-down path. For a given velocity $\bv{U}^k$ the mixed Cahn-Hilliard system \eqref{Cahn1schau}--\eqref{Cahn2schau} is well-defined and widely-studied system which can be reduced to a single positive definite system in terms of the phase (see for instance \cite{ShenReview} and references therein). System \eqref{MagnSchau1}--\eqref{MagnSchau2} can be rewritten as follows:
\begin{align}\label{magproblem}
\begin{split}
\alp \widehat{\bv{M}}^{k},\Mtest \arp 
- \dt \trilm\lp\bv{U}^{k},\Mtest,\widehat{\bv{M}}^{k}\rp
+ \tfrac{\dt}{\chartime} \lp\widehat{\bv{M}}^k,\Mtest\rp
- \tfrac{\dt}{\chartime} \lp \suscepch \widehat{\bv{H}}^{k}, \Mtest \rp
&= \alp \bv{M}^{k-1},\Mtest \arp \\
(\widehat{\bv{M}}^k,\nabla\Phitest) + ( \nabla \widehat{\hdpoth}^k,\nabla\Phitest)
&= (\ha^k,\nabla\Phitest) \, .
\end{split}
\end{align}
Multiply the second line by $\tfrac{\dt \permit}{\chartime}$ and add both lines.
Taking $\Mtest = \widehat{\bv{M}}^{k}$ and $\Phitest = \widehat{\hdpoth}^k$,
and using that $\widehat{\bv{H}}^k=\nabla\widehat{\Phi}^k$, one verifies that the
associated bilinear form is coercive, provided $\permit \leq 4$.
Once the magnetization problem is solved, we will have the functions $\widehat{\bv{M}}^k$ and $\widehat{\hdpoth}^k$ which can be used as data for the Stokes problem \eqref{NS1schau}--\eqref{NS2schau}, which is also well posed.
\item
\textbf{Boundedness}. Given $\alpha \in [0,1]$, we must verify that all $\widehat{x} = \lb\widehat{\phvarh}^k,\widehat{\chpoth}^k,\widehat{\bv{M}}^k,\widehat{\hdpoth}^k,\widehat{\bv{U}}^k,\widehat{P}^k \rb$ satisfying $\tfrac{1}{\alpha} \widehat{x} = \mathcal{L}\widehat{x}$ can be bounded in terms of the local data 
$\lb \phvarh^{k-1},\bv{M}^{k-1},\hdpoth^{k-1},\bv{U}^{k-1},\ha^{k}\rb$ uniformly with respect to $\alpha$. In other words, we want to analyze the local boundedness of solutions to the system
\begin{align}
\label{IterationScheme}
\begin{aligned}
\alp  \tfrac{\widehat{\phvarh}^{k} - \alpha \phvarh^{k-1}}{\dt}, \Phasetest \arp 
- (\widehat{\bv{U}}^{k} \alpha \, \phvarh^{k-1},\nabla\Phasetest)
- \mobility ( \nabla\widehat{\chpoth}^k,\nabla\Phasetest) &= 0 \, , \\
(\widehat{\chpoth}^k, \Chpotest ) 
+ \tfrac{1}{\eta} ( \widehat{\phvarh}^k,\Chpotest) + \layerthick ( \nabla\widehat{\phvarh}^k,\nabla\Chpotest) 
+ \tfrac{\alpha}{\layerthick} (f(\phvarh^{k-1}),\Chpotest)&= 
 \tfrac{\alpha}{\eta} (\phvarh^{k-1},\Chpotest) \, , \\
\alp \tfrac{\widehat{\bv{M}}^{k} - \alpha\bv{M}^{k-1}}\dt,\Mtest \arp
- \trilm\lp\widehat{\bv{U}}^{k},\Mtest, \widehat{\bv{M}}^{k}\rp
+ \tfrac{1}{\chartime} \lp \widehat{\bv{M}}^k,\Mtest\rp &= 
\tfrac{1}{\chartime} \lp \suscepch \, \widehat{\bv{H}}^{k}, \Mtest \rp  \, , \\
(\nabla\widehat{\hdpoth}^k,\nabla\Phitest) &= (\alpha \, \ha^k - \widehat{\bv{M}}^k,\nabla\Phitest) \, , \\
\alp \tfrac{\widehat{\bv{U}}^{k} - \alpha \,\bv{U}^{k-1} }\dt, \Utest \arp
+ \lp \nu_{\phvarh} \, \bv{T}(\widehat{\bv{U}}^{k}),\bv{T}(\Utest)\rp
+  \tril_h\lp\bv{U}^{k-1}, \widehat{\bv{U}}^{k},\Utest\rp 
- \lp \widehat{P}^{k}, \diver{}\Utest \rp & = \mu_0 \trilm\lp \Utest,
\widehat{\bv{H}}^{k},\widehat{\bv{M}}^{k}\rp \\
&+ \tfrac{\capcoeff}{\layerthick} (\phvarh^{k-1} \nabla\widehat{\chpoth}^k , \Utest) \, , \\
\lp \diver{} \widehat{\bv{U}}^k, \Ptest \rp & = 0 \, . 
\end{aligned}
\end{align}
We now proceed exactly as in Proposition \ref{discenerglemm00}
(discrete energy stability). We
set $\Phasetest = \tfrac{2 \capcoeff \dt}{\layerthick} \widehat\chpoth^k $,
$\Chpotest = \tfrac{2 \capcoeff}{\layerthick} \delta_\alpha\widehat{\phvarh}^k$,
$\Utest = 2 \dt \widehat{\bv{U}}^k$,
$\Mtest = 2 \dt\mu_0 \widehat{\bv{M}}^k$,
$\Mtest = 2 \dt\mu_0 \widehat{\bv{H}}^k$,
and  $\Phitest = \tfrac{2 \mu_0 \dt}{\chartime} \widehat{\hdpoth}^k$
in \eqref{IterationScheme}, where the operator $\delta_\alpha$ is
defined to be
\[
\delta_\alpha \widehat{\rho}^k = \widehat{\rho}^k - \alpha \rho^{k-1}.
\]
Applying the identity \eqref{sumid}, we obtain
the following variant of \eqref{discenergy01}
\begin{align}\label{iterativeenergy2}
\begin{split}
\delta_\alpha \|\widehat{\bv{U}}^{k} \|_{\ltwods}^2
& + \mu_0 \delta_\alpha \|\widehat{\bv{M}}^{k} \|_{\ltwods}^2 
+ \capcoeff \delta_\alpha \|\nabla\widehat{\phvarh}^k\|_{\ltwods}^2
+ \|\inc_\alpha \widehat{\bv{U}}^{k} \|_{\ltwods}^2
+ \mu_0 \|\inc_\alpha \widehat{\bv{M}}^{k} \|_{\ltwods}^2
+ \capcoeff \|\inc_\alpha \nabla\widehat{\phvarh}^k\|_{\ltwods}^2
\\
&
+ \tfrac{2 \alpha \capcoeff}{\layerthick^2}
(f(\phvarh^{k-1}), \inc_\alpha \widehat{\phvarh}^k)
+ \tfrac{2 \capcoeff}{\eta\layerthick} \|\inc_\alpha \widehat{\phvarh}^k\|_{\ltwods}^2
+ \tfrac{2 \capcoeff \mobility \dt}{\layerthick}
\|\nabla\widehat{\chpoth}^k\|_{\ltwods}^2 \\
& + 2 \dt   \|\sqrt{\nu_{\phvarh}} \, \bv{T}(\widehat{\bv{U}}^{k})\|_{\ltwods}^2 
+ \tfrac{2 \mu_0 \dt }{\chartime} \|\widehat{\bv{M}}^k\|_{\ltwods}^2
+ \tfrac{2 \mu_0 \dt}{\chartime} \| \sqrt{\suscepch} \,
\widehat{\bv{H}}^{k}\|_{\ltwods}^{2}
+ \tfrac{2 \mu_0 \dt}{\chartime} \|\nabla\widehat{\hdpoth}^k\|_{\ltwods}^2 \\
& = \tfrac{2 \mu_0 \dt}{\chartime} \lp \suscepch
\widehat{\bv{H}}^{k}, \widehat{\bv{M}}^k \rp  
+2 \mu_0 \alp \inc_\alpha \widehat{\bv{M}}^{k},\nabla\widehat{\hdpoth}^k \arp 
+ \tfrac{2 \mu_0 \alpha \dt }{\chartime} (\ha^k,\nabla\widehat{\hdpoth}^k) \, .
\end{split}
\end{align}
The argument now differs from that in Proposition
\ref{discenerglemm00} because we cannot apply $\delta_\alpha$ to the
fourth equation in \eqref{IterationScheme}. Instead, setting
$\Phitest = \alpha \widehat{\hdpoth}^k$ yields
\begin{align*}
\|\nabla\widehat{\hdpoth}^k \|_{\ltwods}^2 &=  (\alpha \ha^k - \widehat{\bv{M}}^k, \nabla\widehat{\hdpoth}^k) \, , 
\end{align*}
whence we can rewrite \eqref{iterativeenergy2} as follows:
\begin{align*}
\begin{split}
\delta_\alpha \|\widehat{\bv{U}}^{k} \|_{\ltwods}^2
& + \mu_0 \delta_\alpha \|\widehat{\bv{M}}^{k} \|_{\ltwods}^2 
+ \capcoeff \delta_\alpha \|\nabla\widehat{\phvarh}^k\|_{\ltwods}^2
+ \|\inc_\alpha \widehat{\bv{U}}^{k} \|_{\ltwods}^2
+ \mu_0 \|\inc_\alpha \widehat{\bv{M}}^{k} \|_{\ltwods}^2
+ \capcoeff \|\inc_\alpha \nabla\widehat{\phvarh}^k\|_{\ltwods}^2
\\
&
+ \tfrac{2 \alpha \capcoeff}{\layerthick^2}
(f(\phvarh^{k-1}), \inc_\alpha \widehat{\phvarh}^k)
+ \tfrac{2 \capcoeff}{\eta\layerthick} \|\inc_\alpha \widehat{\phvarh}^k\|_{\ltwods}^2
+ \tfrac{2 \capcoeff \mobility \dt}{\layerthick}
\|\nabla\widehat{\chpoth}^k\|_{\ltwods}^2 \\
& + 2 \dt   \|\sqrt{\nu_{\phvarh}} \, \bv{T}(\widehat{\bv{U}}^{k})\|_{\ltwods}^2 
+ \tfrac{2 \mu_0 \dt }{\chartime} \|\widehat{\bv{M}}^k\|_{\ltwods}^2
+ \tfrac{2 \mu_0 \dt}{\chartime} \| \sqrt{\suscepch} \,
\widehat{\bv{H}}^{k}\|_{\ltwods}^{2}
+ 2 \mu_0\big(1+\tfrac{ \dt}{\chartime}\big) \|\nabla\widehat{\hdpoth}^k\|_{\ltwods}^2 \\
& = \tfrac{2 \mu_0 \dt}{\chartime} \lp \suscepch
\widehat{\bv{H}}^{k}, \widehat{\bv{M}}^k \rp  
- 2 \mu_0 \alpha \alp \bv{M}^{k-1},\nabla\widehat{\hdpoth}^k \arp 
+ 2 \mu_0\alpha\big(1+\tfrac{ \dt }{\chartime}\big) (\ha^k,\nabla\widehat{\hdpoth}^k) \, .
\end{split}
\end{align*}
To conclude it remains to bound the right hand side using Cauchy-Schwarz and Young's inequalities with appropriate constants, use that $\alpha \leq 1$, and the bound $\nu_{\phvarh} \geq \nu_w$ of \eqref{coeffbounds}. 

\smallskip
\item[\itemizebullet] \textbf{Compactness: }this is automatically satisfied since we are working with finite dimensional spaces.
\end{enumerate}
Finally, we apply the Leray-Schauder's theorem to prove the assertion. \end{proof}
\subsection{Practical space discretization}
\label{addnot}
Having understood what is required from a Galerkin technique to
achieve stability of the scheme \eqref{firstscheme} we now specify our
choices of discrete spaces using finite elements. We assume that
$\Omega$ is convex and $\Gamma$ is polyhedral, and that we have at
hand a quasi-uniform triangulation $\triangulation = \{T\}$ of
$\Omega$ with meshsize $h$. As Proposition~\ref{discenerglemm00}
shows, to gain stability it is instrumental to have $\nabla
\FEspacePhi \subset \FEspaceM$. Since the space $\FEspacePhi$ is used
to approximate the solution of an elliptic problem with Neumann
boundary conditions, the simplest choice for $\FEspacePhi$ is a space
of continuous piecewise polynomials of degree $\polydegree\ge1$
\begin{align}
\label{choice1}
\FEspacePhi := \lb \Phitest \in \mathcal{C}^0\bigl(\overline\Omega\bigr) \ | \ \Phitest|_{T} \in \simplex_{\polydegree}(\element) \, , \forall \element \in \triangulation \rb  \subset \hone .
\end{align}
Therefore, to achieve $\nabla\FEspacePhi \subset \FEspaceM$ we set
\begin{align}
\label{choice2}
\FEspaceM := \lb \bv{M} \in \ltwods(\Omega) \ | \ \bv{M}|_{T} \in [\simplex_{\polydegree-1}(\element)]^d \, , \forall \, \element \in \triangulation \rb \, ,
\end{align}
that is $\FEspaceM$ must be a space of discontinuous piecewise
polynomials of degree $\le\polydegree-1$. Consequently, the trilinear form $\trilm(\cdot, \cdot, \cdot)$ must be defined accordingly, namely
\begin{align}
\label{trilineardef}
\trilm(\bv{U},\bv{V},\bv{W})
= \sum_{\element \in \triangulation} \bulkint{\element}{ (\bv{U} \cdot \nabla )\bv{V} \cdot \bv{W}
    + \tfrac{1}{2} \diver{U} \, \bv{V}\cdot\bv{W} } 
-  \sum_{F \in \mathcal{F}^i}
\bdryint{F}{ ( \lj\bv{V}\rj \cdot \lbb\bv{W}\rbb)(\bv{U} \cdot \normal_F) } \, , 
\end{align}
where $F$ denotes an element face, $\mathcal{F}^i$ is the set of all internal faces (faces which are not part of the boundary $\bdry$), and $\normal_F$ is the normal of the face $\mathcal{F}$ which can be chosen arbitrarily. 
The bulk integrals in \eqref{trilineardef} are the classical Temam
\cite{Temam} modification of the convective term
\eqref{eq:defoftrilb}, whereas the face integrals are
consistency terms. This discretization of convection for discontinuous
finite elements dates back to \cite{DiErn2012,DiPi10,Gir2005,Les1974}. From these references it is also known that $\trilm(\cdot,\cdot, \cdot)$ is skew symmetric, that is \eqref{bmskew} holds, provided the first argument $\bv{U} \in \hdiv$ and $\bv{U}\cdot \normal = 0$ on $\bdry$.

The choice of the remaining spaces is now straightforward: for $\polydegree \geq 2$ we set
\begin{align}
\label{FEspaces}
\begin{aligned}
\FEspacePhase & := \lb \Phasetest \in \mathcal{C}^0\bigl(\overline\Omega\bigr) \ | \ \Phasetest|_{T} \in \simplex_{\polydegree}(\element) \quad \forall \, \element \in \triangulation \rb \, , \\
\FEspaceChemP & := \lb \Chpotest \in \mathcal{C}^0\bigl(\overline\Omega\bigr) \ | \ \Chpotest|_{T} \in \simplex_{\polydegree}(\element) \quad \forall \, \element \in \triangulation \rb \, , \\
\FEspaceU & := \lb \Utest \in \bm{\mathcal{C}}^0\bigl(\overline\Omega\bigr) \ |
\ \Utest|_{T} \in [\simplex_{\polydegree}(\element)]^d \quad \forall \, \element \in \triangulation,  \rb \cap \hzerod \, , \\
\FEspaceP & := \lb \Ptest \in \lzerotwo \ |
\ \Ptest|_{T} \in \simplex_{{\polydegree-1}(\element)} \quad \forall \, \element \in \triangulation \rb \, .\\
\end{aligned}
\end{align}
The finite element spaces $\FEspacePhase$, $\FEspaceChemP$, $\FEspaceM$, $\FEspacePhi$, $\FEspaceU$, and $\FEspaceP$ are defined using polynomial spaces $\simplex_{\polydegree}$, of total degree at most $\polydegree$, usually associated to simplicial elements. However, the fact that the scheme \eqref{firstscheme} is energy stable is independent of whether we choose simplices or quadrilaterals/hexahedrons. If we replace $\simplex_{\polydegree}$ by $\quadrilateral_{\polydegree}$ (polynomials of degree at most $\polydegree$ in each variable) in \eqref{choice1}, \eqref{choice2} and \eqref{FEspaces}, we would only need to do minor changes in the choice of polynomial degrees in order to guarantee that the inclusion $\nabla \FEspacePhi \subset \FEspaceM$ holds true. To simplify our exposition we will always assume that our elements are simplicial and develop our theory under this assumption. We will provide remarks describing the required modifications if quadrilaterals are to be used.

We further assume that the pair $( \FEspacePhase,\FEspaceChemP )$ is inf-sup stable for a mixed discretization of the bilaplacian.
It is well known that using equal-order polynomial spaces for the pair $( \FEspacePhase,\FEspaceChemP )$ (as in \eqref{FEspaces}) is enough to satisfy this condition (\cf \cite{Boff2013}).

Some additional definitions follow:
\begin{enumerate}[\itemizebullet]
\item
Let $\Vspace$ and $\Hspace$ denote the classical function spaces of
divergence-free functions
\begin{equation}\label{spaces:VH}
\begin{aligned}
\Vspace &= \lb \utest \in \hzerod \, | \, \diver{}\utest = 0 \, \text{in } \Omega \rb \\
\Hspace &= \lb \utest \in \ltwod \, | \, \diver{}\utest = 0  \text{ in } \Omega \text{ and } \utest \cdot \normal = 0 \text{ on } \bdry\rb
\end{aligned} 
\end{equation}

\item
We denote by $\FEspaceUdivfree$ the space of discretely divergence-free functions:
\begin{align*}
\FEspaceUdivfree = \lb \Utest \in \FEspaceU \ |
\ (\Ptest,\diver{}\Utest) = 0 \, \, \forall \, \Ptest \in \FEspaceP \rb 
\end{align*} 

\item
Let $\Pi_{\FEspaceUdivfree}:\ltwod \longrightarrow \FEspaceUdivfree$ denote the $\ltwo$ projection operator satisfying:
\begin{align}\label{ltwodef}
(\Pi_{\FEspaceUdivfree}\bv{v},\Utest) = (\bv{v},\Utest) \ \ \forall \, \Utest \in \FEspaceUdivfree
\end{align}

\item
Similarly, we denote $\Pi_{\FEspacePhase}:\ltwo \longrightarrow \FEspacePhase$ the $\ltwo$ projection operator onto the space $\FEspacePhase$. 

\item
We define the Stokes projection
$(\SPvel\bv{w},\SPpress r) \in \FEspaceU \times \FEspaceP$ of
$(\bv{w},r)\in \hzerod\times\ltwod$ as the 
pair that solves
\begin{align}
\label{stokesproj}
\left\{
\begin{aligned}
( \gradv{}\SPvel\bv{w}, \nabla\Utest ) - ( \SPpress r, \diver{}\Utest ) &= 
( \gradv{w} , \nabla\Ptest ) - ( r, \diver{}\Utest ) &&\forall \Utest \in \FEspaceU \\
( \Ptest, \diver{}\SPvel\bv{w} ) &= ( \Ptest, \diver{}\bv{W} ) &&\forall \Ptest \in \FEspaceP  \, .
\end{aligned}
\right.
\end{align}
which has the following well-known approximation properties (see for instance \cite{Girault,HeyRann})
\begin{align}\label{approximab}
\|\bv{w} -\SPvel\bv{w} \|_{\ltwods}
+ h \|\bv{w} -\SPvel\bv{w} \|_{\honeds}
+ h \|r - \SPpress r \|_{\ltwos} &
\leq c \, h^{\polydegree+1} \,
\big(\|\bv{w}\|_{\bv{H}^{\polydegree+1}}
+ \|r\|_{H^{\polydegree}} \big) \, , 
\end{align}
for all $(\bv{w},r) \in \bv{H}^{\polydegree+1}(\Omega)\times H^{\polydegree}(\Omega)$, with $c$ independent of $h$, $\bv{w}$ and $r$.
\end{enumerate}
%
%
%
%
%
%
%
%
%
%
%
%
%
\section{Simplification of the model}
\label{Ssimplification}
We can simplify model \eqref{Themodel} by eliminating the magnetostatics problem \eqref{systemPhi} and setting the effective magnetizing field to be $\heff := \ha$. The purpose of this section is to explain, at least with an heuristic argument, under which circumstances this is a reasonable physical approximation. The variational formulation for the magnetostatics problem \eqref{phiNeuIIItwo} is
\begin{align}
\label{varforphi05}
\int_{\Omega} \nabla\hdpot \cdot \nabla \phitest \, dx 
=  \int_{\Omega} ( \ha - \bv{m} )\cdot \nabla \phitest  \, dx \ \ \ \forall \, \phitest \in \hone / \mathbb{R} \, ,
\end{align}
and $\heff=\nabla\hdpot$ according to \eqref{totalhTwPh}.
We realize that $\heff \approx \ha$ provided the magnetization
$\bv{m}$ is small relative to $\ha$.
On the other hand, as explained in Section~\ref{Sderivation}, the
evolution of the magnetization $\bv{m}$ close to equilibrium is
dictated by \eqref{functional} and \eqref{magderiv}, which imply
\begin{align*}
\|\bv{m} - \suscepch \heff\|_{\ltwods} \approx 0 \, \ \ \Longrightarrow  \ \ 
\bv{m} \approx \suscepch \heff.
\end{align*}
It thus becomes clear that if $\suscepch < < 1$ the magnetization
$\bv{m}$ is small, and we can neglect its contribution
in \eqref{varforphi05}.

Therefore, a simplification of the model \eqref{Themodel} is to discard the Poisson problem \eqref{systemPhi} and set $\heff := \ha$. This will only be physically realistic for ferrofluids with a small susceptibility $\suscep$. Water based ferrofluids subject to slowly varying magnetic fields (and/or small characteristic times $\chartime$) could be modeled under these assumptions, since they usually exhibit a small magnetic susceptibility in the low frequency regime \cite{Pak2006,FerrotecWebpage}. It is worth mentioning that the simplification $\heff = \ha$ is not particularly new: it has been used for analytic computations of the Rosensweig model and still retains a significant amount of valid quantitative information as shown for instance in \cite{Rinal02,Zahn95,Sunil2007}; it has also been suggested in the analysis of stationary configurations of free surfaces of ferrofluids \cite{Tob2006}.

We consider the following ultra-weak formulation for the model defined
by equations \eqref{systemCH1}, \eqref{systemCH2}, \eqref{systemM},
\eqref{systemNS1},  \eqref{systemNS2}
upon integrating by parts the time derivatives and dealing with
test functions that vanish at $t = \tf$:
find $(\phvar,\chpot,\bv{m},\bv{u})\in L^2(([0,\tf);\hone) \times L^2(([0,\tf);\hone) \times L^2(([0,\tf);\ltwod) \times L^2(([0,\tf);\Vspace) $ that satisfy
\begin{subequations}
\label{weakfor}
\begin{align}
\label{SsystemCH1}
- \int_{0}^{\tf}( \phvar,\phasetest_t ) + (\bv{u} \phvar,\nabla\phasetest) + \mobility (\nabla\chpot,\nabla\phasetest) &= (\phvar(0),\phasetest(0) ) \, ,  \\
\label{SsystemCH2} 
\int_{0}^{\tf} \layerthick (\nabla\phvar,\nabla\chpotest)
+ \tfrac{1}{\layerthick} (f(\phvar),\chpotest )
+ (\chpot,\chpotest) &= 0 \, , \\
\label{SsystemM}
- \int_{0}^{\tf} (\bv{m},\mtest_t) + \tril(\bv{u},\mtest,\bv{m})
- \tfrac{1}{\chartime} (\bv{m},\mtest) &= 
(\bv{m}(0),\mtest(0)) 
+ \tfrac{1}{\chartime} \int_{0}^{\tf} (\suscepc \heff,\mtest) \, ,  \\  %
\label{SsystemNS1} %
\int_{0}^{\tf} - ( \bv{u},\utest_t ) + \tril(\bv{u},\bv{u},\utest) +( \nu_{\phvar} \, \bv{T}(\bv{u}),\bv{T}(\utest)) &=
( \bv{u}(0),\utest(0) )
+ \int_{0}^{\tf} \mu_0 \tril( \bv{m} ,\heff, \utest) 
+ \tfrac{\capcoeff}{\layerthick} (\phvar\nabla\chpot,\utest) \, ,
\end{align}
\end{subequations}
for all $\phasetest,\chpotest \in \mathcal{C}^{\infty}_0([0,\tf) \times \Omega)$, $\mtest \in \boldsymbol{\mathcal{C}}^\infty_0([0,\tf) \times \Omega)$, $\utest$ $\in$ $\{ \bv{w}\in \boldsymbol{\mathcal{C}}^\infty_0([0,\tf) \times \Omega)$ $\, | \, \diver{}\bv{w} = 0 $ $\text{ in } \Omega \, \}$, where now the magnetic field $\heff$ is not determined by the Poisson problem \eqref{systemPhi}, but rather $\heff = \ha$ is a given harmonic (curl-free and div-free) smooth vector field.

\subsection{A convergent scheme}
\label{Sconvscheme}
We first initialize the system \eqref{weakfor} as in \eqref{initialization}.
We do not discretize \eqref{weakfor} directly but rather
we compute $\{\phvarh^{k},\chpoth^k, \bv{M}^k,\bv{U}^k,P^k\} \in \FEspacePhase \times \FEspaceChemP \times \FEspaceM \times 
\FEspaceU \times \FEspaceP$ for every $k\in \{1,\ldots,K\}$ that solves
\begin{subequations}
\label{convsystem}
\begin{align}
\label{SCahn1}
\alp \tfrac{\inc\phvarh^{k}}{\dt},\Phasetest \arp 
- (\bv{U}^{k} \phvarh^{k-1},\nabla\Phasetest)
- \mobility (\nabla\chpoth^k,\nabla\Phasetest) &= 0 \, , \\
\label{SCahn2}
(\chpoth^k, \Chpotest ) 
+ \layerthick (\nabla\phvarh^k,\nabla\Chpotest) 
+ \tfrac{1}{\layerthick} (f(\phvarh^{k-1}),\Chpotest)
+ \tfrac{1}{\eta} (\inc\phvarh^k,\Chpotest)
&= 0 \, , \\
\begin{split}\label{SMag1}
\alp \tfrac{\inc\bv{M}^{k}}\dt,\Mtest \arp 
- \trilm\lp \bv{U}^{k},\Mtest,\bv{M}^{k}\rp
+ \tfrac{1}{\chartime} \lp\bv{M}^k,\Mtest\rp &= 
\tfrac{1}{\chartime} \lp \suscepch \bv{H}^{k}, \Mtest \rp \, ,
\end{split} \\
\begin{split}\label{SNS1}
\alp \tfrac{ \inc\bv{U}^{k}}\dt, \Utest \arp
+ \tril_h\lp\bv{U}^{k-1},\bv{U}^{k},\Utest\rp
+ \lp \nu_{\phvarh} \bv{T}(\bv{U}^{k}),\bv{T}(\Utest) \rp - \lp P^{k}, \diver{}\Utest \rp & =  \mu_0 \trilm\lp \Utest, \bv{H}^{k}, \bv{M}^{k}\rp
\\
&+ \tfrac{\capcoeff}{\layerthick} (\phvarh^{k-1} \nabla\chpoth^k, \Utest) \, ,
\end{split} \\
\label{SNS2}
\lp \Ptest, \diver{U}^k\rp &= 0 \, ,
\end{align}
\end{subequations}
for all $\lb \Phasetest,\Chpotest,\Mtest,\Utest,\Ptest\rb \in 
\FEspacePhase \times \FEspaceChemP \times \FEspaceM \times \FEspaceU
\times \FEspaceP$. As before, we set $\nu_{\phvarh} = \nu( \phvarh^{k-1})$ and 
$\suscepch = \varkappa( \phvarh^{k-1})$. The magnetic field $\bv{H}^k$ is given by 
\begin{align}\label{newHdef}
\bv{H}^k := \text{I}_{\FEspaceM}[\ha^k] \, ,
\end{align}
where $\text{I}_{\FEspaceM}$ was defined in \eqref{interpolants}. The latter implies that $\bv{H}^k$ does not jump across interelement boundaries.

The choice of spaces $\FEspacePhase$, $\FEspaceChemP$, $\FEspaceM$, $\FEspaceU$ and $\FEspaceP$ does need to be made precise now, we will provide a specific construction in Remark \ref{choicespacessec}. Right now we only need to say that, in order for \eqref{convsystem} to be convergent (in addition to the requirement that the spaces $\FEspacePhase$ and $\FEspaceChemP$ are stable for a mixed discretization of the bilaplacian, and the LBB compatibility condition \eqref{discreteinfsup} for the spaces $\lb \FEspaceU, \FEspaceP \rb$) we will also require that:
\begin{enumerate}
\item[(A1).] The pressure space $\FEspaceP$ should be made of discontinuous elements, and it should contain a continuous subspace of degree 1 or higher, such that $\FEspaceP \cap \mathcal{C}^{0}(\overline\Omega) \neq \emptyset$.
\item[(A2).] For all $\Mtest \in \FEspaceM$, we want each space component $\Mtest^i$ ($i=1, \cdots , d$) to belong to the same finite element space as the pressure, i.e. we require $\FEspaceM = [\FEspaceP]^d$.
\item[(A3).] The Stokes projector $\SPvel$ (defined in \eqref{stokesproj}) should guarantee strong convergence in $\bv{W}_{\!\infty}^1$-norm for smooth functions. More precisely, we require that
\begin{align}\label{LinfConvAss}
\|\nabla(\utest -\SPvel\utest) \|_{\linfds} \xrightarrow[]{h \rightarrow 0} \utest \ \ \forall \utest \text{ in }\boldsymbol{\mathcal{C}}_0^{\infty}(\Omega) \, .
\end{align}
Property \eqref{LinfConvAss} cannot be taken for granted for any arbitrary construction (i.e. choice of finite element spaces $\lb \FEspaceU,\FEspaceP\rb$ and mesh $\triangulation$). For $d= 2$, a partial list of finite element pairs $\lb \FEspaceU,\FEspaceP\rb$ and the requirements on the triangulation $\triangulation$ such that \eqref{LinfConvAss} can be guaranteed can be found in \cite{DurNoch1990}. For the case of $d = 3$ see Remark \ref{StkStab3d}.
\item[(A4).] The $L^2$-projection operator $\Pi_{\FEspacePhase}$, is $\hone$-stable, namely
\begin{align}\label{honestab}
\|\nabla\Pi_{\FEspacePhase}\phasetest\|_{\ltwos} \leq c \, \|\phasetest\|_{\hones} \ \
\forall \, \phasetest \in \hone
\end{align}
with $c$ independent of $h$ and $\phasetest$. In the context of quasi-uniform meshes the reader can check the classical references \cite{Ciar78,Girault}, and for non quasi-uniform meshes and different norms \cite{Stein2002,Crou1987,MR3150226}. 
\item[(A5).] The $L^2$-projection operator $\Pi_{\FEspaceUdivfree}$, defined in \eqref{ltwodef}, is $\hone$-stable meaning that
\begin{align}\label{Vhonestab}
\|\nabla\Pi_{\FEspaceUdivfree}\utest\|_{\ltwos} \leq c \, \|\utest\|_{\honeds} \ \ 
\forall \, \utest \in \Vspace
\end{align}
with $c$ independent of $h$ and $\utest$. For the proof of this
stability result we refer to \cite{Girault,HeyRann}. Property \eqref{Vhonestab} can be easily established provided that the pair $\lb \FEspaceU,\FEspaceP\rb$ admits a Fortin projector with optimal approximation properties in $\ltwod$. In \cite[p. 226]{salga2013} it was proved that such a Fortin operator exists for every LBB stable pair, provided that $\Omega$ is convex or of class $\mathcal{C}^{1,1}$.
\end{enumerate}

Using discontinuous pressures (requirement (A1)) allows us to localize the incompressibility constraint \eqref{SNS2} from the Stokes problem to each element, that is
\begin{align}\label{localizorth}
(\Ptest,\diver{}\bv{U}^k)_{\element} = 0 \ \ \forall \Ptest \in \FEspaceP, \, \forall \element \in \triangulation \, .
\end{align}
The constraint $\FEspaceM = [\FEspaceP]^d$ (requirement A2) together with \eqref{localizorth} means that:
\begin{align}\label{localizorth2}
\lp \Mtest^i, \diver{U}^k\rp_{\element} = 0  \ \ \forall \Mtest \in \FEspaceM \quad
\forall \, i = 1, \cdots , d \, , \ \forall \element \in \triangulation \, .
\end{align}

Note that \eqref{SNS1} utilizes the definition \eqref{trilineardef} for the  Kelvin force $\mu_0 \trilm\lp \Utest, \bv{H}^{k}, \bv{M}^{k}\rp$. However, not all the terms of the trilinear form $\trilm(\cdot,\cdot,\cdot)$ are used. More precisely, all the jump terms disappear since $\lj \bv{H}^{k} \rj\!\big|_{F} = 0$ for all $F \in \mathcal{F}^i$, which is a consequence of definition \eqref{newHdef}. This is a very convenient feature which will greatly simplify the a priori estimates and consistency analysis.

The main difference between schemes \eqref{firstscheme} and
\eqref{convsystem}, apart from the fact that the Poisson problem for
$\hdpoth^k$ was eliminated, are the new requirements on the spaces
$\FEspaceU$, $\FEspaceP$ and $\FEspaceM$. Adapting the arguments of
Propostion~\ref{discenerglemm00} we can show that the scheme
\eqref{convsystem} is stable, and proceeding as in
Theorem~\ref{thm:existence} we can establish existence of
solutions. We do this next.
\begin{remark}[max-norm error estimates for the Stokes projector in three dimensions]\label{StkStab3d} The development of max-norm error estimates in three dimensions is rather recent (\cf \cite{GirNochScott,GuzLeyk2012,Demlow2013,Manuel2015,MR3422453}), and is limited to a handful of finite element pairs $\lb \FEspaceU,\FEspaceP\rb$, such as the second and third order Taylor-Hood element, and the lowest order Bernardi-Raugel element. All the max-norm estimates reported in \cite{GirNochScott,GuzLeyk2012,Demlow2013,Manuel2015,MR3422453} use finite element pairs $\lb \FEspaceU,\FEspaceP\rb$ with continuous velocities combined with continuous pressures, or discontinuous pressures of order zero (piecewise constants), which do not satisfy the assumption (A1) of our list above. We are not aware of max-norm error estimates for stable finite element pairs $\lb \FEspaceU,\FEspaceP\rb$ using continuous velocities $\FEspaceU$ and higher-order (at least first-order) discontinuous pressures $\FEspaceP$, thus, satisfying (A1).
\end{remark}

\begin{proposition}[properties of the scheme]\label{discenerglemma2} Assume that $\eta \leq \layerthick$. In this setting, for every $k=1,\ldots,K$ there is $\lb\phvarh^k, \chpoth^k, \bv{M}^k,\bv{U}^k,P^k \rb\in \FEspacePhase \times \FEspaceChemP \times \FEspaceM \times \FEspaceU \times \FEspaceP$ that solves \eqref{convsystem}, with $\bv{H}^k$ defined in \eqref{newHdef}. Moreover this solution satisfies the following stability estimate
\begin{align}
\label{discenergy04}
\begin{split}
\|\bv{U}^{K} \|_{\ltwods}^2 
& + \tfrac{\mu_0}{2} \|\bv{M}^{K} \|_{\ltwods}^2
+ \capcoeff \|\nabla\phvarh^K\|_{\ltwods}^2 
+ \tfrac{2 \capcoeff}{\layerthick^2} (F(\phvarh^{K}),1) 
\\
&+ \sum_{k=1}^{K} \Big( \|\inc\bv{U}^{k} \|_{\ltwods}^2
+ \mu_0 \|\inc\bv{M}^{k} \|_{\ltwods}^2 
+ \capcoeff \|\inc\nabla\phvarh^k\|_{\ltwods}^2 
+ \tfrac{2 \capcoeff \mobility \dt}{\layerthick} \|\nabla\chpoth^k\|_{\ltwods}^2 \Big)\\
%
&+ \sum_{k=1}^{K} \Big( 2 \dt   \|\sqrt{\nu_{\phvarh}} \, \bv{T}(\bv{U}^{k})\|_{\ltwods}^2 
+ \tfrac{\mu_0 \dt }{\chartime} \|\bv{M}^k\|_{\ltwods}^2 
+ \tfrac{2 \mu_0 \dt}{\chartime} \| \sqrt{\suscepch} \, \bv{H}^{k}\|_{\ltwods}^2 \Big) \\
& \leq \|\bv{U}^{0} \|_{\ltwods}^2
+ 2\mu_0 \|\bv{M}^{0} \|_{\ltwods}^2 
+ \capcoeff \|\nabla\phvarh^{0}\|_{\ltwods}^2
+ \tfrac{2 \capcoeff}{\layerthick^2} (F(\phvarh^{0}),1) 
+ \mu_0 \|\bv{H}^0\|_{\ltwods}^2 
+ 2 \mu_0  \|\bv{H}^K \|_{\ltwods}^2 \\
&+ \sum_{k=1}^{K} \tfrac{3 \mu_0 \dt}{\chartime} (1 + \suscep^2) \|\bv{H}^{k}\|_{\ltwods}^2
+ 3 \mu_0 \chartime \sum_{k=1}^{K-1} \dt
\LN\tfrac{\inc\bv{H}^{k+1}}{\dt} \RN_{\ltwods}^2 \, .
\end{split}
\end{align}
The scheme is mass preserving
\begin{align}\label{masspres}
(\phvarh^k, 1) = (\phvarh^0, 1) \ \ \forall \, 1 \leq k \leq K \, ,
\end{align}
and the following additional estimate holds
\begin{align}\label{H1chempot}
 \|\chpoth^{\dt}\|_{\ell^{2}(H^1)} \leq c < \infty \, .
\end{align}
\end{proposition}
\begin{proof}
We proceed as in Proposition \ref{discenerglemm00} (discrete energy stability),
except that now $\bv{H}^k$ is given by \eqref{newHdef}:
set $\Phasetest = \tfrac{2 \capcoeff \dt}{\layerthick} \chpoth^k $,
$\Chpotest = \tfrac{2 \capcoeff}{\layerthick} \inc\phvarh^k$, $\Utest
= 2\dt \bv{U}^k$, $\Mtest = 2\dt\mu_0 \bv{M}^k$, and $\Mtest =
2\dt\mu_0 \bv{H}^k$ in \eqref{convsystem} and add the resulting
expressions. We thus obtain again equality \eqref{discenergy01}, now
without terms involving $\nabla\Phi^k$ and with the right-hand side
\[
\tfrac{2 \mu_0 \dt}{\chartime} \lp \suscepch \bv{H}^{k}, \bv{M}^k \rp  
+ 2 \mu_0 \alp \inc\bv{M}^{k},\bv{H}^k \arp 
+ \tfrac{2 \mu_0 \dt }{\chartime} \lp\bv{M}^k,\bv{H}^k\rp.
\]
The rest is just a matter of adding over $k$ from $1$ to $K$,
applying summation by parts 
\begin{align}\label{summation}
\sum_{k = 1}^{K} a^k \,  \inc b^k = a^K b^K - a^0 b^0 
- \sum_{k = 1}^{K-1} b^k \, \inc a^{k+1}
\end{align}
to $\sum_{k = 1}^{K} \alp \inc\bv{M}^{k},\bv{H}^k \arp$, and employing
Cauchy-Schwarz and Young's inequalities with appropriate
constants. Estimate \eqref{discenergy04}, in conjunction with
arguments similar to those of Theorem~\ref{thm:existence}, yield local
existence of solutions via the Leray-Schauder's theorem \cite[p. 280]{GT2001}.

The mass preserving property \eqref{masspres} can be easily verified by taking $\Phasetest = 1$ in \eqref{SCahn1}. Notice that \eqref{masspres} also implies the following Poincar\'e-type inequality
\begin{align}\label{poinca}
\|\phvarh^k\|_{\ltwos} \leq c \, \bigg( \|\nabla\phvarh^k\|_{\ltwods} + \Big|\int_{\Omega}  \phvarh^k\Big| \bigg) = 
c \, \bigg(\|\nabla\phvarh^k\|_{\ltwods} + \Big|\int_{\Omega}  \phvarh^0\Big| \bigg) \, . 
\end{align}

Estimate \eqref{H1chempot} follows by taking $\Chpotest = 1$ in \eqref{SCahn2}
\begin{align}\label{chppotmean}
 \Big| \int_{\Omega} \chpoth^{k} \Big| &\lesssim \int_{\Omega} |\inc\phvarh^k| + |f(\phvarh^{k-1})|  \,.
\end{align}
To control $\int_{\Omega} |f(\phvarh^{k-1})|$ we use \eqref{doublewellbound} together with the bound on $\|\nabla\phvarh^{\dt}\|_{\ell^{\infty}(\ltwods)}$ given in \eqref{discenergy04} and Poincar\'e inequality \eqref{poinca}, from this we conclude that
\begin{align}\label{meanchempot}
 \max_{k}\Big| \int_{\Omega} \chpoth^{k} \Big| \leq c < \infty \, .
\end{align}
Finally, \eqref{H1chempot} follows by combining \eqref{meanchempot} with \eqref{discenergy04}. 
This concludes the proof.
\end{proof}

\begin{remark}[technical assumption] From now on we will assume that $\int_{\Omega} \phvarh^0 = 0$ in order to simplify the presentation.
\end{remark}

\begin{remark}[range of susceptibility]
Notice that, in contrast to Proposition~\ref{prop:formalenergy}, we do
not require that the susceptibility satisfies $\suscep \leq 4$ 
in Proposition~\ref{discenerglemma2}. This is due to the fact that the magentic field $\bv{H}^\dt$ is part of the problem data. However, as explained at the beginning of this Section, the simplification $\heff = \ha$ is physically meaningful only for small values of $\suscep$.
\end{remark}

\begin{lemma}[estimates for the discrete time derivatives]\label{derivestlemma}
The following estimates for $\tfrac{\inc\phvarh^{k}}{\dt}$ and $\tfrac{\inc\bv{U}^{k}}{\dt}$ hold 
\begin{align}\label{addest}
\LN \tfrac{\inc\phvarh^\dt}{\dt}\RN_{\ell^2(H^*)} + \LN \tfrac{\inc\bv{U}^{\dt}}{\dt}\RN_{\ell^{4/3}(\Vspace^{*}) } \leq c < \infty
\end{align}
with $c$ depending only on $\ha$ but not on $h$ and $\dt$.
Here $H^*$ is the dual of $H^1(\Omega)$ and $\Vspace^{*}$ that of
  $\Vspace$ defined in \eqref{spaces:VH}.
\end{lemma}
\begin{proof} Following \cite{Feng2006} we first use definition
  \eqref{ltwodef} and stability estimate \eqref{Vhonestab} to obtain
\begin{align*}
\begin{split}
\LN \tfrac{\inc\bv{U}^{k}}{\dt}\RN_{\Vspace^*}
&= \sup_{ \bv{v}\in \Vspace}
\frac{\alp \frac{ \inc\bv{U}^{k}}\dt, \bv{v} \arp}{\ \|\bv{v}\|_{\honeds}}
= \sup_{ \bv{v}\in \Vspace }
\frac{\alp \frac{ \inc\bv{U}^{k}}\dt, \Pi_{\FEspaceUdivfree}[\bv{v}] \arp}{\ \|\bv{v}\|_{\honeds}} \lesssim
\sup_{ \bv{v}\in \Vspace } 
\frac{\alp \frac{ \inc\bv{U}^{k}}\dt, \Pi_{\FEspaceUdivfree}[\bv{v}] \arp}{\ \|\Pi_{\FEspaceUdivfree}[\bv{v}]\|_{\honeds}} \, . 
\end{split}
\end{align*}
We next utilize \eqref{SNS1} and definition \eqref{trilineardef} to get
\begin{align*}
\begin{split}
\LN \tfrac{\inc\bv{U}^{k}}{\dt}\RN_{\Vspace^*}
&\lesssim \|\gradv{U}^{k}\|_{\ltwods}
+ \|\bv{U}^{k-1}\|_{\bv{L}^3} \|\bv{U}^{k}\|_{\bv{L}^6}
+ \|\diver{U}^{k-1}\|_{L^2} \|\bv{U}^{k}\|_{\bv{L}^3}  \\
&+ \|\nabla\bv{H}^{k}\|_{\linfds} \|\bv{M}^{k}\|_{\ltwods}
+ \|\bv{H}^{k}\|_{\linfds} \|\bv{M}^{k}\|_{\ltwods}
+ \|\phvarh^{k-1}\|_{L^3} \|\nabla\chpoth^k\|_{\bv{L}^2} \, . \\
\end{split}
\end{align*}
We employ estimate \eqref{discenergy04} and the interpolation inequality
\begin{align}
\|\bv{U}^{k}\|_{\bv{L}^3} \leq \|\bv{U}^{k}\|_{\bv{L}^2}^{1/2} 
\|\bv{U}^{k}\|_{\bv{L}^6}^{1/2} 
\lesssim \|\bv{U}^{k}\|_{\bv{L}^6}^{1/2}
\lesssim \|\nabla\bv{U}^{k}\|_{\ltwods}^{1/2} \, , 
\end{align}
to deduce that 
\begin{align}\label{disctimederest}
\begin{split}
\LN \tfrac{\inc\bv{U}^{k}}{\dt}\RN_{\Vspace^*} 
&\lesssim \|\gradv{U}^{k}\|_{\ltwods}
+ \|\nabla\bv{U}^{k-1}\|_{\ltwods}^{3/2}
+ \|\nabla\bv{U}^{k}\|_{\ltwods}^{3/2}
+ \|\bv{M}^{k}\|_{\ltwods} 
+ \|\nabla\chpoth^k\|_{\bv{L}^2} \\
&\lesssim \Big( \|\gradv{U}^{k}\|_{\ltwods}^{4/3}
+ \|\nabla\bv{U}^{k-1}\|_{\ltwods}^{2} 
+ \|\nabla\bv{U}^{k}\|_{\ltwods}^{2} 
+ \|\bv{M}^{k}\|_{\ltwods}^{4/3}
+ \|\nabla\chpoth^k\|_{\bv{L}^2}^{4/3}  \Big)^{\frac{3}{4}}  \, . 
\end{split}
\end{align}
Raise \eqref{disctimederest} to the power $4/3$, multiply by $\dt$, add in time, and resort to \eqref{discenergy04} to get the desired estimate on $\dt^{-1} \inc \bv{U}^k$.

For the term $\LN\tfrac{\inc\phvarh^k}{\dt}\RN_{\ell^2(H^\star)}$ we proceed
analogously, this time using \eqref{SCahn1} and \eqref{honestab}:
\begin{align}
\label{phasetiemder}
\begin{split}
\LN \tfrac{\inc\phvarh^k}{\dt}\RN_{H^\star} 
&= \sup_{ \Phasetest\in \hone} 
\frac{\alp \tfrac{\inc\phvarh^{k}}{\dt},\Phasetest \arp }{\|\Phasetest\|_{\hones}} 
\lesssim \sup_{ \Phasetest\in \hone} 
\frac{\alp \tfrac{\inc\phvarh^{k}}{\dt},\Pi_{\FEspacePhase}[\Phasetest] \arp }{\|\Pi_{\FEspacePhase}[\Phasetest]\|_{\hones}} \\
&\lesssim \|\bv{U}^k\|_{\bv{L}^4} \|\phvarh^{k-1}\|_{L^4} + \|\nabla\chpoth^k\|_{\ltwods}
\lesssim \|\nabla\bv{U}^k\|_{\ltwods} + \|\nabla\chpoth^k\|_{\ltwods} \, ; 
\end{split}
\end{align}
here we have used the Sobolev embedding inequality in three
dimensions, the equivalence between $\|\phvarh^k\|_{\hones}$ and
$\|\nabla\phvarh^k\|_{\ltwods}$ given by the Poincar\'e inequality
\eqref{poinca}, and the fact that $\|\nabla\phvarh^{\dt}\|_{\ell^{\infty}(\ltwods)}$
is bounded in view of estimate \eqref{discenergy04}.
We finally square \eqref{phasetiemder} and add over $k$ to get the desired estimate for $\dt^{-1} \inc\phvarh^k$.
\end{proof}
%
%
\subsection{Convergence}
\label{convschemesec}
We want to show that solutions generated by the scheme \eqref{convsystem} converge to the ultra-weak solutions of \eqref{weakfor}. The proof relies on classical compactness arguments. We first need the (already proved) basic energy estimates, and then the estimates on the time derivatives in dual norms. Applying Aubin's lemma we can establish existence of strongly convergent subsequences in $L^2(L^2)$ norms, which is enough to pass to the limit in each term. The construction combines some elements from both \cite{Feng2006} and \cite{Walk2005}.

The scheme \eqref{convsystem} generates a sequence of functions $\lb \phvarh^{\dt}, \chpoth^{\dt}, \bv{M}^{\dt}, \bv{U}^{\dt}, \bv{P}^{\dt}\rb$ corresponding to the nodes $\lb t^k \rb_{k=0}^K$, rather than space-time functions. In addition, the scheme \eqref{convsystem} does not have a variational structure in time. In order to reconcile these differences we rewrite scheme \eqref{convsystem} as a space-time variational formulation. For this purpose, we start by defining the functions $\phvarh_{h\dt}$, $\chpoth_{h\dt}$, $\bv{M}_{h\dt}$, $\bv{U}_{h\dt}$, $P_{h\dt}$ such that
\begin{align}\label{pwconst}
\phvarh_{h\dt} = \phvarh^{k} , \, 
\chpoth_{h\dt} = \chpoth^{k} , \, 
\bv{M}_{h\dt} = \bv{M}^{k} , \,
\bv{U}_{h\dt} = \bv{U}^{k} , \, 
P_{h\dt} = P^{k} \ \ \forall \, t \in (t_{k-1},t_k], \  k=1,\ldots,K \, ,
\end{align}
which are piecewise constant in time. Note that these functions are discontinuous in time, but their point values are well defined, in particular they are left-continuous at the nodes $\lb t_k \rb_{k=0}^K$, for instance:
\begin{align}\label{lefcont}
\bv{U}^k = \bv{U}_{h\dt} (t_{k}) = 
\lim_{t \nearrow t_k} \bv{U}_{h\dt} (t) \ \ \ \forall \ 0 \leq k \leq K \, .
\end{align}
Given a finite element space $\mathbb{A}_h$, say either 
$\FEspacePhase, \FEspaceChemP, \FEspaceM, \FEspaceU$ or $\FEspaceP$,
we define the space-time finite element space $\mathbb{A}_{h\tau}$ as
\[
\mathbb{A}_{h\tau} :=
\lb A_{h,\dt} \in L^2(0,\tf;\mathbb{A}_h) \ \Big| \ A_{h,\dt}\big|_{(t_{k-1},t_k]} \in
\mathbb{A}_h \otimes \mathbb{P}_0((t_{k-1},t_k]) \, , \ 1 \leq k \leq K  \rb.
\]
To interpret \eqref{convsystem} variationally, we multiply each
equation by $\tau$ and add over $k$ using \eqref{summation}. We end up
with the following discrete version of \eqref{weakfor}: find
$\lb \phvarh_{h\dt}, \chpoth_{h\dt}, \bv{M}_{h\dt}, \bv{U}_{h\dt}, P_{h\dt} \rb \in
\FEspacePhaseht \times \FEspaceChemPht \times \FEspaceMht \times
\FEspaceUht \times \FEspacePht$ satisfying the system of equations
\begin{subequations}
\label{timeDGsystem}
\begin{align}
\begin{split}
\label{SCahn1DG}
&(\phvarh_{h\dt}(\tf),\Phasetest_{h\dt}(\tf)) 
- \int_{0}^{\tf-\dt} \alp \phvarh_{h\dt},\tfrac{\Phasetest_{h\dt}(\cdot+\dt) - \Phasetest_{h\dt}}{\dt} \arp 
- \int_{0}^{\tf} (\bv{U}_{h\dt} \phvarh_{h\dt}(\cdot-\dt),\nabla\Phasetest_{h\dt}) \\
& \hskip3.cm
+ \mobility (\nabla\chpoth_{h\dt},\nabla\Phasetest_{h\dt}) = (\phvarh_{h\dt}(0),\Phasetest_{h\dt}(0)) \, , 
\end{split} \\
\label{SCahn2DG}
&\int_{0}^{\tf} (\chpoth_{h\dt}, \Chpotest_{h\dt} ) 
+ \tfrac{1}{\eta} (\phvarh_{h\dt} - \phvarh_{h\dt}(\cdot - \dt),\Chpotest_{h\dt}) 
+ \layerthick (\nabla\phvarh_{h\dt},\nabla\Chpotest_{h\dt}) 
+ \tfrac{1}{\layerthick}
(f(\phvarh_{h\dt}(\cdot-\dt)),\Chpotest_{h\dt}) = 0 \, , \\ 
\begin{split}\label{SMag1DG}
&(\bv{M}_{h\dt}(\tf),\Mtest_{h\dt}(\tf)) 
- \int_{0}^{\tf-\dt} \alp \bv{M}_{h\dt},\tfrac{\Mtest_{h\dt}(\cdot+\dt) - \Mtest_{h\dt}}{\dt} \arp 
- \int_{0}^{\tf}  \trilm\lp \bv{U}_{h\dt},\Mtest_{h\dt},\bv{M}_{h\dt}\rp \\
& \hskip3.cm - \tfrac{1}{\chartime} \lp\bv{M}_{h\dt},\Mtest_{h\dt}\rp 
= (\bv{M}_{h\dt}(0),\Mtest_{h\dt}(0) ) 
+ \tfrac{1}{\chartime} \int_{0}^{\tf} \lp \suscepch \bv{H}_{h\dt}, \Mtest_{h\dt} \rp \, , 
\end{split} \\
\begin{split}\label{SNS1DG}
&(\bv{U}_{h\dt}(\tf),\Utest_{h\dt}(\tf)) 
- \int_{0}^{\tf-\dt} \alp \bv{U}_{h\dt},\tfrac{\Utest_{h\dt}(\cdot+\dt) - \Utest_{h\dt}}{\dt} \arp \\
& \hskip3.cm
+ \int_{0}^{\tf} \lp \nu_{\phvarh} \bv{T}(\bv{U}_{h\dt}),\bv{T}(\Utest_{h\dt}) \rp
+ \tril_h\lp\bv{U}_{h\dt},\bv{U}_{h\dt},\Utest_{h\dt}\rp - \lp P_{h\dt}, \diver{}\Utest_{h\dt} \rp \\
&\hskip3.cm
= (\bv{U}_{h\dt}(0),\Utest_{h\dt}(0))
+ \mu_0 \int_{0}^{\tf} \trilm\lp \Utest_{h\dt}, \bv{H}_{h\dt}, \bv{M}_{h\dt}\rp
+ \tfrac{\capcoeff}{\layerthick} (\phvarh_{h\dt}(\cdot - \dt) \nabla\chpoth_{h\dt}, \Utest_{h\dt}) \, , 
\end{split} \\
\label{SNS2DG}
& \int_{0}^{\tf} \lp \Ptest_{h\dt}, \diver{U}_{h\dt}\rp = 0 \, , 
\end{align}
\end{subequations}
for every $\lb \Chpotest_{h\dt}, \Phasetest_{h\dt}, \Mtest_{h\dt}, \Utest_{h\dt}, \Ptest_{h\dt}\rb \in \FEspacePhaseht \times \FEspaceChemPht \times \FEspaceMht \times \FEspaceUht \times \FEspacePht$, where
$\cdot + \dt$ and $\cdot - \dt$ denote positive and negative shifts in time of size $\dt$. Expressions \eqref{pwconst}--\eqref{timeDGsystem} are the reinterpretation of the Backward-Euler method as a zero-order discontinuous Galerkin scheme (see for instance \cite{ErnGuermond,Walk2005,LiuWalk2007,MR2249024}). The difference between \eqref{convsystem} and \eqref{timeDGsystem} is merely cosmetic, since they are equivalent formulations of the same scheme, but clearly \eqref{timeDGsystem} has the right structure if we want to compare it with \eqref{weakfor}. Note also that the choice of half-open intervals $(t_{k-1},t_k]$ in \eqref{pwconst}, leading to the left-continuity \eqref{lefcont}, is consistent with upwinding fluxes (i.e. we chose traces from the direction where information is coming from, which is also consistent with causality).
\begin{lemma}[weak convergence]\label{weakconv} The family of functions $\lb\phvarh_{h\dt},\chpoth_{h\dt},\bv{M}_{h\dt},\bv{U}_{h\dt}\rb_{h,\dt>0}$, defined in \eqref{pwconst} have the following convergence properties:
\begin{align*}
&\phvarh_{h\dt} \xrightharpoonup[]{h,\dt \rightarrow 0}_*  \phvar^* \ \text{in }L^{\infty}(0,\tf;\hone) \, , \\
&\chpoth_{h\dt} \xrightharpoonup[]{h,\dt \rightarrow 0 } \chpot^* \ \text{in }L^2(0,\tf;\hone) \, , \\
&\bv{M}_{h\dt} \xrightharpoonup[]{h,\dt \rightarrow 0 }_* \bv{m}^* \ \text{in }L^{\infty}(0,\tf;\ltwod) \, , \\
&\bv{U}_{h\dt}  \xrightharpoonup[]{h,\dt \rightarrow 0}_* \bv{u}^* \ \text{in }L^{\infty}(0,\tf;\ltwod) \, ,\\
&\bv{U}_{h\dt} \xrightharpoonup[]{h,\dt \rightarrow 0 } \bv{u}^* \ \text{in }L^2(0,\tf;\honed) \, ,
\end{align*}
for some functions $\phvar^*$, $\chpot^*$, $\bv{m}^*$ and $\bv{u}^*$. Here $\xrightharpoonup[]{}_*$ denotes weak-star convergence. 
\end{lemma}
\begin{proof} The proof is a direct consequence of Proposition \ref{discenerglemma2}, definition \eqref{pwconst}, and the Banach-Alaoglu theorem. 
\end{proof}
Note that these modes of convergence are not strong enough to pass to the limit in every term of \eqref{timeDGsystem}, so that the weak limits $\phvar^*$, $\chpot^*$, $\bv{m}^*$ and $\bv{u}^*$ of the previous lemma might not necessarily be solutions of \eqref{weakfor}. In order to improve these estimates we will use the classical Aubin-Lions Lemma in the following form \cite{Lions1969}.
\begin{lemma}[Aubin-Lions]\label{aubinlemma} Let $B_0$, $B$ and $B_1$ denote three Banach spaces such that
\begin{align*}
 B_0 \subset B \subset B_1 \, , 
\end{align*}
with $B_0$ and $B_1$ being reflexive, and $ B_0 \subset \subset B$. We define the space $W$ \begin{align*}
W = \lb w \, \big| \, w \in L^{p_0}(0,\tf;B_0) , \, w_t \in L^{p_1}(0,\tf;B_1)\rb \, , 
\end{align*}
with $1 < p_0, p_1 < \infty$, endowed with the following norm
\begin{align*}
\|w\|_{W} = \|w\|_{L^{p_0}(0,\tf;B_0)} +  \|w_t\|_{L^{p_1}(0,\tf;B_1)} \, .
\end{align*}
Then, the space $W$ is compactly embedded in $L^{p_0}(0,\tf;B)$.
\end{lemma}
\begin{lemma}[strong $L^2(0,t_F;\ltwo)$ convergence]\label{strongL2lem} The family of functions $\lb\phvarh_{h\dt},\bv{U}_{h\dt}\rb_{h,\dt>0}$ defined in \eqref{pwconst} has the following additional convergence properties:
\begin{align*}
&\phvarh_{h\dt} \xrightarrow[]{h,\dt \rightarrow 0} \phvar^* \ \text{in }L^2(0,\tf;\ltwo) \, , \\
&\bv{U}_{h\dt} \xrightarrow[]{h,\dt \rightarrow 0} \bv{u}^* \ \text{in }L^2(0,\tf;\ltwod) \, , 
\end{align*}
for some functions $\phvar^*$ and $\bv{u}^*$.
\end{lemma}

\begin{proof} We would like to apply the estimates from
Proposition~\ref{discenerglemma2} and Lemmas~\ref{derivestlemma} and
\ref{aubinlemma} directly to the family of functions
$\lb\phvarh_{h\dt},\bv{U}_{h\dt}\rb_{h,\dt>0}$. However, this
is not possible because they are discontinuous functions in time. Therefore, we define the auxiliary functions $\widehat{\phvarh}_{h\dt}$ and $\widehat{\bv{U}}_{h\dt}$ by:
\begin{align*}
\widehat{\phvarh}_{h\dt} = \ell_{k-1}(t) \phvarh^{k-1} 
+ \ell_{k}(t) \phvarh^{k} \ \ \forall \, t \in (t_{k-1},t_k] \, , \\
\widehat{\bv{U}}_{h\dt} = \ell_{k-1}(t) \bv{U}^{k-1} 
+ \ell_{k}(t) \bv{U}^k \ \ \forall \, t \in (t_{k-1},t_k] \, , 
\end{align*}
where $\ell_{k-1}(t) = (t_k - t)/\dt$ and $\ell_{k}(t) = (t - t_{k-1})/\dt$.
Since $\widehat{\phvarh}_{h\dt}$ and $\widehat{\bv{U}}_{h\dt}$ are
continuous functions in time, we have:
\begin{enumerate}[\itemizebullet]
\item
The functions $\widehat{\bv{U}}_{h\dt}$ and $\widehat{\phvarh}_{h\dt}$ converge strongly to some $\bv{u}^*$ and $\phvar^*$ in the $L^2(L^2)$ norm, i.e.
 \begin{align}\label{strongL2L2}
\|\widehat{\bv{U}}_{h\dt} - \bv{u}^*\|_{L^2(0,\tf;\ltwods)} + \|\widehat{\phvarh}_{h\dt} - \phvar^*\|_{L^2(0,\tf;L^2)} \xrightarrow[]{h,\dt \rightarrow 0} 0  \, .
\end{align} 
This is a direct consequence Proposition \ref{discenerglemma2}, the dual norm estimates for the time derivatives of Lemma~\ref{derivestlemma}, and an application of Lemma~\ref{aubinlemma}.

\item
The previous bullet implies that $\bv{U}_{h\dt}$ and $\phvarh_{h\dt}$ also converge strongly to the same limits $\bv{u}^*$ and $\phvar^*$ in the $L^2(L^2)$ norm. For the velocity $\bv{U}_{h\dt}$ this is easy to show using the triangle inequality
\begin{align*}
\|\bv{U}_{h\dt} - \bv{u}^*\|_{L^2(0,\tf;\ltwods)} \leq 
\|\bv{U}_{h\dt} - \widehat{\bv{U}}_{h\dt}\|_{L^2(0,\tf;\ltwods)} 
+ \|\widehat{\bv{U}}_{h\dt} - \bv{u}^*\|_{L^2(0,\tf;\ltwods)} \, ,
\end{align*}
and the fact that both terms tend to zero, the second because of
\eqref{strongL2L2} and the first in view of the property
\begin{align*}
\|\bv{U}_{h\dt} - \widehat{\bv{U}}_{h\dt}\|_{L^2(0,\tf;\ltwods)}^2 = 
\tfrac{\dt}{3} \sum_{k = 1}^{K} \|\inc\bv{U}^k\|_{\ltwods}^2
\end{align*}
and estimate \eqref{discenergy04}.
The same argument applies to the phase field $\phvarh_{h\dt}$.
\end{enumerate} 
This completes the proof. \end{proof}
At this point we are in the position to show the main convergence result.
\begin{theorem}[convergence]\label{mainconvlemma} The family of functions $\lb\phvarh_{h\dt},\chpoth_{h\dt},\bv{M}_{h\dt},\bv{U}_{h\dt}\rb_{h,\dt>0}$, defined in \eqref{pwconst} has the following convergence properties
\begin{align}\label{convmodes}
\begin{split}
&\phvarh_{h\dt} \xrightarrow[]{h,\dt \rightarrow 0} \phvar^* \ \text{in } L^2(0,\tf;\ltwo) \, , \\
&\phvarh_{h\dt} \xrightharpoonup[]{h,\dt \rightarrow 0 } \phvar^* \ \text{in } L^2(0,\tf;\hone) \, , \\
&\chpoth_{h\dt} \xrightharpoonup[]{h,\dt \rightarrow 0 } \chpot^* \ \text{in }L^2(0,\tf;\hone) \, , \\
&\bv{M}_{h\dt} \xrightharpoonup[]{h,\dt \rightarrow 0 } \bv{m}^* \ \text{in }L^2(0,\tf;\ltwod) \, , \\
&\bv{U}_{h\dt} \xrightarrow[]{h,\dt \rightarrow 0} \bv{u}^* \ \text{in }L^2(0,\tf;\ltwod) \, , \\
&\bv{U}_{h\dt} \xrightharpoonup[]{h,\dt \rightarrow 0 } \bv{u}^* \ \text{in }L^2(0,\tf;\honed) \, ,
\end{split}
\end{align}
where $\{\phvar^*, \chpot^*, \bv{m}^*,\bv{u}^*\} \in L^2(0,\tf;\hone) \times L^2(0,\tf;\hone) \times L^2(0,\tf;\ltwod) \times L^2(0,\tf;\honed)$ is an ultra-weak solution, namely it satisfies \eqref{weakfor}.
\end{theorem}

\begin{proof} The modes of convergence (weak or strong and their norm) in \eqref{convmodes} are a consequence of Lemmas \ref{weakconv} and \ref{strongL2lem}. It only remains to show that weak limits $\phvar^*$, $\chpot^*$, $\bv{m}^*$ and $\bv{u}^*$ are solutions of the variational problem \eqref{weakfor}. For this purpose we set $\lb\Phasetest_{h\dt},\Chpotest_{h\dt},\Mtest_{h\dt},\Utest_{h\dt} \rb$ to be the space-time interpolants/projections of the smooth test functions $\lb \phasetest, \chpotest, \mtest, \utest \rb$ of the variational formulation \eqref{weakfor}:
\begin{align}\label{disctest}
\Phasetest_{h\dt} := \text{I}_{\FEspacePhase}[\phasetest^k] \, ,  \ 
\Chpotest_{h\dt} := \text{I}_{\FEspaceChemP}[\chpotest^k] \, ,  \ 
\Mtest_{h\dt} := \text{I}_{\FEspaceM}[\mtest^k] \, , \
\Utest_{h\dt} := \SPvel[\utest^k] \, ,  \
\ \ \forall \, t \in (t_{k-1},t_k] \, . 
\end{align}
Note that we employ the Stokes projector $\SPvel\utest^k$ as a
discrete test function $\Utest_{h\dt}$, which means that we test with
discretely divergence-free functions. Inserting these discrete test
functions into \eqref{timeDGsystem} we get:
\begin{subequations}
\label{timeDGsystem2}
\begin{align}
\label{SCahn1DG2}
\begin{split}
&- \int_{0}^{\tf-\dt} \alp \phvarh_{h\dt},\tfrac{\Phasetest_{h\dt}(\cdot+\dt) - \Phasetest_{h\dt}}{\dt} \arp 
- \int_{0}^{\tf} (\bv{U}_{h\dt} \phvarh_{h\dt}(\cdot-\dt),\nabla\Phasetest_{h\dt}) +  \mobility (\nabla\chpoth_{h\dt},\nabla\Phasetest_{h\dt}) = (\phvarh_{h\dt}(0),\Phasetest_{h\dt}(0)) \, , 
\end{split} \\
\label{SCahn2DG2}
&\int_{0}^{\tf} (\chpoth_{h\dt}, \Chpotest_{h\dt} ) 
+ \tfrac{1}{\eta} (\phvarh_{h\dt} - \phvarh_{h\dt}(\cdot - \dt),\Chpotest_{h\dt}) 
+ \layerthick (\nabla\phvarh_{h\dt},\nabla\Chpotest_{h\dt}) 
+ \tfrac{1}{\layerthick} (f(\phvarh_{h\dt}(\cdot-\dt)),\Chpotest_{h\dt}) = 0 \, , \\
\begin{split}\label{SMag1DG2}
&- \int_{0}^{\tf-\dt} \alp \bv{M}_{h\dt},\tfrac{\Mtest_{h\dt}(\cdot+\dt) - \Mtest_{h\dt}}{\dt} \arp 
- \int_{0}^{\tf} \trilm\lp \bv{U}_{h\dt},\Mtest_{h\dt},\bv{M}_{h\dt}\rp
- \tfrac{1}{\chartime} \lp\bv{M}_{h\dt},\Mtest_{h\dt}\rp \\
& \hskip3.cm
= (\bv{M}_{h\dt}(0),\Mtest_{h\dt}(0) ) 
+ \tfrac{1}{\chartime} \int_{0}^{\tf} \lp \suscepch \bv{H}_{h\dt}, \Mtest_{h\dt} \rp \, ,
\end{split} \\
\label{SNS1DG2}
\begin{split}
&- \int_{0}^{\tf-\dt} \alp \bv{U}_{h\dt},\tfrac{\Utest_{h\dt}(\cdot+\dt) - \Utest_{h\dt}}{\dt} \arp 
+ \int_{0}^{\tf} \lp \nu_{\phvarh} \bv{T}(\bv{U}_{h\dt}),\bv{T}(\Utest_{h\dt}) \rp 
+ \tril_h\lp\bv{U}_{h\dt},\bv{U}_{h\dt},\Utest_{h\dt}\rp - \lp P_{h\dt}, \diver{}\Utest_{h\dt} \rp \\
& \hskip3.cm
= (\bv{U}_{h\dt}(0),\Utest_{h\dt}(0)) + \int_{0}^{\tf} \trilm\lp \Utest_{h\dt}, \bv{H}_{h\dt}, \bv{M}_{h\dt}\rp 
+ \tfrac{\capcoeff}{\layerthick} (\phvarh_{h\dt}(\cdot - \dt)
\nabla\chpoth_{h\dt}, \Utest_{h\dt}) \, ,
\end{split} \\
\label{div}
& \int_{0}^{\tf} \lp \Ptest_{h\dt}, \diver{U}_{h\dt}\rp = 0 \, ,  
\end{align}
\end{subequations}
where the terms evaluated at time $t = \tf$ have disappeared because
of the compact support of the test functions $\lb \phasetest,
\chpotest, \mtest, \utest \rb$.  Note also that the pressure term of
the Navier Stokes equation has vanished as well, which is a
consequence of the definition \eqref{disctest} of the discrete test function
$\Utest_{h\dt} := \SPvel[\utest^k]$ for all $t \in (t^{k-1},t^k]$ involving the Stokes projector $\SPvel$ defined in \eqref{stokesproj}. Now we will pass to the limit term by term in \eqref{timeDGsystem2}:
\begin{enumerate}[\itemizebullet]
\item
{\it Time derivatives}.
The convergence of the following three terms is straightforward
\begin{align*}
- \int_{0}^{\tf-\dt} \alp \phvarh_{h\dt},\tfrac{\Phasetest_{h\dt}(\cdot+\dt) - \Phasetest_{h\dt}}{\dt} \arp  
\ \ \ &\xrightarrow[]{h,\dt \rightarrow 0} \ \ \ - \int_{0}^{\tf}( \phvar^*,\phasetest_t )  \, , \\
- \int_{0}^{\tf-\dt} \alp \bv{M}_{h\dt},\tfrac{\Mtest_{h\dt}(\cdot+\dt) - \Mtest_{h\dt}}{\dt} \arp  
\ \ \ &\xrightarrow[]{h,\dt \rightarrow 0} \ \ \ - \int_{0}^{\tf} (\bv{m}^*,\mtest_t) \, , \\
- \int_{0}^{\tf-\dt} \alp \bv{U}_{h\dt},\tfrac{\Utest_{h\dt}(\cdot+\dt) - \Utest_{h\dt}}{\dt} \arp 
\ \ \ &\xrightarrow[]{h,\dt \rightarrow 0} \ \ \ - \int_{0}^{\tf} ( \bv{u}^*,\utest_t ) \, , 
\end{align*}
because of the weak $L^2(L^2)$ convergence of $\phvarh_{h\dt}$, $\bv{M}_{h\dt}$ and $\bv{U}_{h\dt}$, and the strong convergence of the finite differences $\tfrac{\Phasetest_{h\dt}(\cdot+\dt) - \Phasetest_{h\dt}}{\dt}$, $\tfrac{\Mtest_{h\dt}(\cdot+\dt) - \Mtest_{h\dt}}{\dt}$ and $\tfrac{\Utest_{h\dt}(\cdot+\dt) - \Utest_{h\dt}}{\dt}$ guaranteed by the regularity of the test functions.

\item
{\it Convective terms.} We start with the convective term of 
\eqref{SCahn1DG2}
\begin{align*}
\int_{0}^{\tf} (\bv{U}_{h\dt} \phvarh_{h\dt}(\cdot-\dt),\nabla\Phasetest_{h\dt})
\ \ \ \xrightarrow[]{h,\dt \rightarrow 0} \ \ \ \int_{0}^{\tf}(\bv{u}^* \phvar^*,\nabla\phasetest) 
\end{align*}
for which the convergence modes in \eqref{convmodes} are more than
enough: for $\bv{U}_{h\dt}$ and $\phvarh_{h\dt}$ we just need
one weak and one strong convergence in $L^2(L^2)$, and the strong $L^{\infty}$-convergence of $\nabla\Phasetest_{h\dt}$ guaranteed by \eqref{optestim} and the regularity of the test function $\phasetest$.

We continue with the convective term of \eqref{SMag1DG2}. Using definition \eqref{trilineardef} we get
\begin{align}
\label{trilimlimit}
\begin{aligned}
&\int_{0}^{\tf} \trilm\lp \bv{U}_{h\dt},\Mtest_{h\dt},\bv{M}_{h\dt}\rp 
= \int_{0}^{\tf} \sum_{\element \in \triangulation} \int_\element \big( (\bv{U}_{h\dt} \cdot \nabla )\bv{Z}_{h\dt} \cdot \bv{M}_{h\dt} 
+ \tfrac{1}{2} \diver{}\bv{U}_{h\dt} \, \bv{Z}_{h\dt} \cdot \bv{M}_{h\dt} \big) \, .
\end{aligned}
\end{align} 
Note that the consistency terms with the jumps have disappeared since $\lj\bv{Z}_{h\dt}\rj\big|_{F} = 0$ for all $F\in \mathcal{F}^i$, which is a consequence of definitions \eqref{interpolants} and \eqref{disctest}. Passage to the limit of the first part of \eqref{trilimlimit}, that is
\begin{align*}
\int_{0}^{\tf} \sum_{\element \in \triangulation} \int_\element (\bv{U}_{h\dt} \cdot \nabla )\bv{Z}_{h\dt} \cdot \bv{M}_{h\dt} \ \ \ \xrightarrow[]{h,\dt \rightarrow 0} \ \ \ \int_{0}^{\tf} \tril(\bv{u}^*,\bv{z},\bv{m}^*) \, , 
\end{align*}
is carried out using the strong $L^2(L^2)$ convergence of $\bv{U}_{h\dt}$, the weak $L^2(L^2)$ convergence of $\bv{M}_{h\dt}$ and the strong $L^{\infty}$-convergence of $\nabla\bv{Z}_{h\dt}$ guaranteed by \eqref{optestim}. By consistency, we need the second part of \eqref{trilimlimit} to vanish when $h,\dt \rightarrow 0$. In view of the local orthogonality property \eqref{localizorth2}, we obtain
\begin{align}
\label{temamconsistency}
\int_{0}^{\tf} \sum_{\element \in \triangulation} \int_\element \tfrac{1}{2} \diver{}\bv{U}_{h\dt} \, \bv{Z}_{h\dt} \cdot \bv{M}_{h\dt} 
=  \int_{0}^{\tf} \sum_{\element \in \triangulation} \int_\element
\tfrac{1}{2} \diver{}\bv{U}_{h\dt} \, (\bv{Z}_{h\dt} -
\langle\bv{Z}_{h\dt} \rangle_{\element} )\cdot \bv{M}_{h\dt} \rp \, ,
\end{align}
where $\langle\bv{Z}_{h\dt} \rangle_{\element} =
\tfrac{1}{|\element|}\int_{\element} \bv{Z}_{h\dt}$. Invoking the
uniform bounds on $\bv{U}_{h\dt}$ and $\bv{M}_{h\dt}$ already proved
in Proposition \ref{discenerglemma2} (properties of the scheme), and the
regularity of the test function $\mtest$, we deduce
\begin{equation*}
\int_{0}^{\tf} \sum_{\element \in \triangulation} \int_\element
\big| \diver{}\bv{U}_{h\dt} \, \bv{Z}_{h\dt} \cdot \bv{M}_{h\dt} \big|
\lesssim \|\nabla\bv{U}_{h\dt}\|_{L^2(0,\tf;\ltwods)} \ h \|\nabla\Mtest\|_{L^\infty(0,\tf,\linfds)} \ \|\bv{M}_{h\dt}\|_{L^2(0,\tf;\ltwods)}
\  \xrightarrow[]{h,\dt \rightarrow 0} \ 0 \, .
\end{equation*}

Passage to the limit of the convective term in \eqref{SNS1DG2} follows standard procedures and its treatment can be found in other works such as \cite{Temam,MarTem}.

\item
For the Kelvin force in \eqref{SNS1DG2}, using again definition \eqref{trilineardef}, we get
\begin{equation*}
\int_{0}^{\tf} \trilm\lp \Utest_{h\dt}, \bv{H}_{h\dt},\bv{M}_{h\dt}\rp 
= \int_{0}^{\tf} \sum_{\element \in \triangulation} \int_\element \lp (\Utest_{h\dt} \cdot \nabla )\bv{H}_{h\dt} \cdot \bv{M}_{h\dt}
+ \tfrac{1}{2} \diver{}\Utest_{h\dt} \, \bv{H}_{h\dt} \cdot
\bv{M}_{h\dt} \rp.
\end{equation*}
Arguments similar to those used in  \eqref{temamconsistency} show
that the second term tends to zero, this time by adding a term of the
form $\langle\bv{H}_{h\dt} \rangle_{\element}\cdot\bv{M}_{h\dt}$. For
the first term we exploit the weak $L^2(L^2)$ convergence of
$\bv{M}_{h\dt}$, the strong $L^2(H^1)$ convergence of $\bv{H}_{h\dt}$,
the strong $L^{\infty}$ convergence the test function $\Utest_{h\dt}$,
and the fact that $\nabla\heff = \nabla\heff^\text{T}$ (because
$\curl\heff = 0$). This is all what we need to pass to the limit
\begin{equation*}
\int_{0}^{\tf}  \trilm\lp \Utest_{h\dt},\bv{H}_{h\dt},\bv{M}_{h\dt}\rp
\xrightarrow[]{h,\dt \rightarrow 0}
\int_{0}^{\tf} \int_{\Omega} (\utest \cdot \nabla)\bv{h} \,  \bv{m}^* = \int_{0}^{\tf} \int_{\Omega} (\bv{m}^* \cdot \nabla)\bv{h} \,  \utest \, . 
\end{equation*}

\item

{\it Stabilization.} To show that the stabilization term
in \eqref{SCahn2DG2} vanishes in the limit
\begin{align*}
\int_{0}^{\tf} \tfrac{1}{\eta} \big(\phvarh_{h\dt} - \phvarh_{h\dt}(\cdot - \dt),\Chpotest_{h\dt}\big) = \tfrac{1}{\eta} \sum_{k=1}^{K} \dt (\inc\phvarh^k,\Chpotest^k )
\leq \tfrac{1}{\eta} \alp \sum_{k=1}^{K} \dt \|\inc\phvarh^k\|_{\ltwos}^2 \arp^{\!\! 1/2}
\alp \sum_{k=1}^{K} \dt \|\Chpotest^k\|_{\ltwos}^2 \arp^{\!\! 1/2} \lesssim \dt^{1/2} \, , 
\end{align*}
we resort to the stability estimate \eqref{discenergy04} and
Poincar\'e inequality \eqref{poinca}.

\item
{\it Dissipation.}
The only critical term left in \eqref{SNS1DG2} is $\int_{0}^{\tf} \lp \nu_{\phvarh} \bv{T}(\bv{U}_{h\dt}),\bv{T}(\Utest_{h\dt}) \rp$. Passage to the limit
\begin{align}
\int_{0}^{\tf} \lp \nu_{\phvarh} \bv{T}(\bv{U}_{h\dt}),\bv{T}(\Utest_{h\dt}) \rp 
\ \ \xrightarrow[]{h,\dt \rightarrow 0} \ \ \int_{0}^{\tf} \lp \nu_{\phvar^*}  \bv{T}(\bv{u}^*),\bv{T}(\utest) \rp
\end{align}
uses strong $L^2(L^2)$ convergence of $\phvar^*$, the Lipschitz continuity property of $\nu_{\phvar}$, the weak $L^2(L^2)$ convergence of $\nabla\bv{U}_{h\dt}$, and the strong $L^{\infty}$ convergence of $\gradv{}\Utest_{h\dt}$ from assumption \eqref{LinfConvAss}.
\end{enumerate}
The remaining terms require little or no explanation, or the passage
to the limit can be found in other works such as
\cite{Feng2006,LiuWalk2007,grun2013}.
This completes the proof.
\end{proof}

\begin{remark}[stabilization] For the sake of simplicity, we have presented the numerical scheme \eqref{convsystem} without any form of stabilization (upwinding). Unlike Continuous Galerkin methods, DG schemes do not need any form additional numerical stabilization in order to work. However without some form of linear stabilization they will deliver sub-optimal convergence rates to smooth solutions (see for instance \cite{DiPi10}). Numerical stabilization can be considered by adding the term
\begin{align}\label{stabterm}
\stab(\bv{U}^k,\bv{M}^k,\Mtest) = \tfrac{1}{2} \sum_{F\in\mathcal{F}^i}
\bdryint{F}{|\bv{U}^k\cdot \normal_F| \lj \bv{M}^k\rj \cdot \lj\Mtest\rj}
\end{align}
to the left-hand side of \eqref{SMag1}. With such a modification all the results presented above remain unchanged.
\end{remark}

\begin{remark}[choice of finite element spaces]\label{choicespacessec} Now we specify a set of finite element spaces which satisfy the key assumptions required in the proof of convergence of scheme \eqref{convsystem}. In space dimension two, the spaces $\FEspacePhase$ and $\FEspaceChemP$  can be the same as those defined in \eqref{FEspaces}, while the pair $\lb\FEspaceU,\FEspaceP\rb$ can be chosen as
\begin{align}
\label{CrouRavPair}
\begin{split}
\FEspaceU &= \lb \bv{U} \in \bm{\mathcal{C}}^0\bigl(\overline\Omega\bigr) \, \big|
\, \bv{U}|_{T} \in [\simplex_{2}(\element) \oplus \text{Span} \, \mathcal{B}(\element)]^d \, , \forall \, \element \in \triangulation \rb \cap \hzerod \, , \\
\FEspaceP &= \lb \Ptest \in L^2\bigl(\overline\Omega\bigr) \ |
\ \Ptest|_{T} \in \simplex_{{1}}(\element) \, , \forall \, \element \in \triangulation \rb \, ,  
\end{split}
\end{align}
where $\mathcal{B}(\element) = \prod_{i =1}^{d+1} \lambda_i$ is the cubic bubble function and $\lb \lambda_i \rb_{i=1}^{d+1}$ are the barycentric coordinates. This finite element pair is known as enriched Taylor-Hood element and it is well known to be LBB stable (\cf \cite{ErnGuermond}) in space dimension two. Approximation (convergence) properties of the Stokes projector in $L^{\infty}$-norm (guaranteeing that assumption \eqref{LinfConvAss} holds true) for this finite element pair were established in \cite{DurNoch1990}. In space dimension three, the reader could consider using again the finite element spaces $\FEspacePhase$ and $\FEspaceChemP$ defined in \eqref{FEspaces} and the second-order Bernardi-Raugel element (see \cite[p. 148]{Girault}) which uses $\simplex_{1}$ discontinuous elements for the pressure space, but the convergence assumption \eqref{LinfConvAss} is to be verified. In every case, the magnetization space $\FEspaceM$ should be set to be $\FEspaceM := [\FEspaceP]^d$.
\end{remark}

%
%
\section{Numerical Experiments}
\label{Sexperiment}
%
Let us now explore model \eqref{Themodel} and scheme \eqref{firstscheme} with a series of examples. The main goal of these experiments is to assess the robustness of scheme \eqref{firstscheme} and to demostrate the ability of the model to capture some well-known phenomena observed in real ferrofluids.  
%
\subsection{General considerations}
In all the experiments we consider the gradient of the potential of a point dipole $\nabla\hapot_s$ as a prototype for an applied magnetizing field where
\begin{align}\label{dipole}
\hapot_s(\bv{x}) = \frac{\dipdir\cdot (\bv{x}_s - \bv{x})}{|\bv{x}_s- \bv{x}|^3} \, , 
\end{align}
$|\dipdir| = 1 $ indicates the direction of the dipole, and $\bv{x}_s = (x_s,y_s,z_s) \in \mathbb{R}^3$ is its location. It is clear that $\curl{}\nabla\hapot_s=0$, and it is straightforward to verify that $\diver{}\nabla\hapot_s = \Delta\hapot_s = 0$ for every $\bv{x} \neq \bv{x}_s$, so that $\nabla\hapot_s$ defines a harmonic vector field. This is a physical requirement in the context of non-conducting media: a magnetic field must satisfy the equations of magnetostatics, which boils down to $\ha$ being harmonic (both curl and div-free). 

Formula \eqref{dipole}, however, is intrinsically three dimensional \cite{Jack1998}. For this reason, we consider an alternative definition which leads to a two-dimensional harmonic vector field in $\mathbb{R}^2$:
\begin{align}\label{dipole2D}
\hapot_s(\bv{x}) = \frac{\dipdir\cdot (\bv{x}_s - \bv{x})}{|\bv{x}_s- \bv{x}|^2} \, ,
\end{align}
where now $\bv{d}, \bv{x} , \bv{x}_s \in \mathbb{R}^2$. In all our numerical experiments we will use linear combinations of dipoles as applied magnetizing field $\ha$:
\begin{align}\label{haformula}
\ha = \sum_{s} \alpha_s \nabla\hapot_s.
\end{align}

On the other hand, in order to carry out meaningful computations of phase-field models it is crucial to resolve the transition layer, otherwise artificial spurious oscillations may arise (\cf \cite{VPN1990,Elli1993,Braides2012}). Even in the context of two dimensional simulations, using for instance $\layerthick = 0.01$, and resolving the transition layer of thickness $\layerthick$ by means of uniform meshes (say for instance, using $h = 10^{-3}$) can turn out to be prohibitively expensive and slow. If we want to obtain results in a timely fashion, computations of phase-field models claiming to have any practical value do invariably need some form of adaptivity. Our work is no exception, and for that reason we use adaptivity in space. This entails using graded meshes and coarsening, the latter not being covered by the theory developed in this paper; we refer to \cite{Bart2010,Kess2004}. 

The implementation is carried out with the help of the \texttt{deal.II} library \cite{BHK2007,DealIIReference}. In particular the parallel-adaptive framework discussed in \cite{Heister2011, TimoThesis} has been extensively used in this work. Regarding error indicators we resort to the simplest element indicator $\eta_{\element}$ \cite{Kelly1983}:
\begin{align}
\label{errorindic}
\eta_{\element}^2 = h_{\element} \bdryint{\partial\element}{ \big| \! \lj \tfrac{\partial \phvarh}{\partial n} \rj \! \big|^2}  \ \ \forall \element \in \triangulation . 
\end{align}
Computationally, it is well-known that \eqref{errorindic} performs reasonably well for second order elliptic and parabolic problems. This error indicator is already implemented in the library \texttt{deal.II}, being that the main reason for its selection. From a mathematical point of view, using \eqref{errorindic} is questionable, since residual a posteriori error indicators for phase-field models (Allen-Cahn and Cahn-Hilliard) have been an area of major research; the interested reader can check \cite{Kess2004,Bart2010} and references therein. The marking strategy follows the D\"orfler (or bulk chasing) approach, which in this time dependent context requires both marking for refinement and coarsening (a judiciously small fraction). The mesh is refined-coarsened once every 5 time steps.
%
\subsection{Parametric study of the Rosensweig instability}

\label{exp1} The purpose of this section is to run a series of
parametric studies with respect to mesh size $h$ and time step
size $\dt$ in order to assess the robustness of scheme
\eqref{firstscheme} with respect to these two parameters. In order to
run such studies, we selected the {\it Rosensweig instability} (also called normal-field instability) as an example. We start by explaining the setup of the parameters of the model. Then we explain what the Rosensweig instability is, provide some well-known analytic results, and point to some background references for the interested reader. Finally, we use these analytic results to tweak the numerical experiment, and run the parametric study.

Regarding inertial scalings, we work in a rectangular domain of $1$ unit of width and $0.6$ units of height, with vertices at $(0,0)$, $(0,0.6)$, $(1,0.6)$ and $(1,0)$, whence $\textsl{diam}\,\Omega \approx 1$. We choose the viscosities
\begin{align*}
\nu_w = 1.0 \ \text{and} \ \nu_f = 2.0 \, , 
\end{align*}
and the density $\rho$ implicitly taken (from the very beginning of
the paper) to be unitary, so that the Reynolds number obeys the rule
$\textsl{Re} = \mathcal{O}(\|\bv{u}\|_{\bv{L}^{\infty}(\Omega \times (0,\tf))})$. On the other hand, we take the parameters
\begin{align*}
\mu_0 = 1 \, , \ 
\permit = 0.5 \, , \ 
\mobility = 0.0002 \, , \ 
\capcoeff = 0.05 \, , \ 
r = 0.1 \, .
\end{align*}
The main goal of such an arbitrary choice of parameters ($\nu_w$, $\nu_f$, $\rho$, $\mu_0$, $\permit$, $\mobility$ and $\capcoeff$) is to have a very stable PDE system. Note that in \eqref{enerestTwPh} and \eqref{finaldiscenergy} all the natural estimates for the phase-field $\phvar$ and chemical potential $\chpot$ depend on $\capcoeff$, thus we should expect the stability of the interfaces (and the whole PDE system \eqref{Themodel} in general) to deteriorate severely for small values of $\capcoeff$. With such a deliberate choice of parameters we have a very stable system of equations. Paradoxically, we now try to come up with a smart scaling of the forces in order to get an interesting (unstable) behavior as it will be detailed in the following paragraphs.

The Rosensweig instability is perhaps the simplest nontrivial phenomena observed in ferrofluids. Basically, if we have a pool of ferrofluid lying horizontally, subject to both the force of gravity and a uniform magnetic field $\ha$ pointing upwards, it is well known that a flat profile is not stable for all values of the magnetic field and a regular pattern of peaks and valleys forms. The formation of these patterns is the result of competing forces: both gravity and surface tension favor a flat surface, but above a critical magnetic field strength, the flat profile is no longer the most stable configuration. A sufficiently strong magnetic field triggers the instability and the pattern formation.

There are analytical expressions for the distance between peaks and for the critical magnetic field strength that triggers the instability. There is a vast literature on this topic and it is impossible to do justice to all possible references; here we will just comment on a few of them as  background for the interested reader. The work of Cowley and Rosensweig \cite{CowRos1967} is most probably the first one to provide analytical results based on linear stability analysis (dependence of the most unstable modes on the constitutive parameters) valid only in the asymptotic limit of vanishing magnetic susceptibility $\suscep$:
the quantities
\begin{align}\label{RosLin}
\distpeaks = 2 \pi \alp \frac{\surftens}{g \Delta\rho} \arp ^{\!\!\nicefrac{1}{2}} \ , \ \
m_c^2 = \tfrac{2}{\mu_0} \big(\tfrac{2 + \permit}{1+\permit} \big) \alp g \Delta\rho \surftens \arp ^{\!\nicefrac{1}{2}} \, , 
\end{align}
are the critical spacing between peaks and the critical magnitude of
magnetization, whereas $\surftens$ is the surface tension coefficient in the sharp interface limit, $g = |\bv{g}|$ is the magnitude of gravity, and $\Delta\rho$ is the jump of the density across the interface. The work of Gailitis \cite{Gaili1977}, considered to be the first attempt to include nonlinear effects, uses an energetic approach (minimization of a functional), and is able to describe the shape of the patterns (hexagons, squares, etc.). However \cite{Gaili1977} suffers from the same limitations of \cite{CowRos1967} (small susceptibilities, finite depth, etc), which have been overcome to some degree in \cite{Engel2001}. Validation of all these analytical results is far from complete, requiring carefully crafted experiments which mimic ideal conditions; some efforts in this direction can be found in \cite{Abou2000,Tob2007} and references therein.

Most of these results possess intrinsic qualitative value, but they are far from quantitative for any realistic context which could include finite magnetic susceptibilities, finite depth of the ferrofluid pool, nonlinear effects (large displacements of the interface between both phases), and diffusive effects (partial mixture). In particular, to the best of our knowledge, there are not analytical results for highly paramagnetic ferrofluids ($\suscep > 1$), and the treatment (or inclusion) of effects related to the demagnetizing field is quite poor.

For instance, we cannot expect the linear stability result \eqref{RosLin} to accurately predict the behavior of system \eqref{Themodel} in the context of bounded domains (with non-periodic boundary conditions), finite depth, finite magnetic susceptibility ($\suscep = \mathcal{O}(1)$), highly deformed transition layer (far from a straight line), and finite interaction length (layer thickness $\layerthick$) involving additional diffusive effects. We also point out that our phase-field model is not a genuine variable density model, so that the term $\Delta\rho$ has very little meaning in the context of the model \eqref{Themodel}, and gravitational effects are only included approximately via \eqref{Bouapprox}. Yet, from  \eqref{Bouapprox} we can easily realize that 
\begin{align}\label{densformula}
r \approx \tfrac{\Delta\rho}{\rho}   \, . 
\end{align}
Finally, the relationship between the capillary coefficient $\capcoeff$ and the surface tension $\surftens$ (see for instance \cite{Lowen1998,Jacq1996}) is only known approximately
\begin{align}\label{surfcaprel}
\capcoeff \approx \surftens \layerthick \, , 
\end{align}
where the constant involved in this relationship is unknown but of $\mathcal{O}(1)$. Nevertheless, it can be shown that the linear relationship \eqref{surfcaprel} is particularly accurate for small mobilities $\mobility$ (see \cite{WaLiu01}), being that the reason why we choose $\gamma =  0.0002$.

\begin{figure}[h!]
\begin{center} 
    \setlength\fboxsep{0pt}
    \setlength\fboxrule{1pt}
  \begin{tabular}{cccc}

    \fbox{\includegraphics[width=32mm]{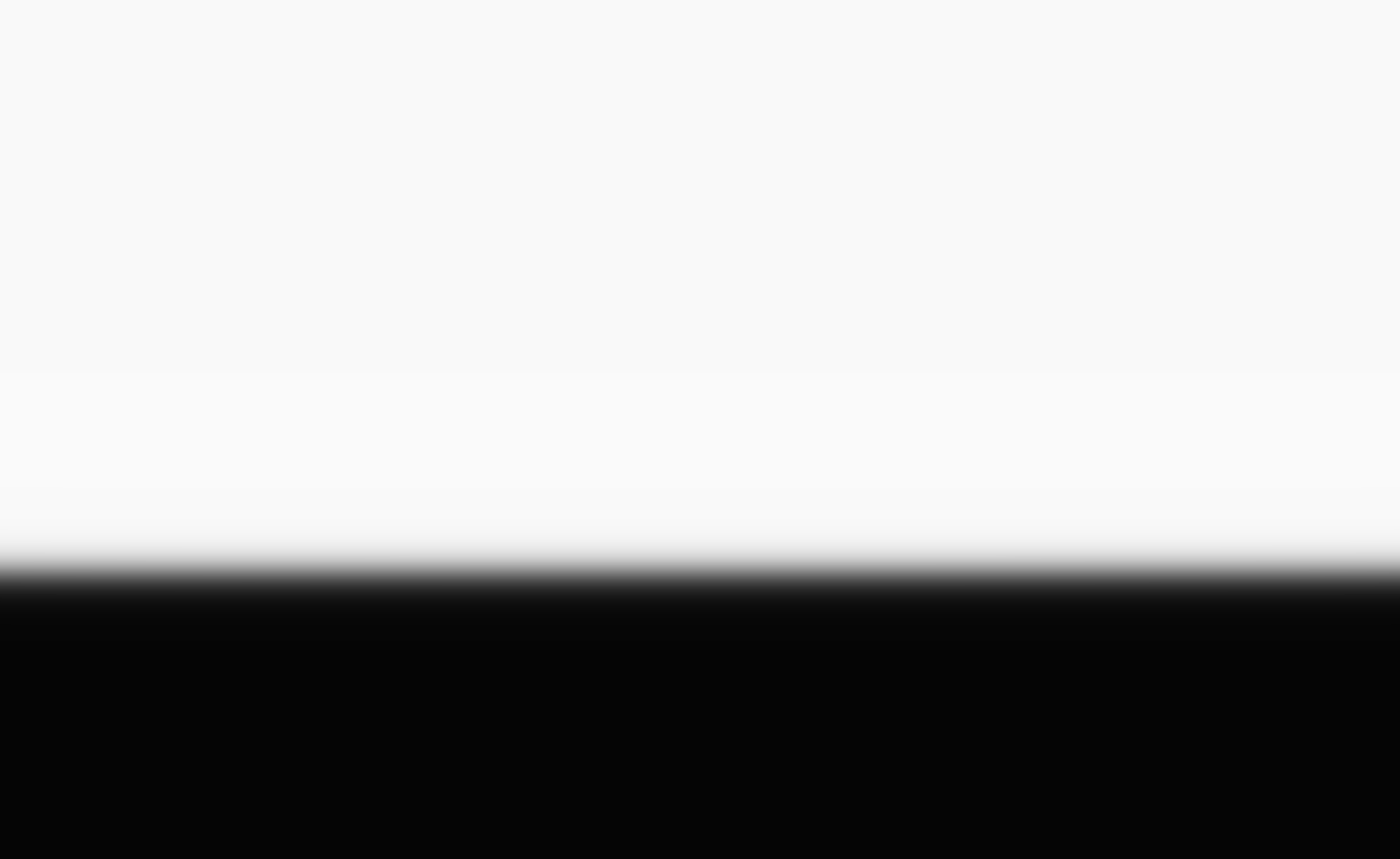}}&
    \fbox{\includegraphics[width=32mm]{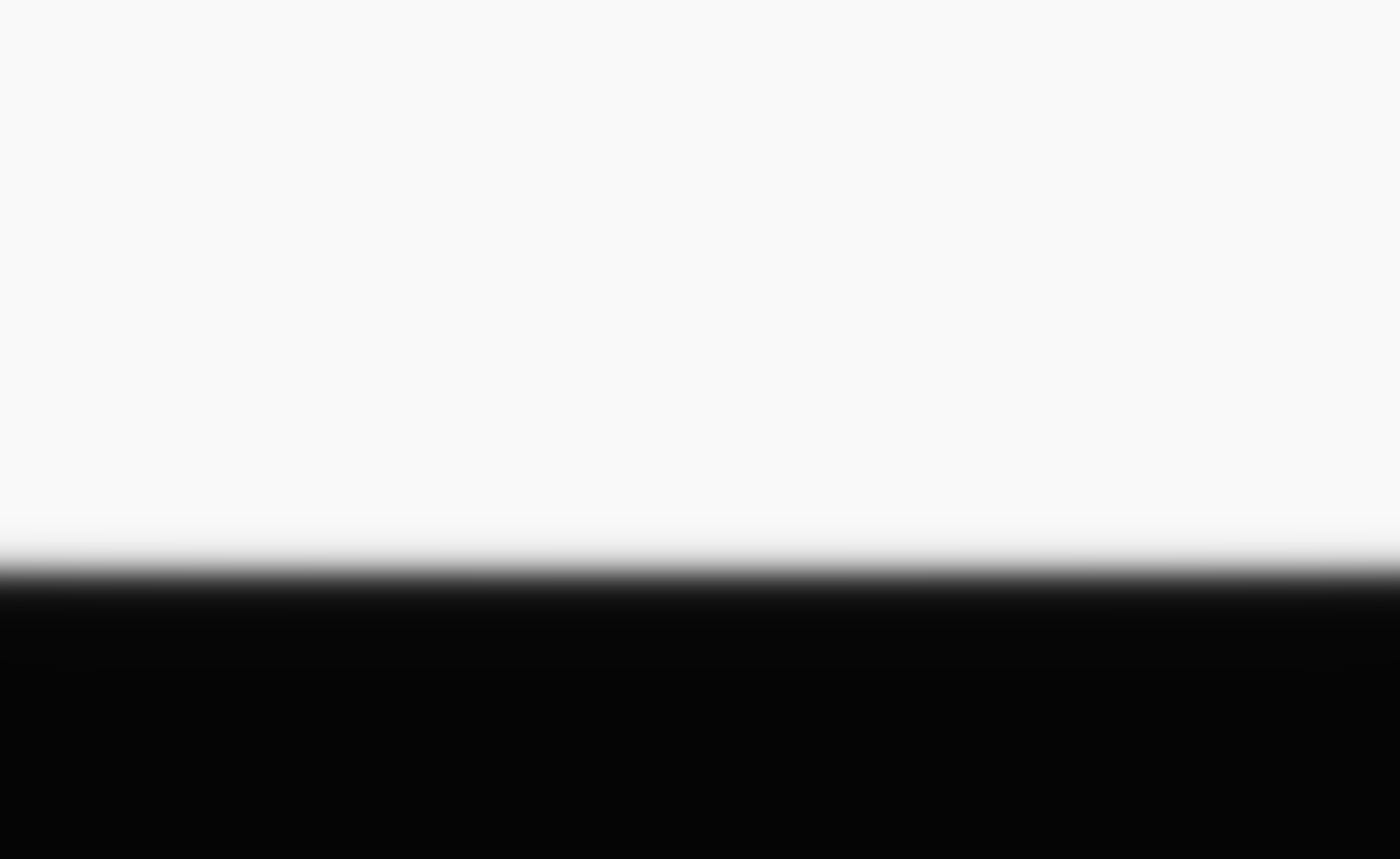}}&
    \fbox{\includegraphics[width=32mm]{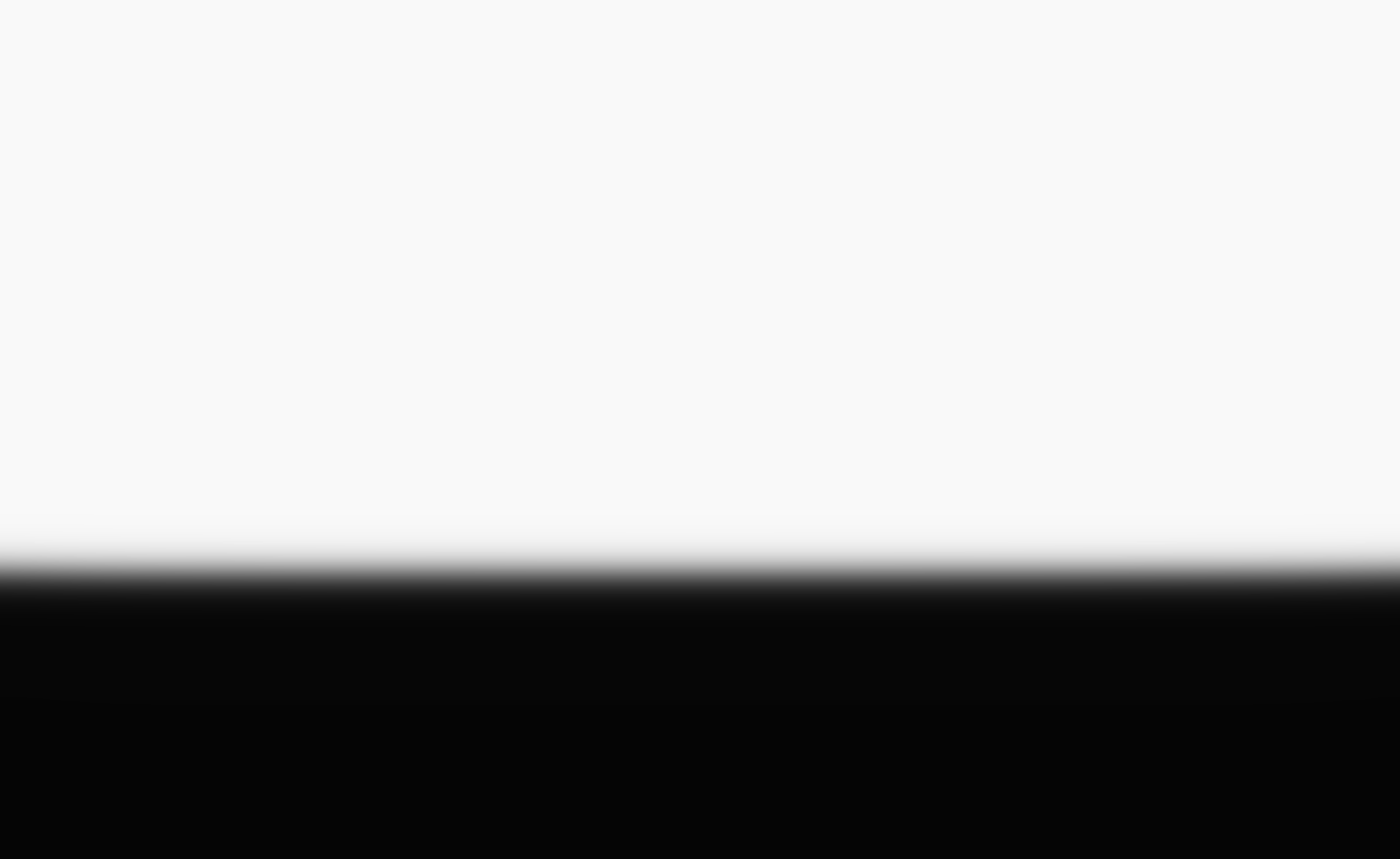}}&
    \fbox{\includegraphics[width=32mm]{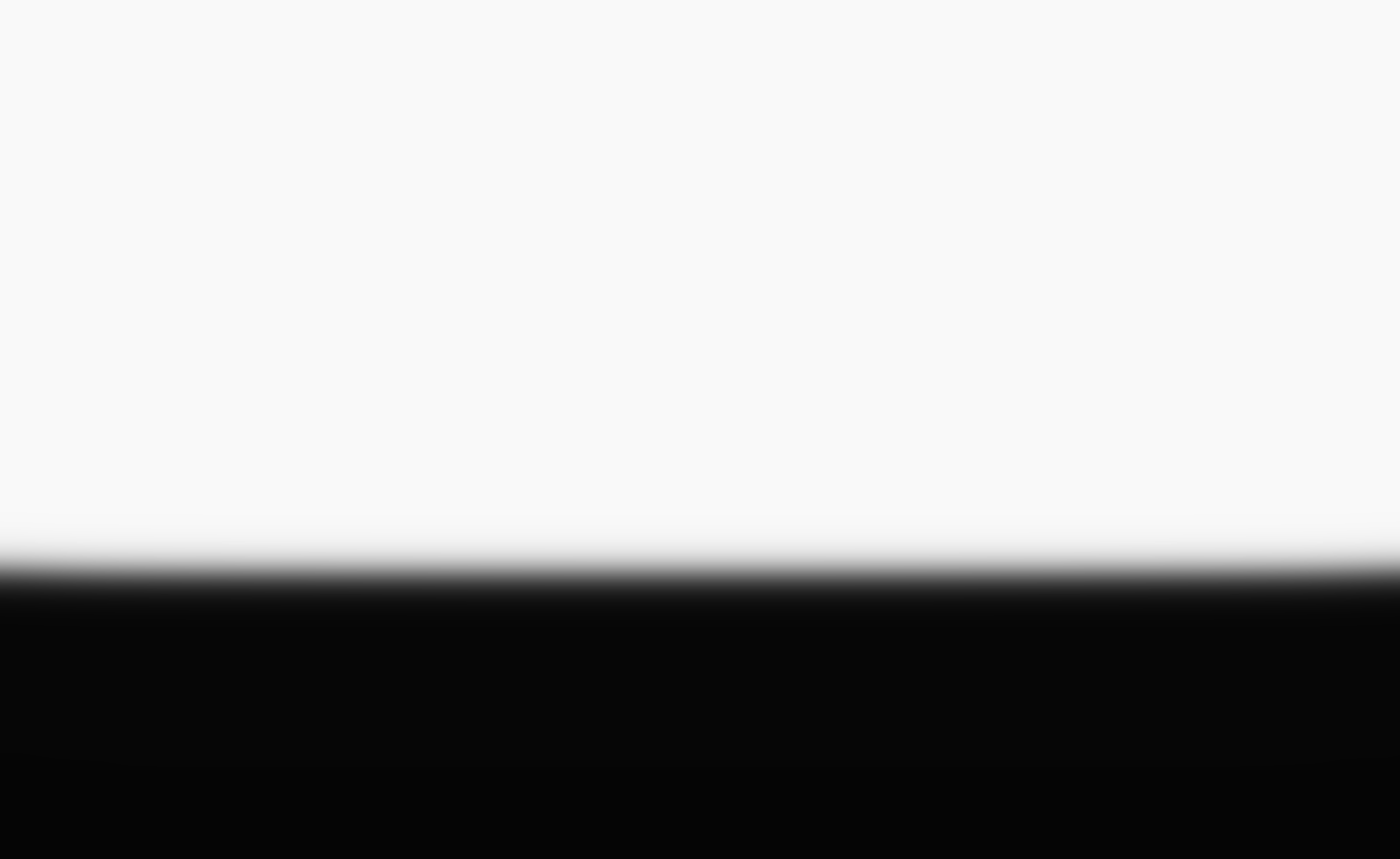}}\\

    \fbox{\includegraphics[width=32mm]{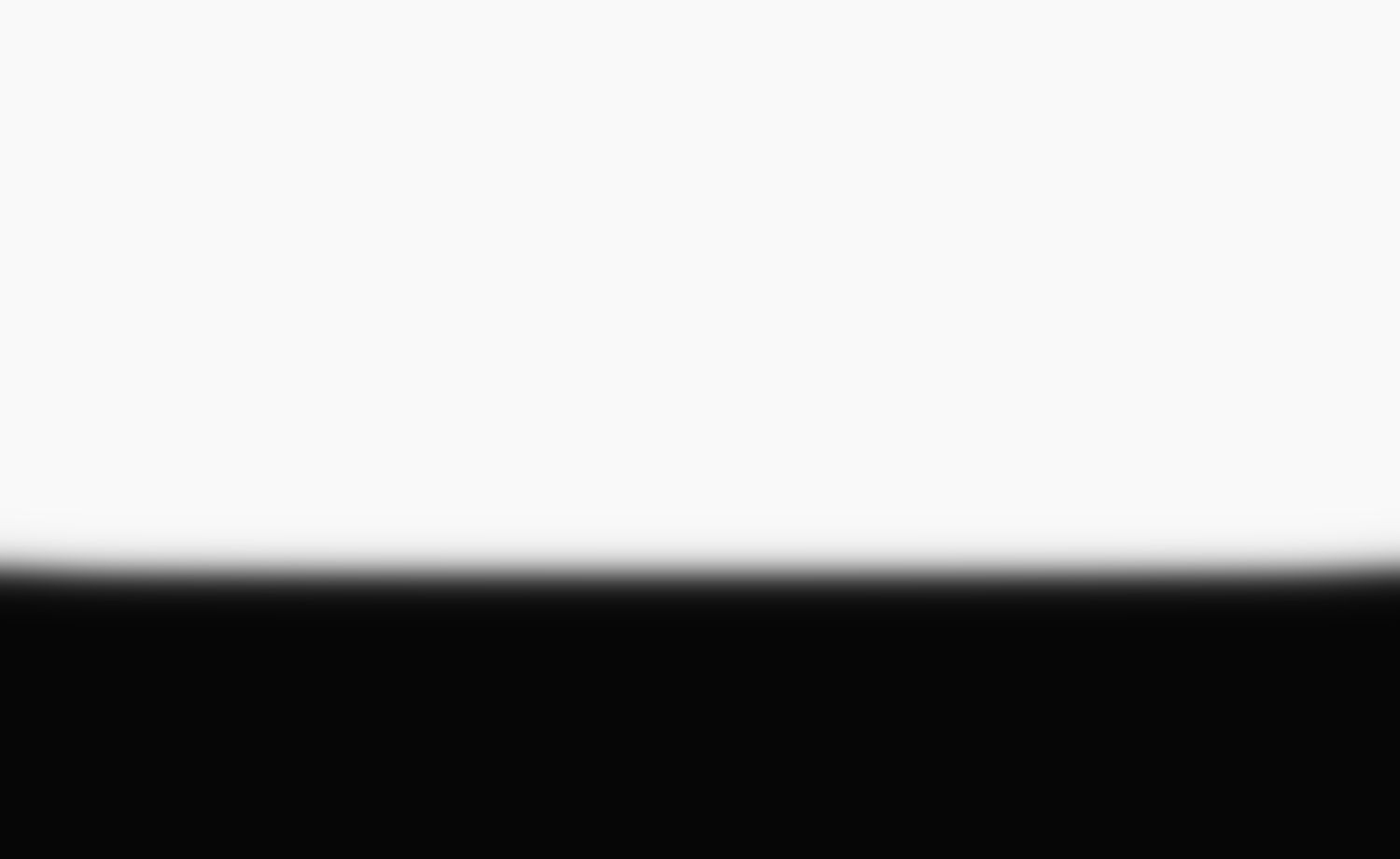}}&
    \fbox{\includegraphics[width=32mm]{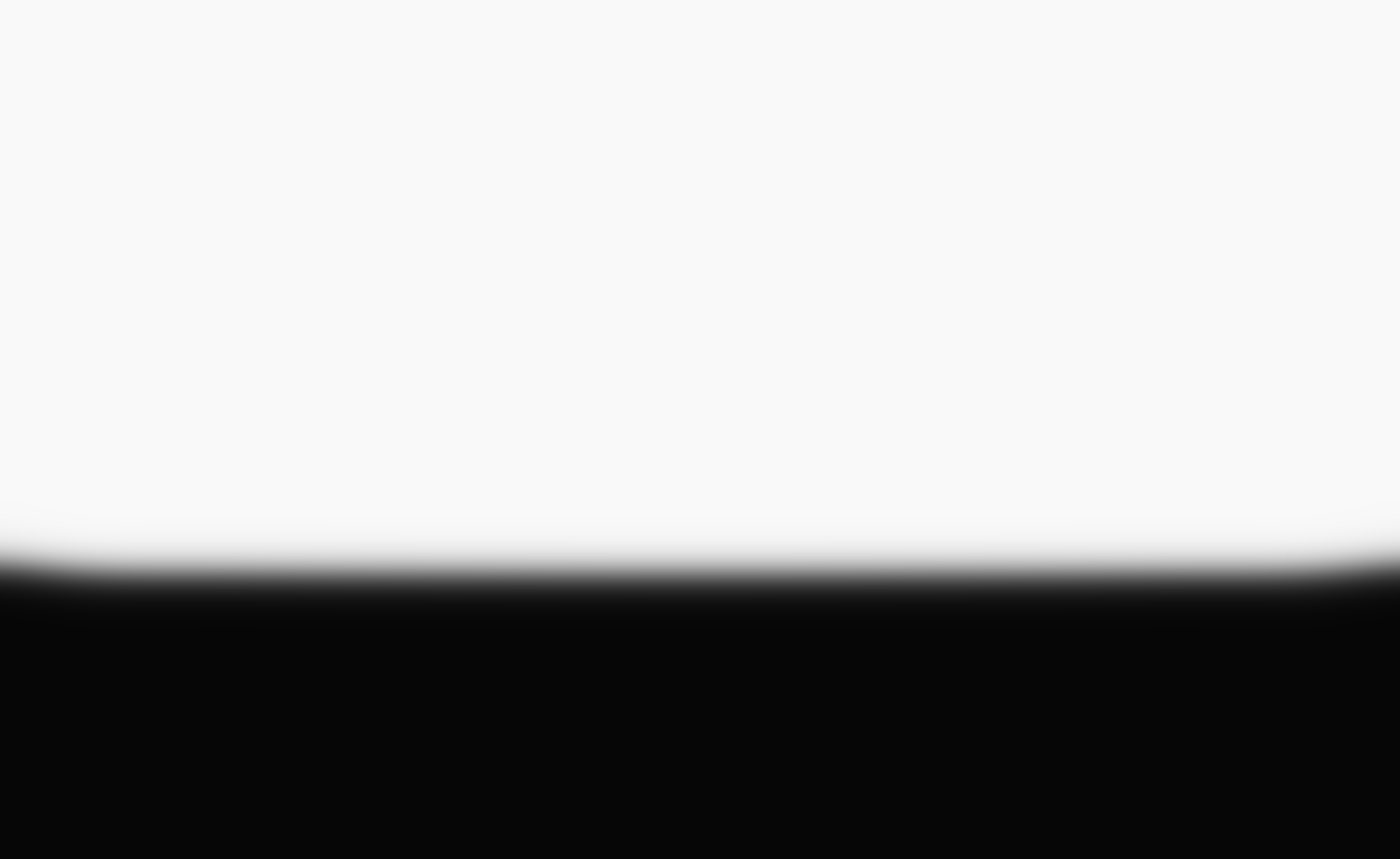}}&
    \fbox{\includegraphics[width=32mm]{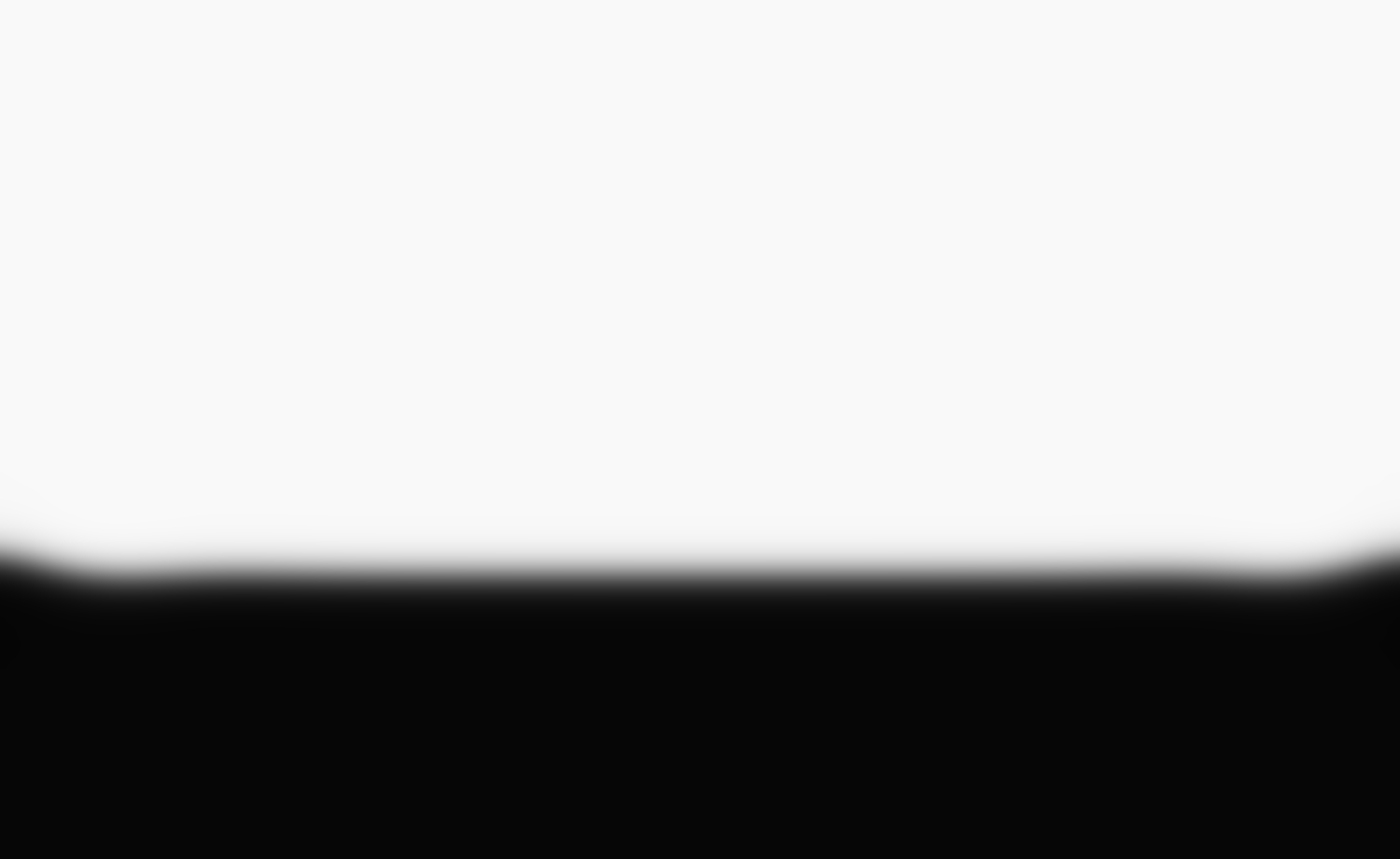}}&
    \fbox{\includegraphics[width=32mm]{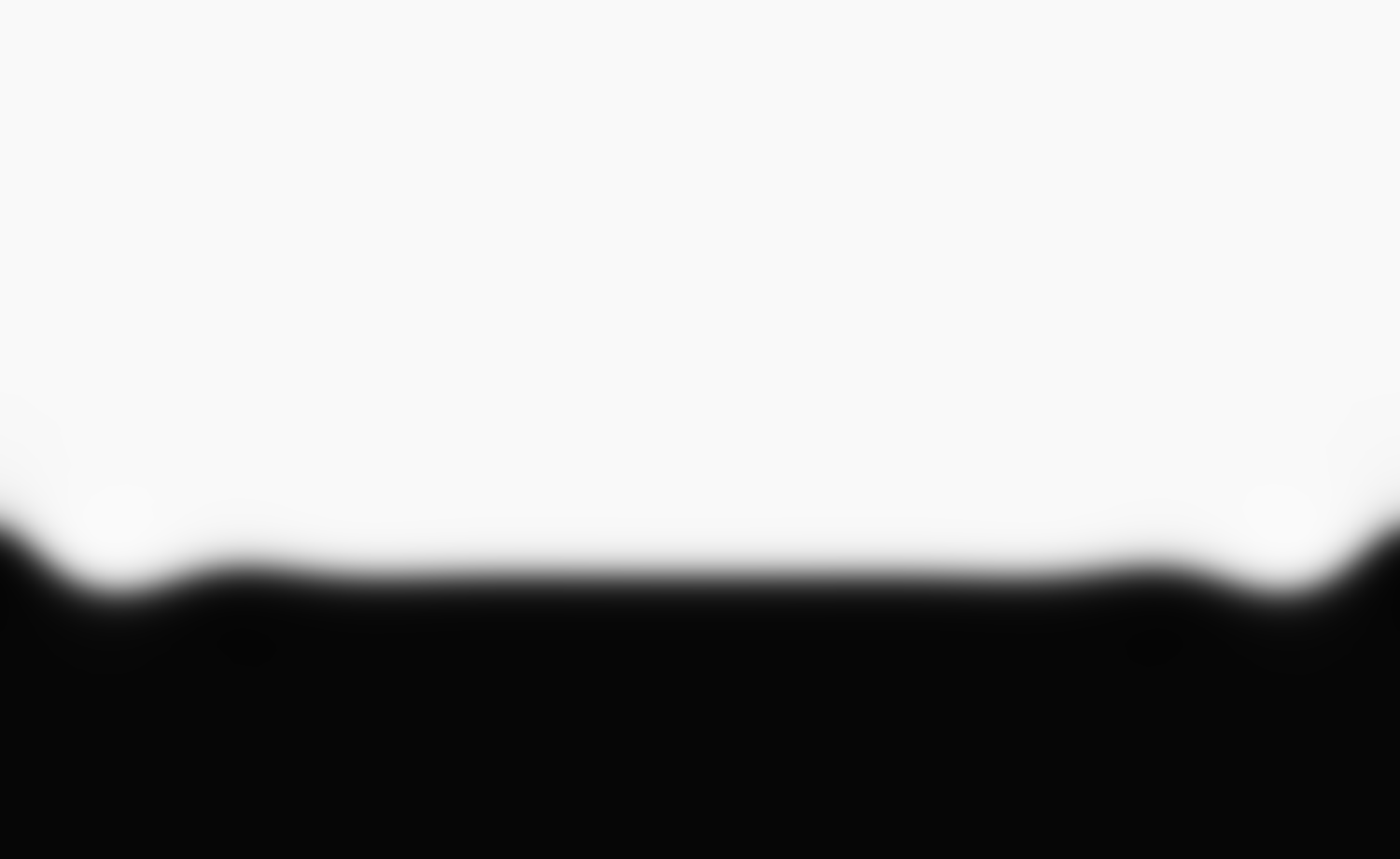}}\\

    \fbox{\includegraphics[width=32mm]{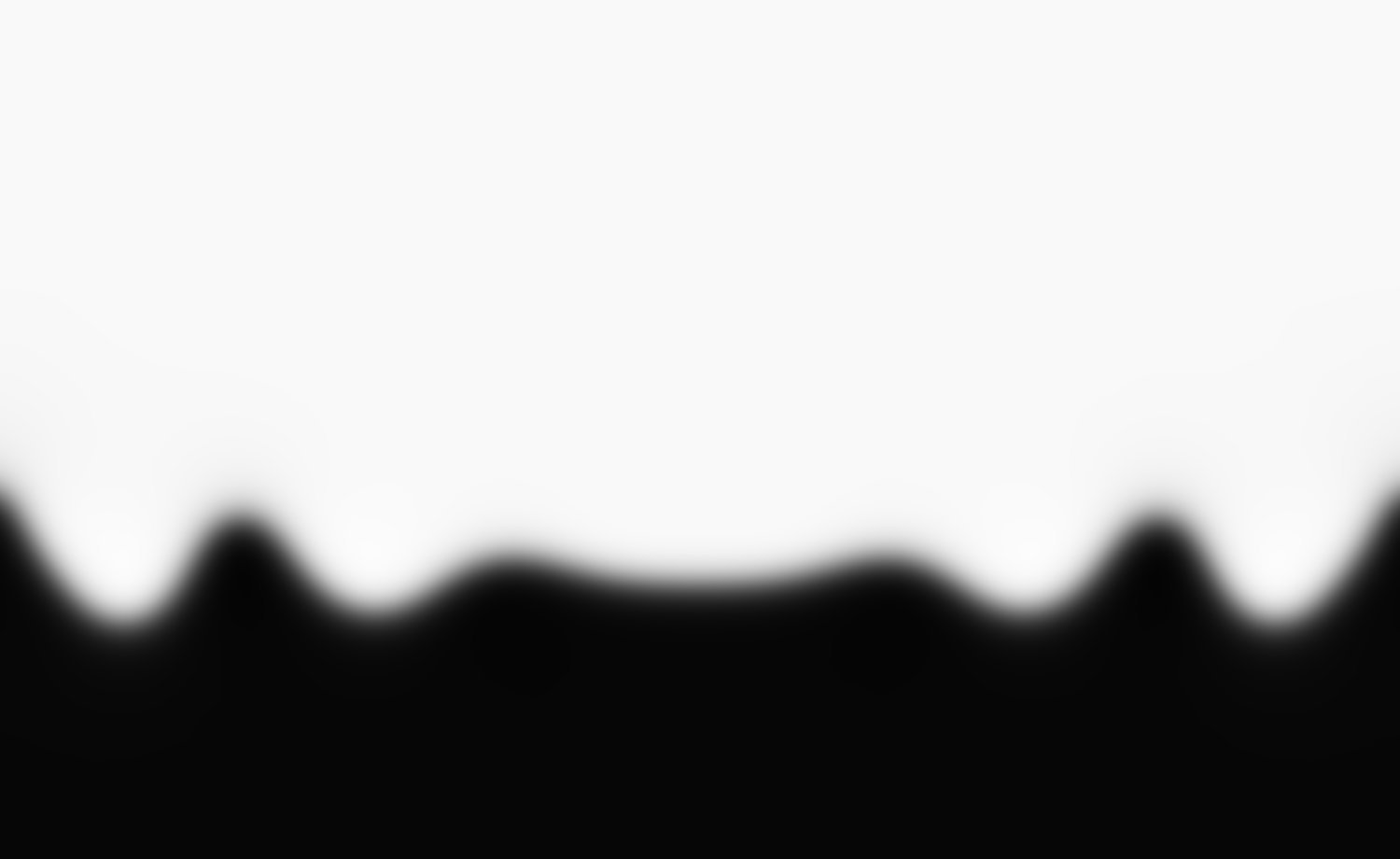}}&
    \fbox{\includegraphics[width=32mm]{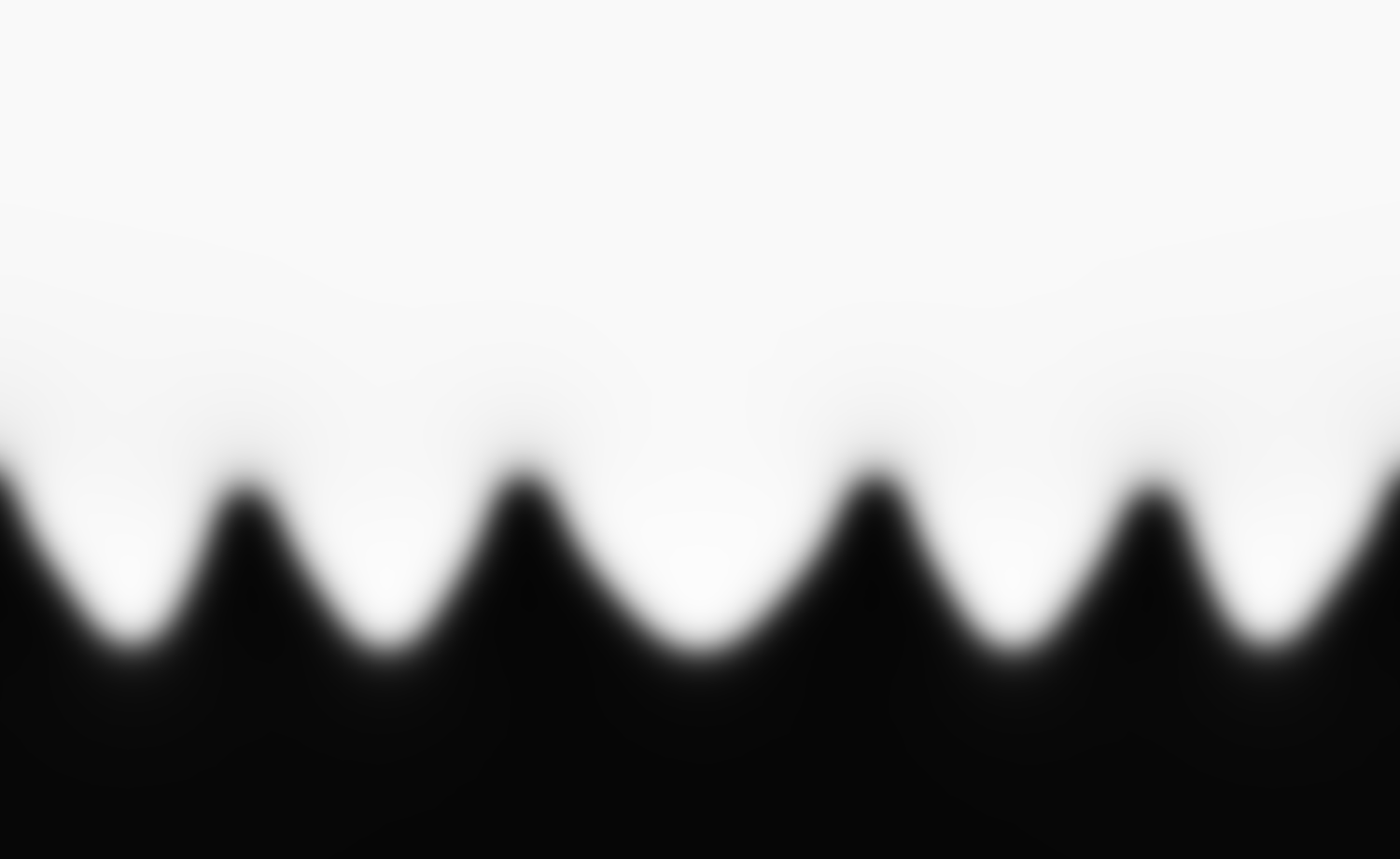}}&
    \fbox{\includegraphics[width=32mm]{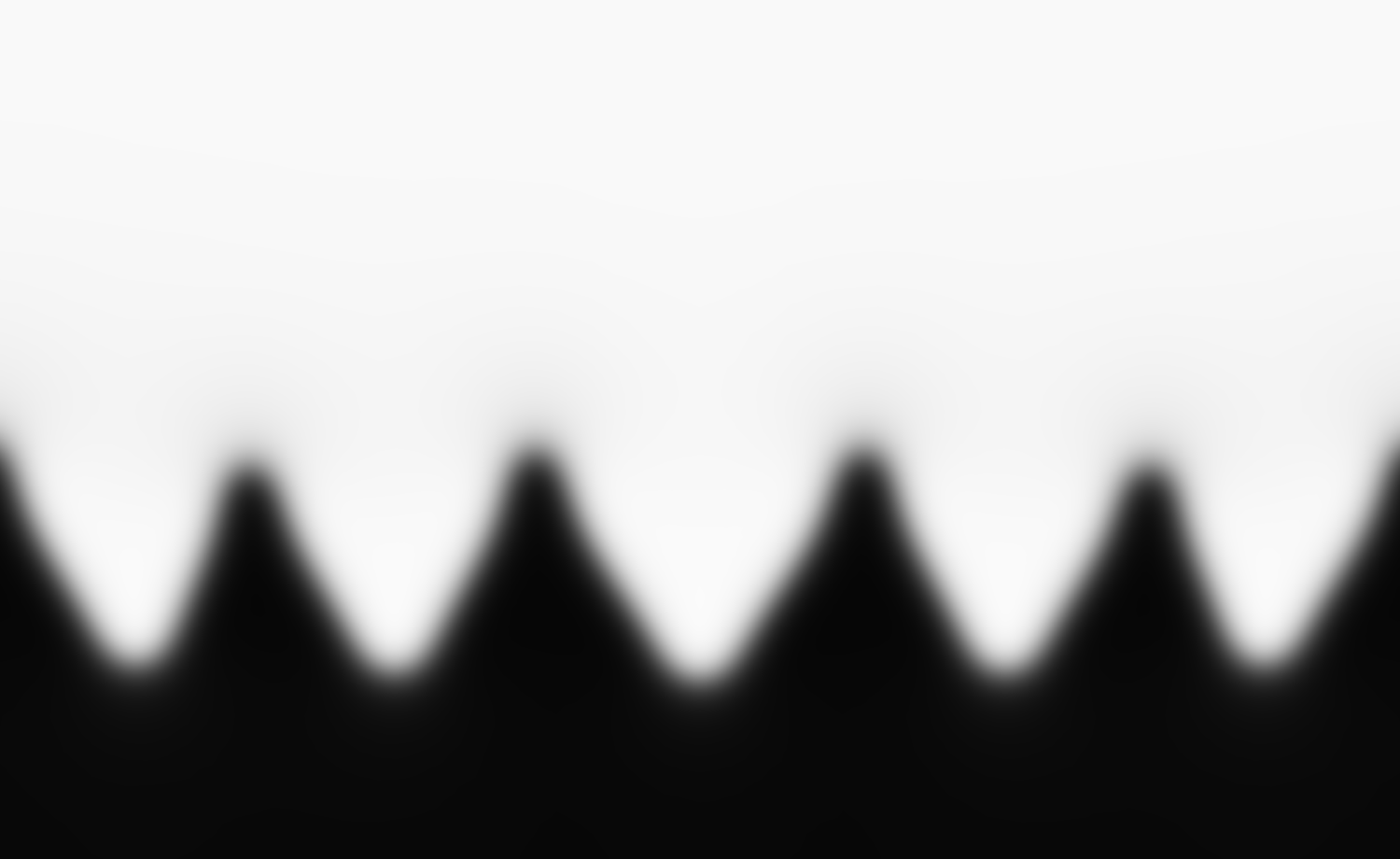}}&
    \fbox{\includegraphics[width=32mm]{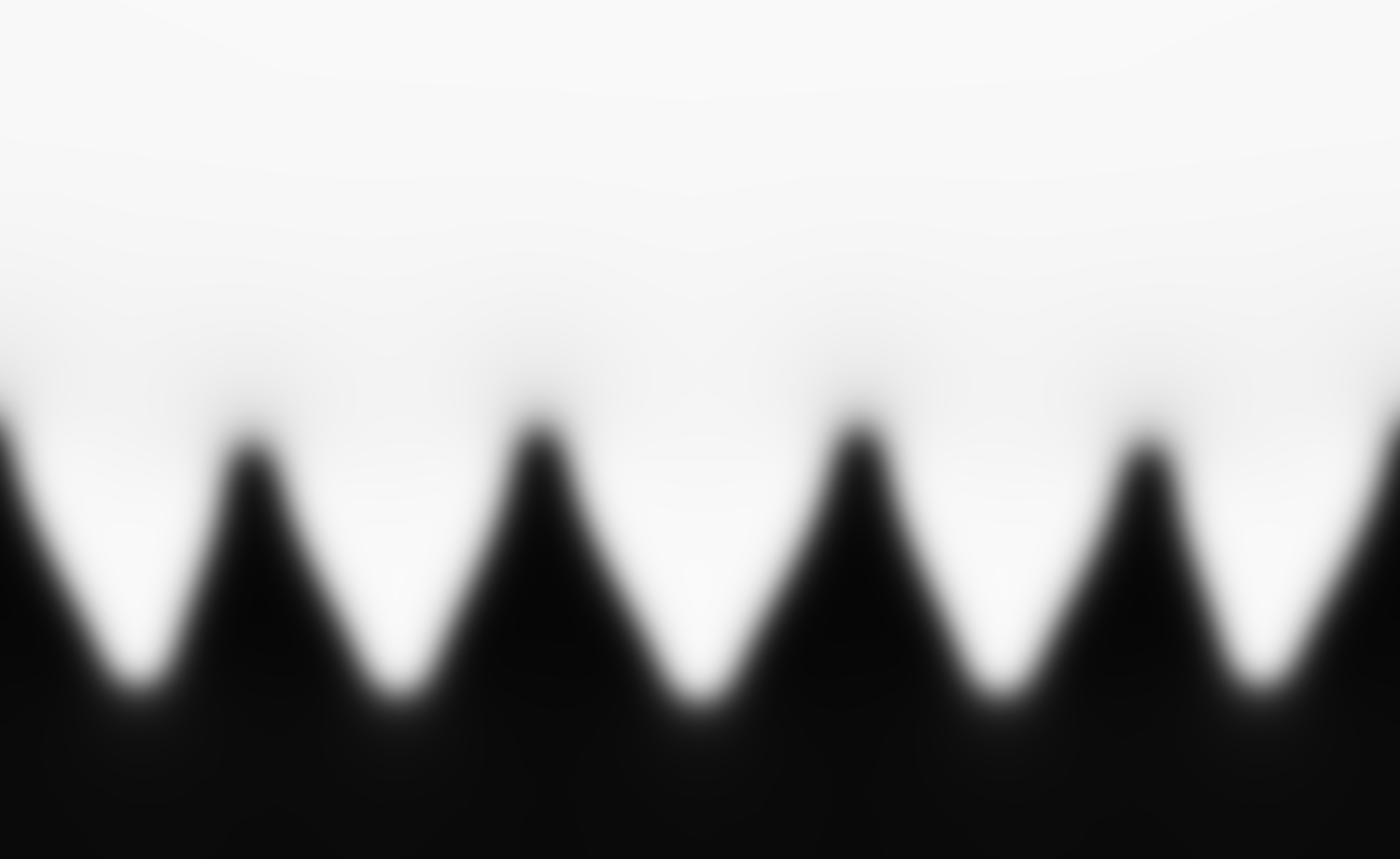}}\\

    \fbox{\includegraphics[width=32mm]{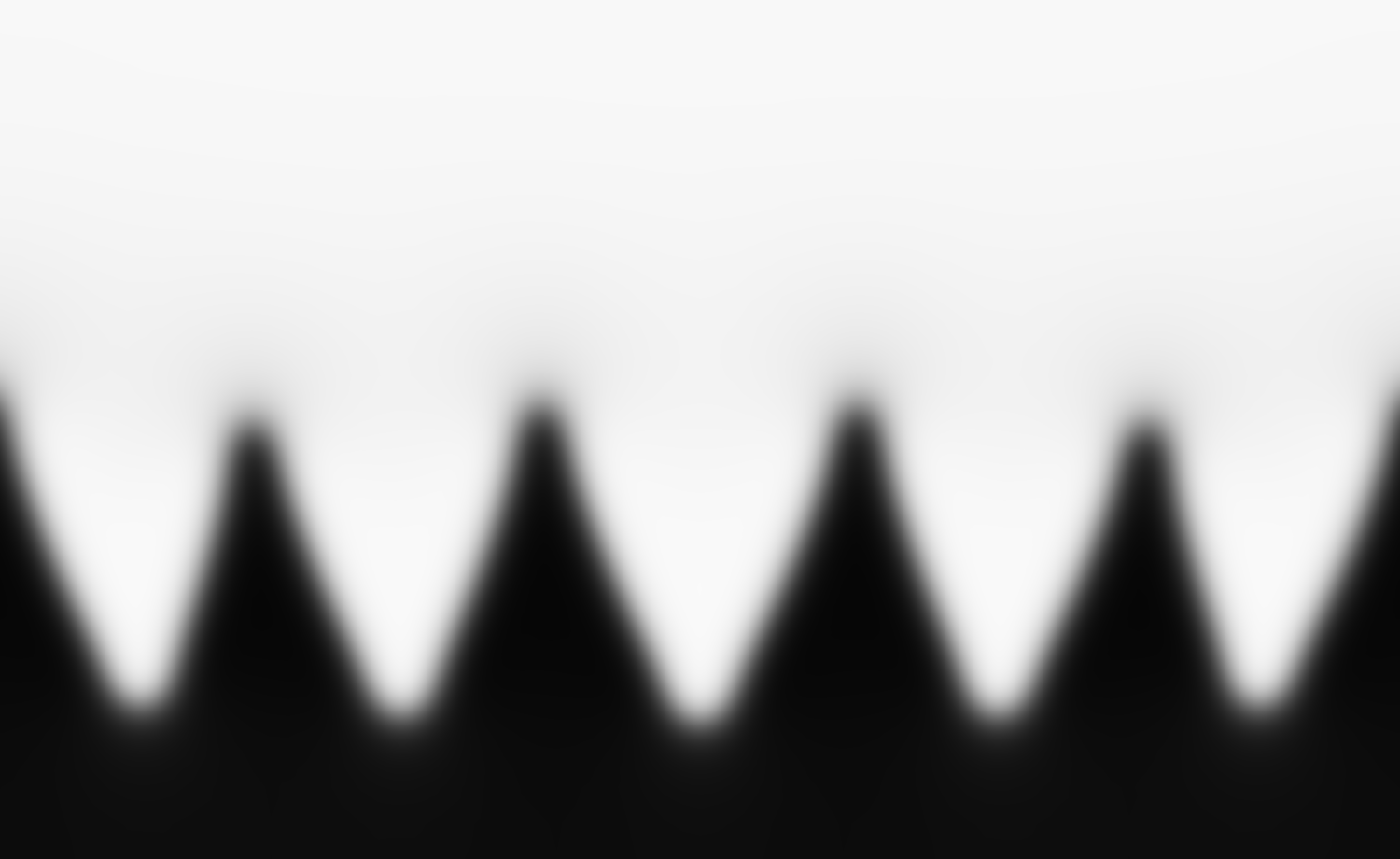}}&
    \fbox{\includegraphics[width=32mm]{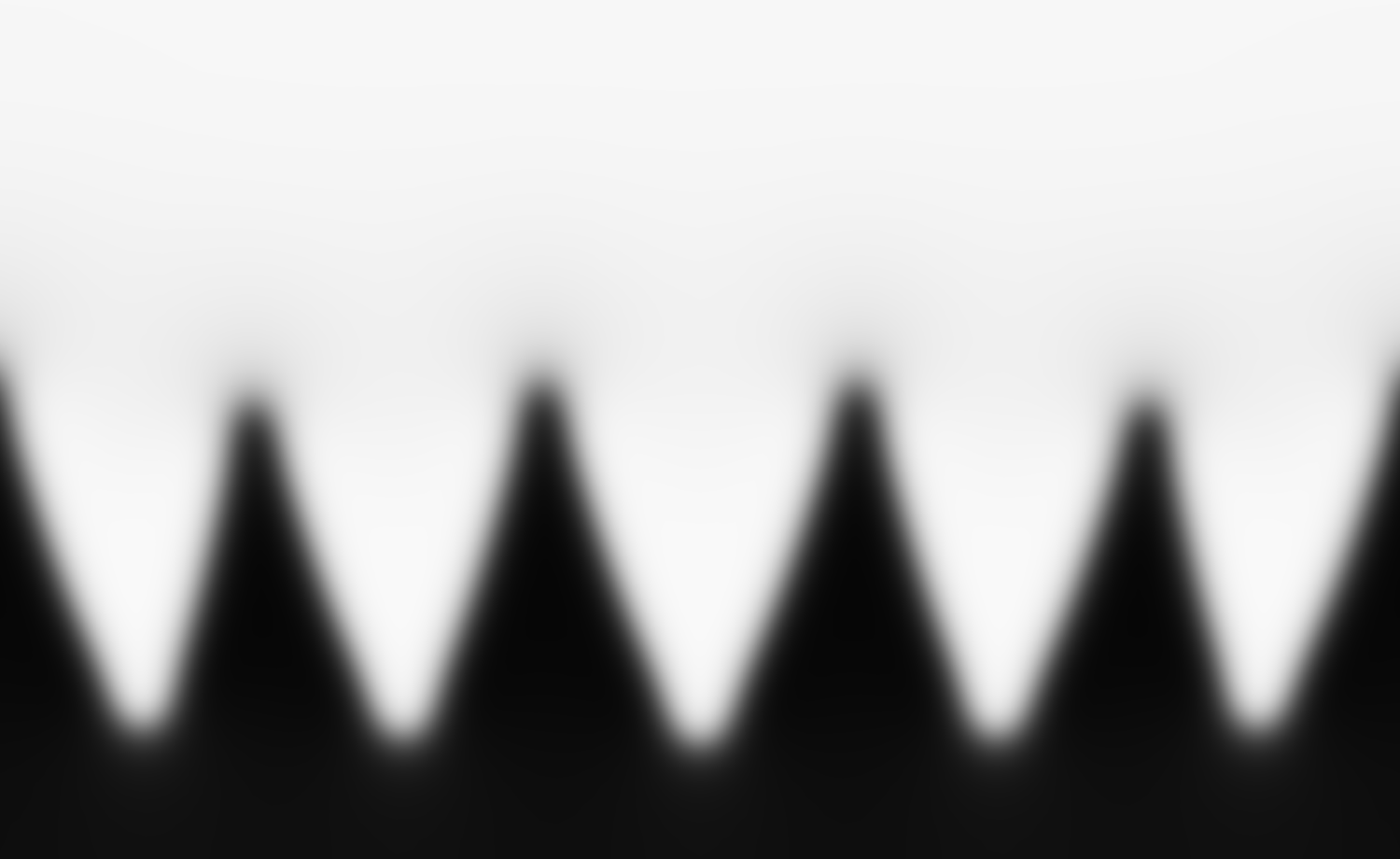}}&
    \fbox{\includegraphics[width=32mm]{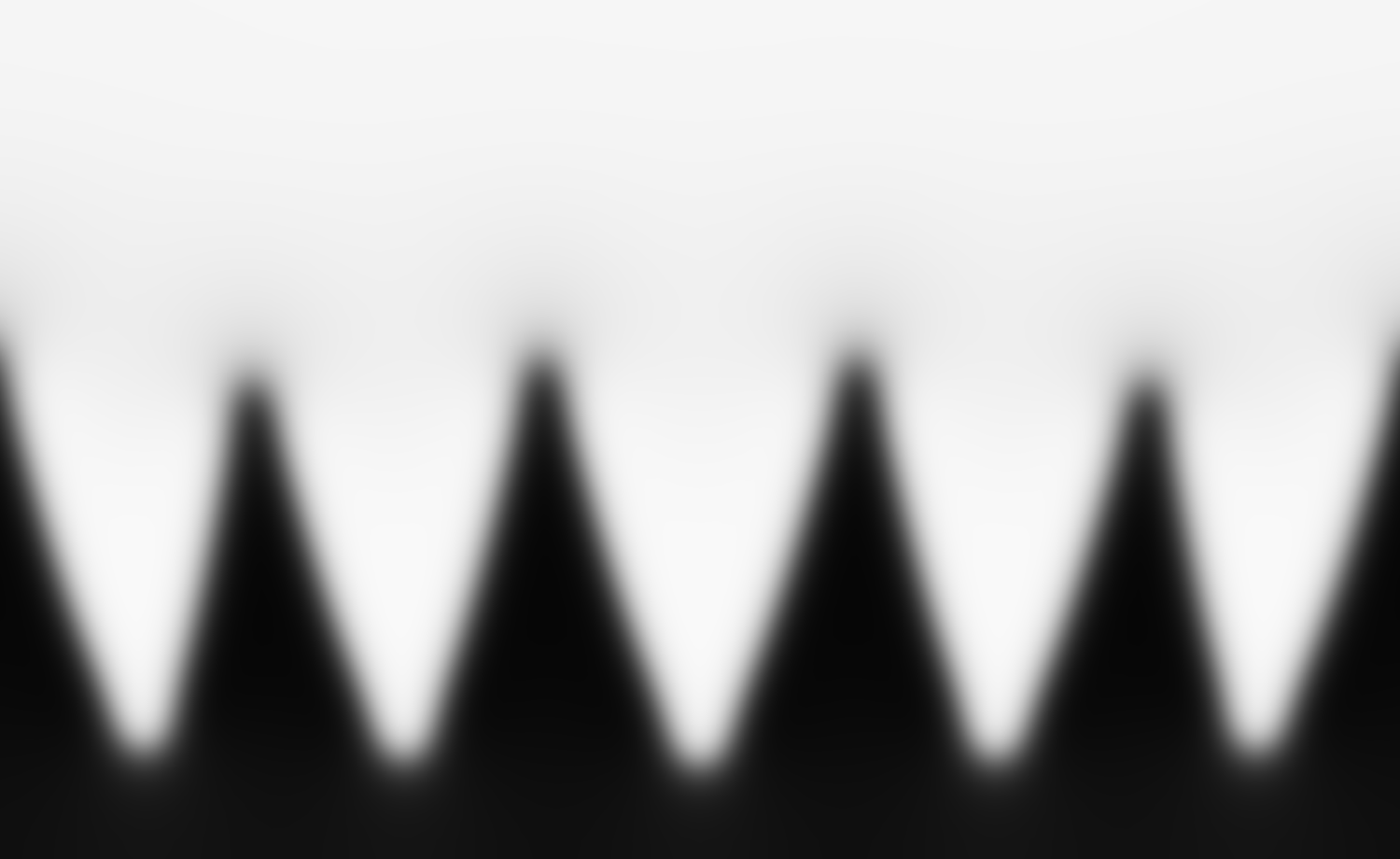}}&
    \fbox{\includegraphics[width=32mm]{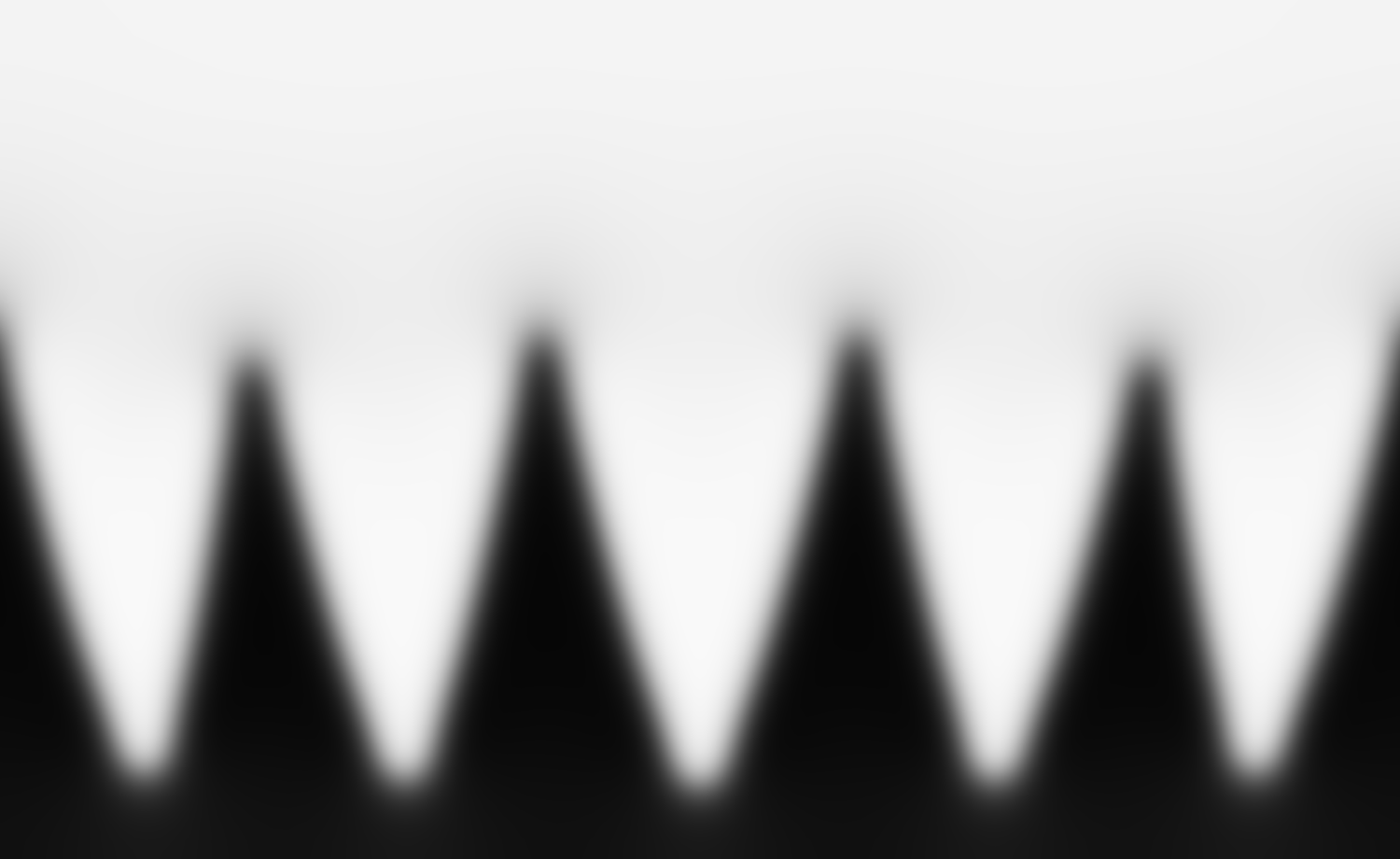}} \\

    \fbox{\includegraphics[width=32mm]{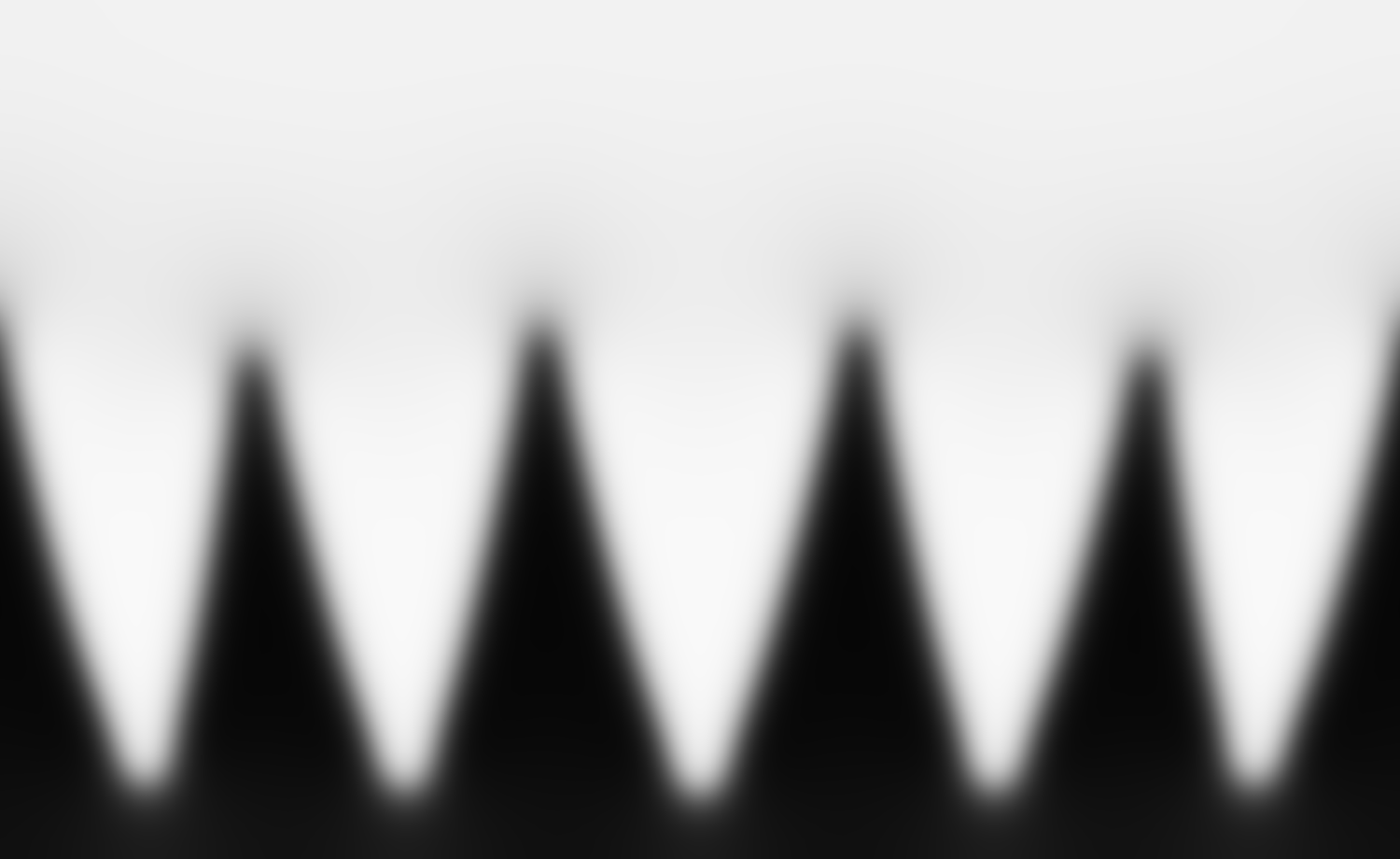}}&
    \fbox{\includegraphics[width=32mm]{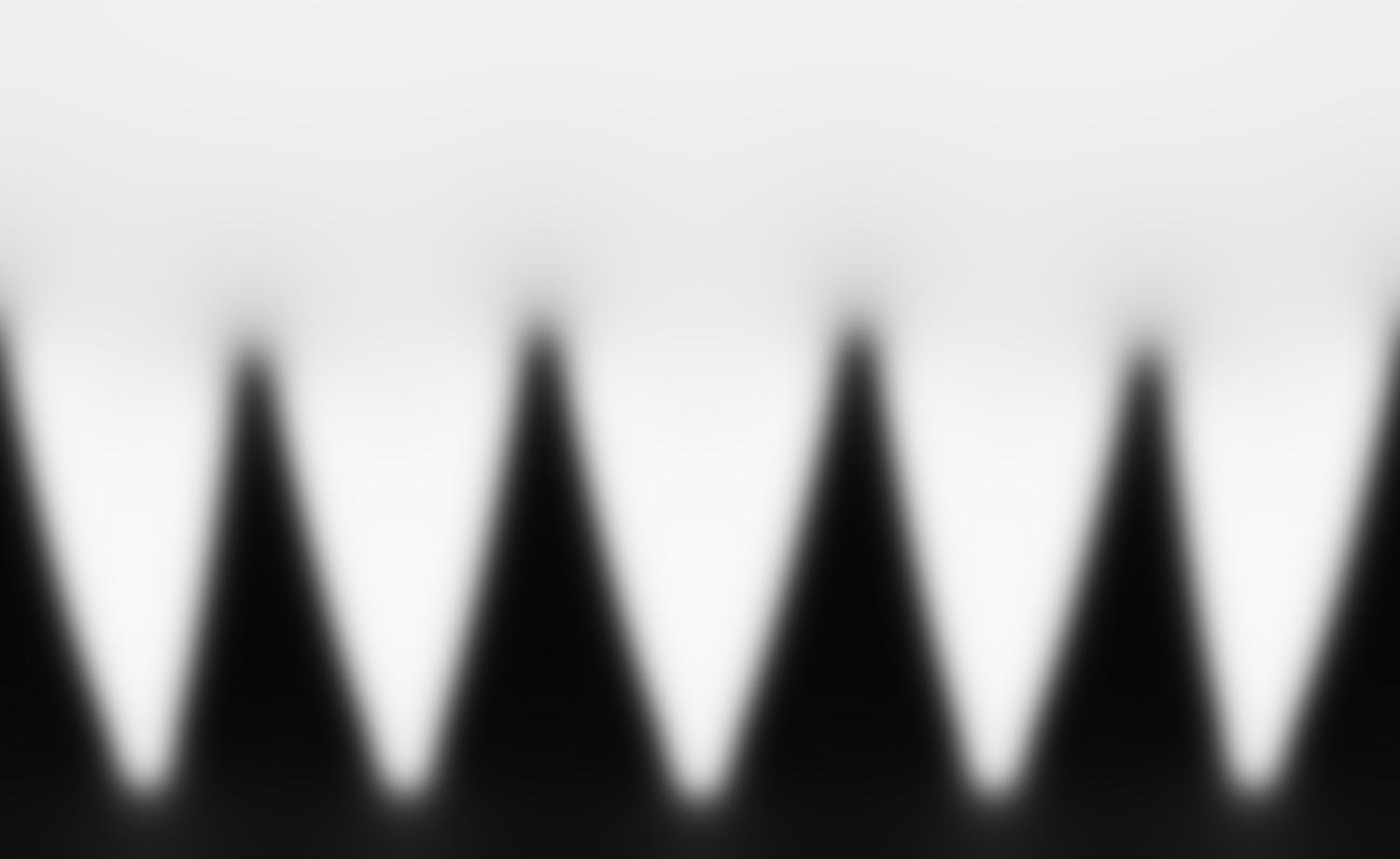}}&
    \fbox{\includegraphics[width=32mm]{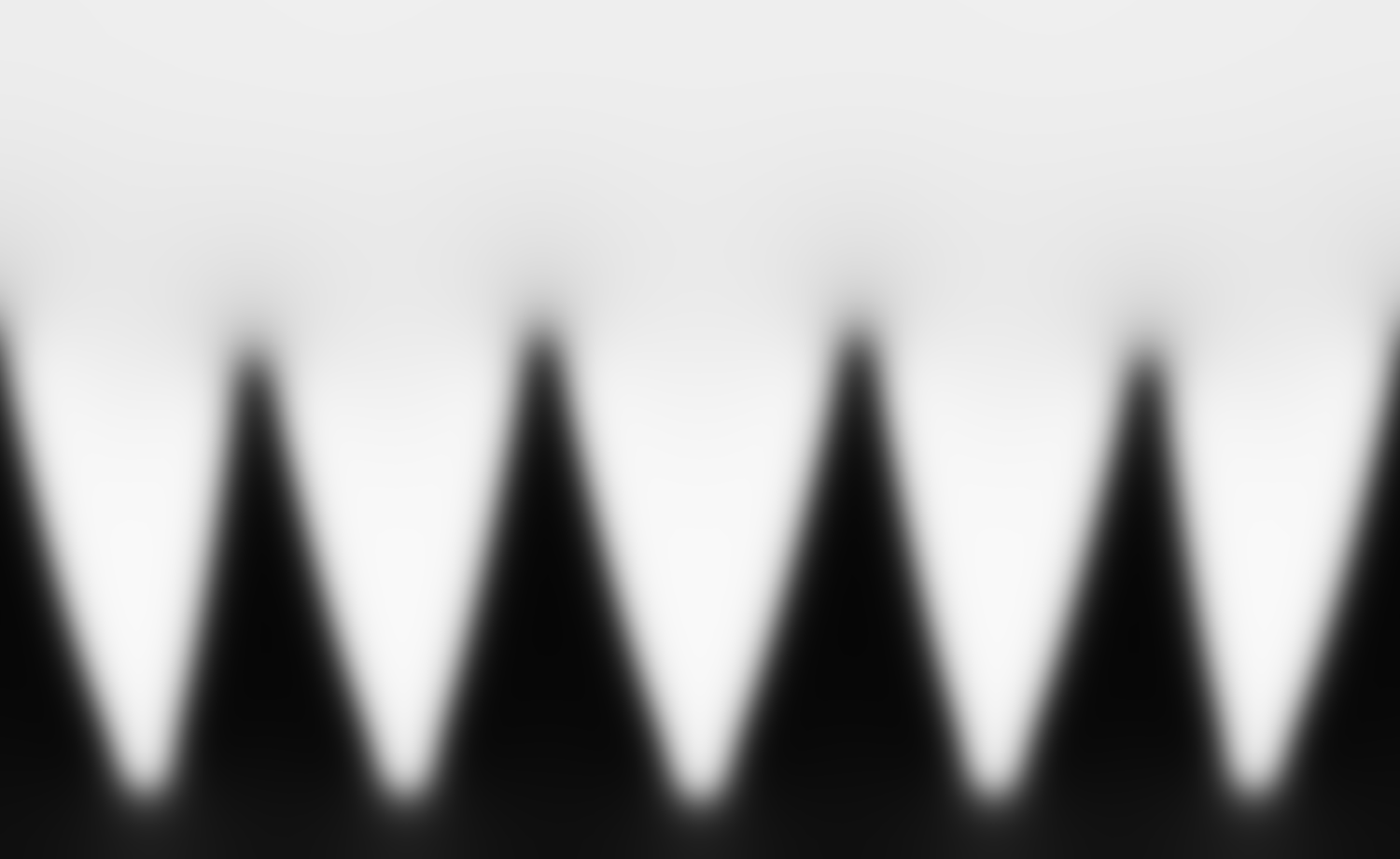}}&
    \fbox{\includegraphics[width=32mm]{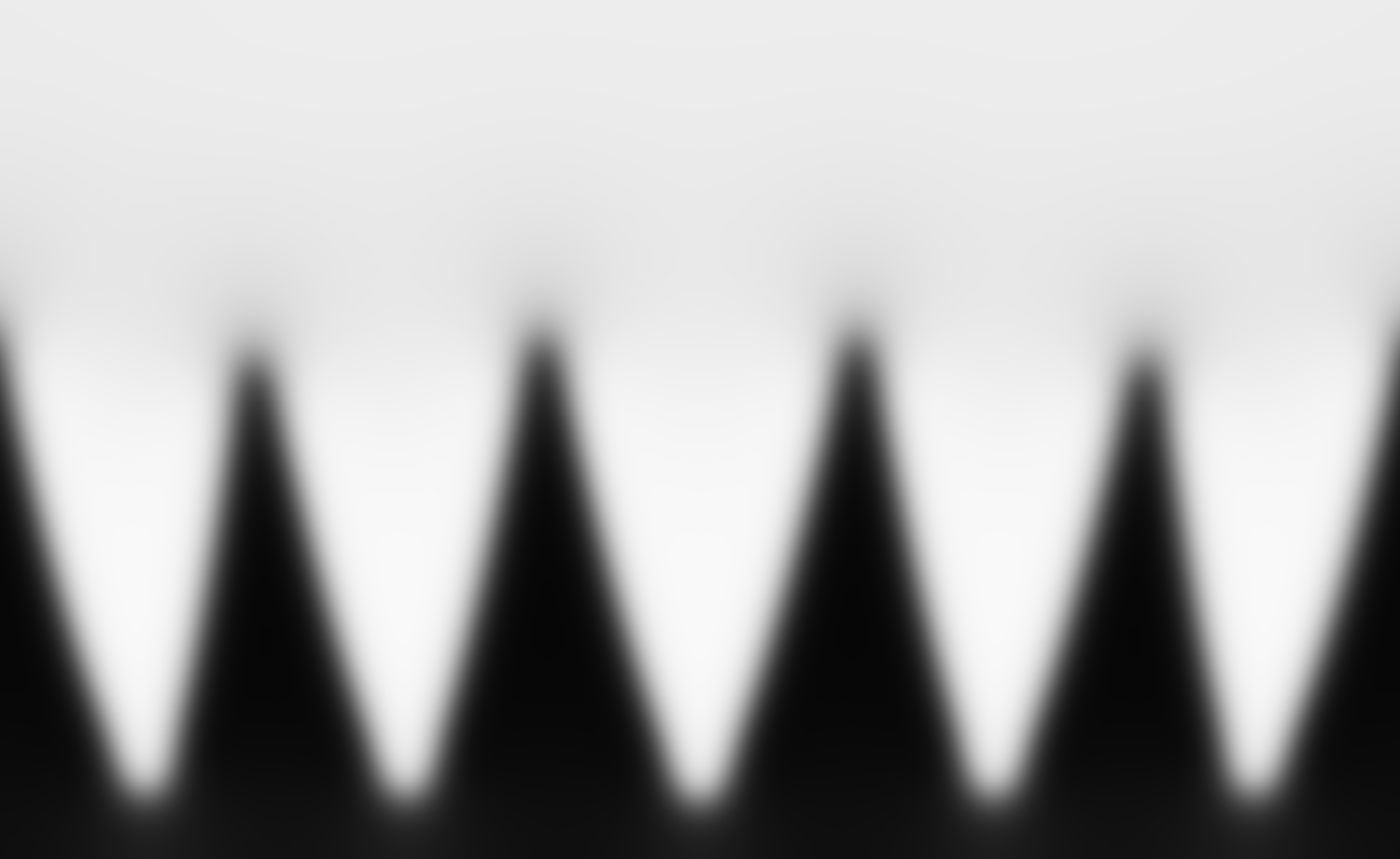}}
  \end{tabular}
\end{center}
\caption[Rosensweig instability: evolution
  screenshots.]{\textbf{Rosensweig instability: evolution
    screenshots.} Sequence of screenshots from time $t= 0.1$ to $t =
  2.0$ in regular time intervals of $0.1$ showing the evolution of the
  phase-field variable $\phvarh$ (read from left-to-right and
  top-to-bottom). As it can be appreciated, we obtained in the order
  of $4$ peaks inside the box, showing us that the crude estimate
  \eqref{gravity} was a very good initial guess for the scaling
  between the capillary coefficient $\capcoeff$ and the gravity
  $g$. Note that diffusive effects are quite noticeable as we are
  using $\layerthick = 0.01$. Most of the interesting transient
  behavior happens from time $t = 0.7$ to $t = 1.3$ (reading from
  left-to-right and top-to-bottom: boxes 7 to 13), so we will focus on
  this interval for a parametric study. This simulation was obtained
  using 6 refinement levels in space, and 4000 time steps for a
  total of $t_F=2$ time units.
  \label{figureResolved}}
\end{figure}

\begin{figure}
\begin{center}
\setlength\fboxsep{0pt}
\setlength\fboxrule{1pt}
\includegraphics[scale=0.25]{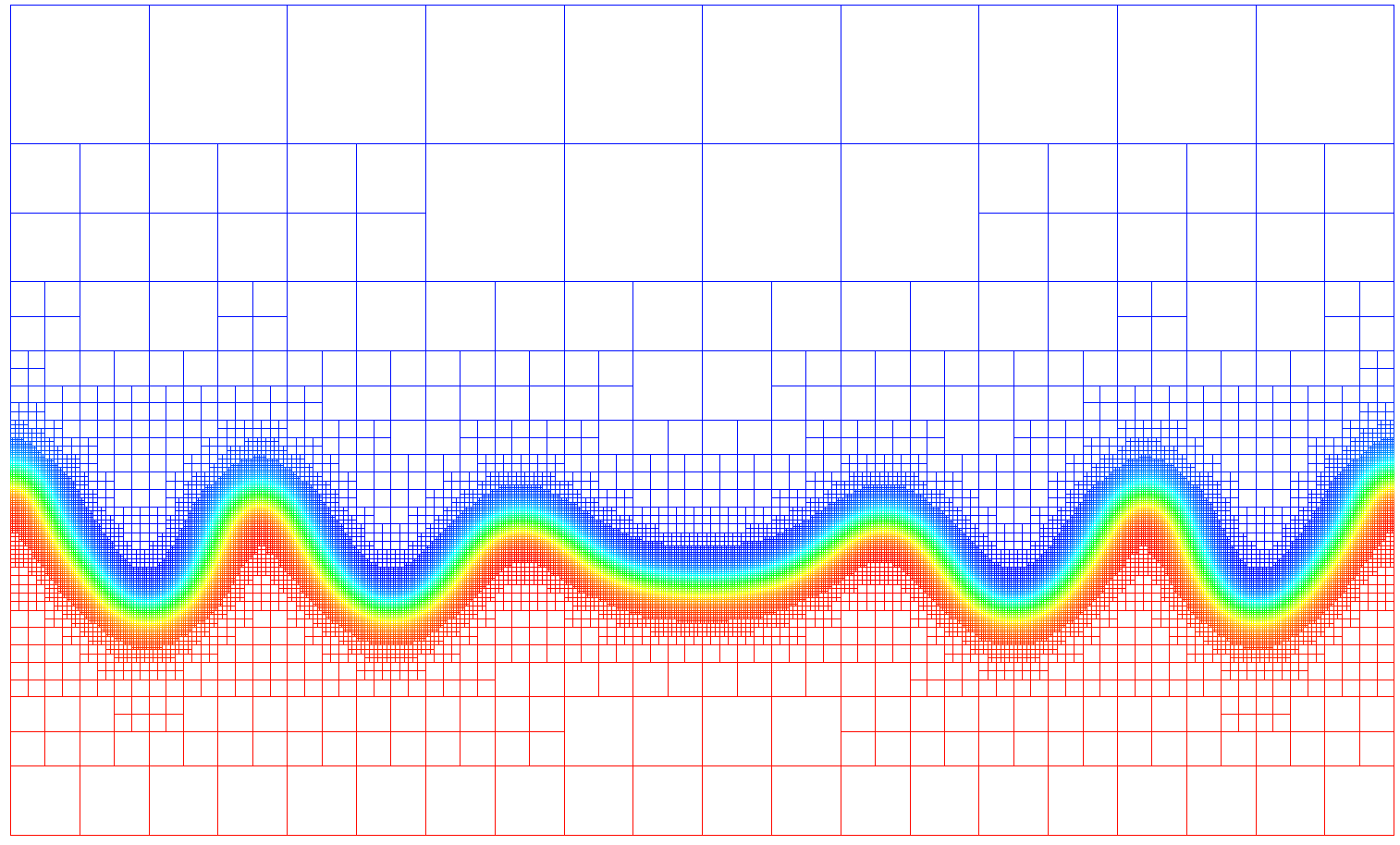}
\end{center}
\caption[Rosensweig instability: sample graded
  mesh]{\textbf{Rosensweig instability: sample graded mesh.} Finite
  element mesh with 6 levels of refinement at time $t=0.92$,
  corresponding with the simulation of Figure \ref{figureResolved}. In
  order to have physically meaningful simulations we must resolve
  the transition layer well. We thus enforce approximately 20 elements of the finest level resolving the transition layer. \label{meshsamplefig}}
\end{figure}

\begin{figure}
\begin{center}
    \setlength\fboxsep{0pt}
    \setlength\fboxrule{1pt}
  \begin{tabular}{cccc}

    \fbox{\includegraphics[width=32mm]{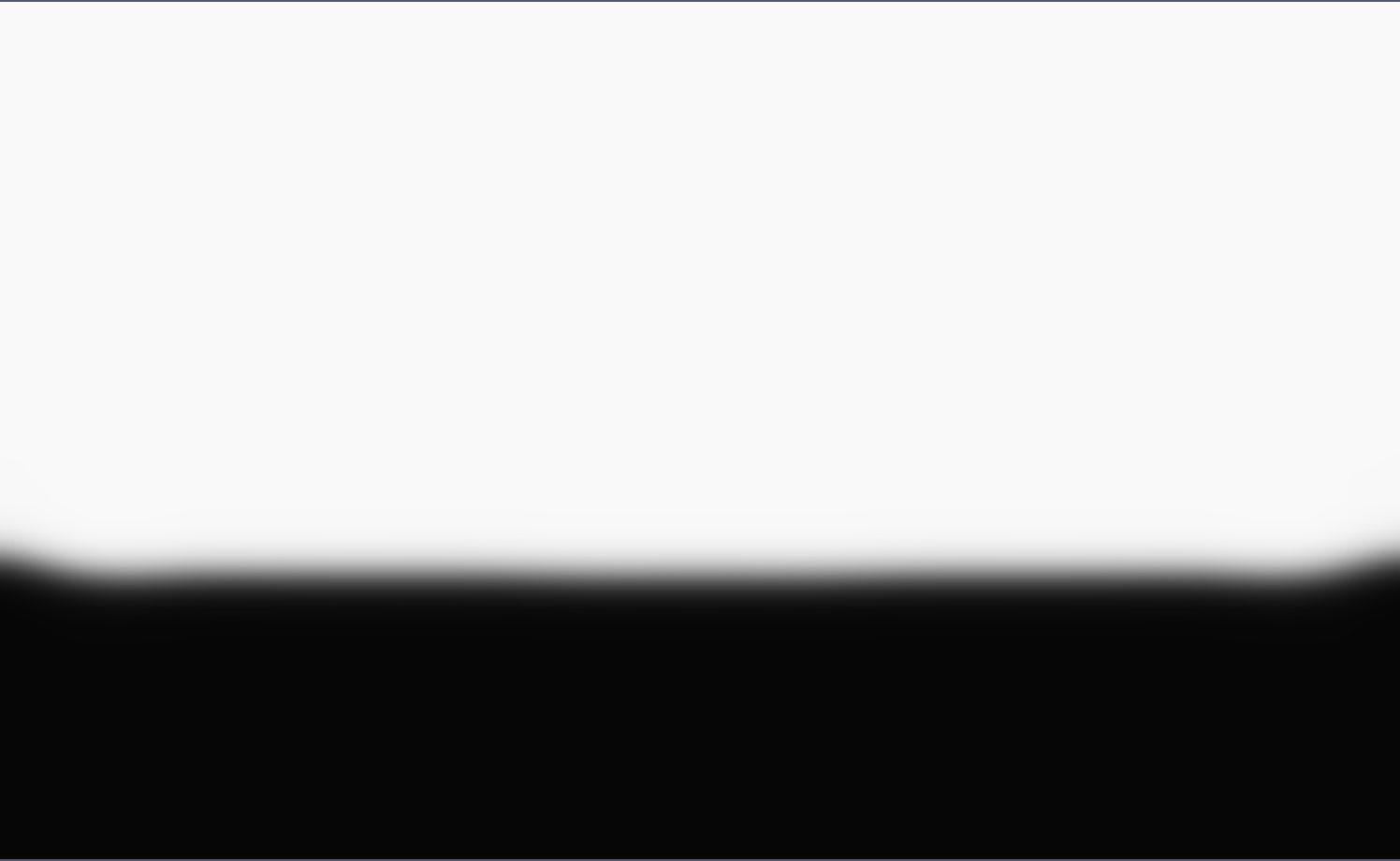}}&
    \fbox{\includegraphics[width=32mm]{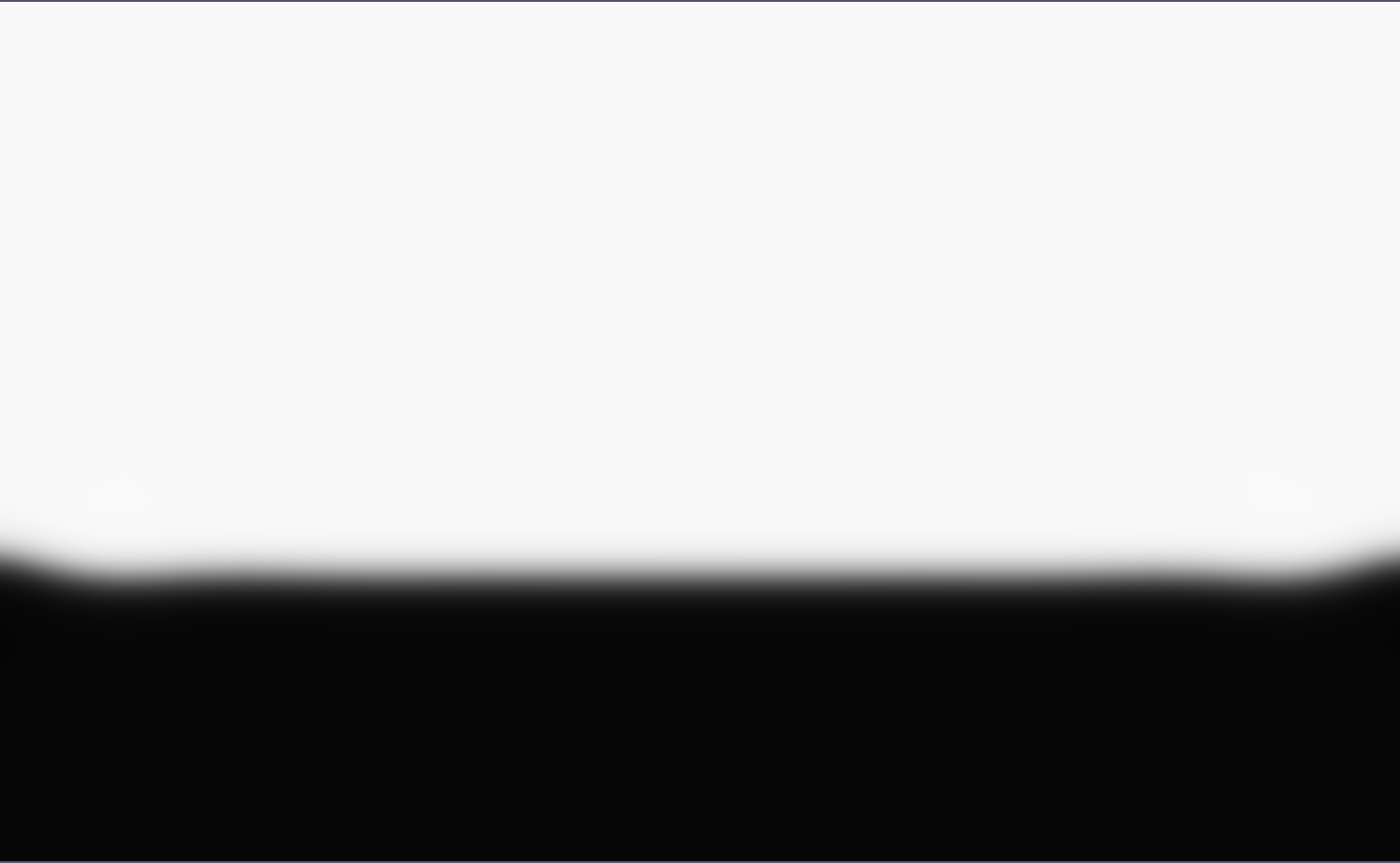}}&
    \fbox{\includegraphics[width=32mm]{Pictures/6L4000ts/20seq0007.png}}&
    \fbox{\includegraphics[width=32mm]{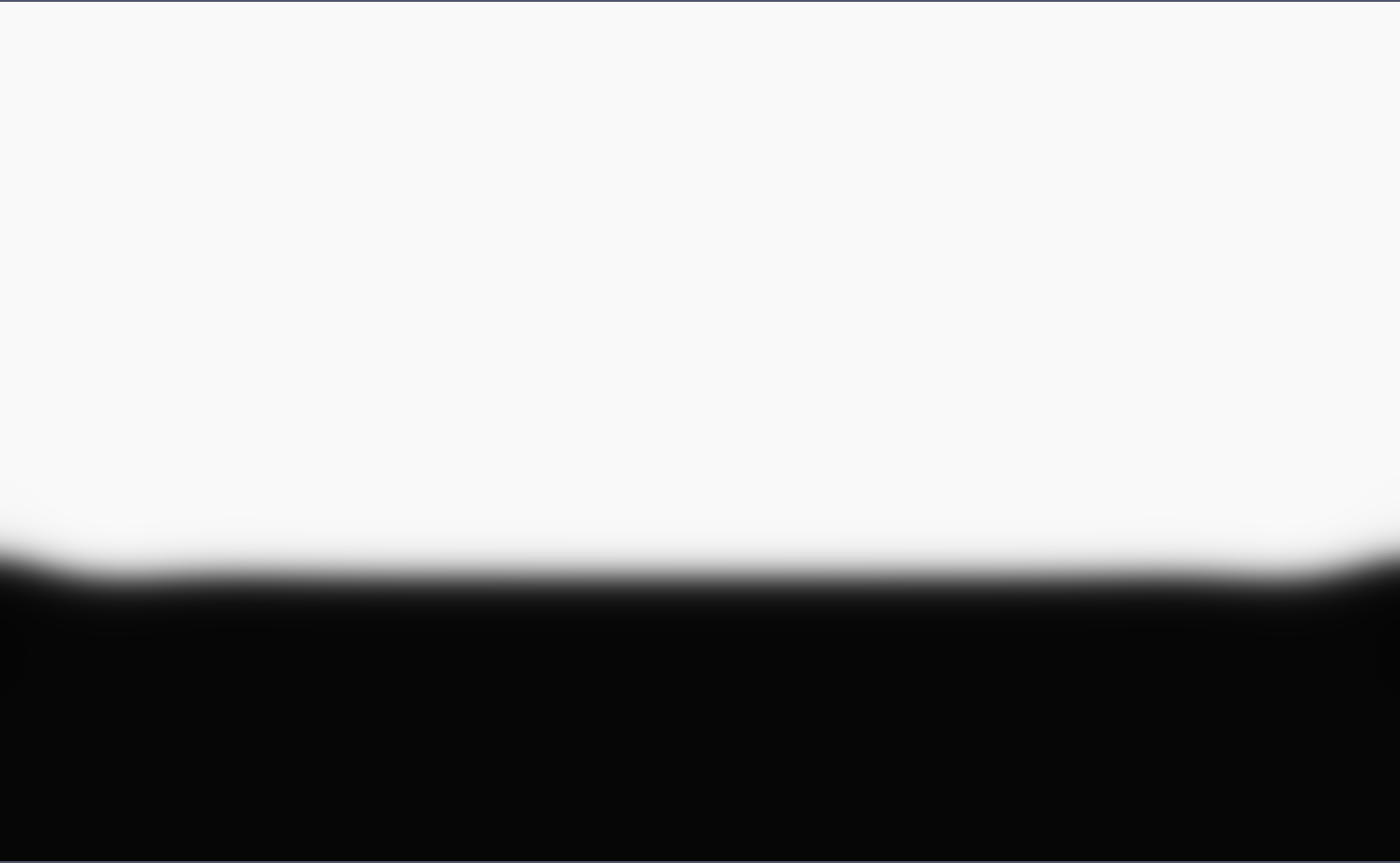}}\\

    \fbox{\includegraphics[width=32mm]{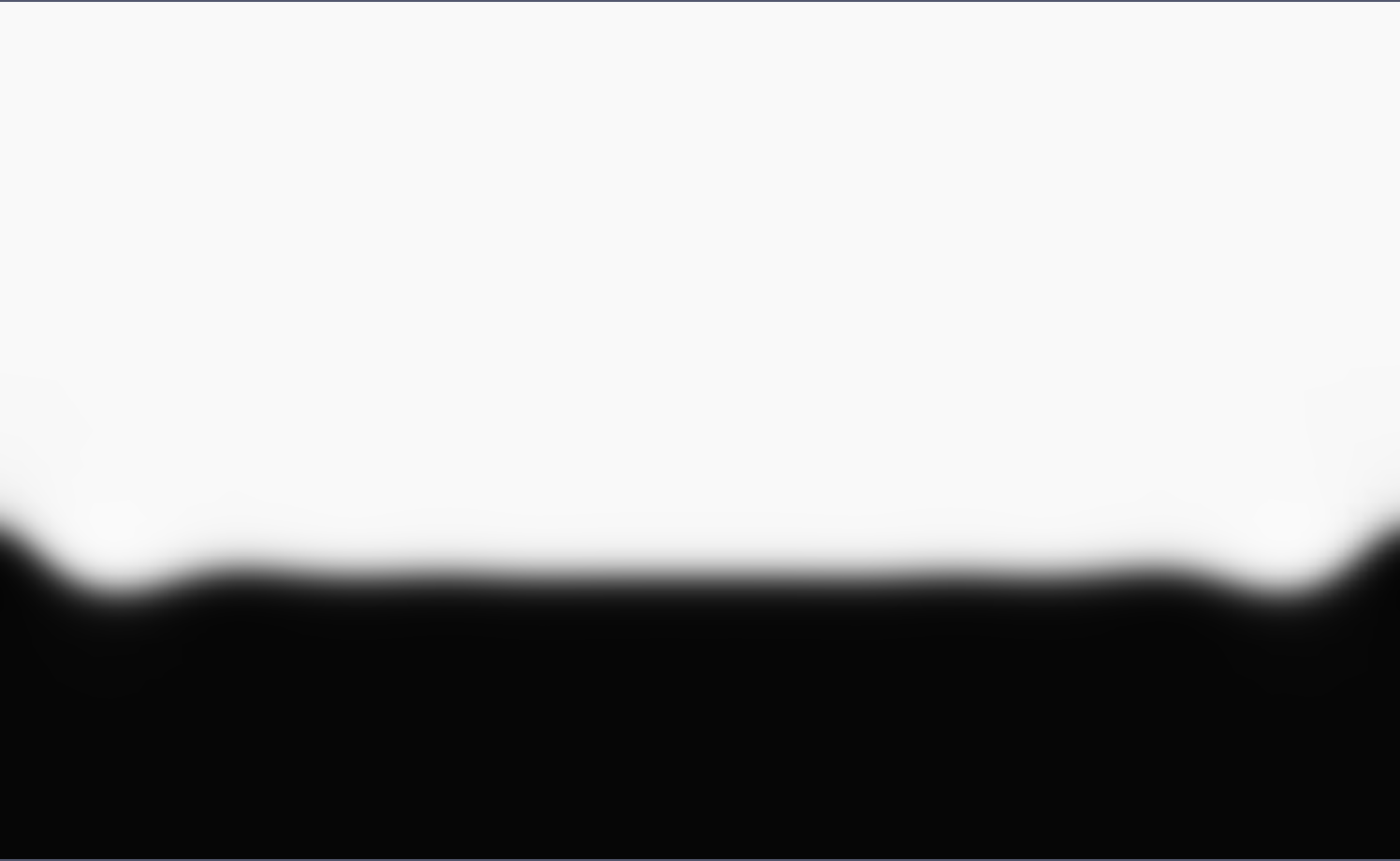}}&
    \fbox{\includegraphics[width=32mm]{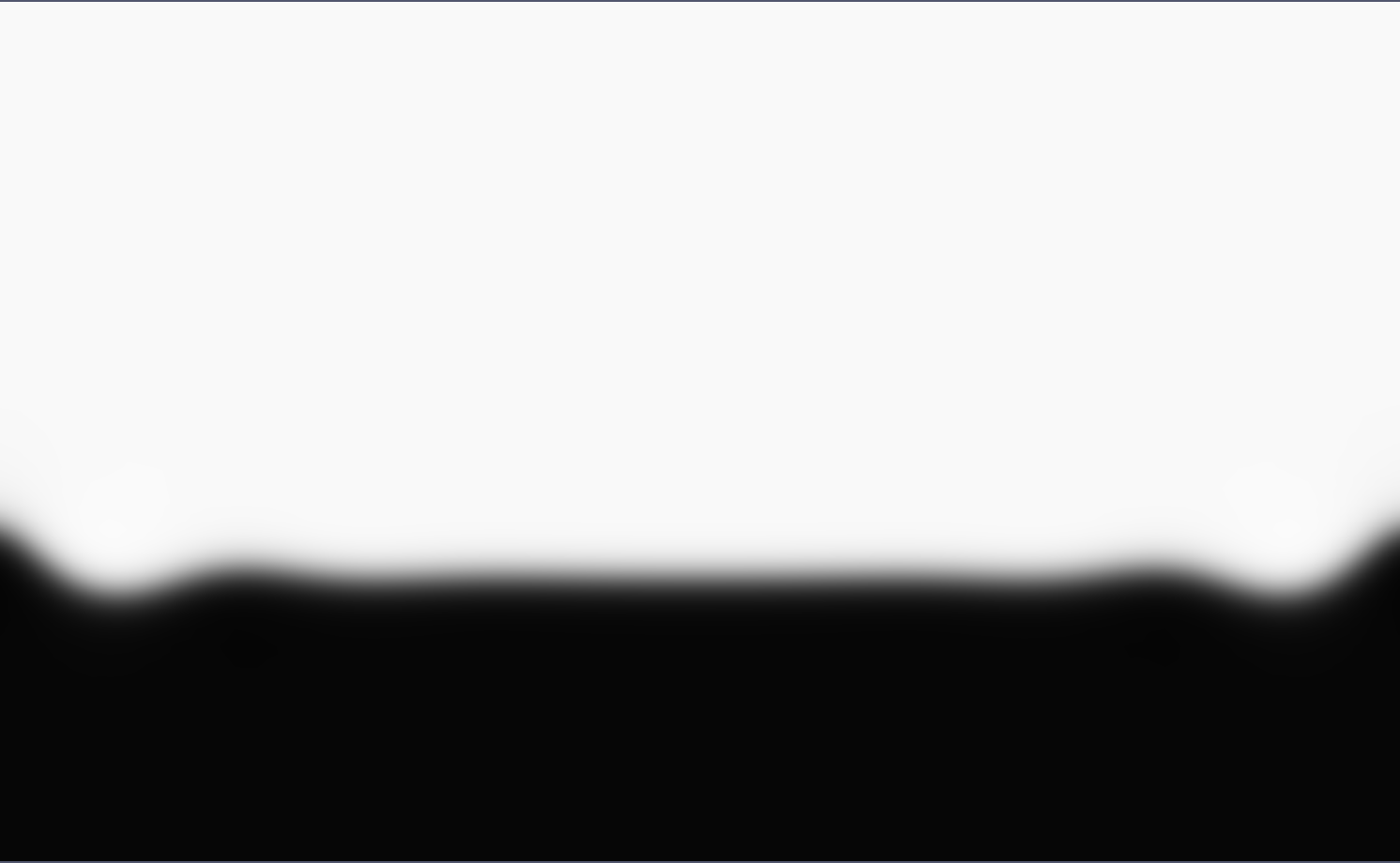}}&
    \fbox{\includegraphics[width=32mm]{Pictures/6L4000ts/20seq0008.png}}&
    \fbox{\includegraphics[width=32mm]{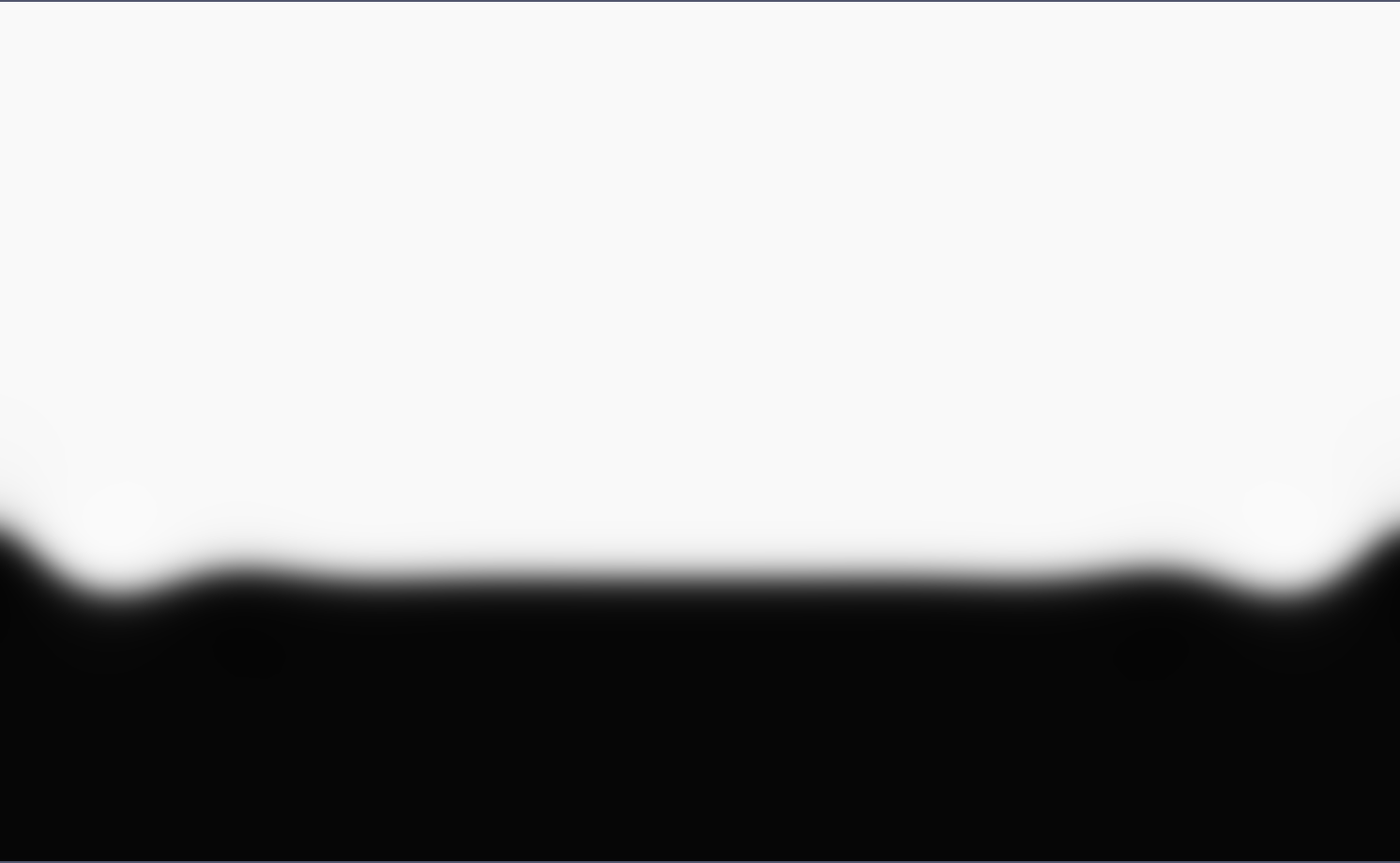}}\\

    \fbox{\includegraphics[width=32mm]{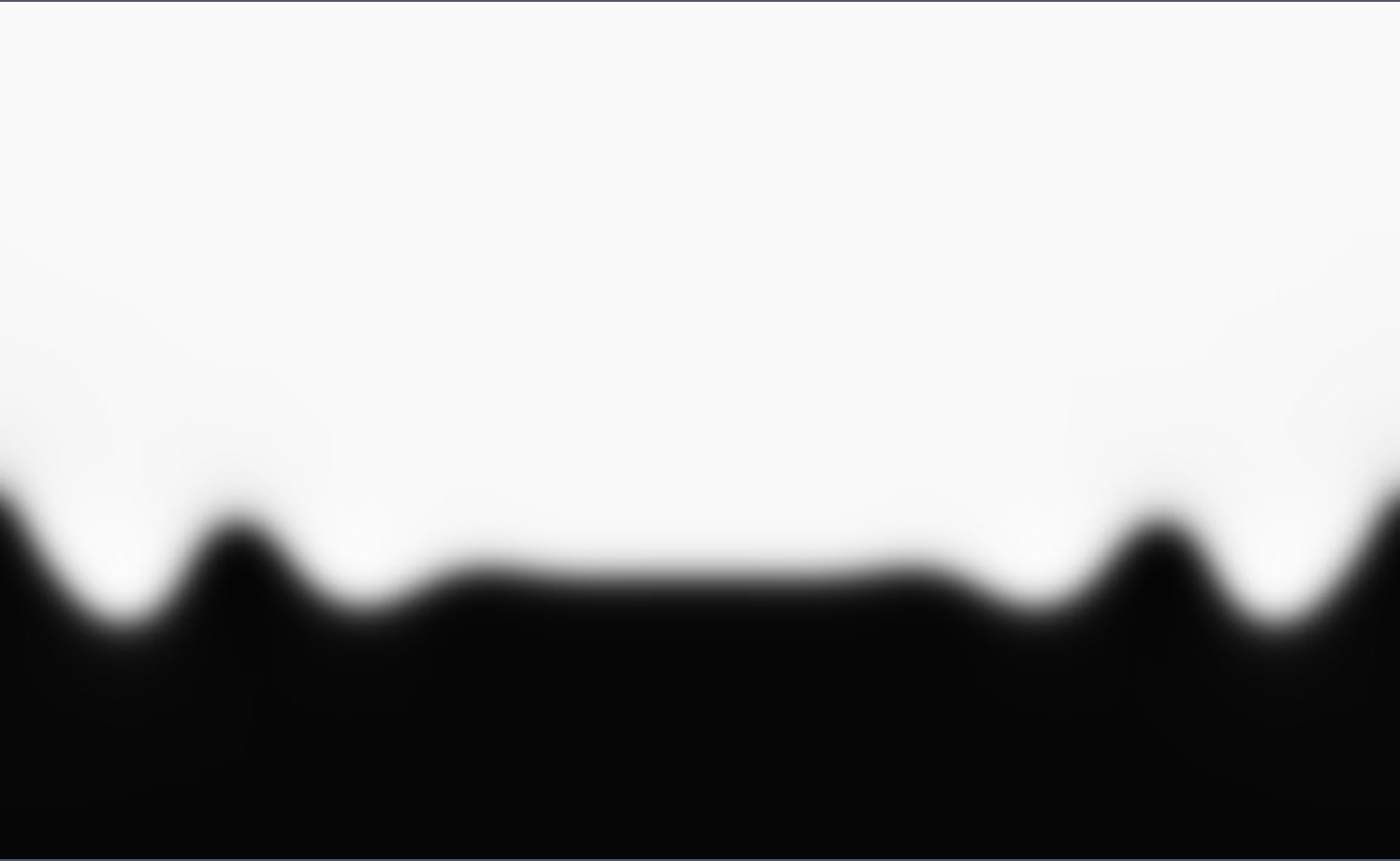}}&
    \fbox{\includegraphics[width=32mm]{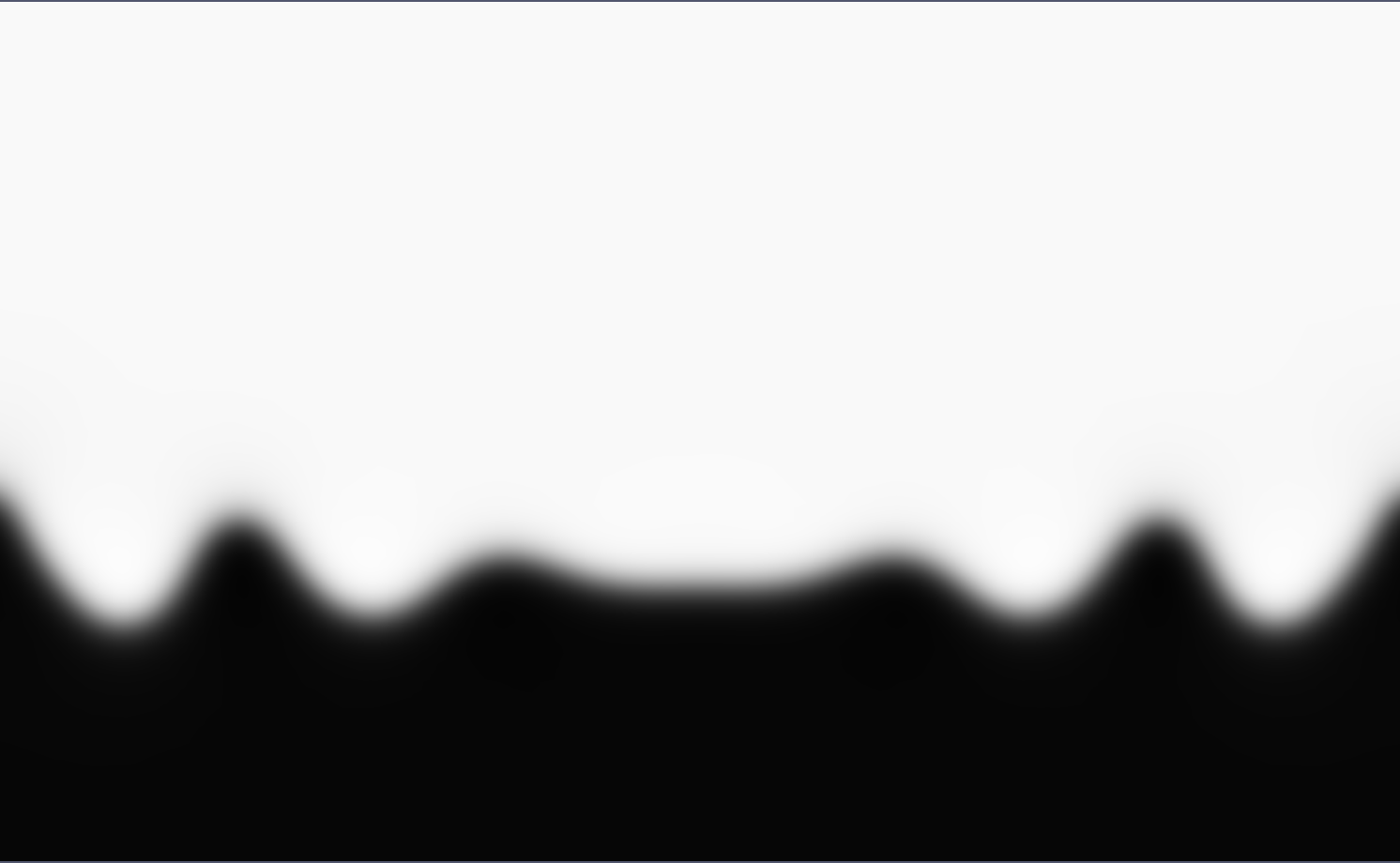}}&
    \fbox{\includegraphics[width=32mm]{Pictures/6L4000ts/20seq0009.png}}&
    \fbox{\includegraphics[width=32mm]{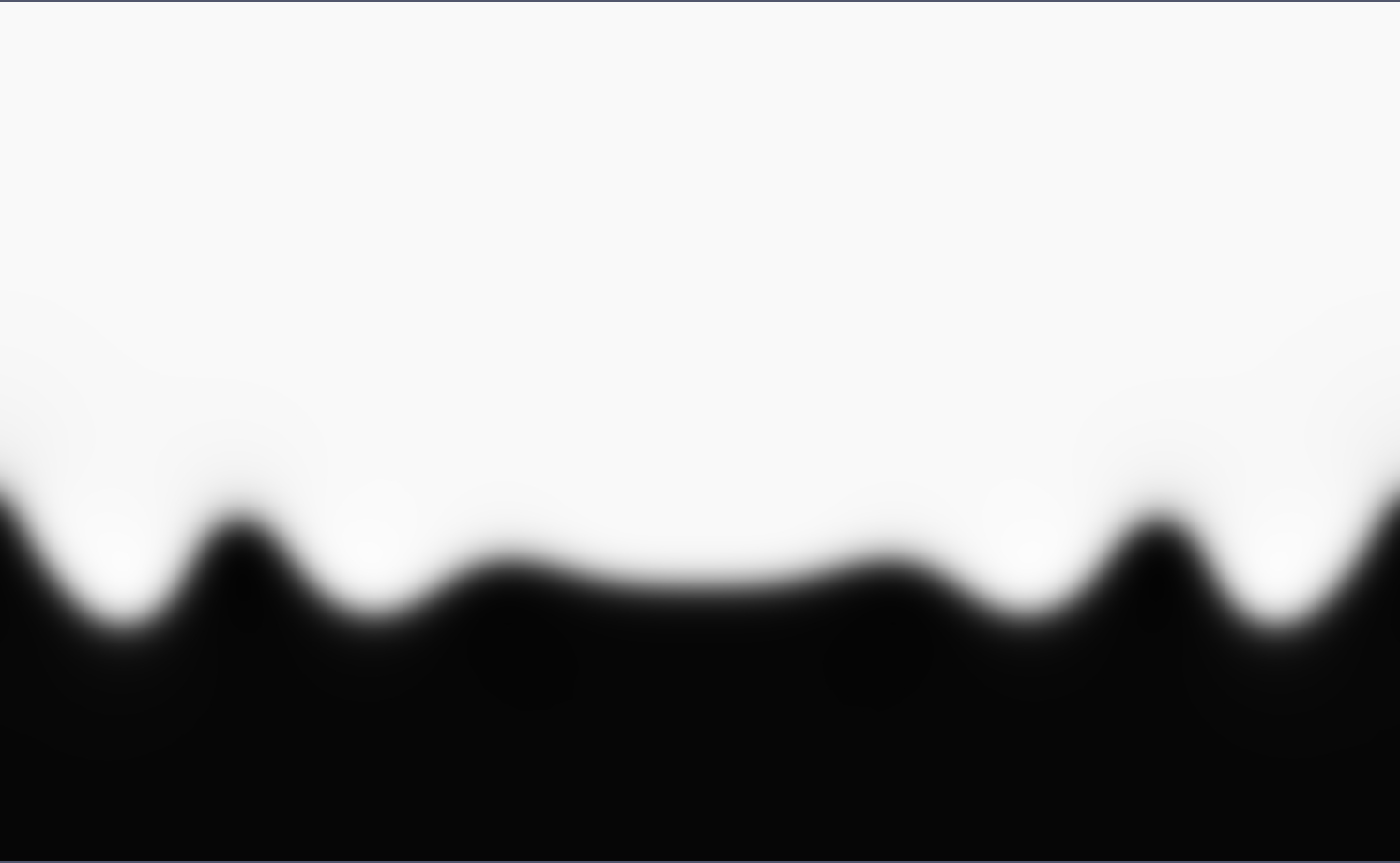}}\\

    \fbox{\includegraphics[width=32mm]{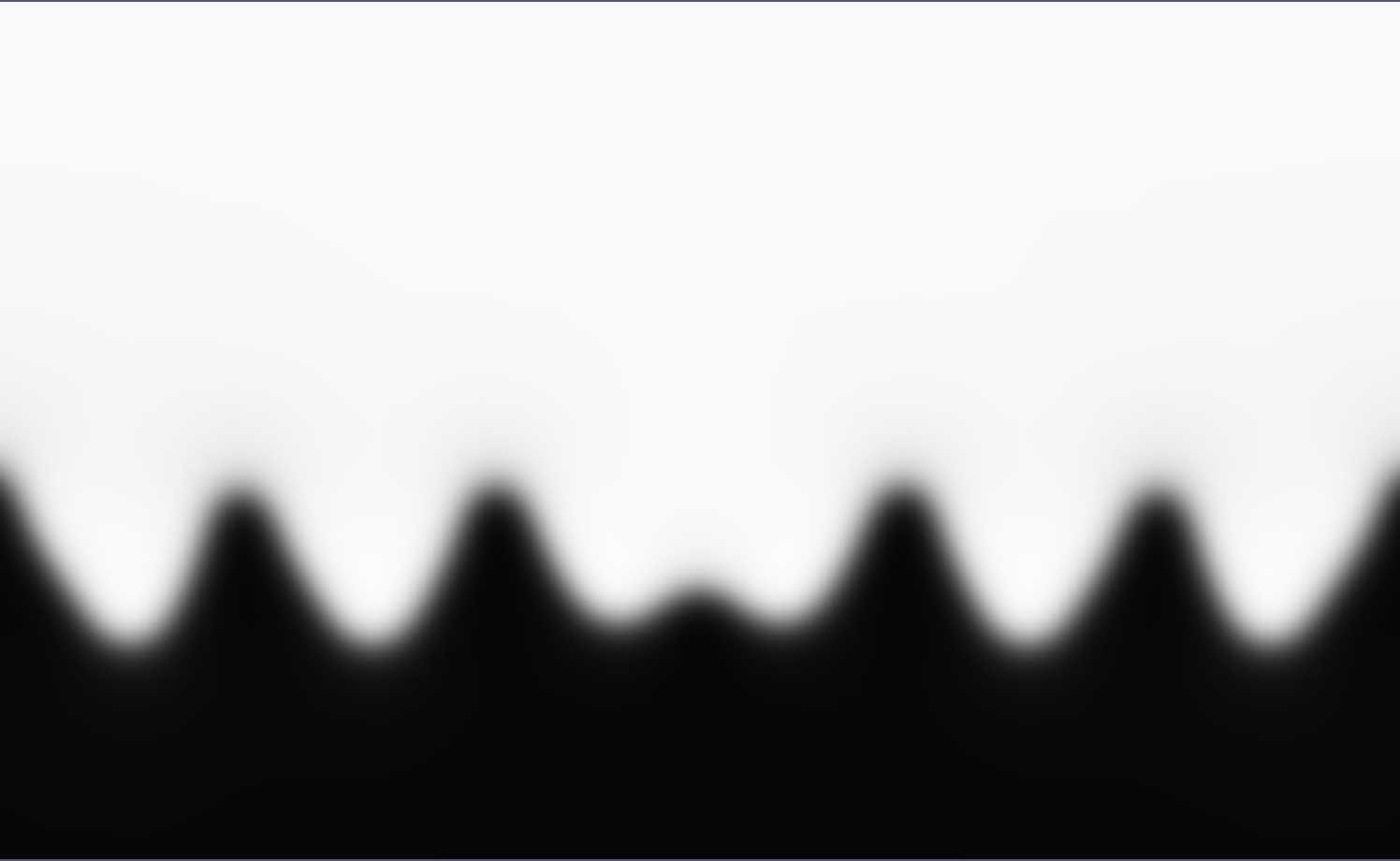}}&
    \fbox{\includegraphics[width=32mm]{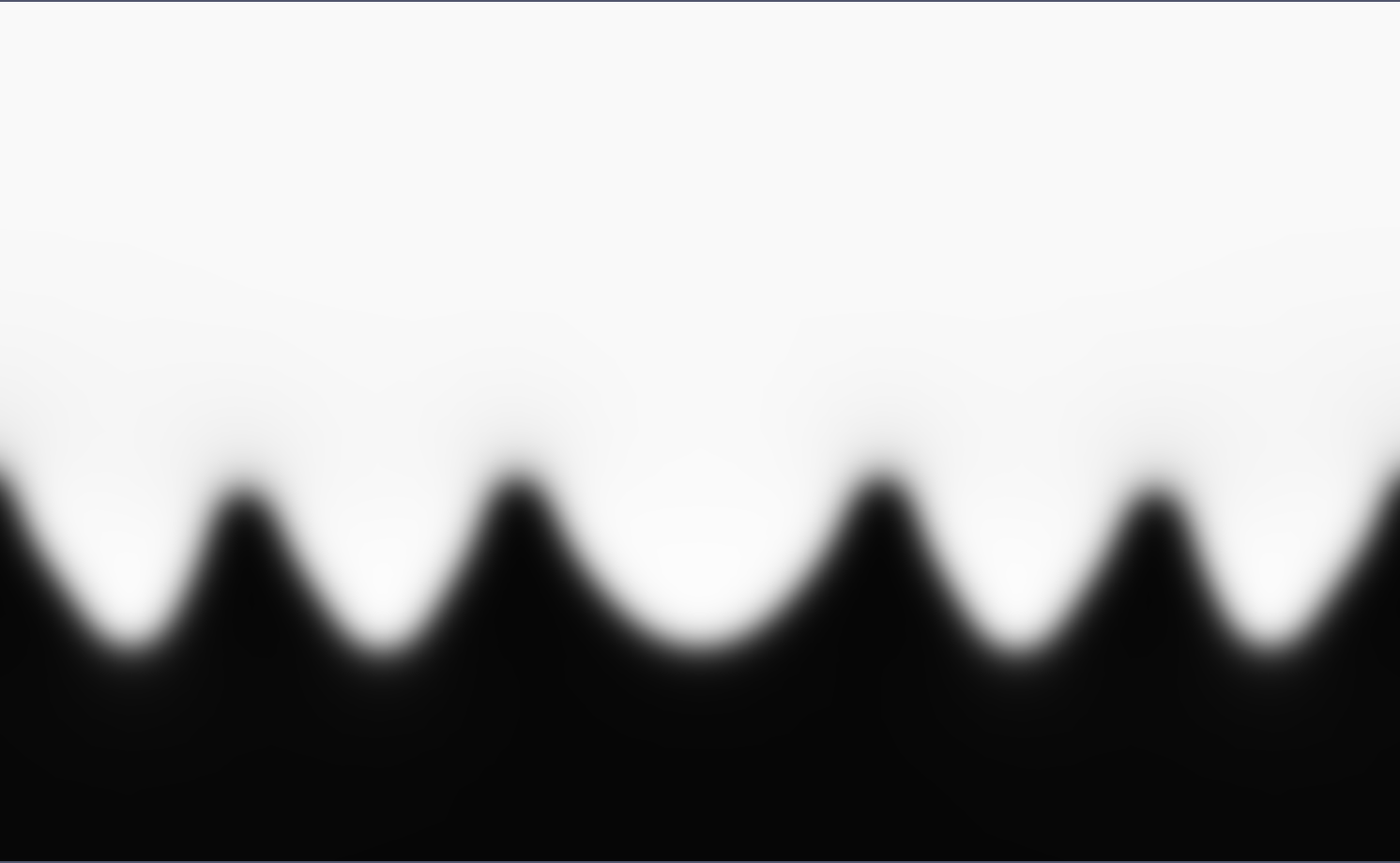}}&
    \fbox{\includegraphics[width=32mm]{Pictures/6L4000ts/20seq0010.png}}&
    \fbox{\includegraphics[width=32mm]{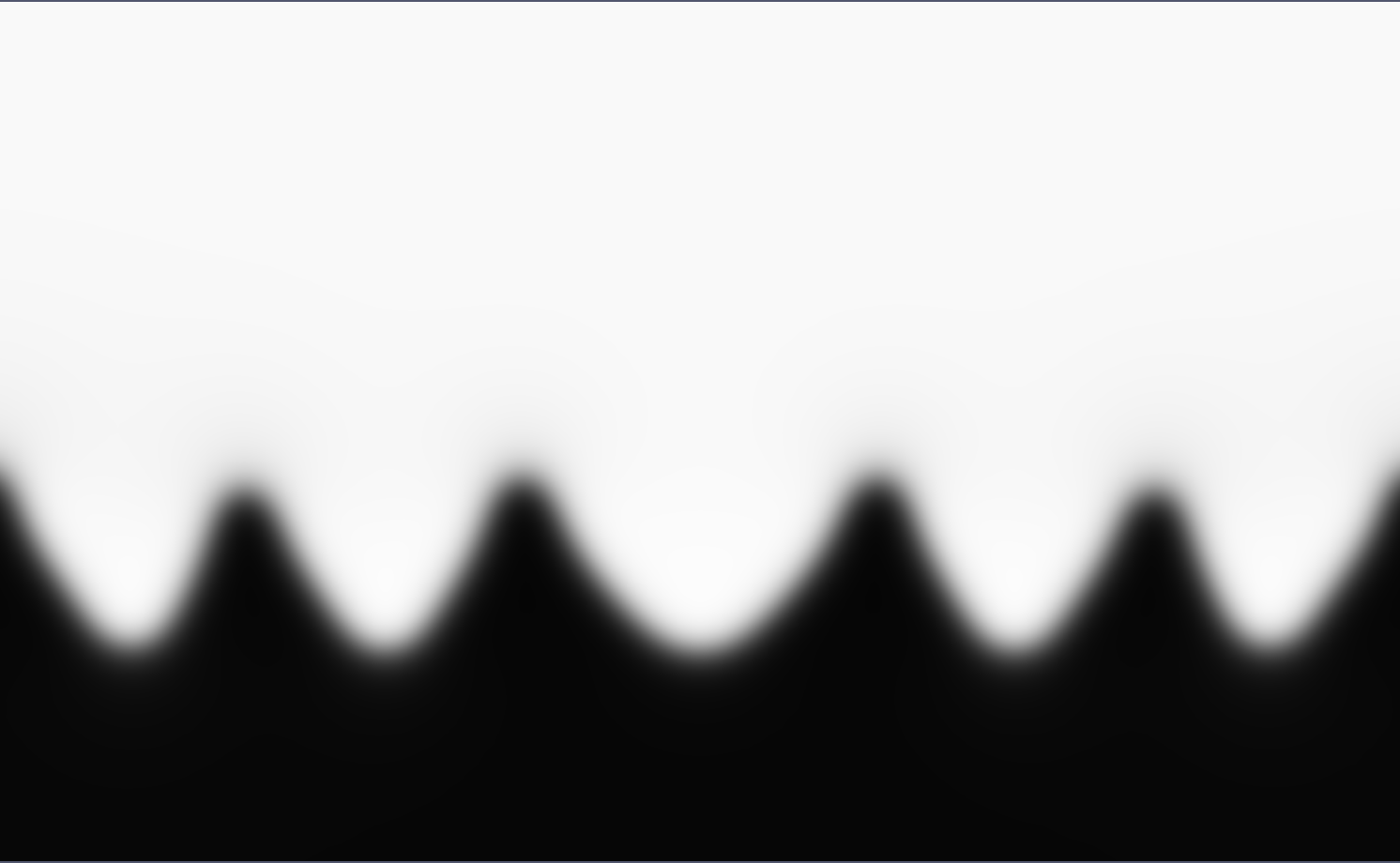}}\\

    \fbox{\includegraphics[width=32mm]{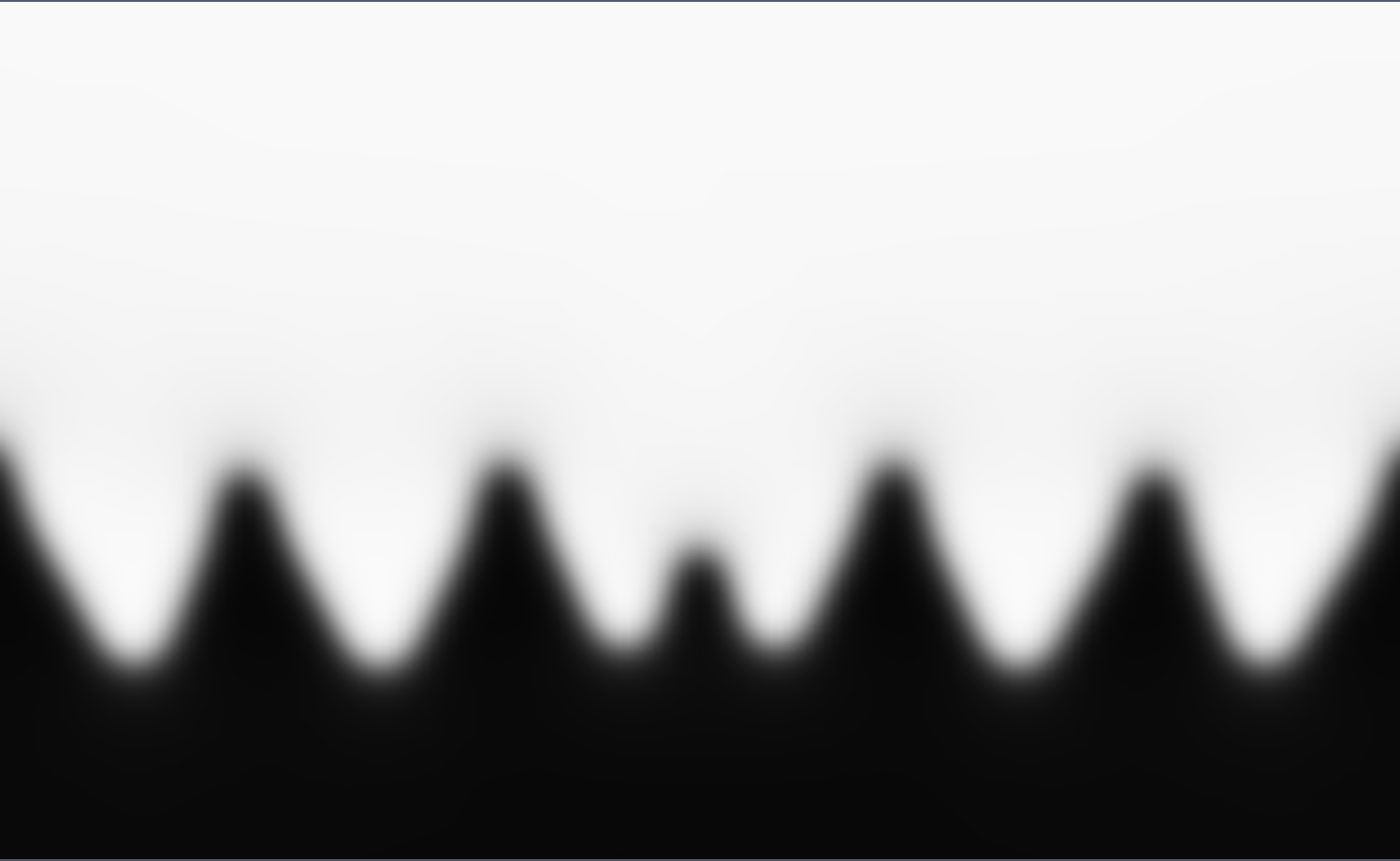}}&
    \fbox{\includegraphics[width=32mm]{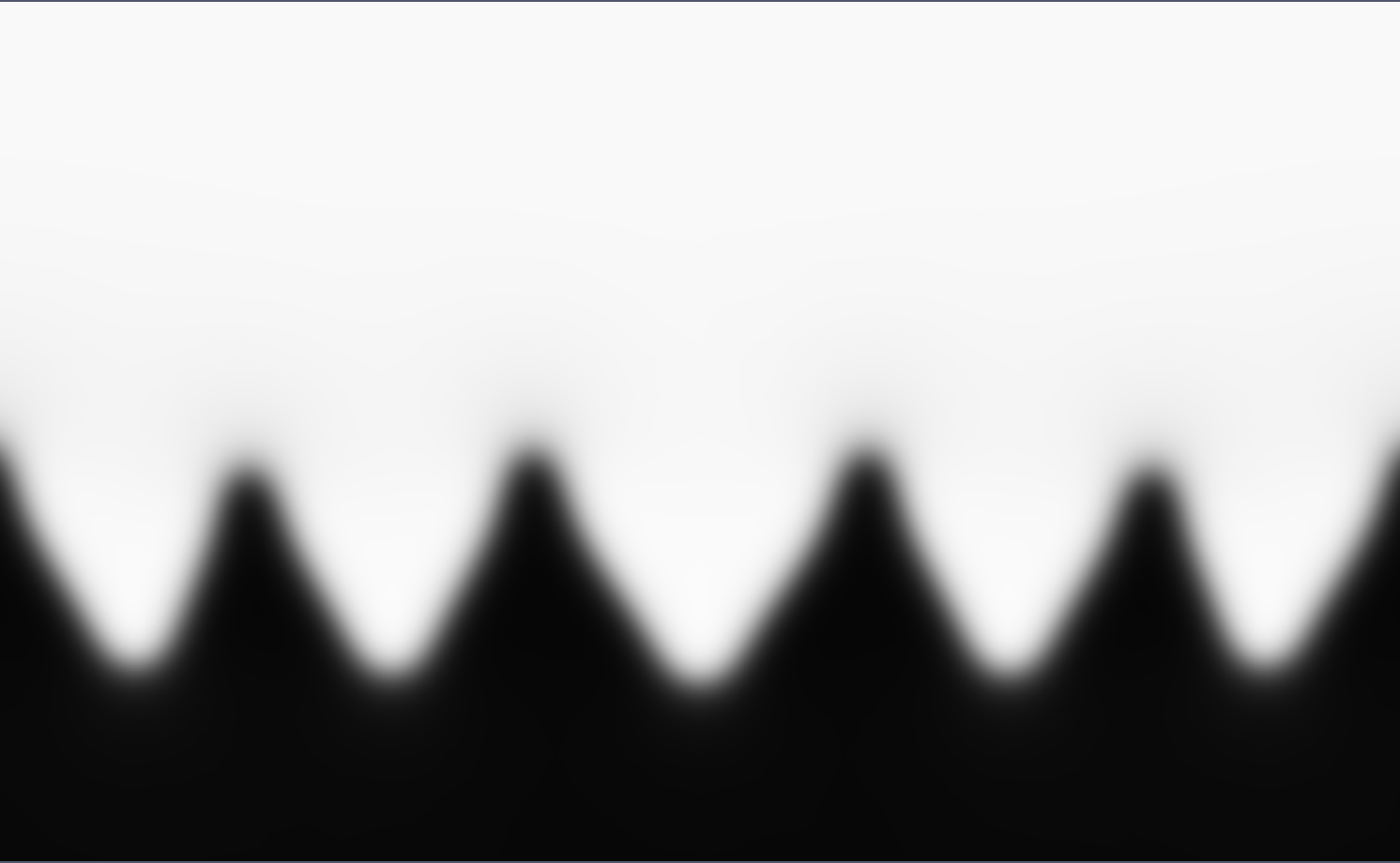}}&
    \fbox{\includegraphics[width=32mm]{Pictures/6L4000ts/20seq0011.png}}&
    \fbox{\includegraphics[width=32mm]{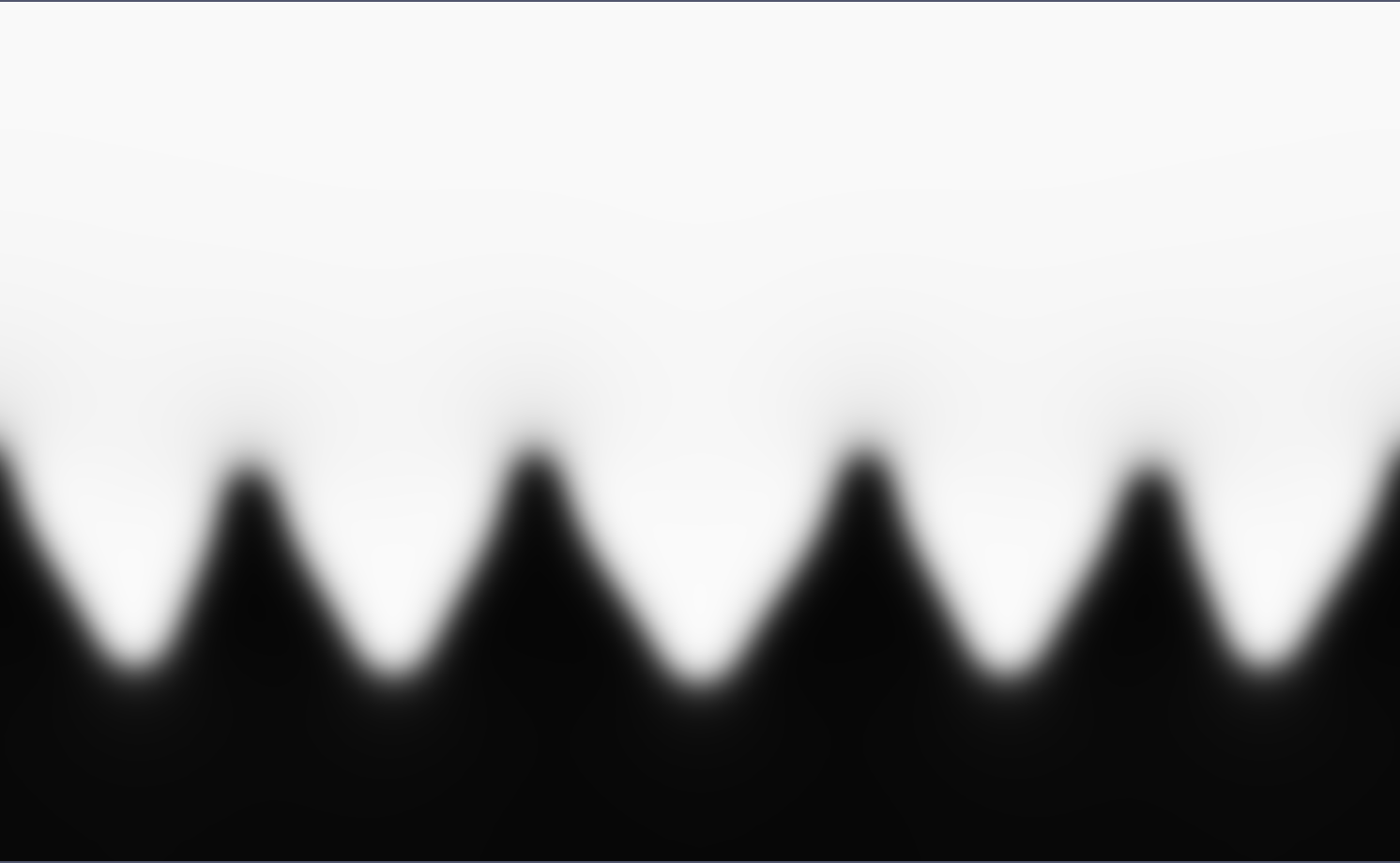}}\\
    
    \fbox{\includegraphics[width=32mm]{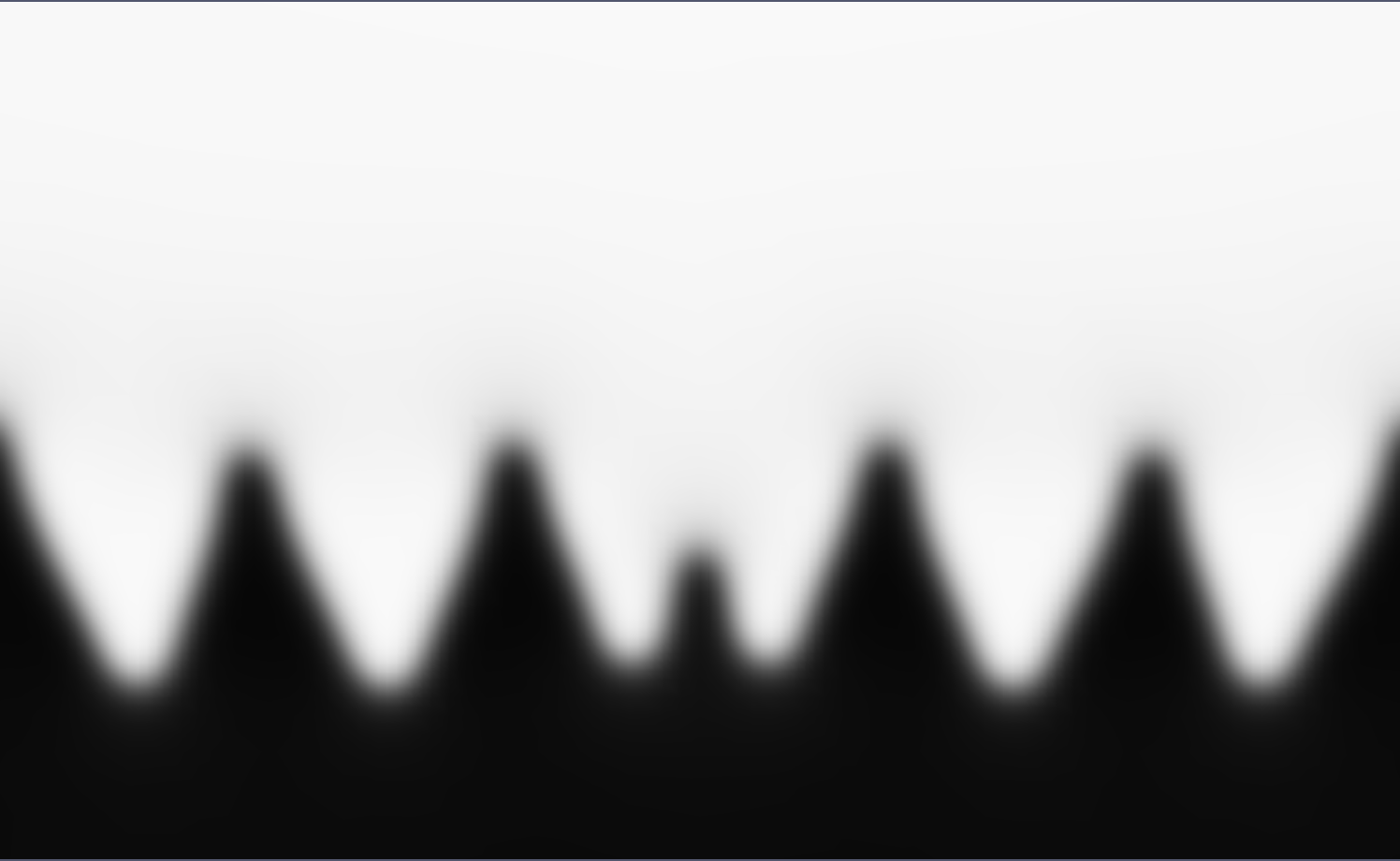}}&
    \fbox{\includegraphics[width=32mm]{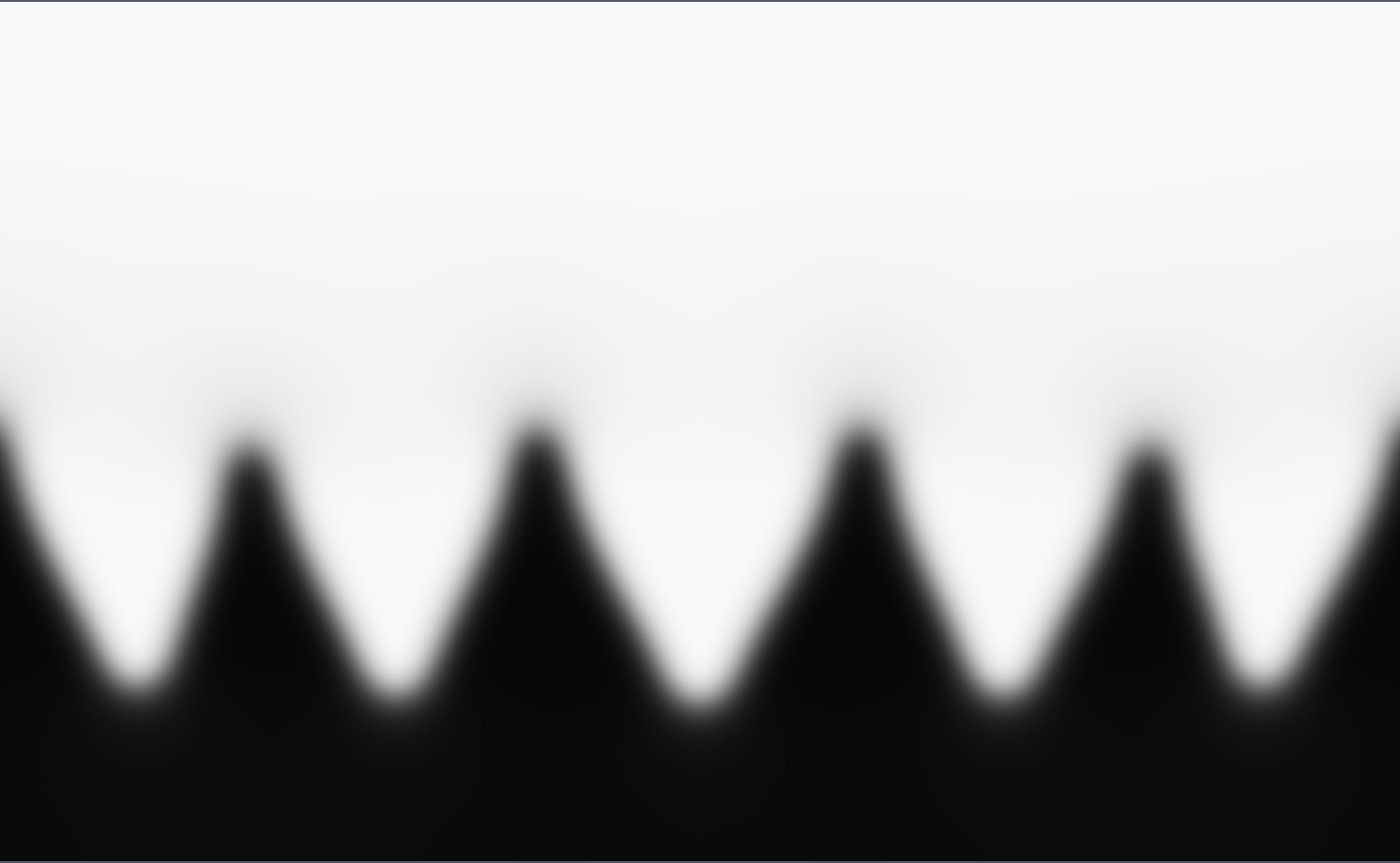}}&
    \fbox{\includegraphics[width=32mm]{Pictures/6L4000ts/20seq0012.png}}&
    \fbox{\includegraphics[width=32mm]{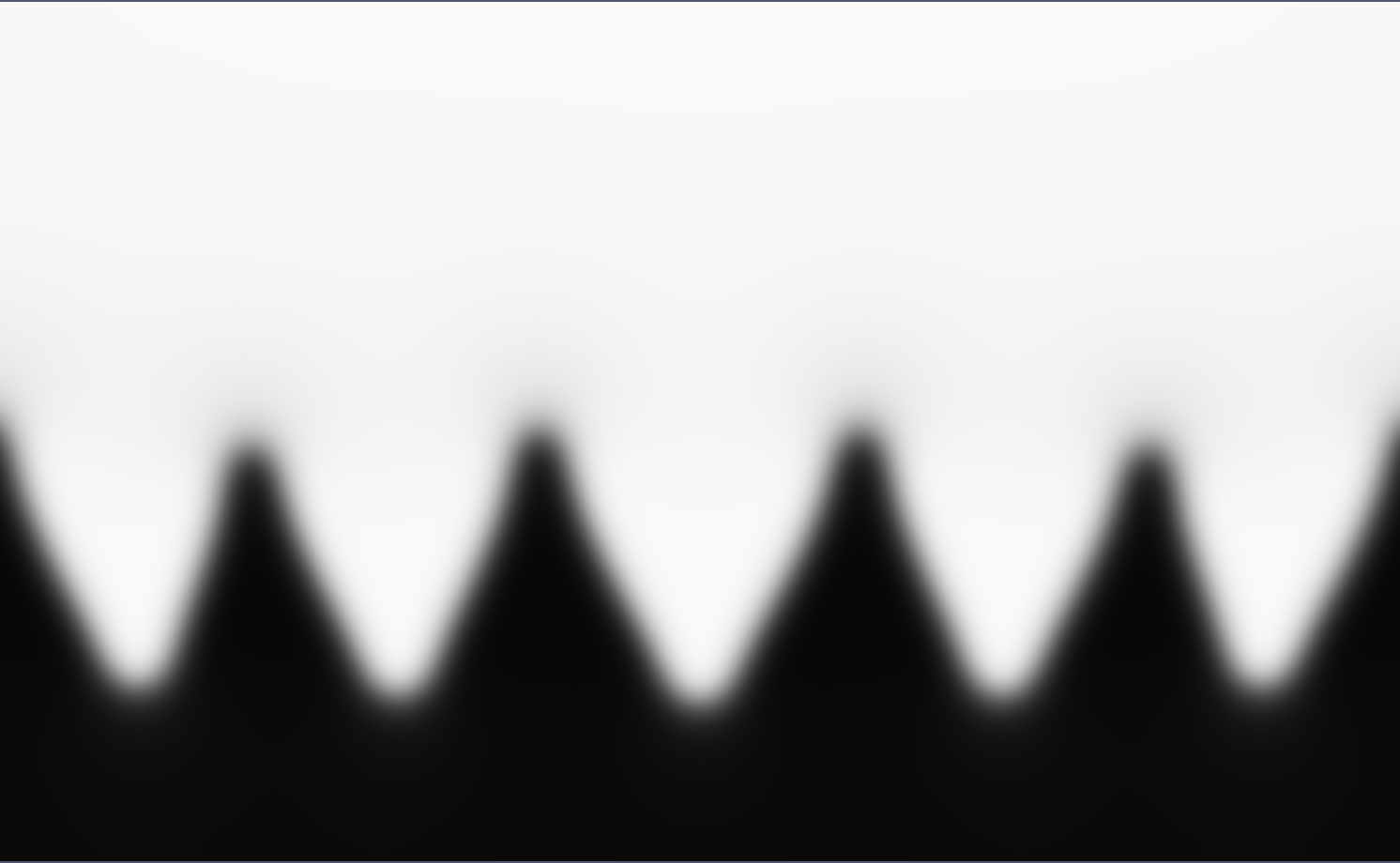}}\\

    \fbox{\includegraphics[width=32mm]{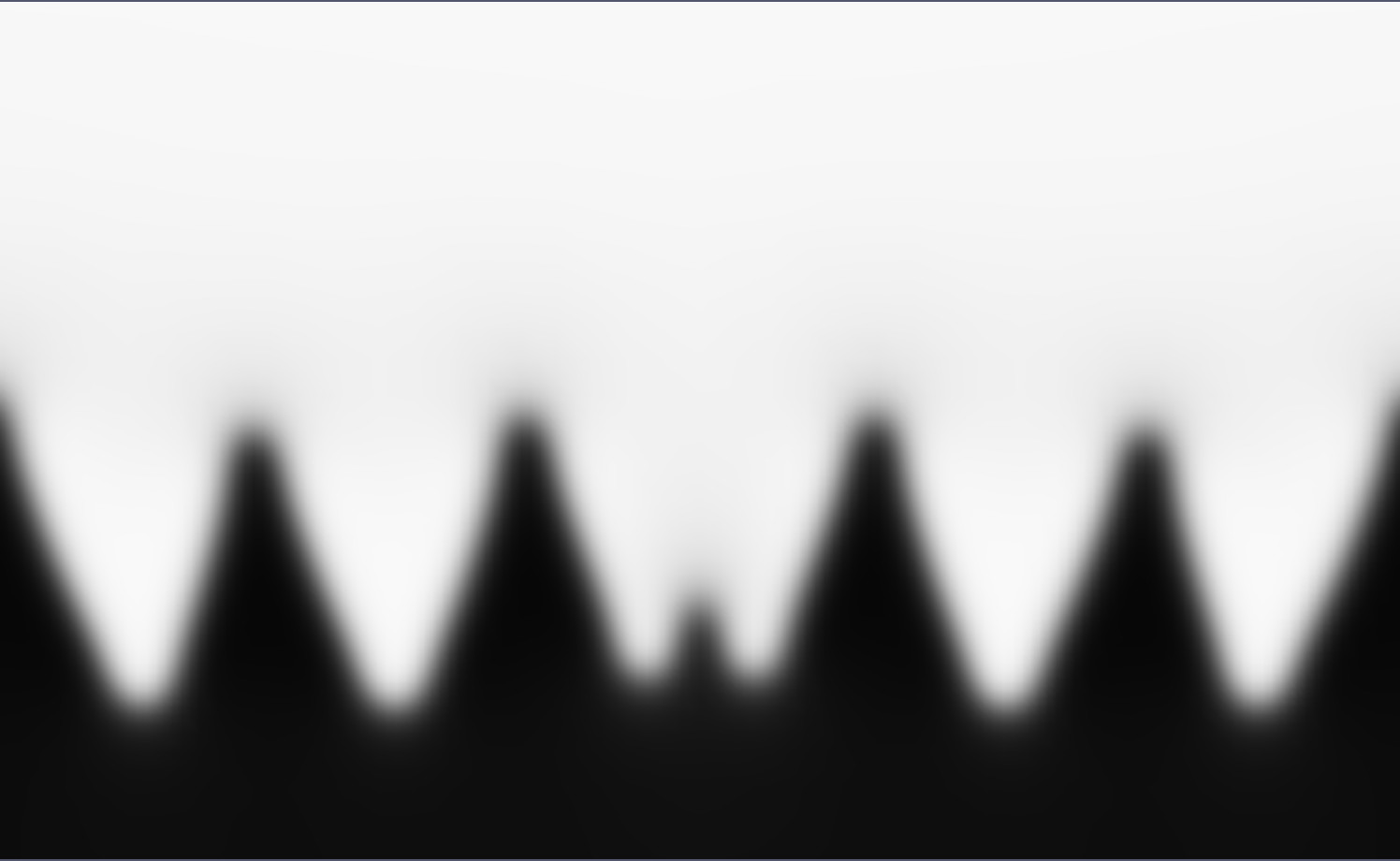}}&
    \fbox{\includegraphics[width=32mm]{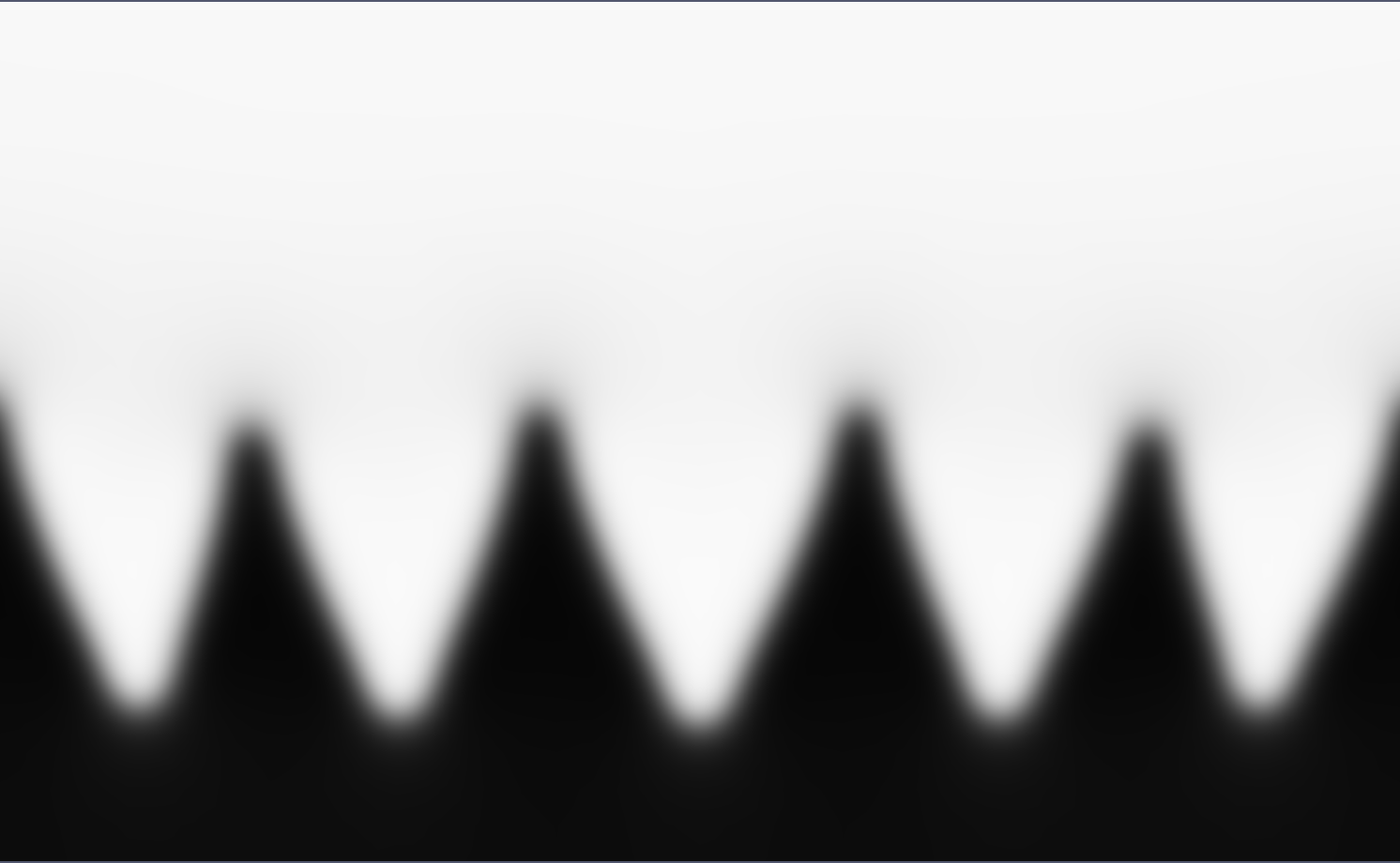}}&
    \fbox{\includegraphics[width=32mm]{Pictures/6L4000ts/20seq0013.png}}&
    \fbox{\includegraphics[width=32mm]{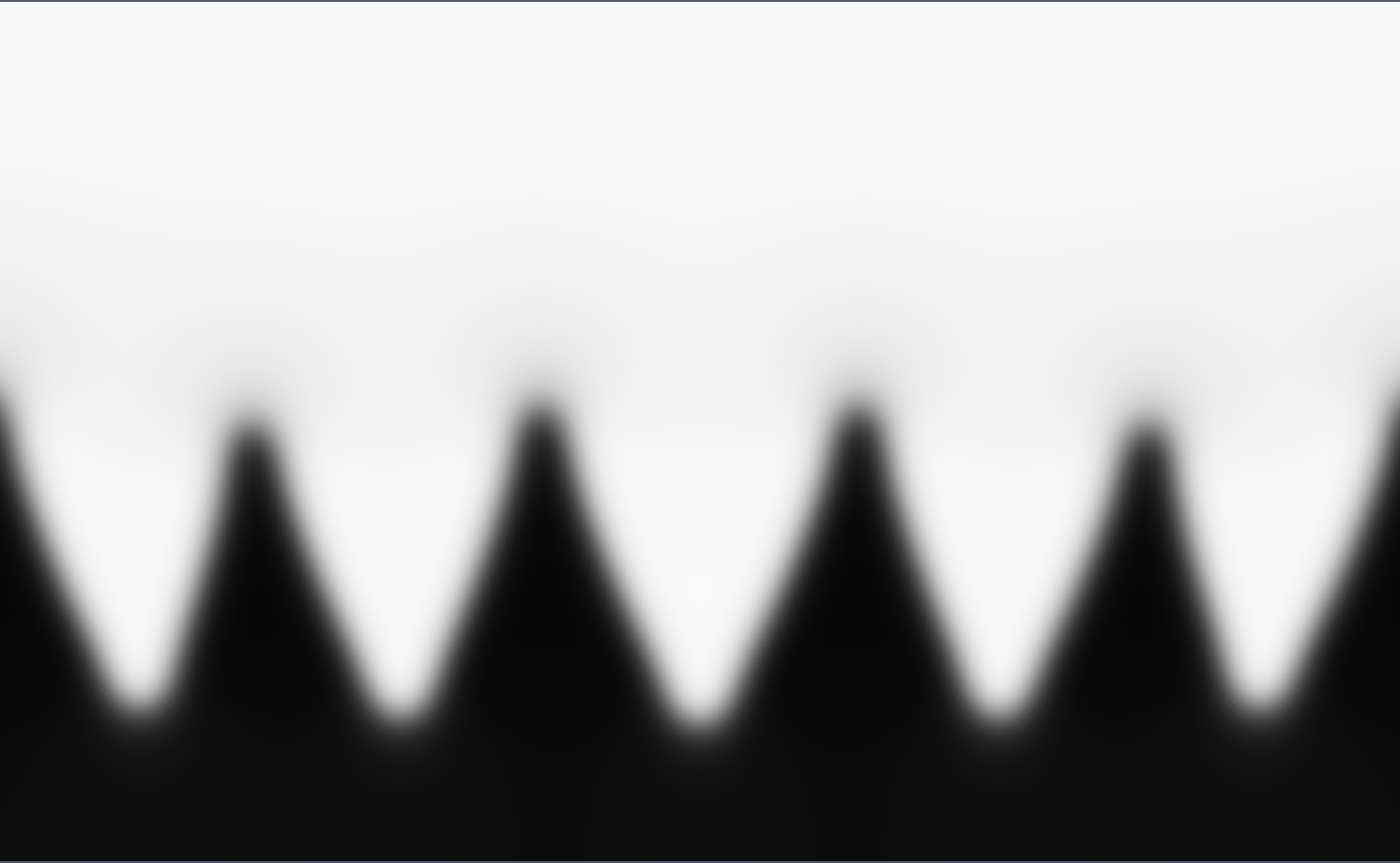}}
  \end{tabular}
\end{center}
\caption[Rosensweig instability: parametric study on the space
  discretization]{\textbf{Parametric study: space discretization.}
  This figure shows the results obtained from time $t=0.7$ (uppermost
  row) to time $t= 1.3$ (lowermost row) in time intervals $\dt = 0.1$
  using 4000 time steps and four different levels of refinement in
  space: the coarsest mesh uses 4 levels of refinement (first column),
  5 levels (second column), 6 levels (third column), and the finest
  mesh uses 7 levels (fourth column). The reader can appreciate that
  with the coarsest mesh (leftmost column) the numerical solution
  exhibits artificial features which do not survive additional
  refinement. One of them is an additional spike in the middle
  of the diffuse interface which is not present in the second, third and fourth columns. This simple example illustrates the importance of parametric studies in the context of phase-field methods. \label{parafigspace}}
%
\end{figure}

\begin{figure}
\begin{center}
    \setlength\fboxsep{0pt}
    \setlength\fboxrule{1pt}
  \begin{tabular}{cccc}

    \fbox{\includegraphics[width=32mm]{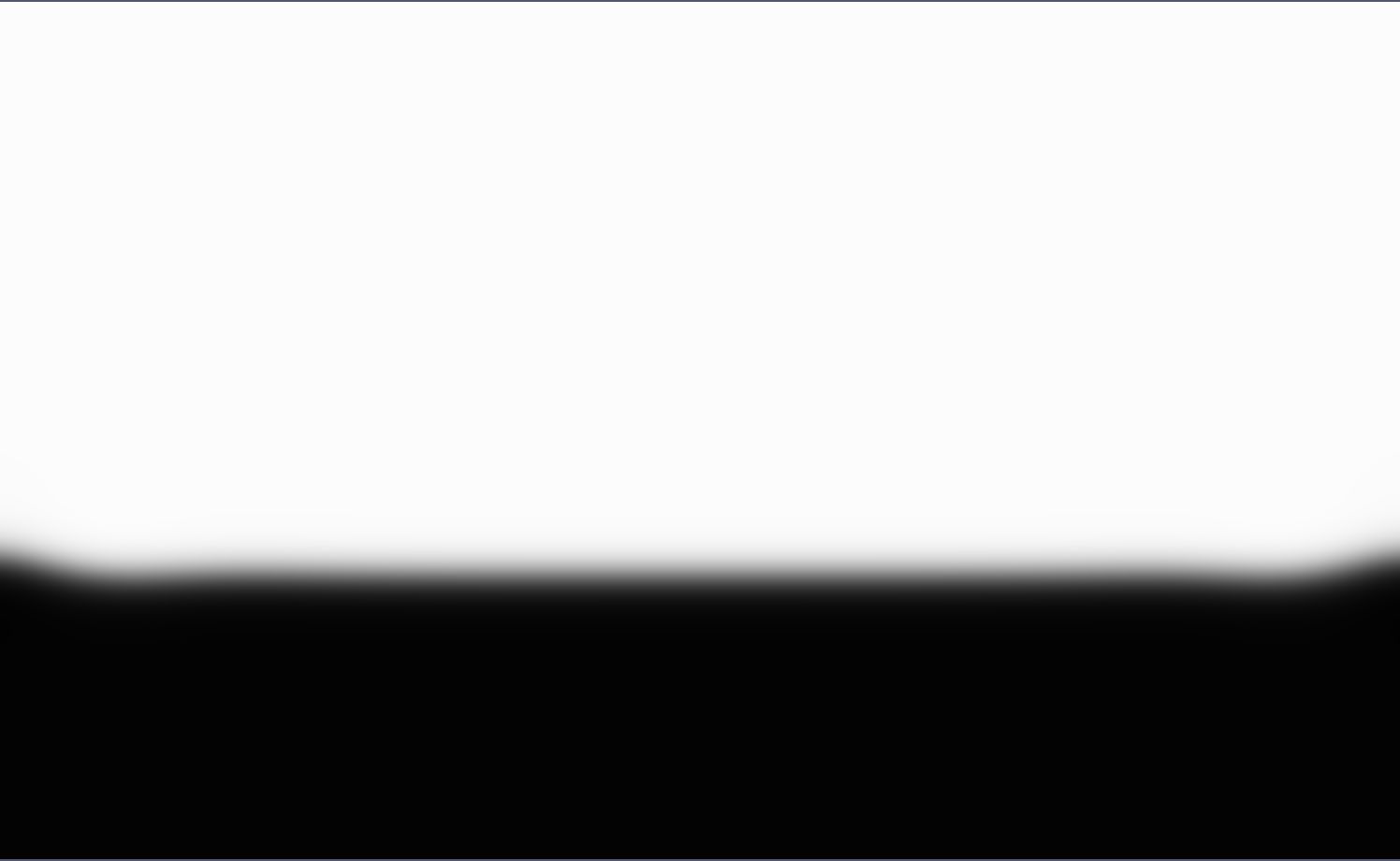}}&
    \fbox{\includegraphics[width=32mm]{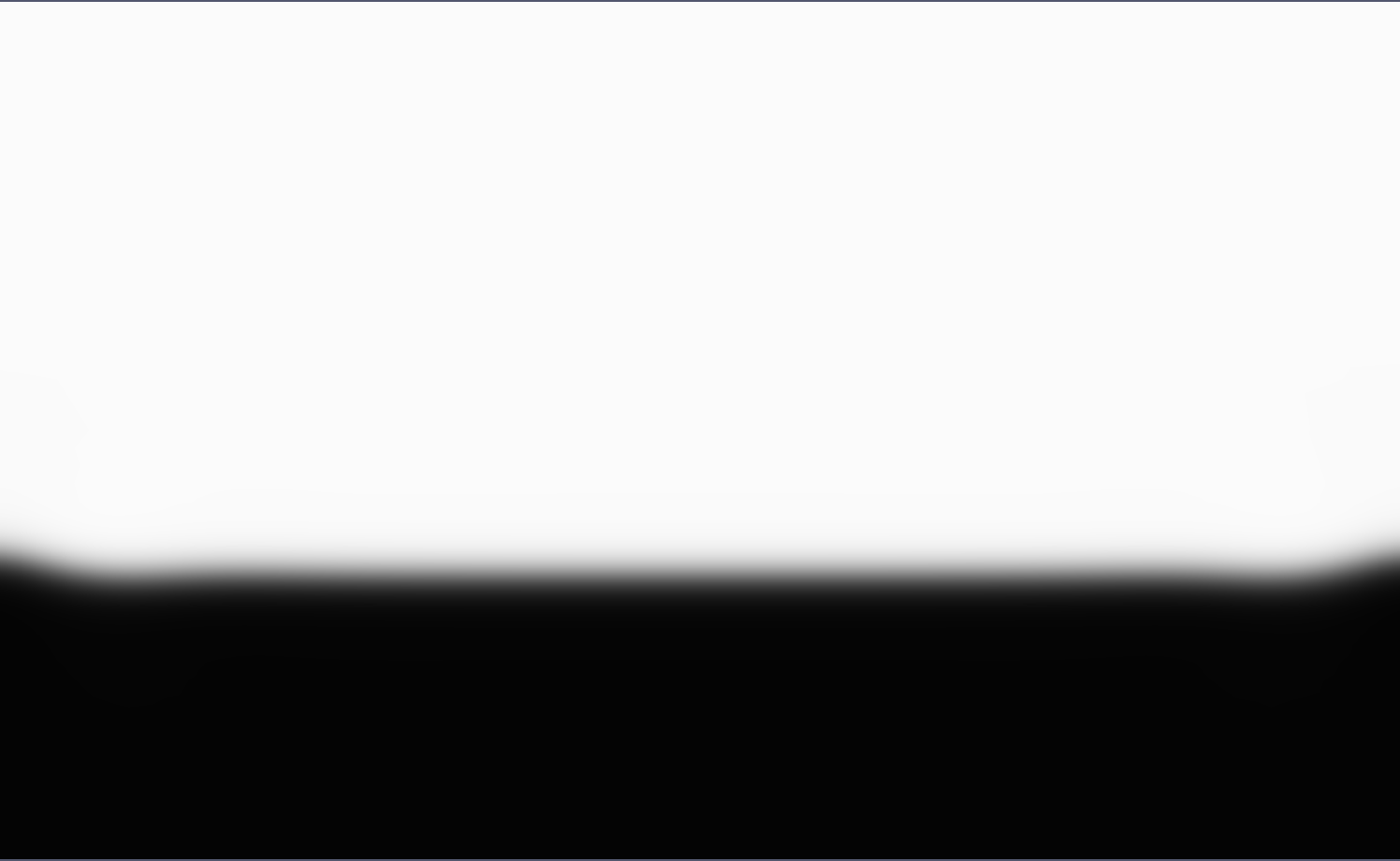}}&
    \fbox{\includegraphics[width=32mm]{Pictures/6L4000ts/20seq0007.png}}&
    \fbox{\includegraphics[width=32mm]{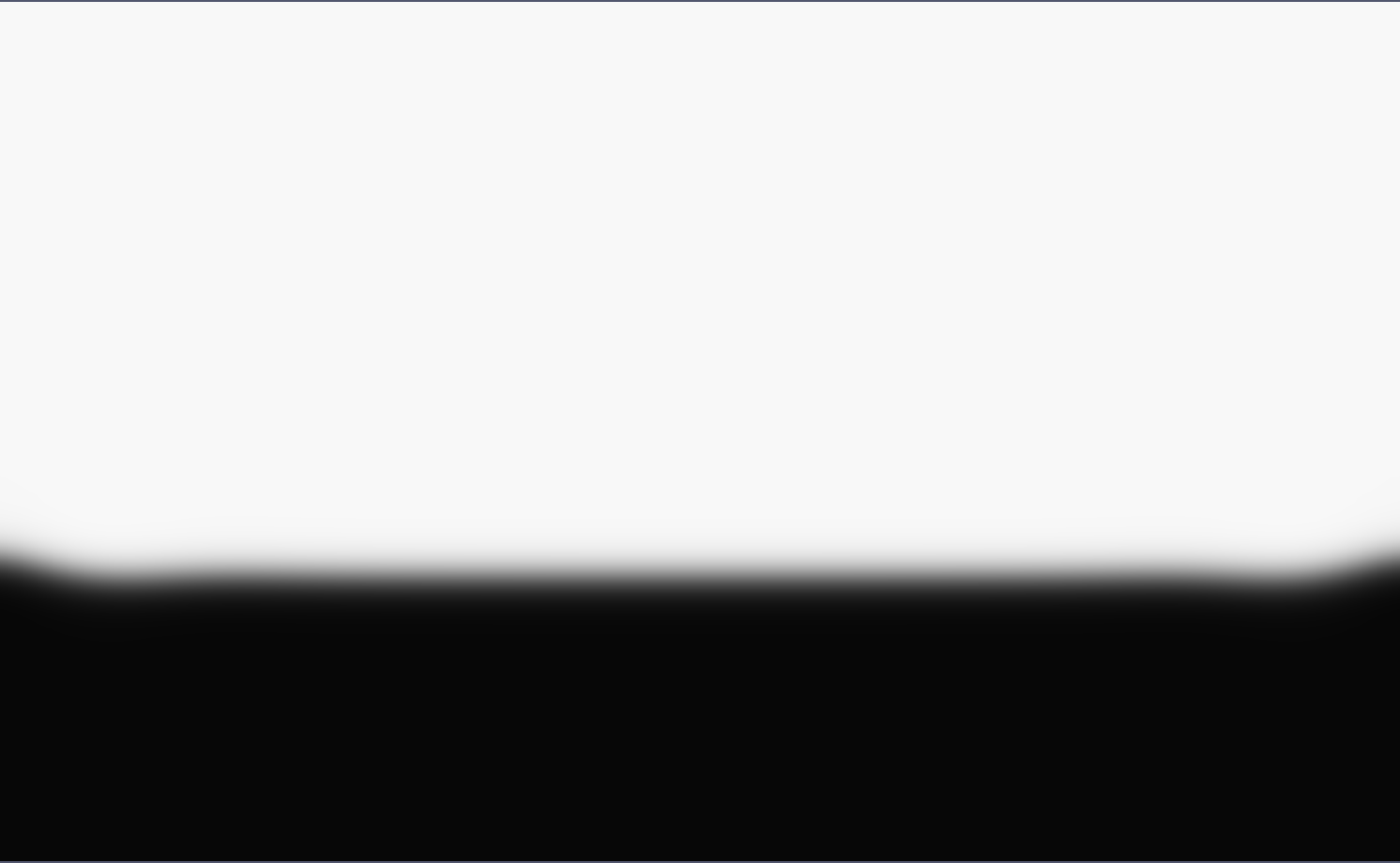}}\\

    \fbox{\includegraphics[width=32mm]{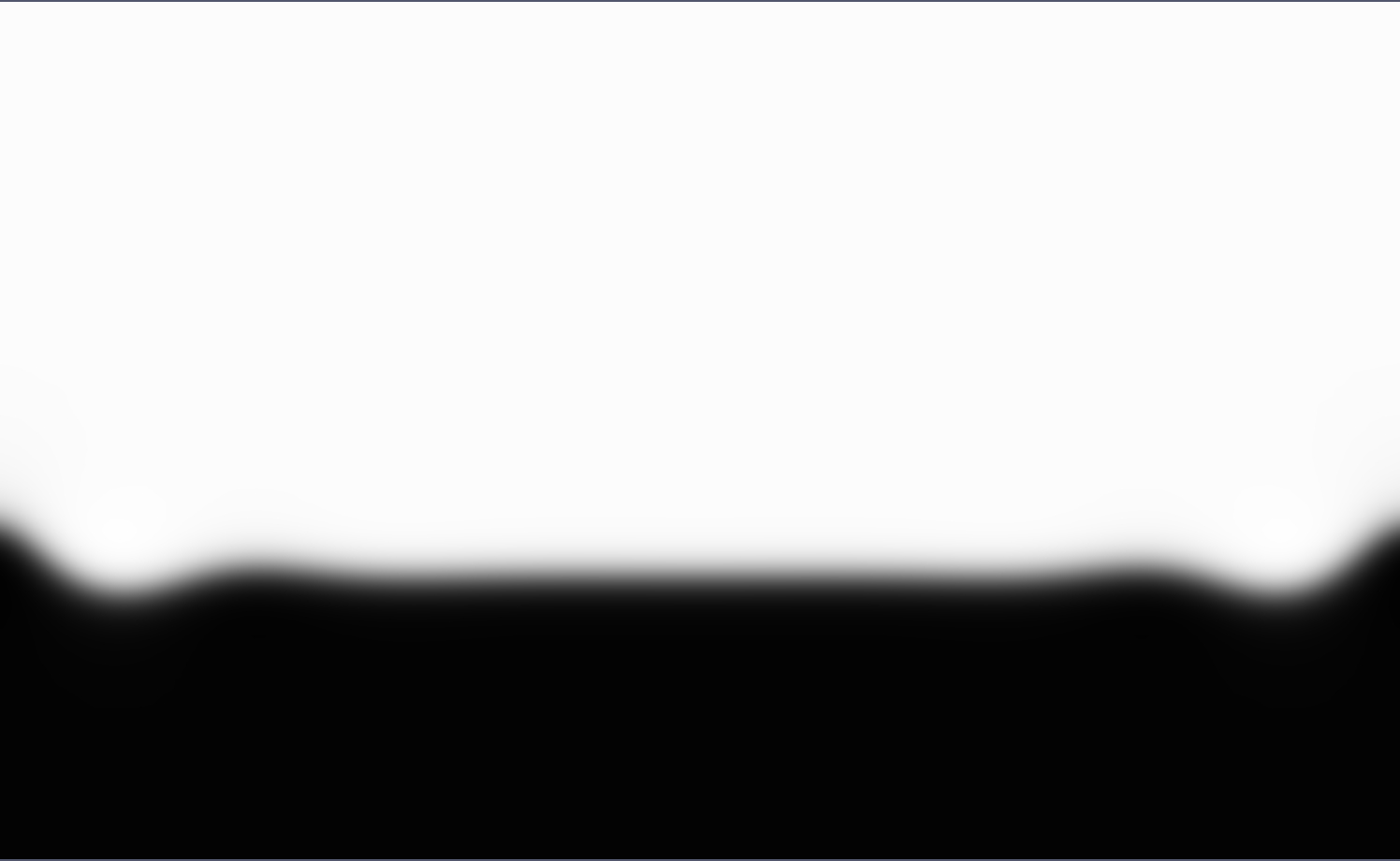}}&
    \fbox{\includegraphics[width=32mm]{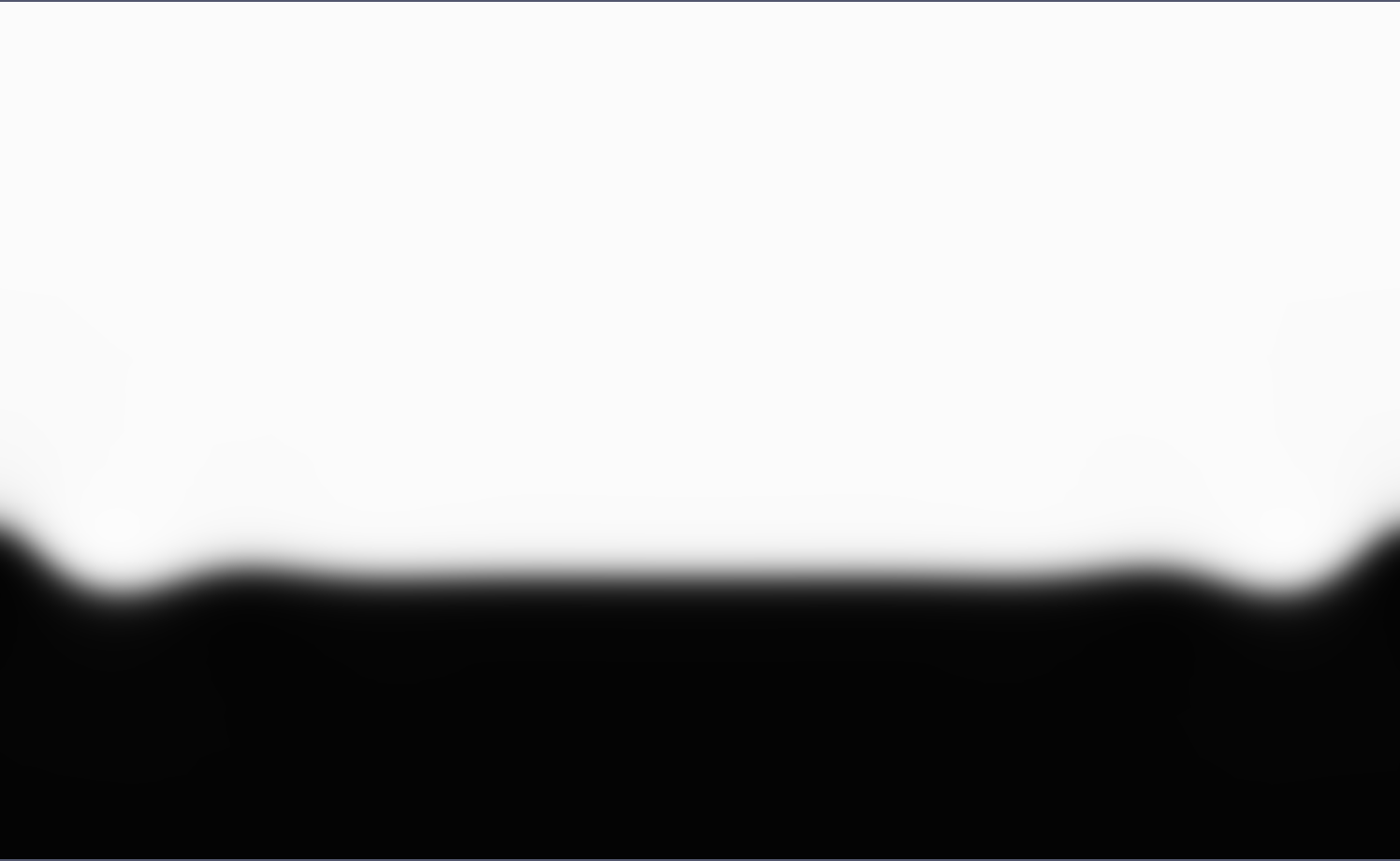}}&
    \fbox{\includegraphics[width=32mm]{Pictures/6L4000ts/20seq0008.png}}&
    \fbox{\includegraphics[width=32mm]{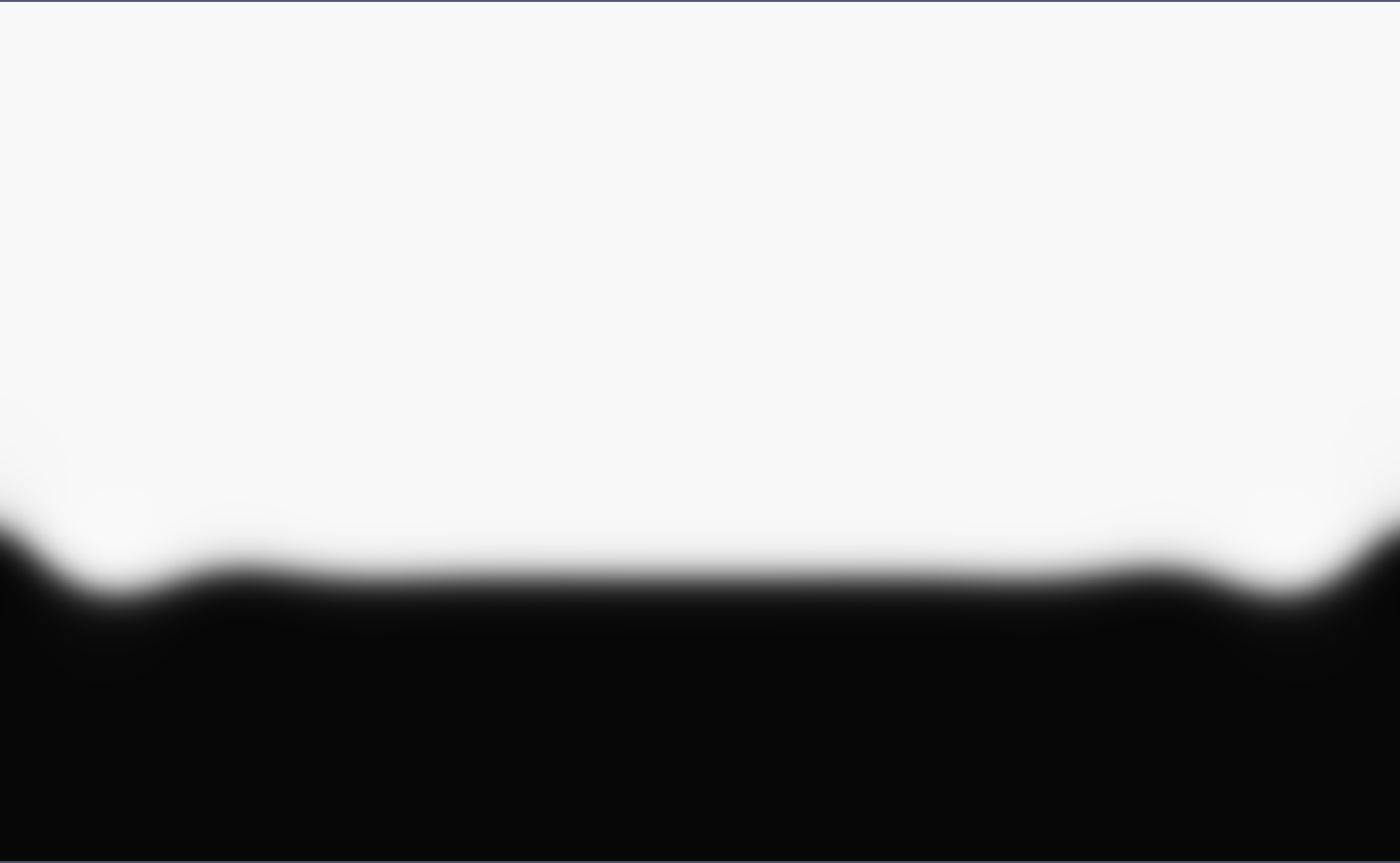}}\\

    \fbox{\includegraphics[width=32mm]{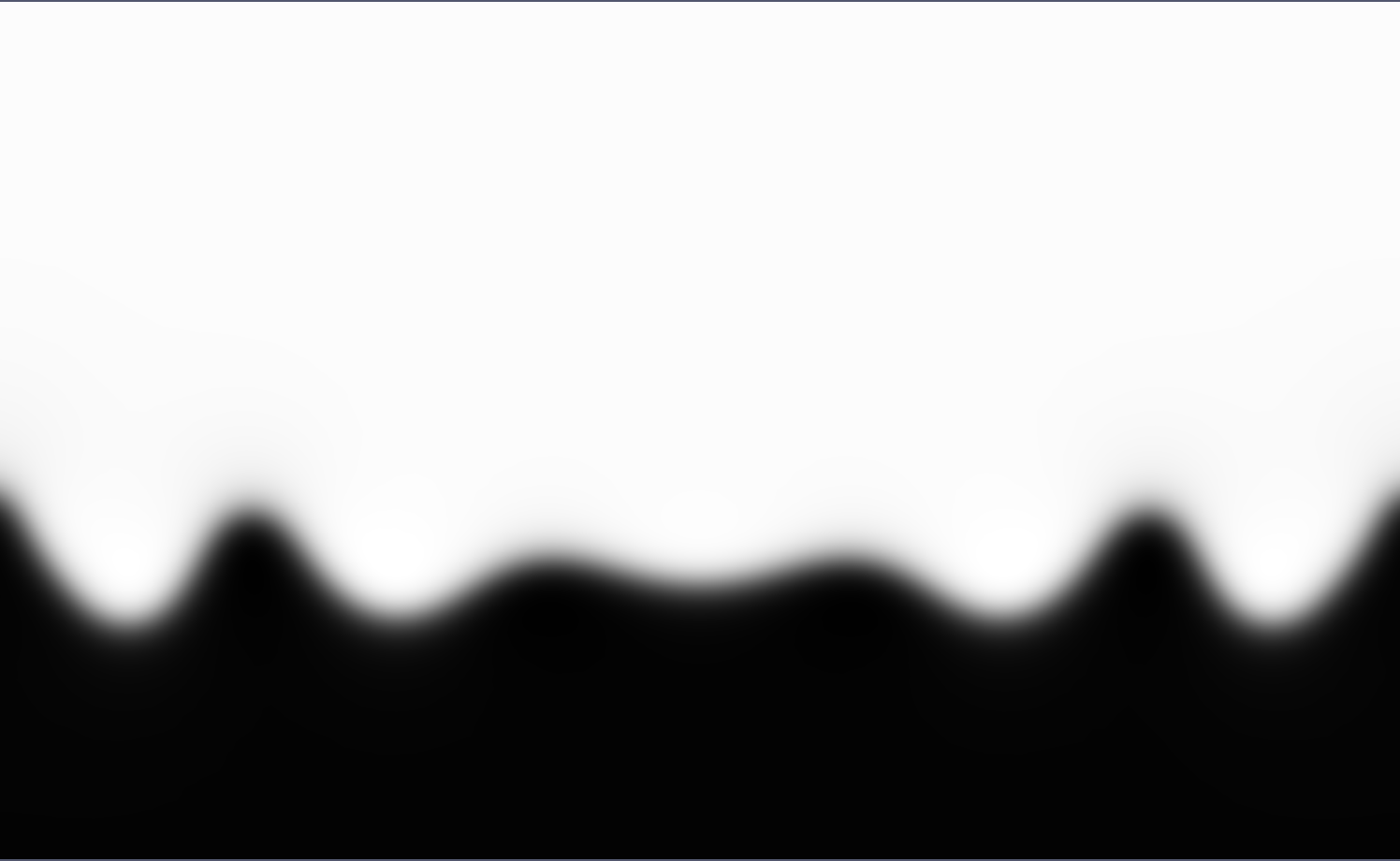}}&
    \fbox{\includegraphics[width=32mm]{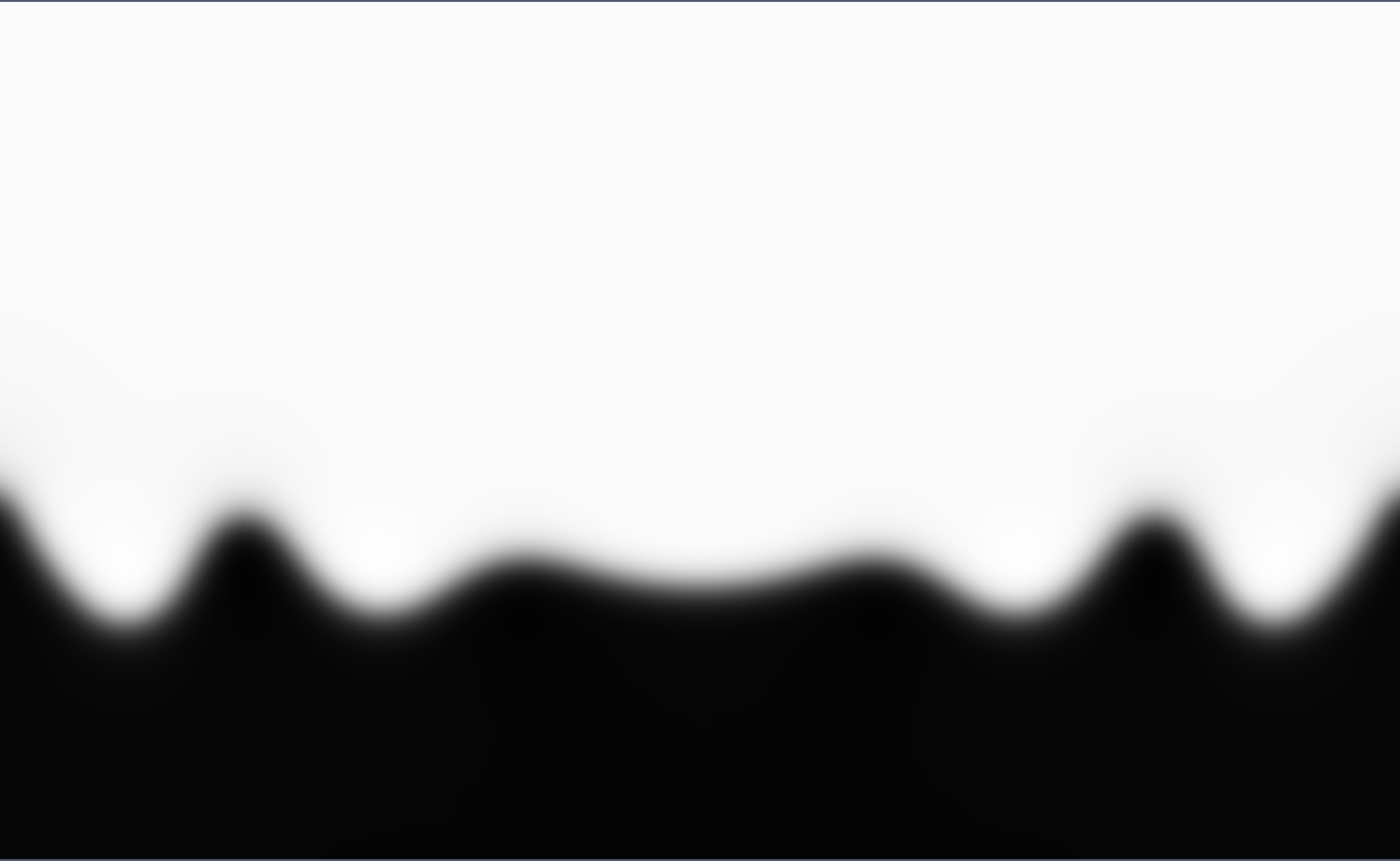}}&
    \fbox{\includegraphics[width=32mm]{Pictures/6L4000ts/20seq0009.png}}&
    \fbox{\includegraphics[width=32mm]{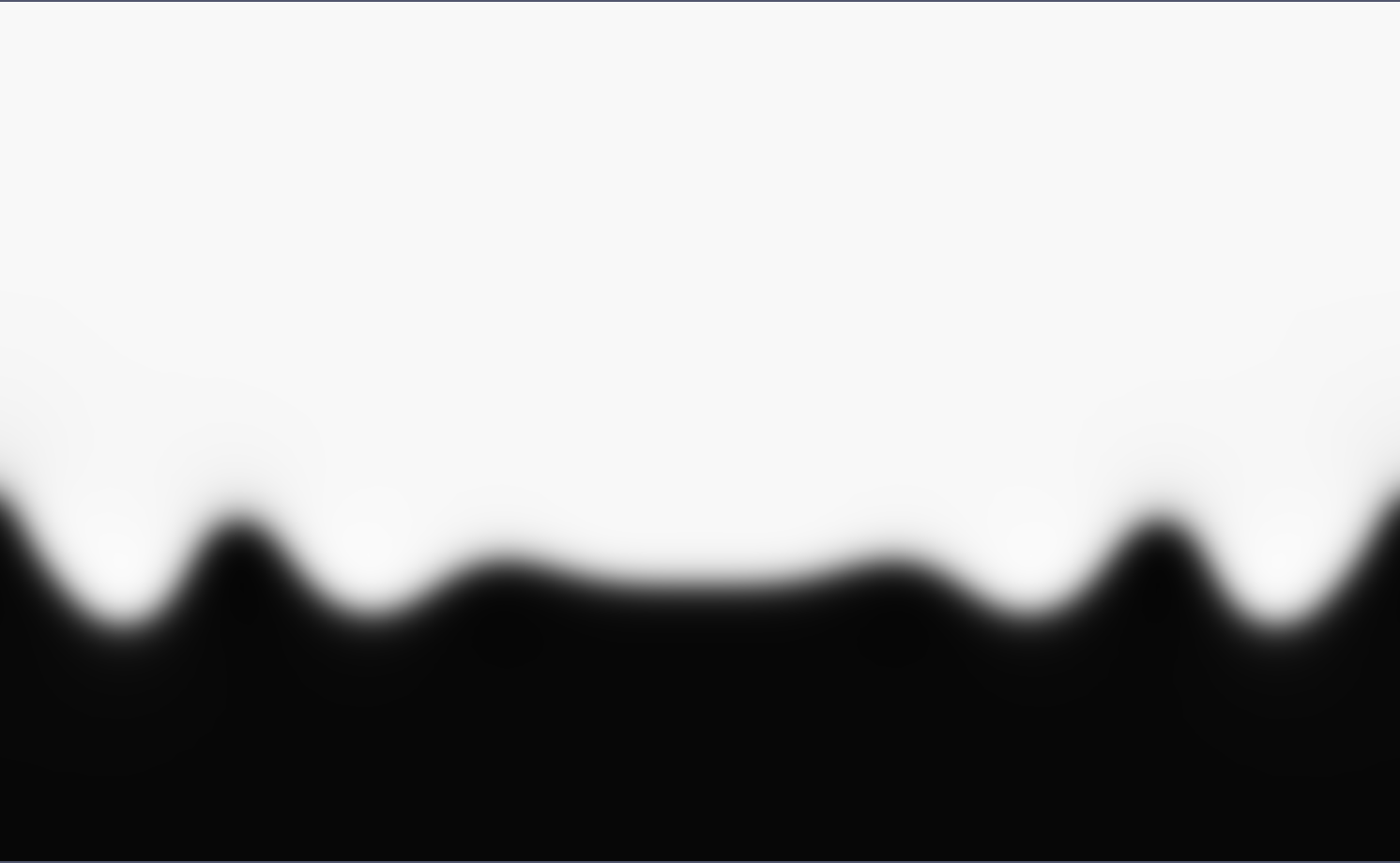}}\\

    \fbox{\includegraphics[width=32mm]{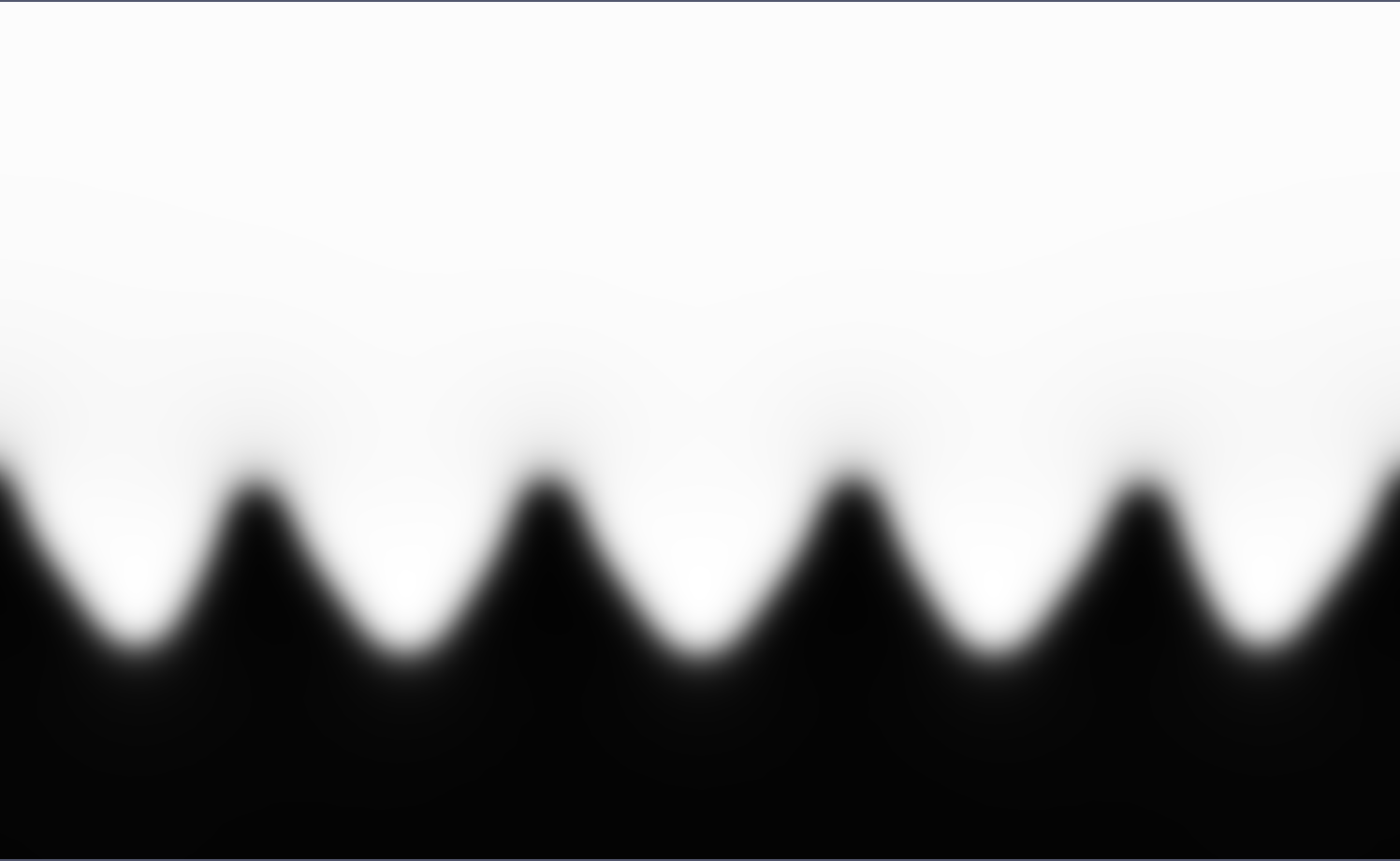}}&
    \fbox{\includegraphics[width=32mm]{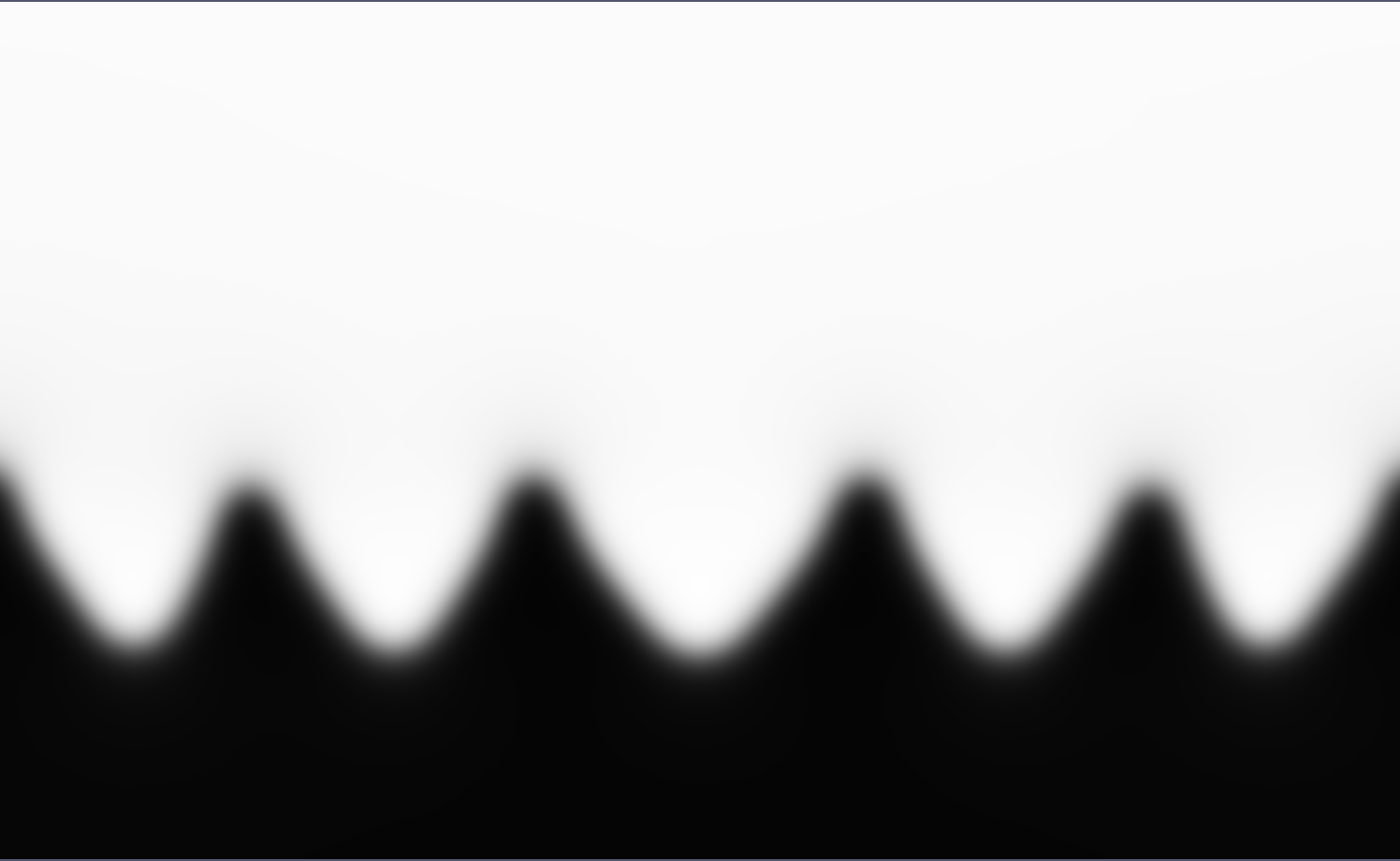}}&
    \fbox{\includegraphics[width=32mm]{Pictures/6L4000ts/20seq0010.png}}&
    \fbox{\includegraphics[width=32mm]{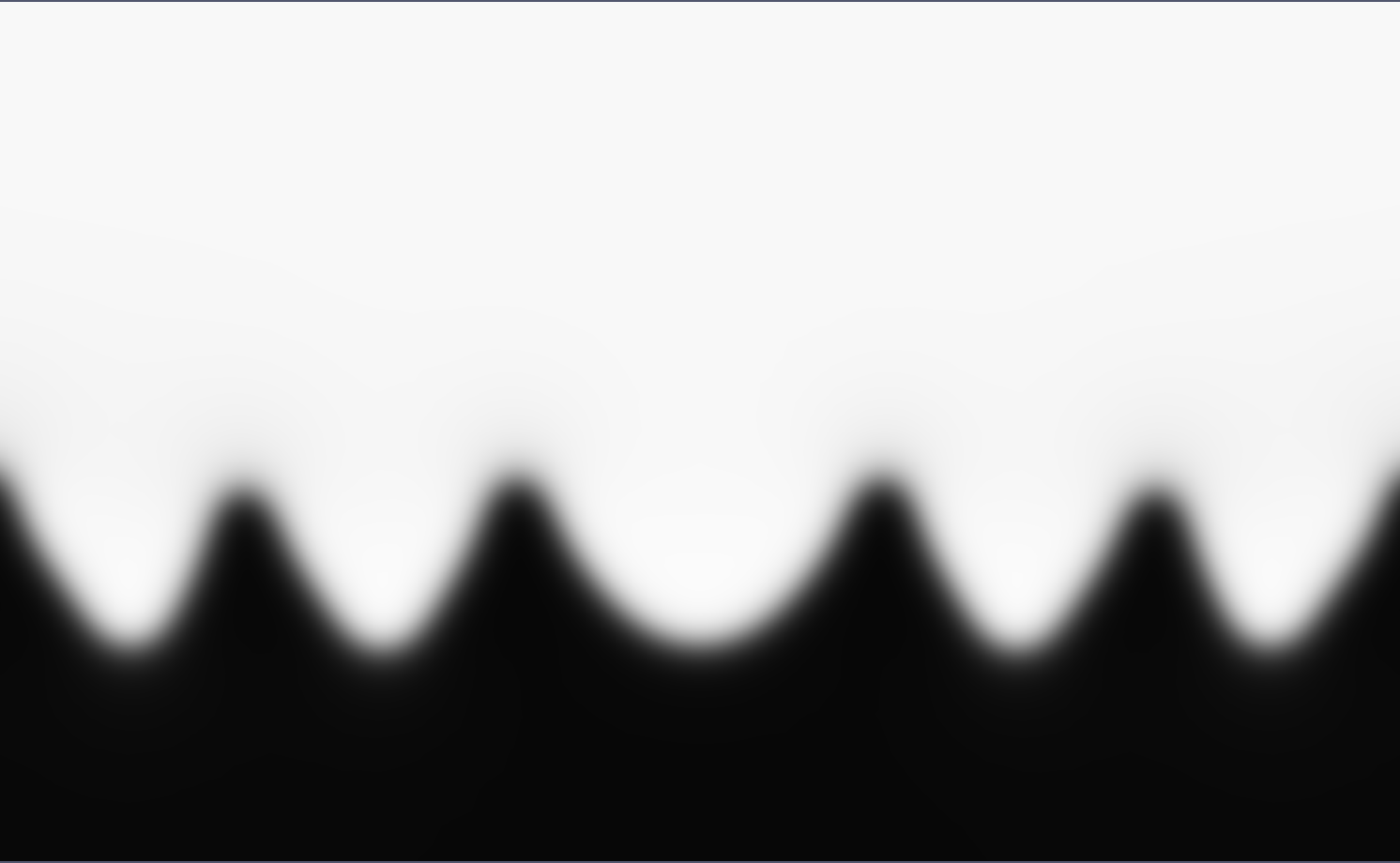}}\\
    
    \fbox{\includegraphics[width=32mm]{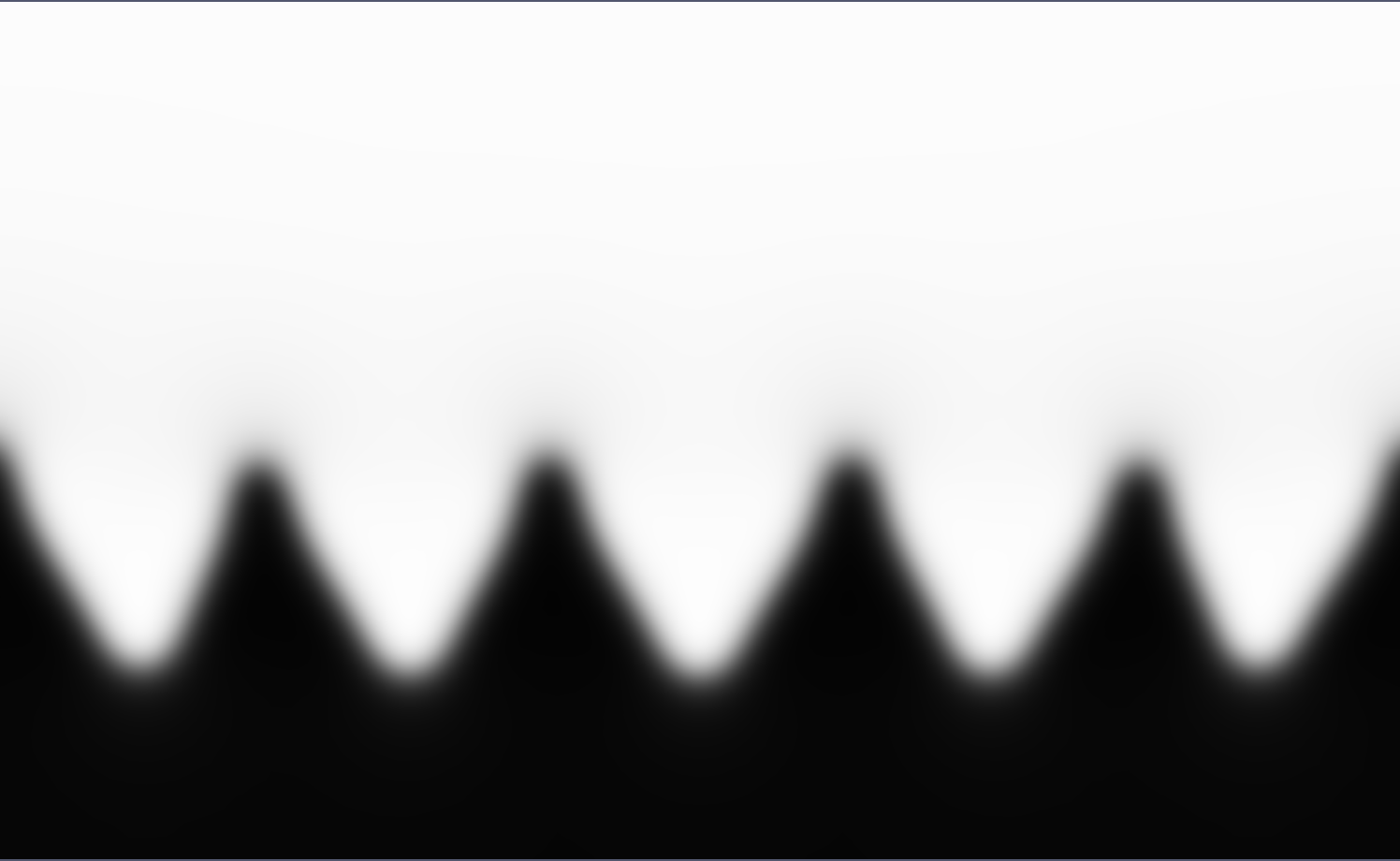}}&
    \fbox{\includegraphics[width=32mm]{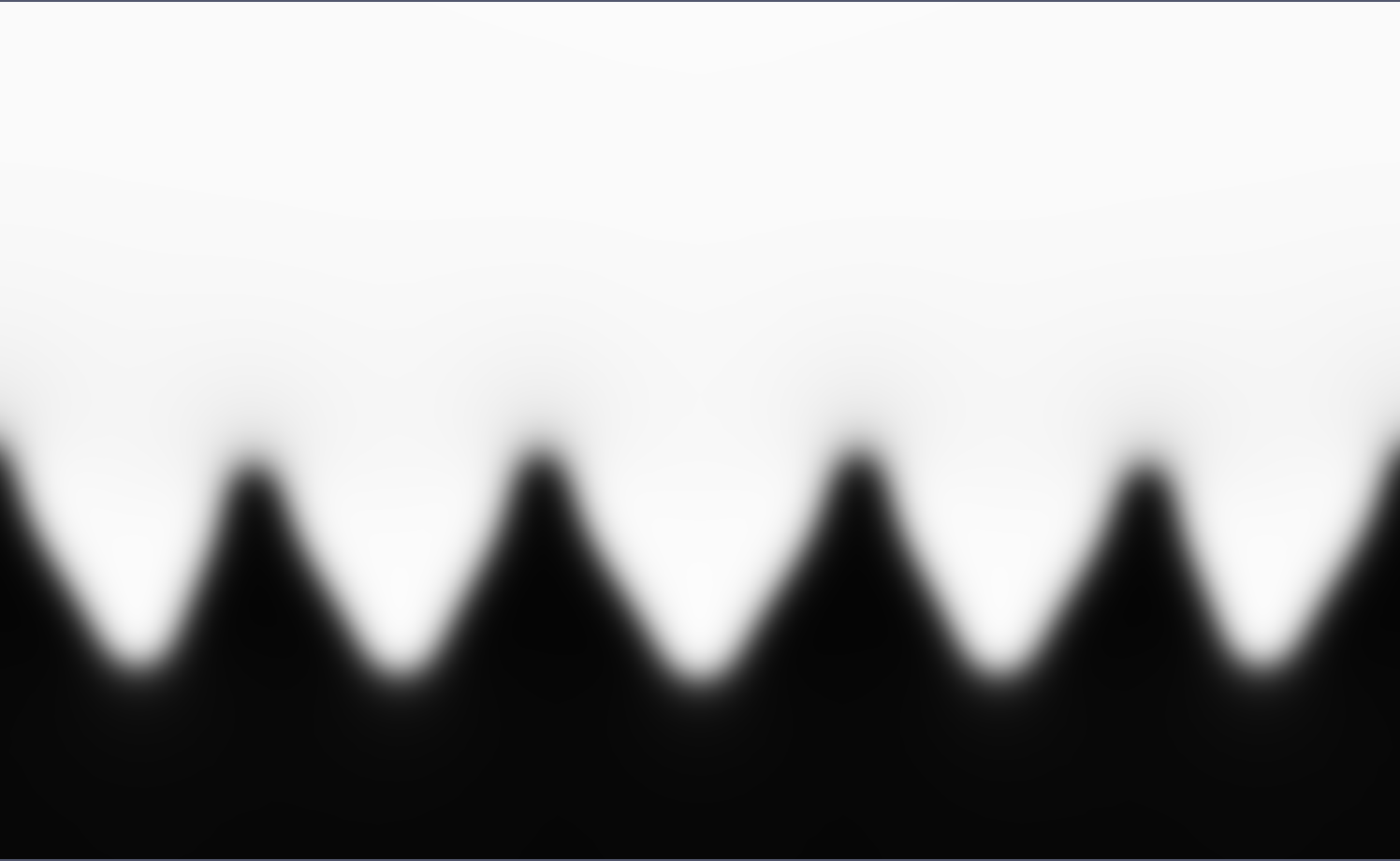}}&
    \fbox{\includegraphics[width=32mm]{Pictures/6L4000ts/20seq0011.png}}&
    \fbox{\includegraphics[width=32mm]{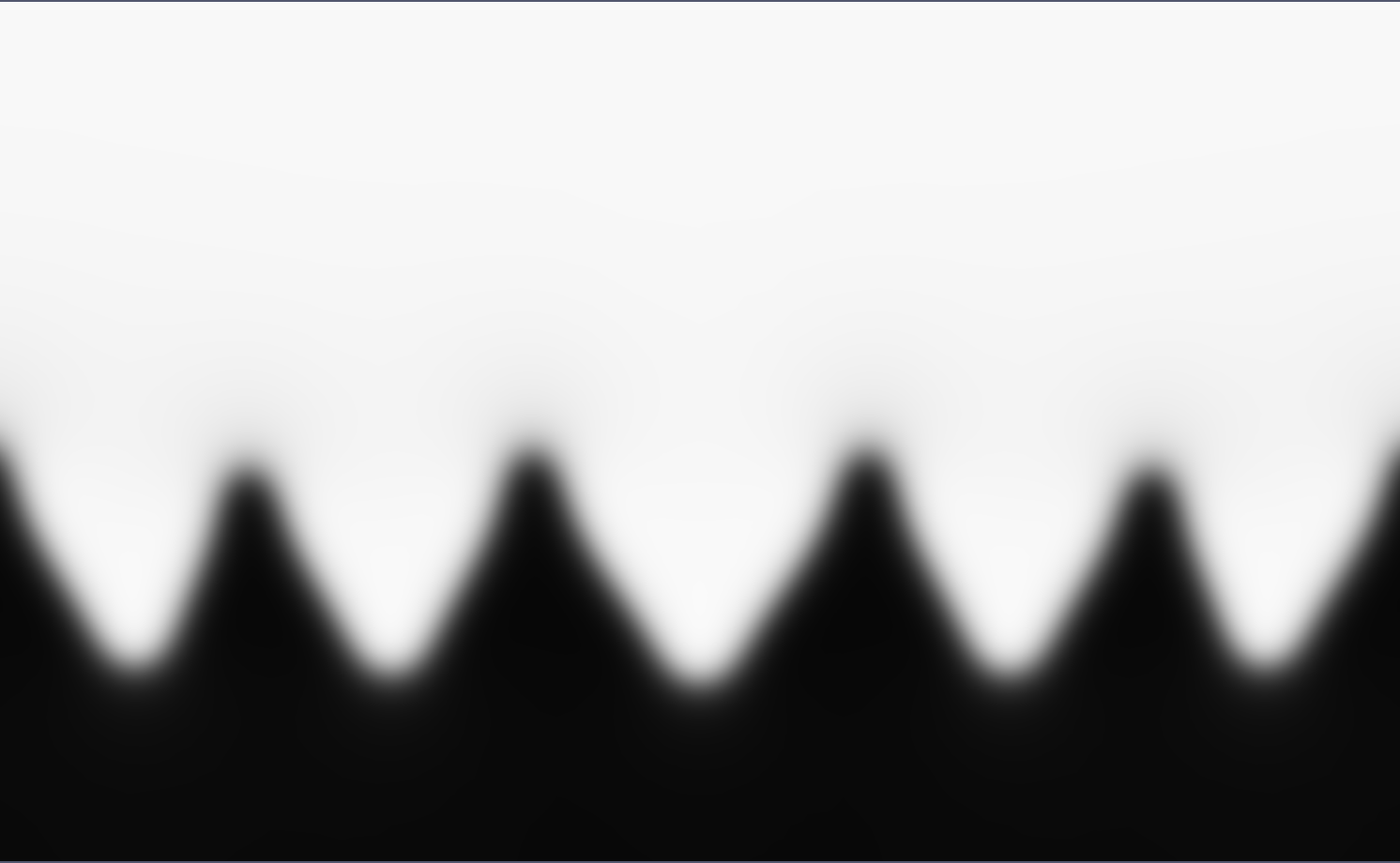}}\\
    
    \fbox{\includegraphics[width=32mm]{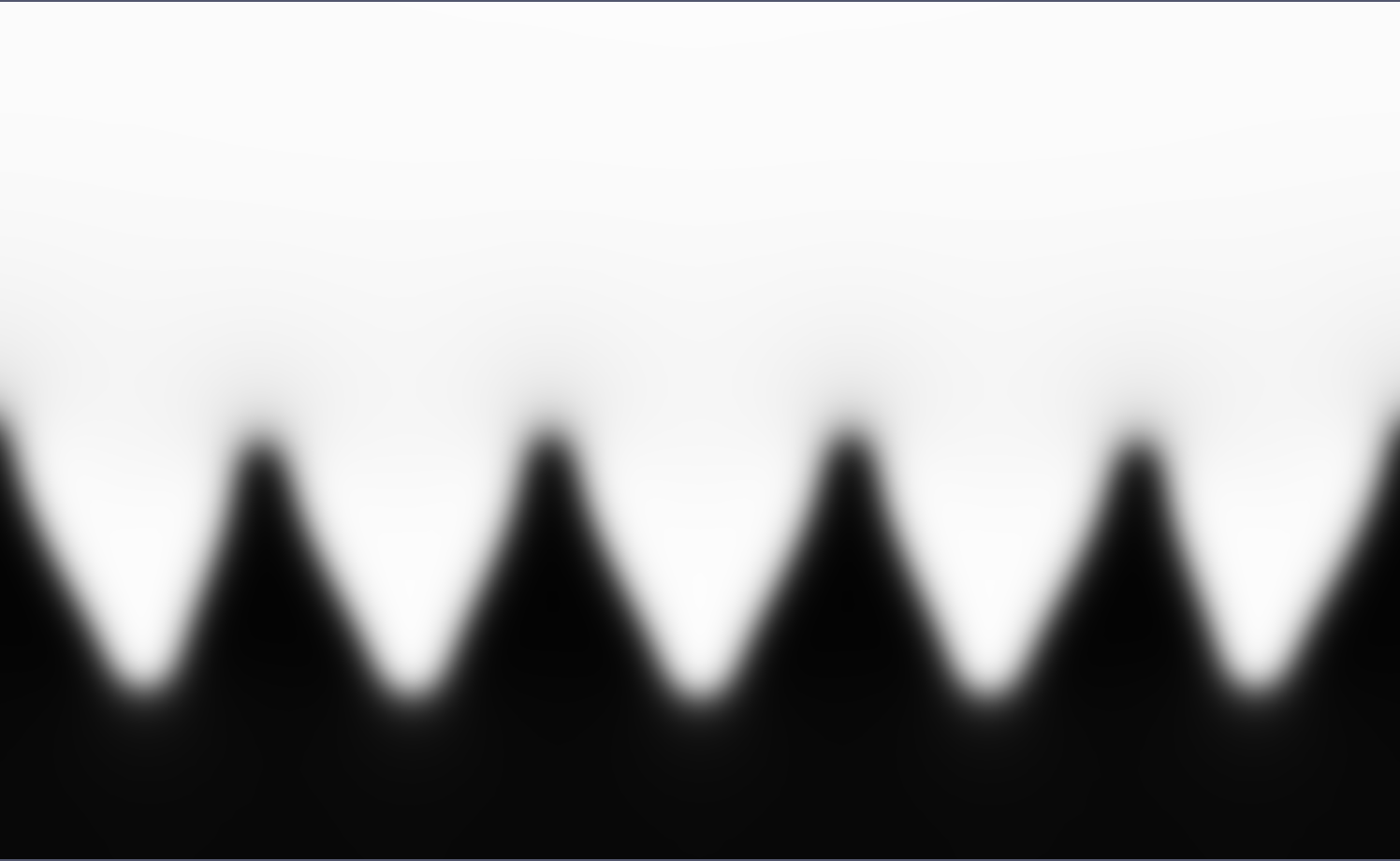}}&
    \fbox{\includegraphics[width=32mm]{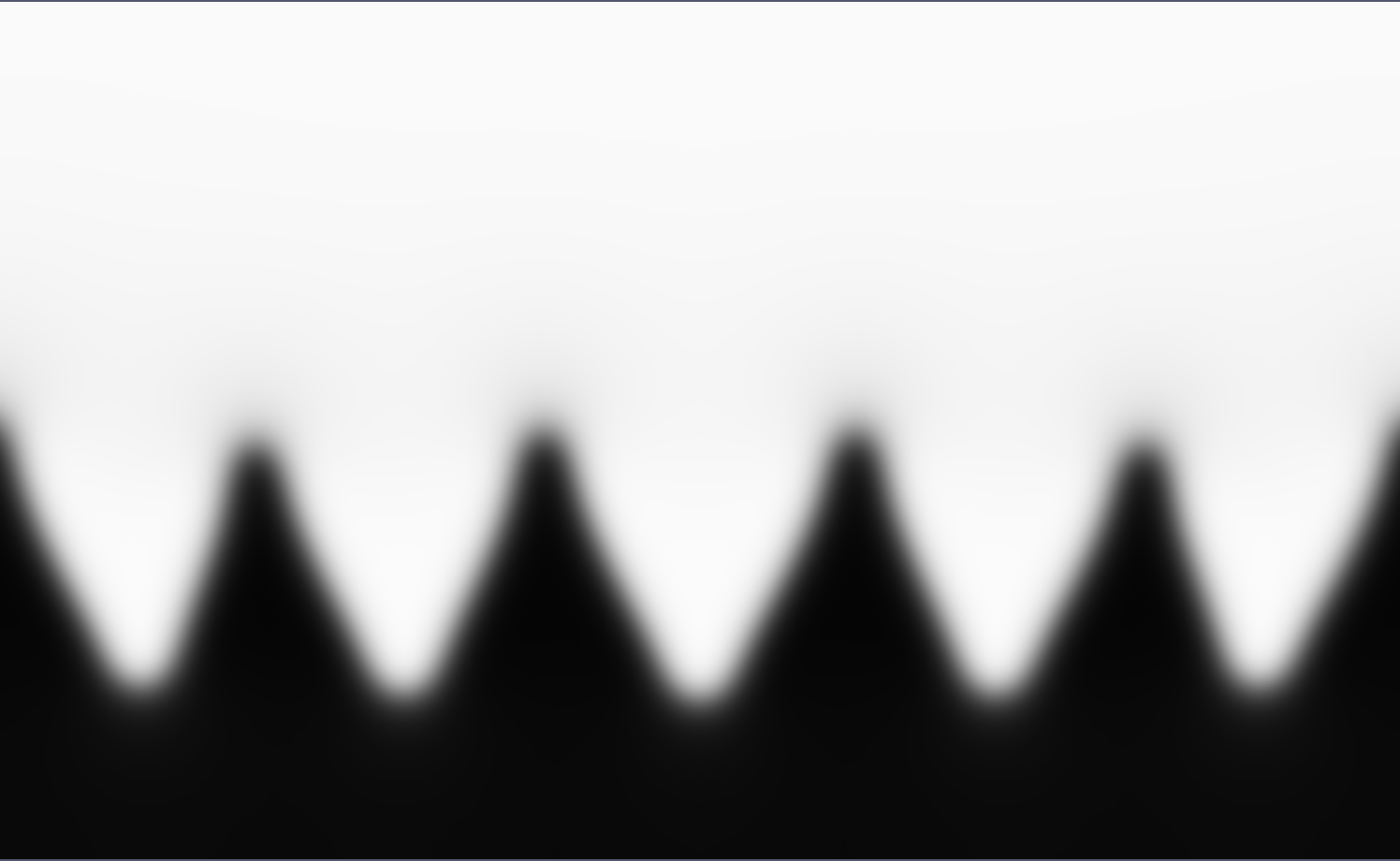}}&
    \fbox{\includegraphics[width=32mm]{Pictures/6L4000ts/20seq0012.png}}&
    \fbox{\includegraphics[width=32mm]{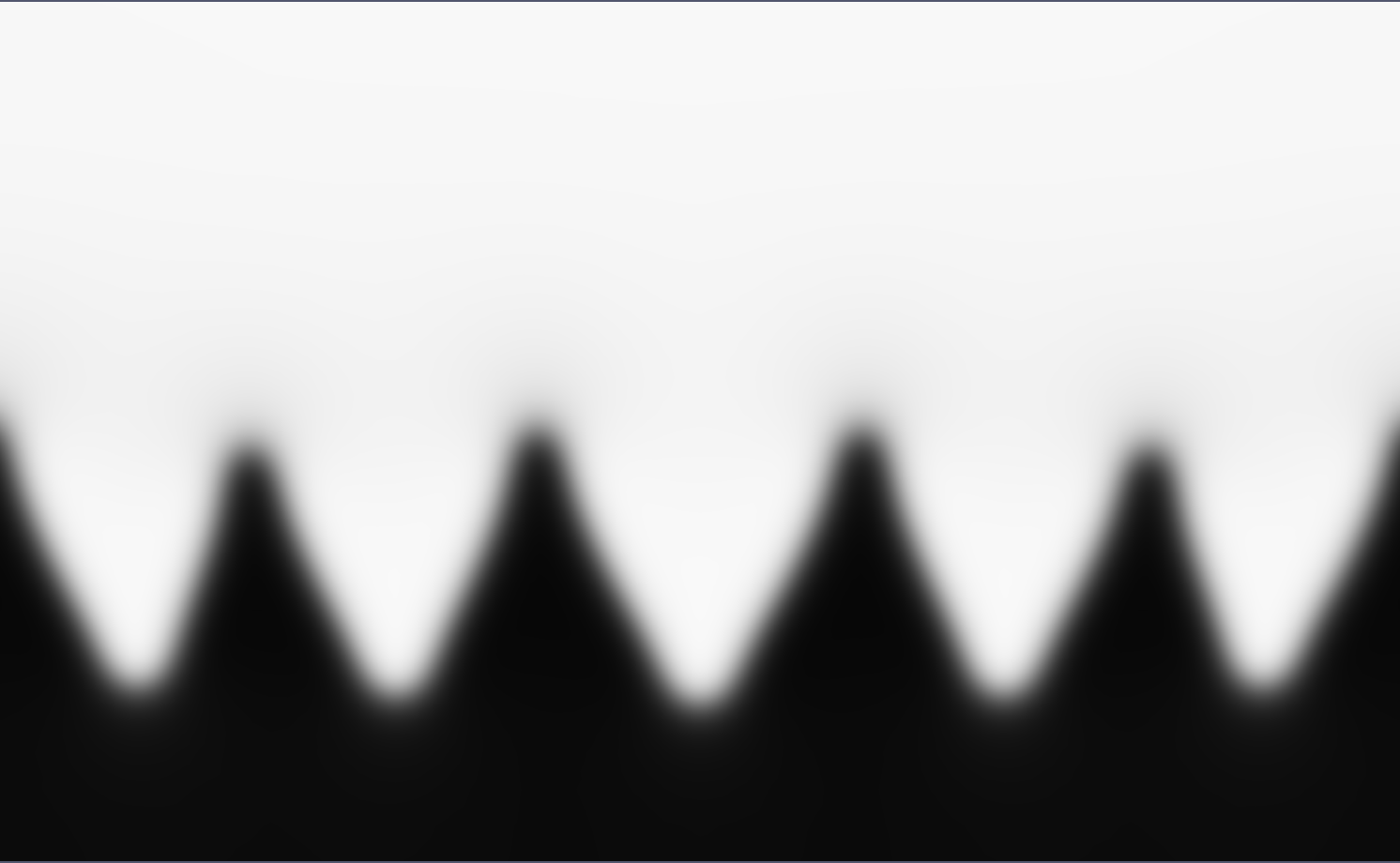}}\\

    \fbox{\includegraphics[width=32mm]{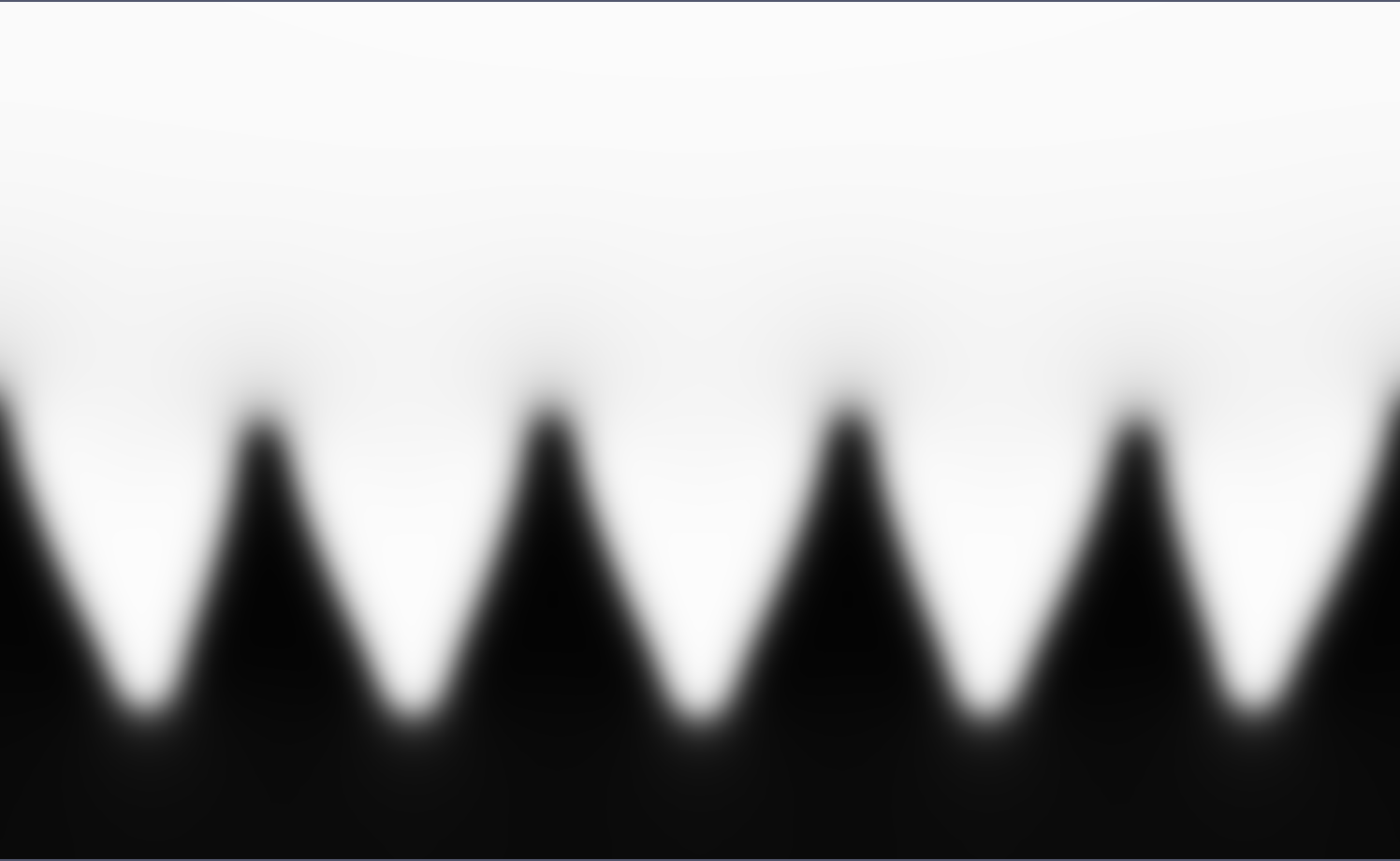}}&
    \fbox{\includegraphics[width=32mm]{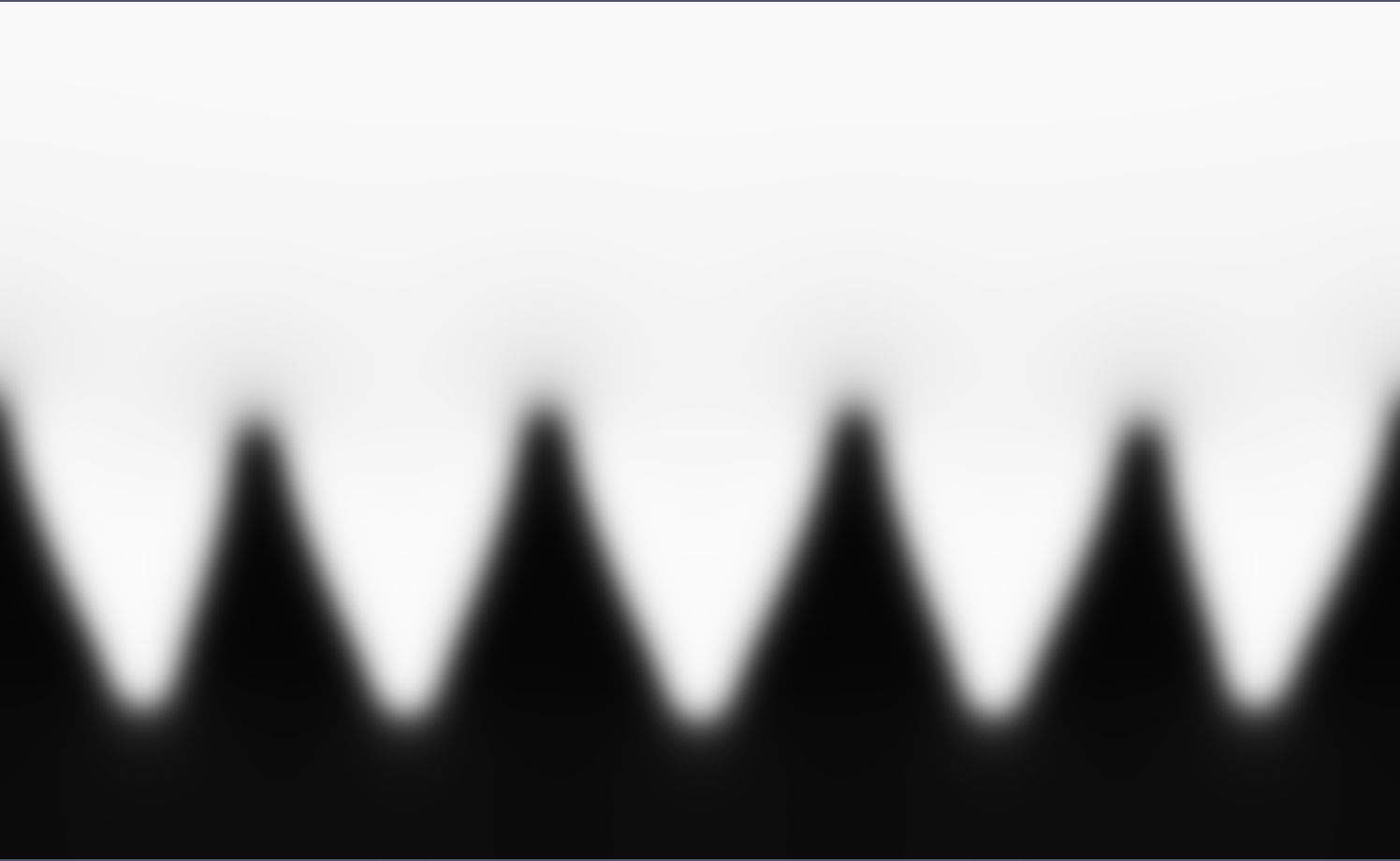}}&
    \fbox{\includegraphics[width=32mm]{Pictures/6L4000ts/20seq0013.png}}&
    \fbox{\includegraphics[width=32mm]{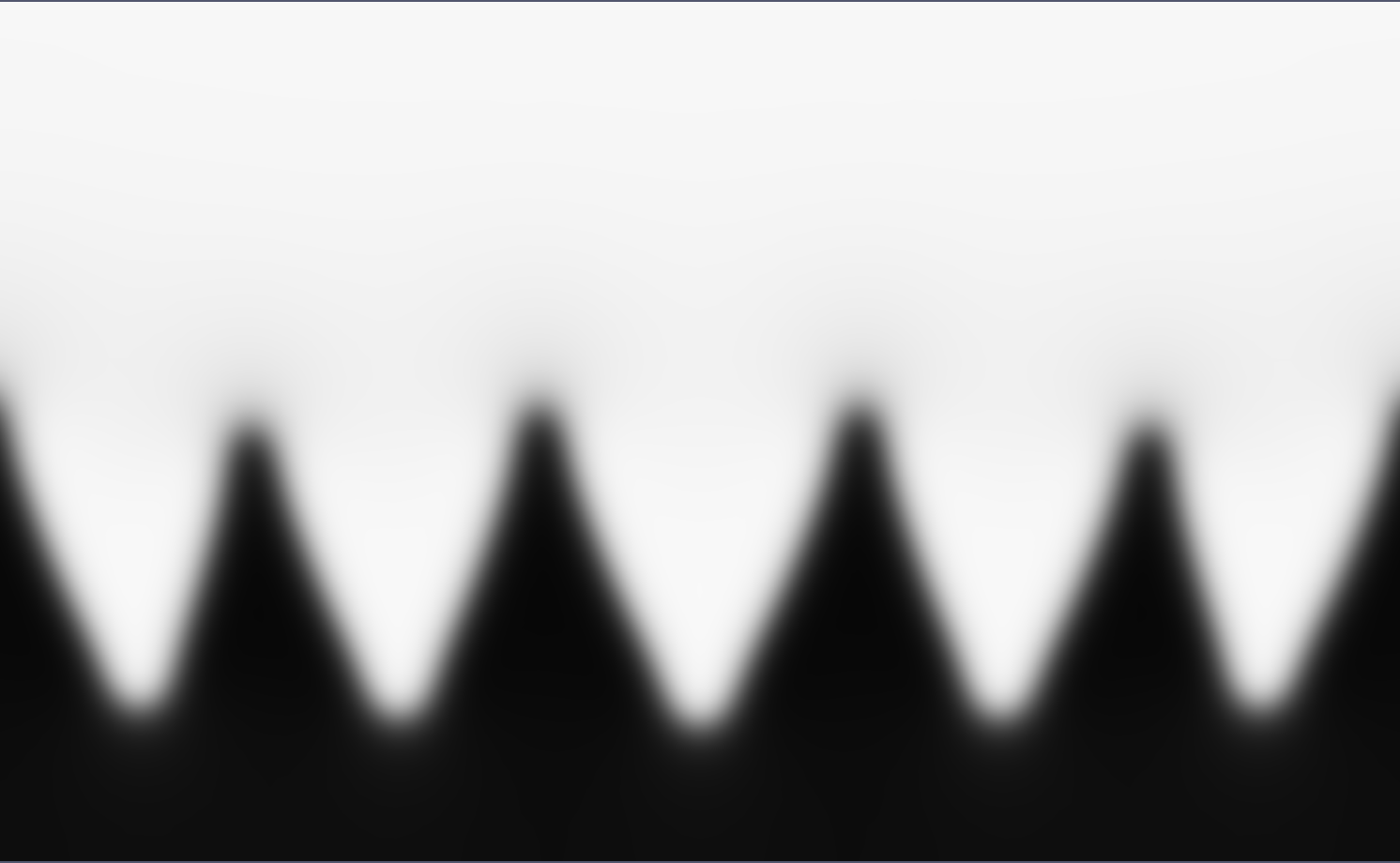}}
  \end{tabular}
\end{center}
\caption[Rosensweig instability: parametric study on the time discretization]{\textbf{Parametric study: time discretization.} This figure shows the results obtained from time $t=0.7$ (uppermost row) to time $t= 1.3$ (lowermost row) in intervals $\dt=0.1$ using 6 levels of refinement in space and four different time discretizations: the coarsest time discretization uses 1000 times steps (first column), 2000 time steps (second column), 4000 time steps (third column), and the finest discretization 8000 time steps (fourth column). The reader can appreciate that even the coarsest time discretization does not generate artificial or spurious features in the numerical solution.\label{parafigtime}}
\end{figure}

Taking a leap of faith, we can only expect \eqref{RosLin} to be able to deliver the right order of magnitude for the relationship between the gravity $g$ and the surface tension coefficient $\surftens$ which could yield a predetermined number of peaks. Consider that we want four peaks inside our unit length box, that is $\distpeaks = 0.25$, combining  \eqref{RosLin}, \eqref{densformula} and \eqref{surfcaprel}, and inserting our choice of parameters ($\layerthick = 0.01$, $\capcoeff = 0.05$, $\Delta\rho \approx 0.1$, $\distpeaks = 0.25$) we get
\begin{align}\label{gravity}
g = \frac{4 \pi^2 \surftens}{\distpeaks^2 \Delta\rho} \approx
\frac{4 \pi^2 \capcoeff}{\distpeaks^2 \Delta\rho \layerthick} \approx 3 \cdot 10^4 \, . 
\end{align}
Therefore, to obtain four peaks inside our unit-size box, the appropriate order of magnitude for the gravity is $10^4$. We use \eqref{gravity} as an educated guess and load $\bv{g} = (0,-30000)^T$ in the computer code.

In order to generate an almost uniform magnetic field, we place 5 dipoles pointing upwards, that is $\bv{d} = (0,1)^T$ (see formula \eqref{dipole2D}), sufficiently far away from our rectangular box, so that for most practical purposes the magnetic field is uniform, having only a slight  gradient (decay) in the $y$ direction. The coordinates $\bv{x}_s$ of the dipoles are $(-0.5,-15)$, $(0,-15)$, $(0.5,-15)$, $(1,-15)$ and $(1.5,-15)$. The intensity $\alpha_s$ (see expression \eqref{haformula}) is the same for each dipole but evolves in time. More precisely, $\alpha_s$ is increased linearly in time, starting from $\alpha_s = 0$ at time $t=0$ to its maximum value $\alpha_s = 6000$ at time $t = 1.6$, and from time $t = 1.6$ to $t = 2.0$ the intensity of the dipoles is kept constant in order to let the system rest and develop a stable configuration. 

Regarding the space discretization, the initial mesh has 10 elements in the $x$ direction and $6$ in the $y$ direction, and allow for a maximum refinement of 4, 5, 6 and 7 levels. On the other hand, regarding time discretization we  use 1000, 2000, 4000, and 8000 time steps for a total of $\tf=2$ units of simulation time.

We display numerical results in Figure \ref{figureResolved} achieved with this non-trivial setup, involving a choice of coefficients, a specific configuration of the external magnetic field $\ha$, and space adaptivity. The simulation starts with a ferrofluid pool of $0.2$ units of depth at rest at time $t = 0$, and at time $t=2.0$ there are five peaks inside the box rather than the expected four. Clearly, \eqref{gravity} is able to deliver a very reasonable initial guess. In Figure \ref{meshsamplefig} we show a sample finite element mesh corresponding to the simulation of Figure \ref{figureResolved}. We can see that most of the interesting dynamics in Figure \ref{figureResolved} happens from times $t = 0.7$ to $t = 1.3$, therefore we focus on the interval of time $[0.7,1.3]$ for a parametric study in order to show the robustness of this simulation with respect to the discretization parameters $h$ and $\dt$. Figures \ref{parafigspace} and \ref{parafigtime} depict the results of the 
parametric study with respect to space and time discretization respectively. The results from Figure \ref{figureResolved} correspond to the third column of Figures \ref{parafigspace} and \ref{parafigtime} which, as it can be appreciated, is a meaningful (well-resolved) solution.

The Rosensweig instability considered in this section, in practice, can only be reproduced under carefully controlled laboratory conditions. That is, this instability is not the most common form of ferrofluid instability observed in everyday experiments (such as commercial ferrofluid toys) since in practice most magnetic fields are by no means uniform nor have magnetic field lines very aligned. This is the reason why in \S\ref{exp2} we consider a much more mundane form of the Rosensweig instability, involving non-uniform magnetic fields with relatively poor alignment of the magnetic field lines.


\subsection{The ferrofluid hedgehog}\label{exp2} In this section we carry out two numerical experiments in order to explore the effects of a non-uniform magnetic field, depth of the ferrofluid pool, and the effects of the demagnetizing field. For these experiments we take the same constitutive constants as in \S\ref{exp1}, with the exception of 
\begin{align*}
 \suscep = 0.9 \ , \ \ \layerthick = 0.005 \ \ \text{and} \ \ \capcoeff = 0.025 \, . 
\end{align*}
We increase the magnetic susceptibility $\suscep$ to make the effects of the demagnetizing field $\hd$ much more pronounced than those of \S\ref{exp2}, we reduce the layer thickness $\layerthick$ to diminish diffusive effects and get sharper interfaces, and reduce $\capcoeff$ to get slightly more unstable (more sensitive to perturbations) interfaces. The depth of the ferrofluid pool is now of $0.11$ units. We use 6 levels of refinement, but the initial mesh has $15$ elements in the $x$ direction and $9$ elements in the $y$ direction. Regarding temporal discretization we use a total of $24000$ times steps for $\tf=6$ units of simulation time.

It is clear that changing many parameters at the same time (magnetic susceptibility, pool depth, capillarity coefficient, and layer thickness) makes it hard to understand the separate influence of each of them on the behavior of the system. Doing a parametric/sensitivity study of each variable at a time would be highly desirable, but that would involve an ambitious separate analysis. Our objective is just to showcase other interesting phenomena (another instance of the Rosensweig instability) that can be captured with this simple PDE model. We call this instability the ``ferrofluid hedgehog'', because of its natural resemblance with the spiny mammal.

The experiments are carried out in the same rectangular domain as in
\S\ref{exp1} (with vertices at $(0,0)$, $(0,0.6)$, $(1,0.6)$ and
$(1,0)$). The magnetic field $\ha = \sum_{s} \alpha_s \nabla\hapot_s$
is generated by a set of $42$ dipoles. More precisely, we want to
create a crude ``discrete'' approximation of what would be the
magnetic field due to a bar magnet of 0.4 units of width and $0.5$
units of height pointing upwards (i.e. $\bv{d} = (0,1)^T$ again). The
dipoles are located in three rows, at $y = -0.5$, $y = -0.75$ and
  $y = -1.0$, and the 14 dipoles within of each row are
equi-distributed in the $x$ direction as shown in Figure
\ref{FigDipsHedge}. The main idea of this setup is to recreate a
non-uniform magnetic field, with an open pattern of magnetic field
lines (as sketched in Figure \ref{FigDipsHedge}) rather than aligned
magnetic field lines (as in \S\ref{exp1}).

\begin{figure}[h!]
\begin{center}
  \includegraphics[scale=0.30]{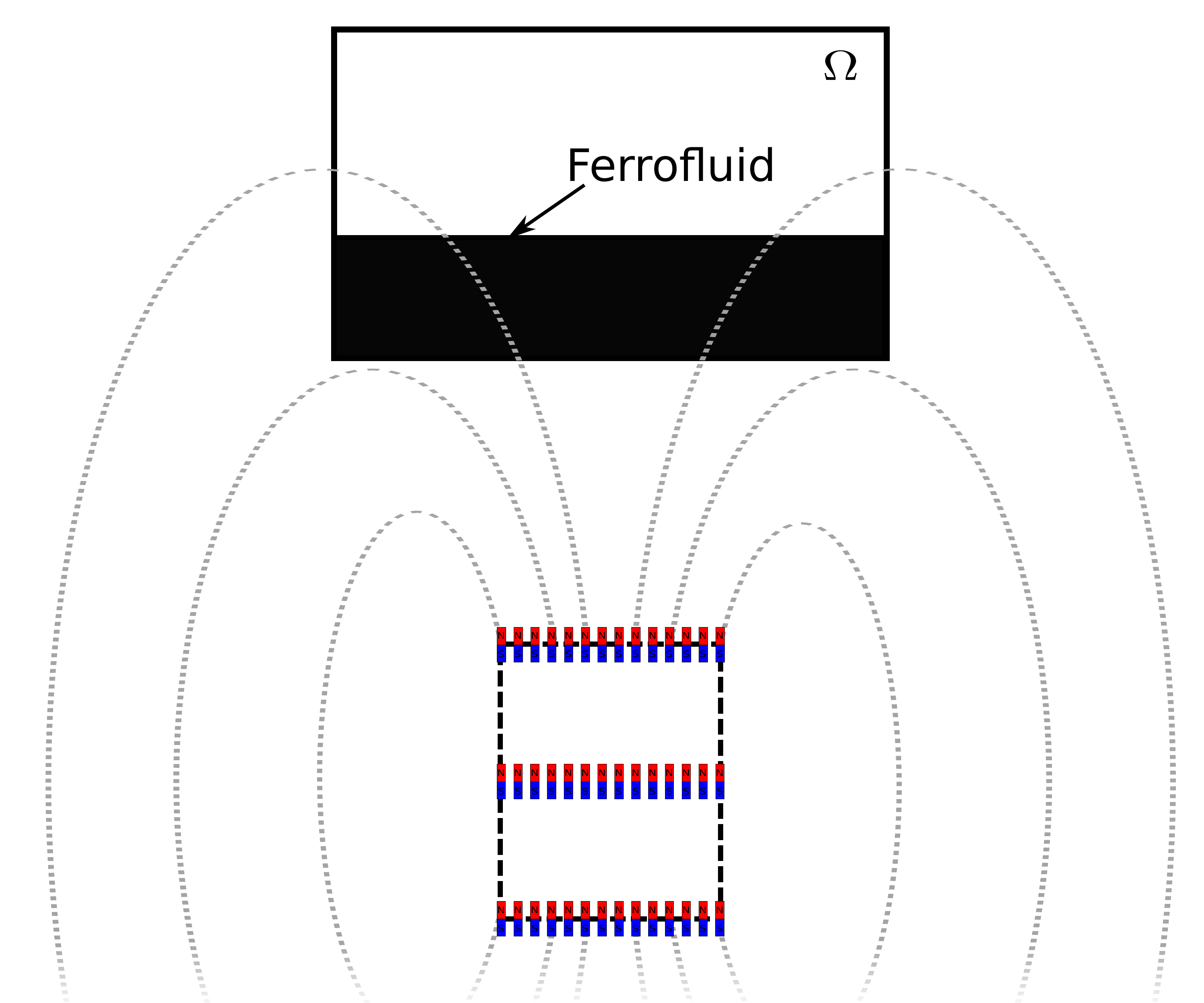}
\end{center}
\caption[The ferrofluid hedgehog: setup of the dipoles]{\textbf{The ferrofluid hedgehog: setup of the dipoles.} Setup of the dipoles for the experiment of \S\ref{exp2}, showing our rectangular domain $\Omega = (0,1)\times(0,0.6)$ with the ferrofluid (dark region) in the bottom of the box, and the arrangement of the dipoles below $\Omega$. The 42 dipoles are located in three rows in the lower part of the picture, here represented like small bar magnets, delivering a coarse approximation of what would be the magnetic field due to a bar magnet. The idea of such a configuration is to obtain an open pattern of magnetic field lines and steeper gradients than those of \S\ref{exp1}. \label{FigDipsHedge}}
\end{figure}

\begin{figure}
\begin{center}
    \setlength\fboxsep{0pt}
    \setlength\fboxrule{1pt}
  \begin{tabular}{cccc}

    \fbox{\includegraphics[width=32mm]{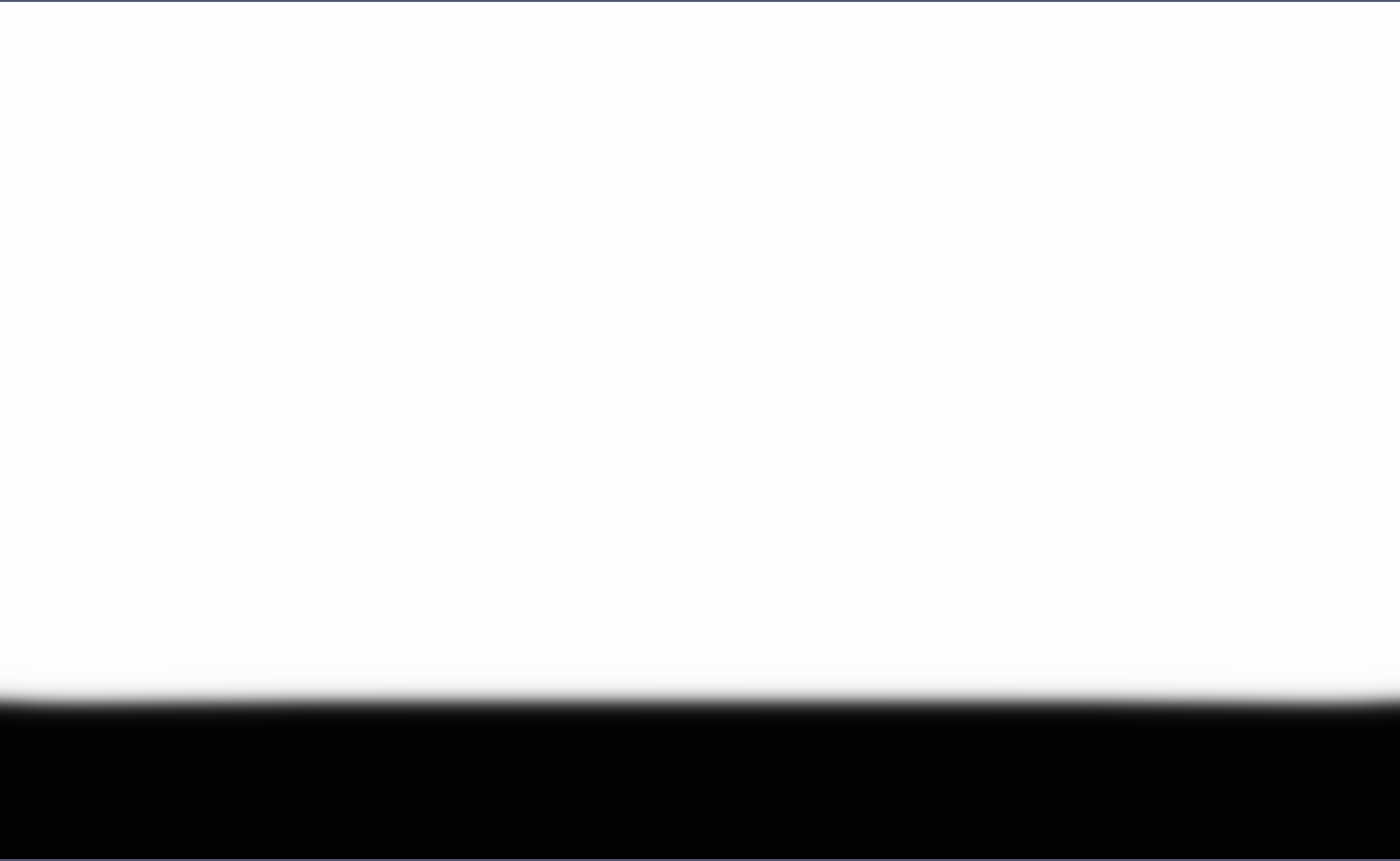}}&
    \fbox{\includegraphics[width=32mm]{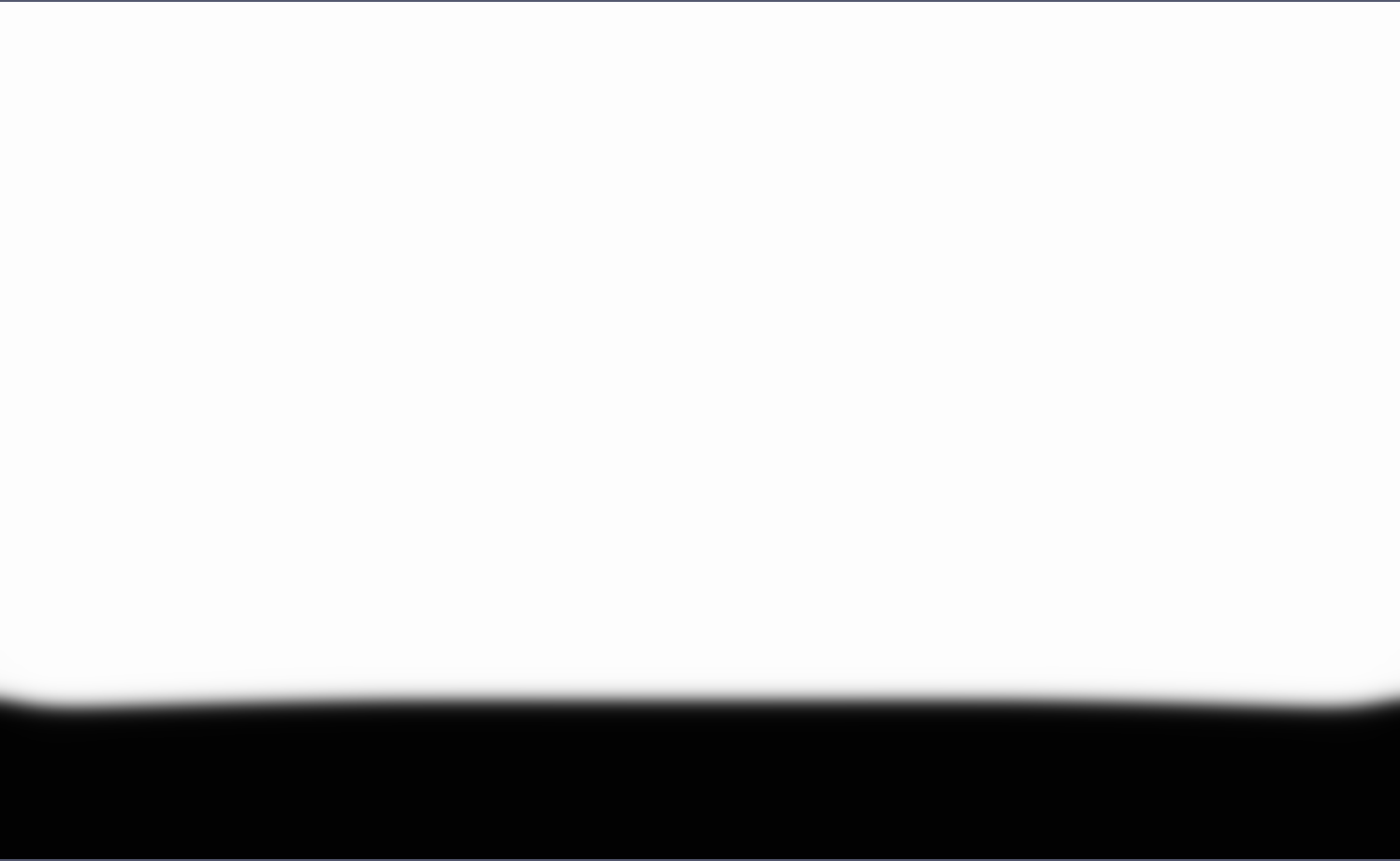}}&
    \fbox{\includegraphics[width=32mm]{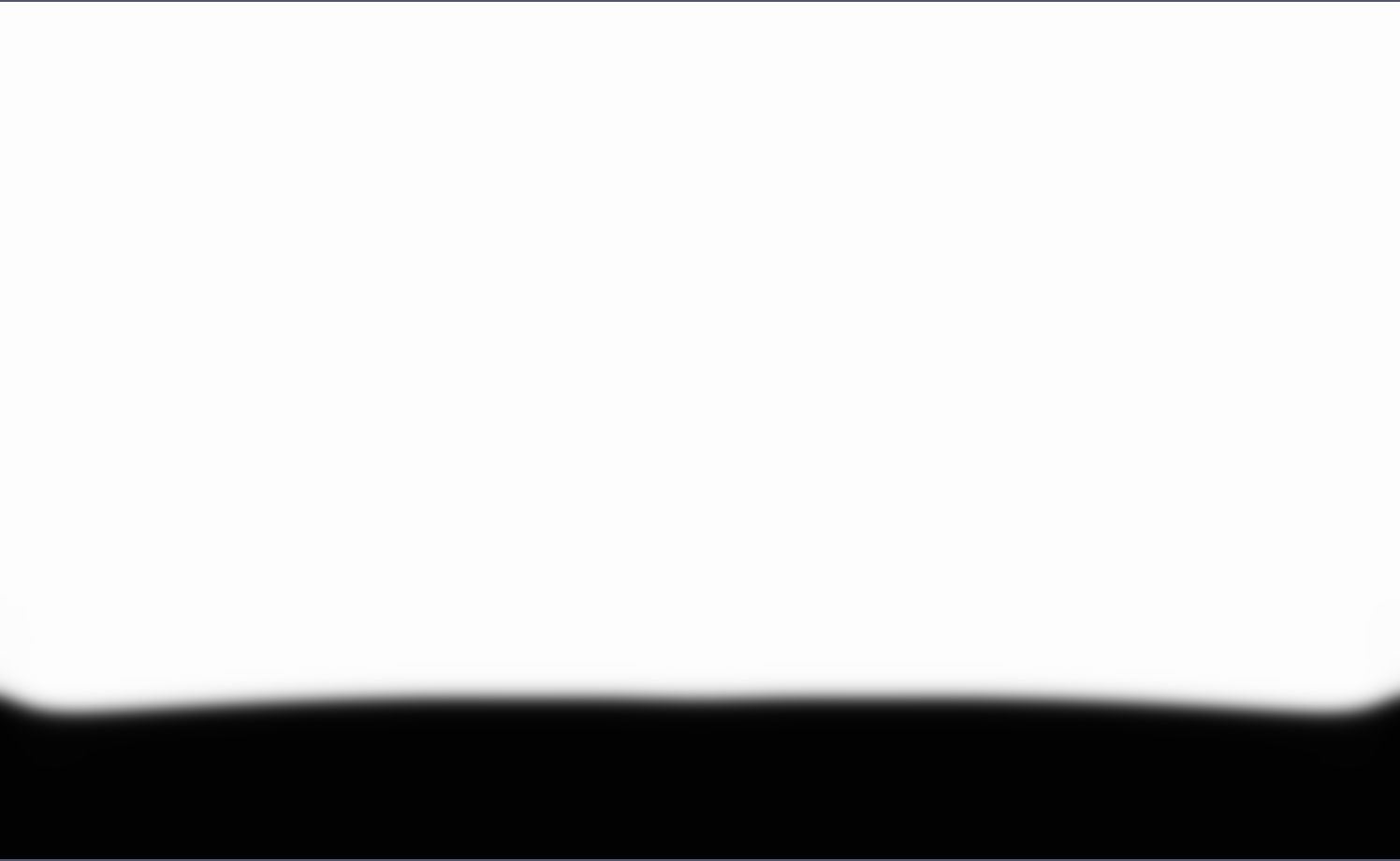}}&
    \fbox{\includegraphics[width=32mm]{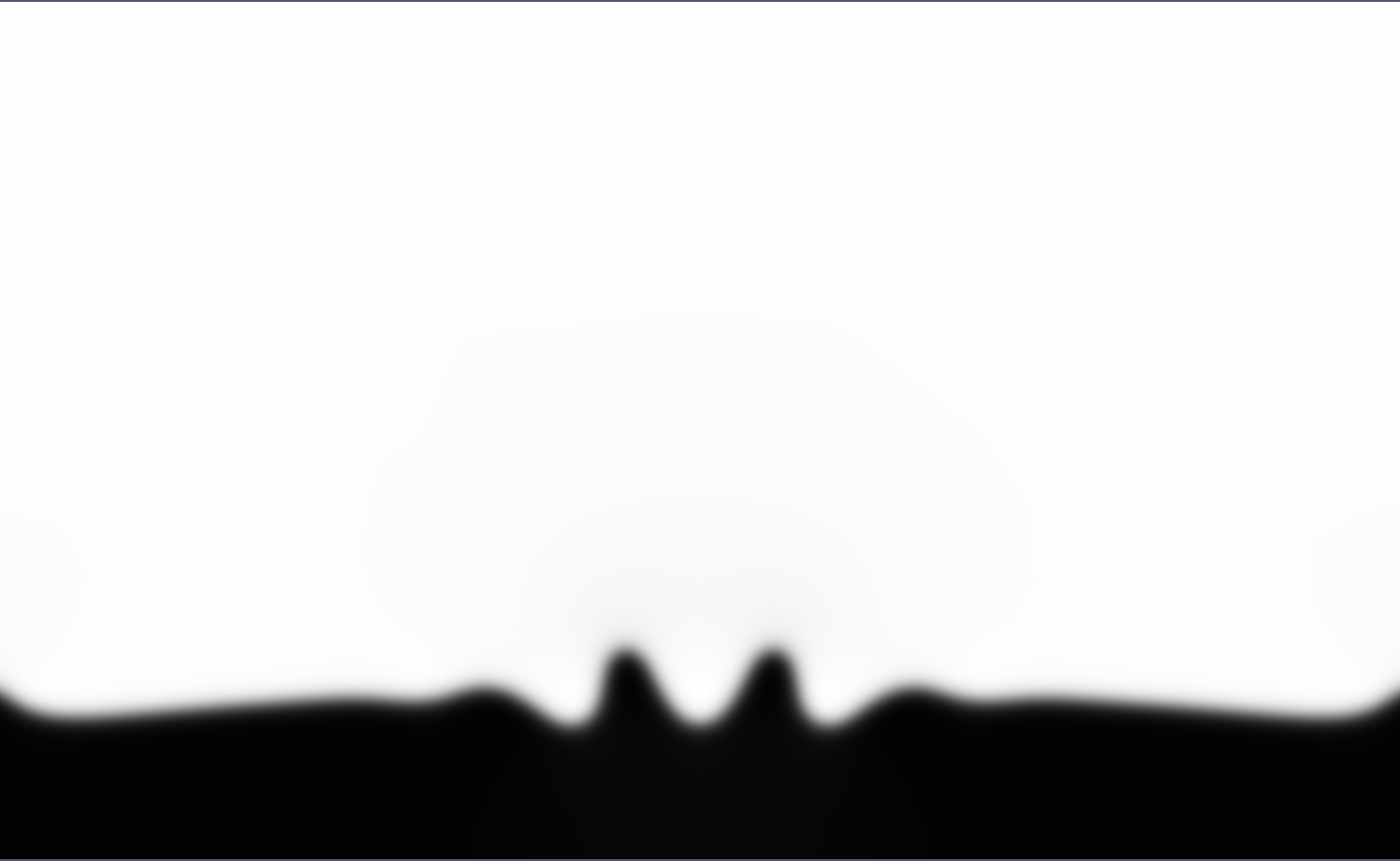}}\\

    \fbox{\includegraphics[width=32mm]{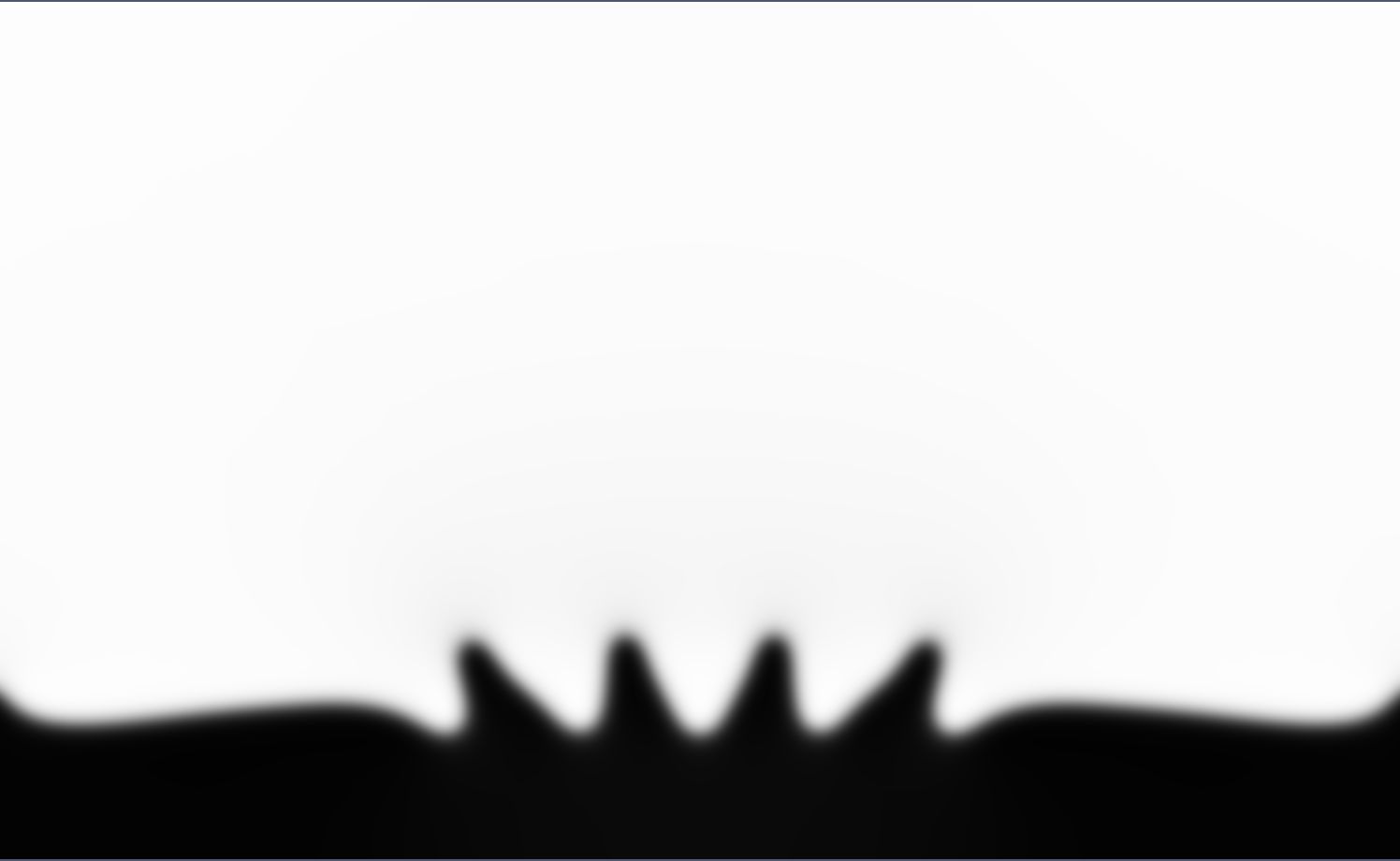}}&
    \fbox{\includegraphics[width=32mm]{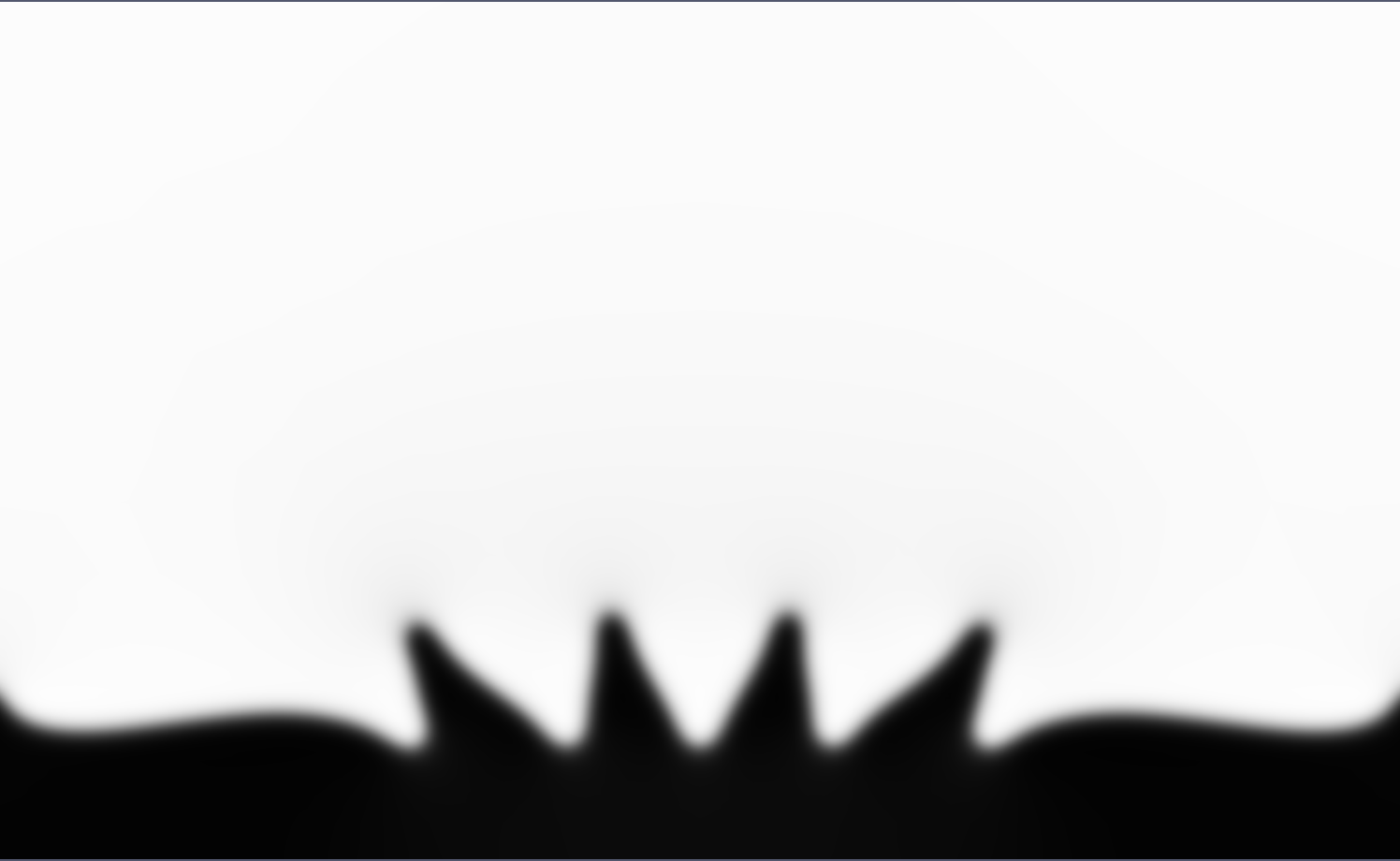}}&
    \fbox{\includegraphics[width=32mm]{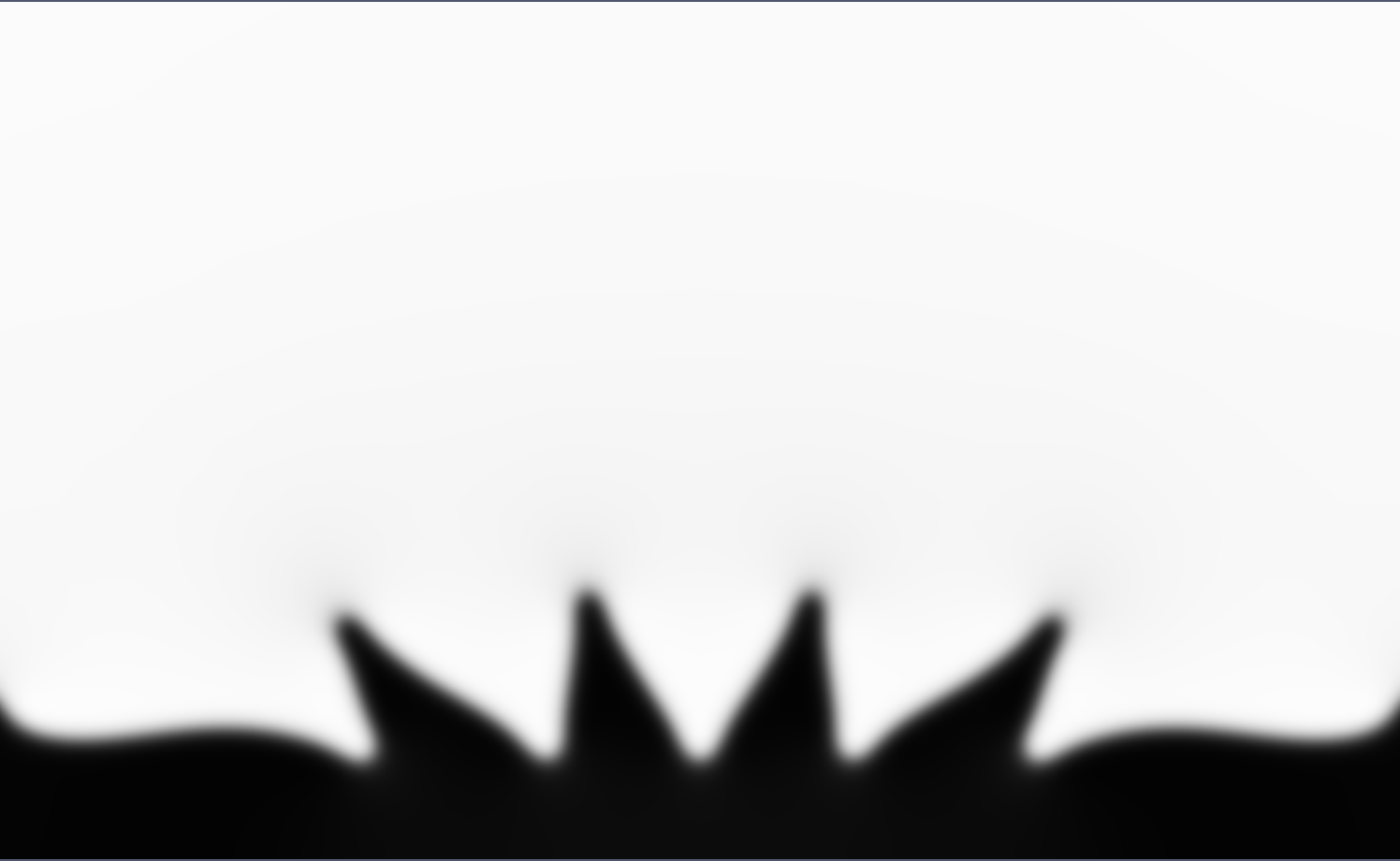}}&
    \fbox{\includegraphics[width=32mm]{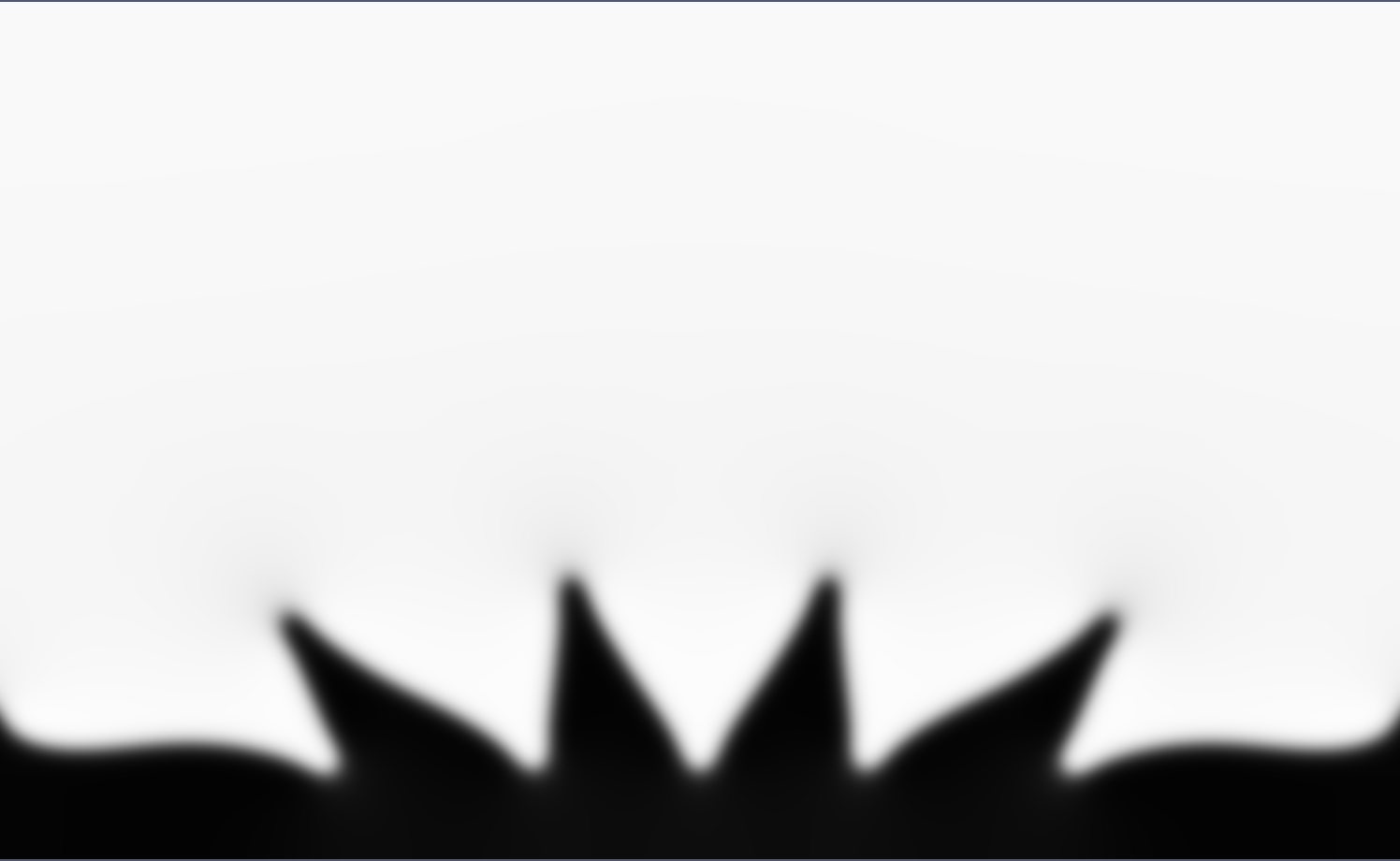}}\\

    \fbox{\includegraphics[width=32mm]{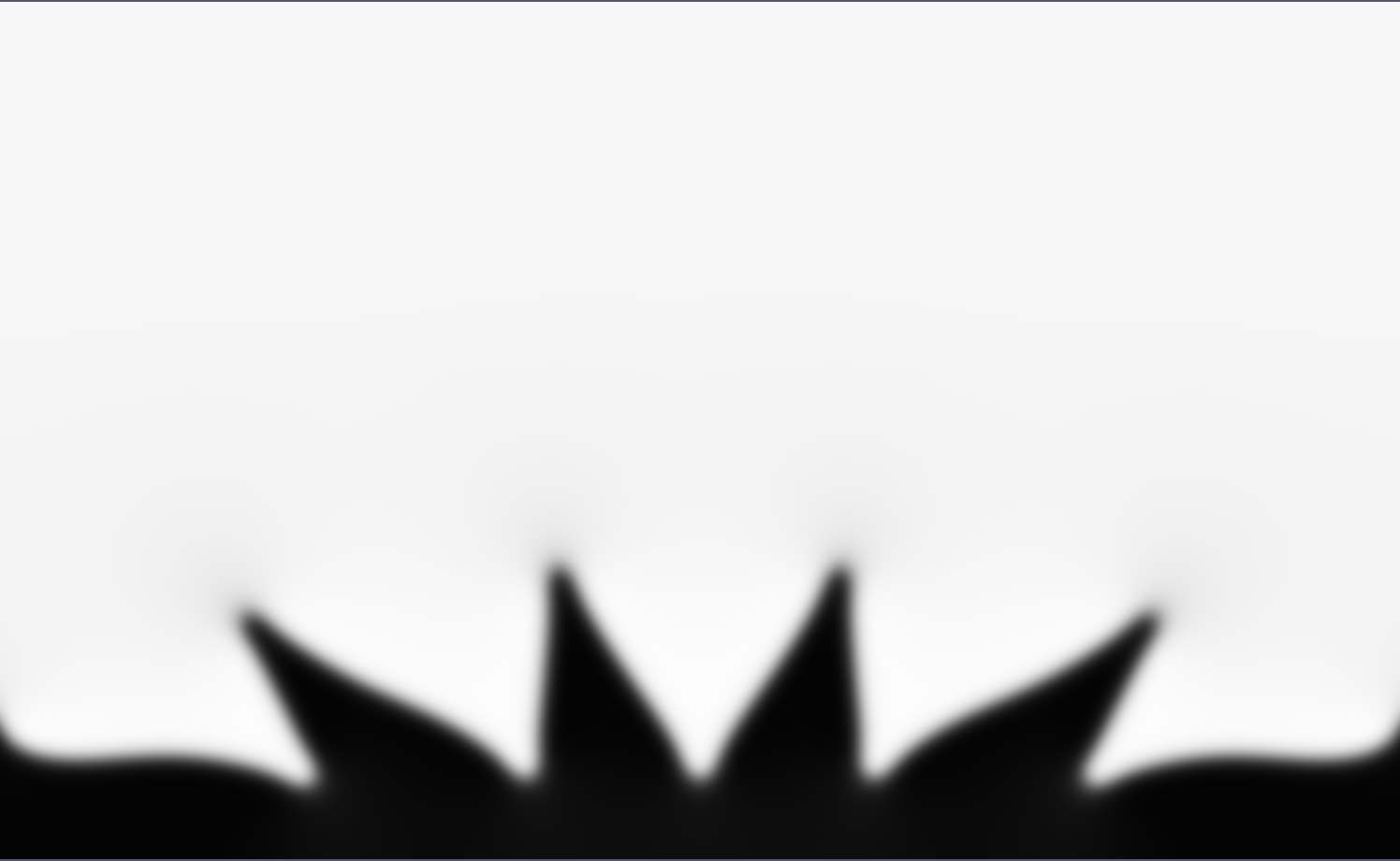}}&
    \fbox{\includegraphics[width=32mm]{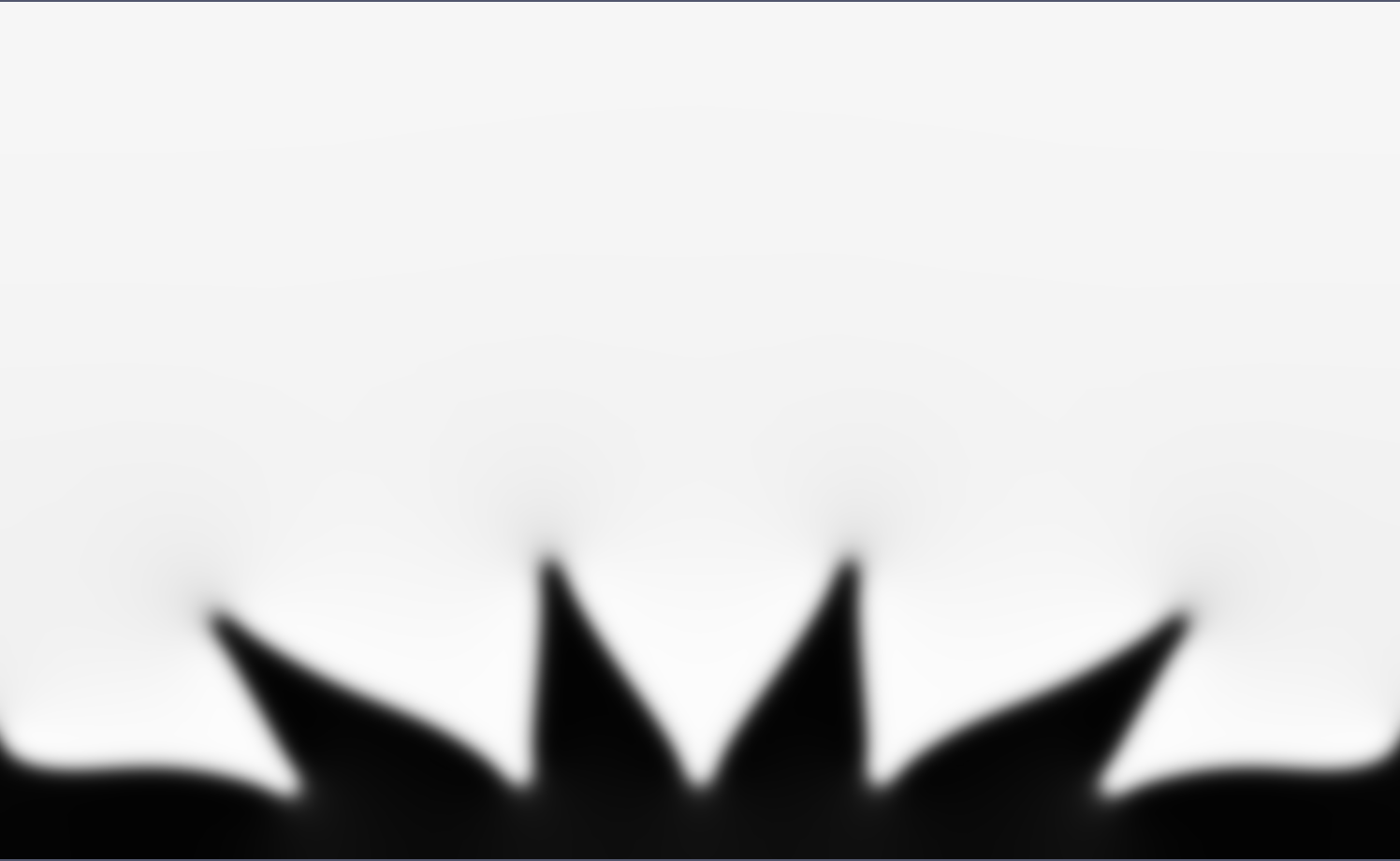}}&
    \fbox{\includegraphics[width=32mm]{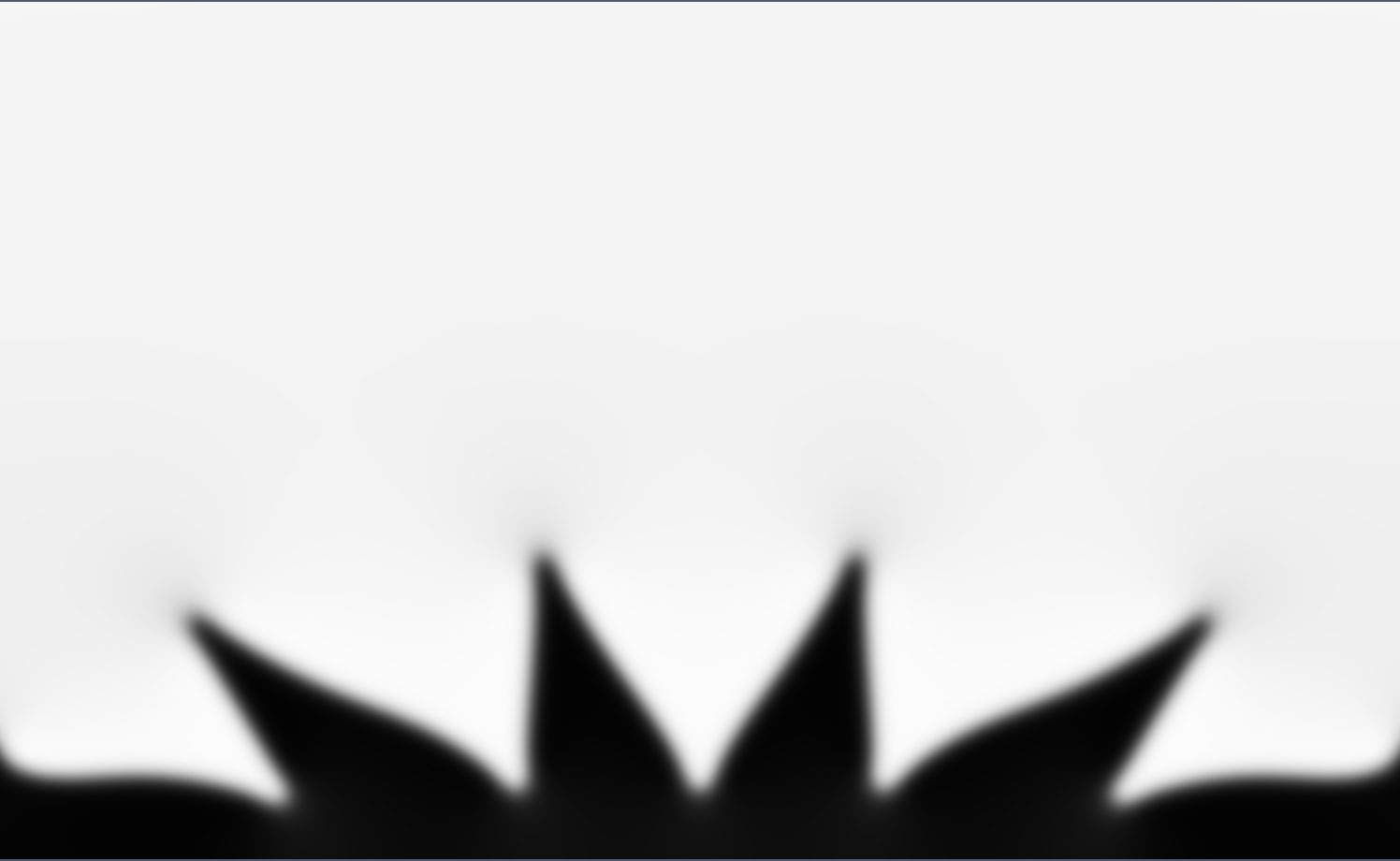}}&
    \fbox{\includegraphics[width=32mm]{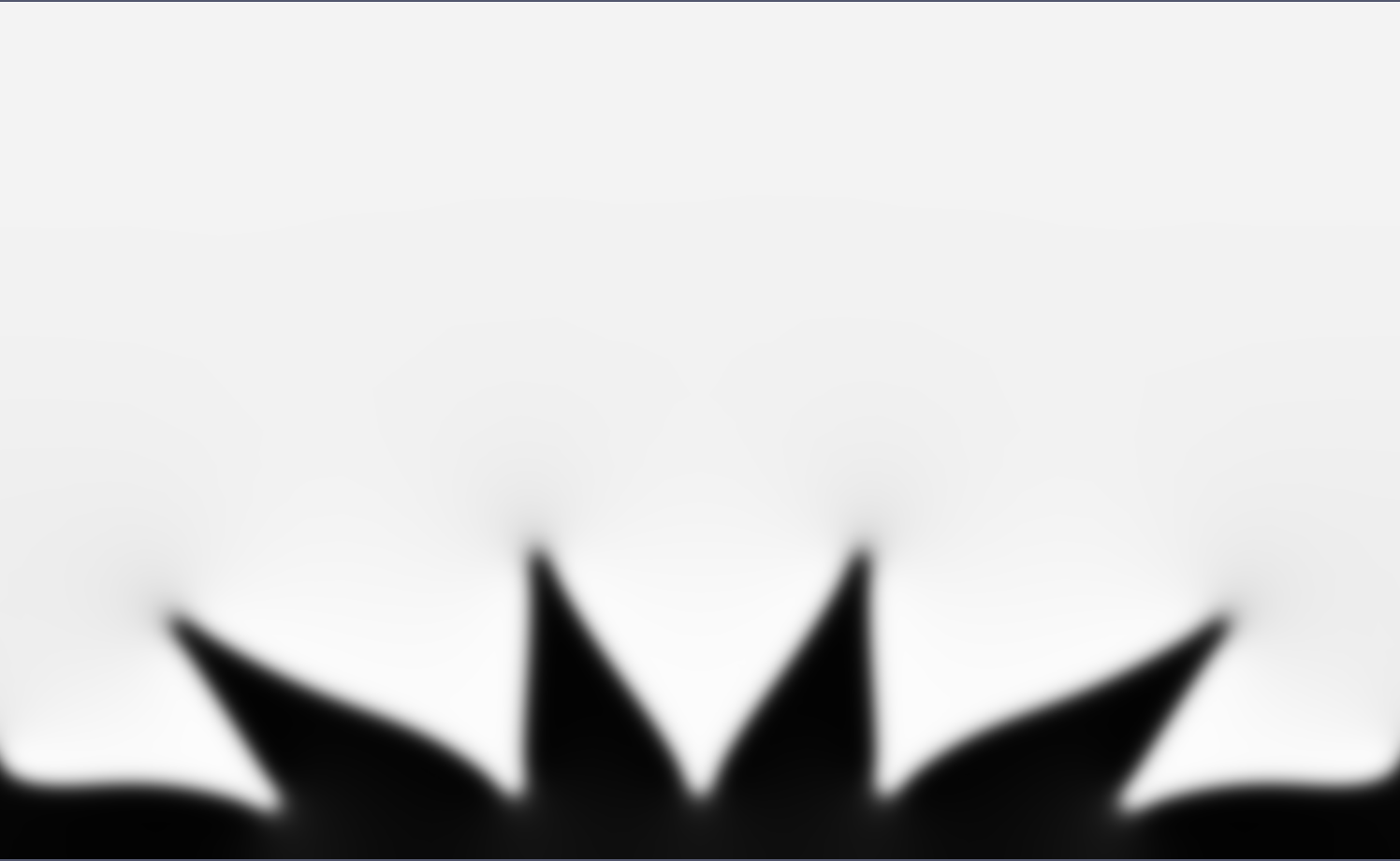}}\\

    \fbox{\includegraphics[width=32mm]{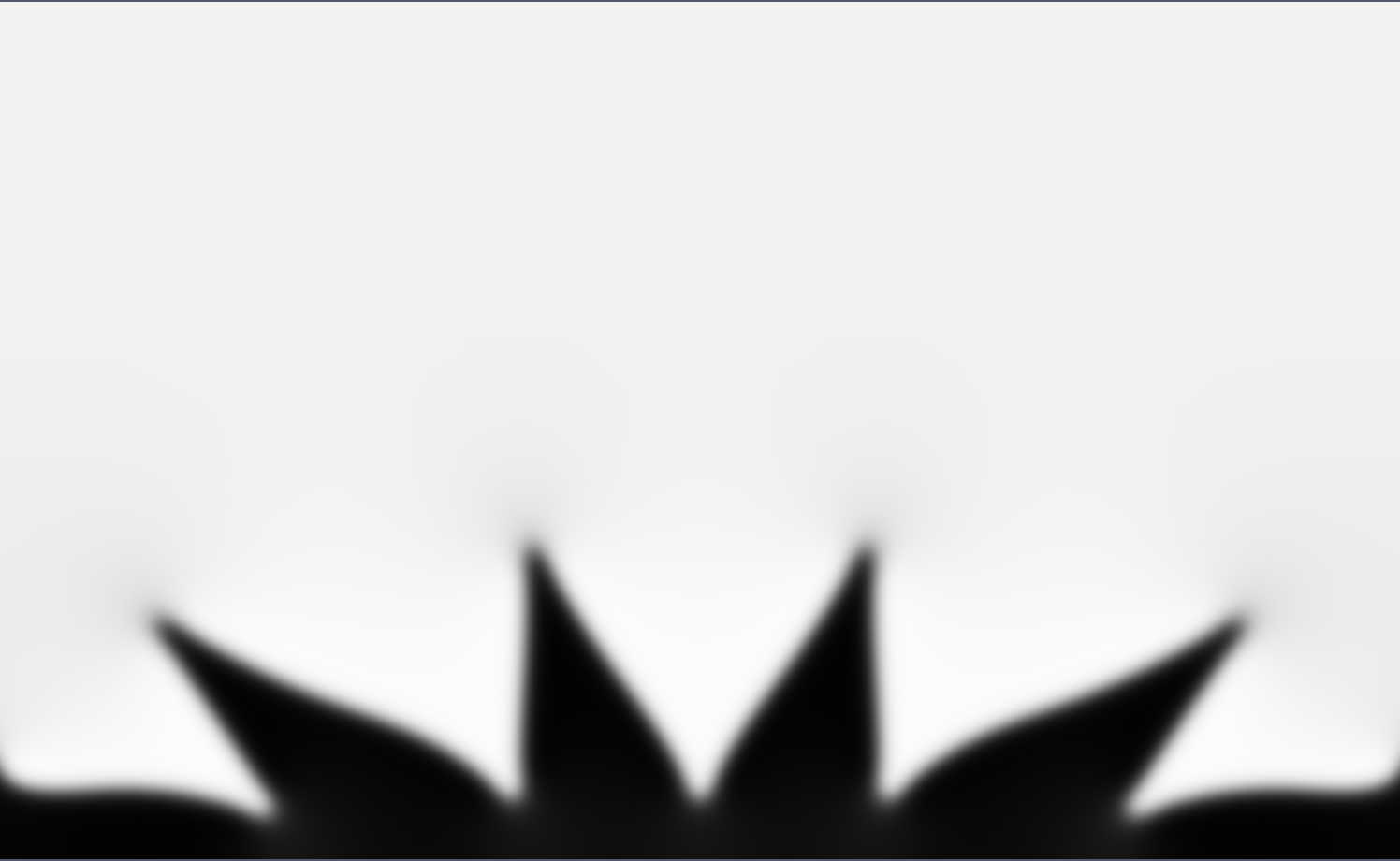}}&
    \fbox{\includegraphics[width=32mm]{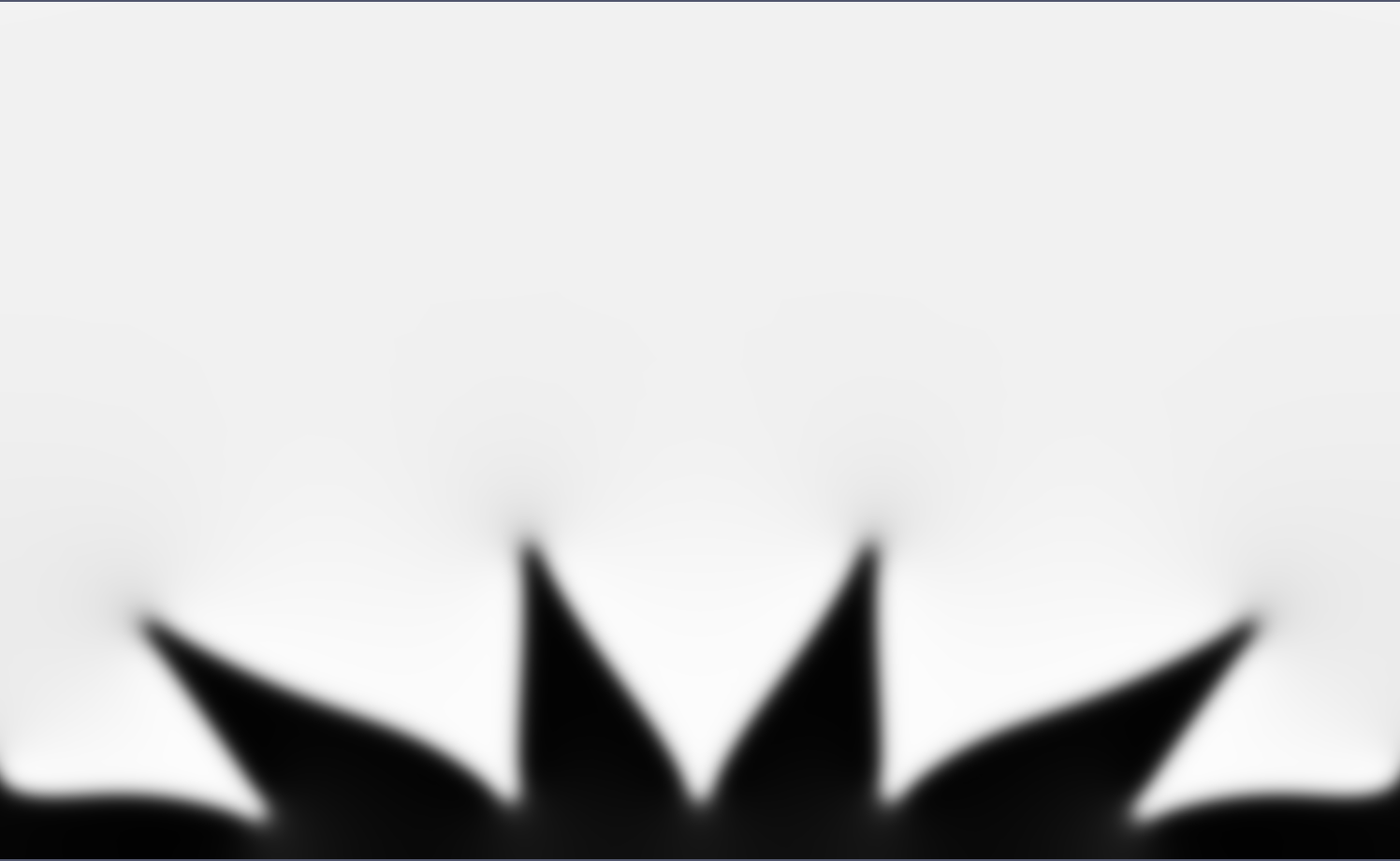}}&
    \fbox{\includegraphics[width=32mm]{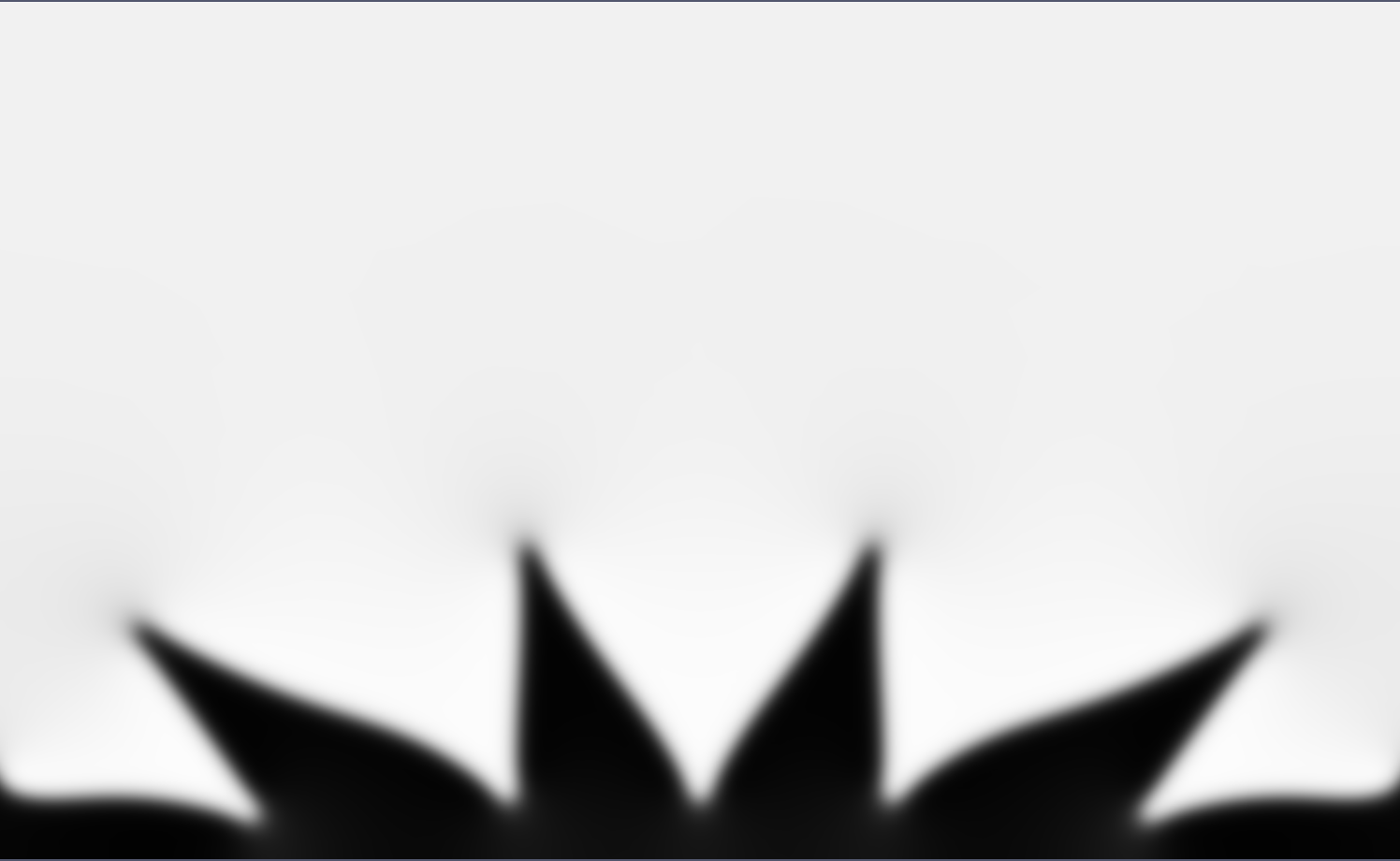}}&
    \fbox{\includegraphics[width=32mm]{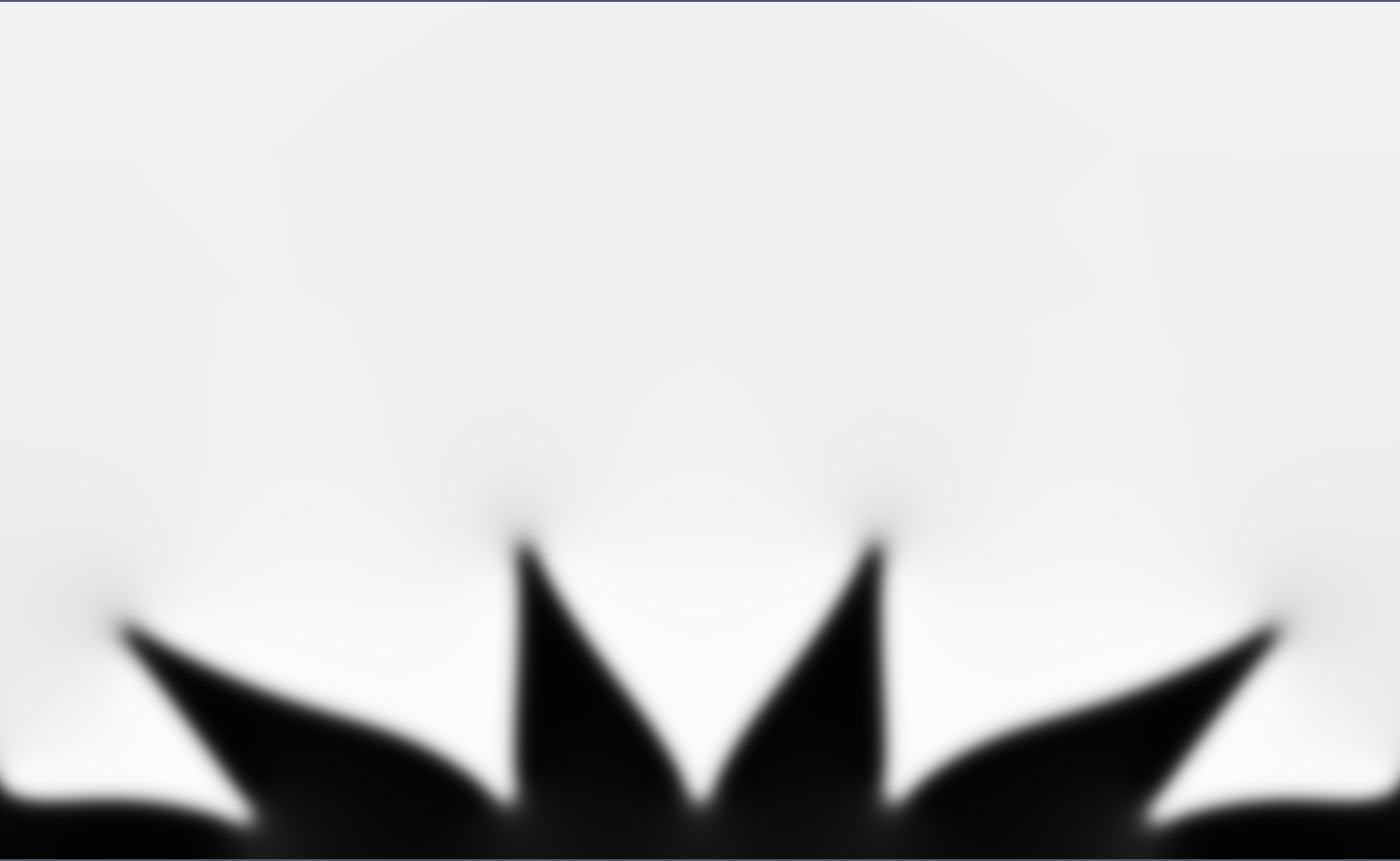}}\\

    \fbox{\includegraphics[width=32mm]{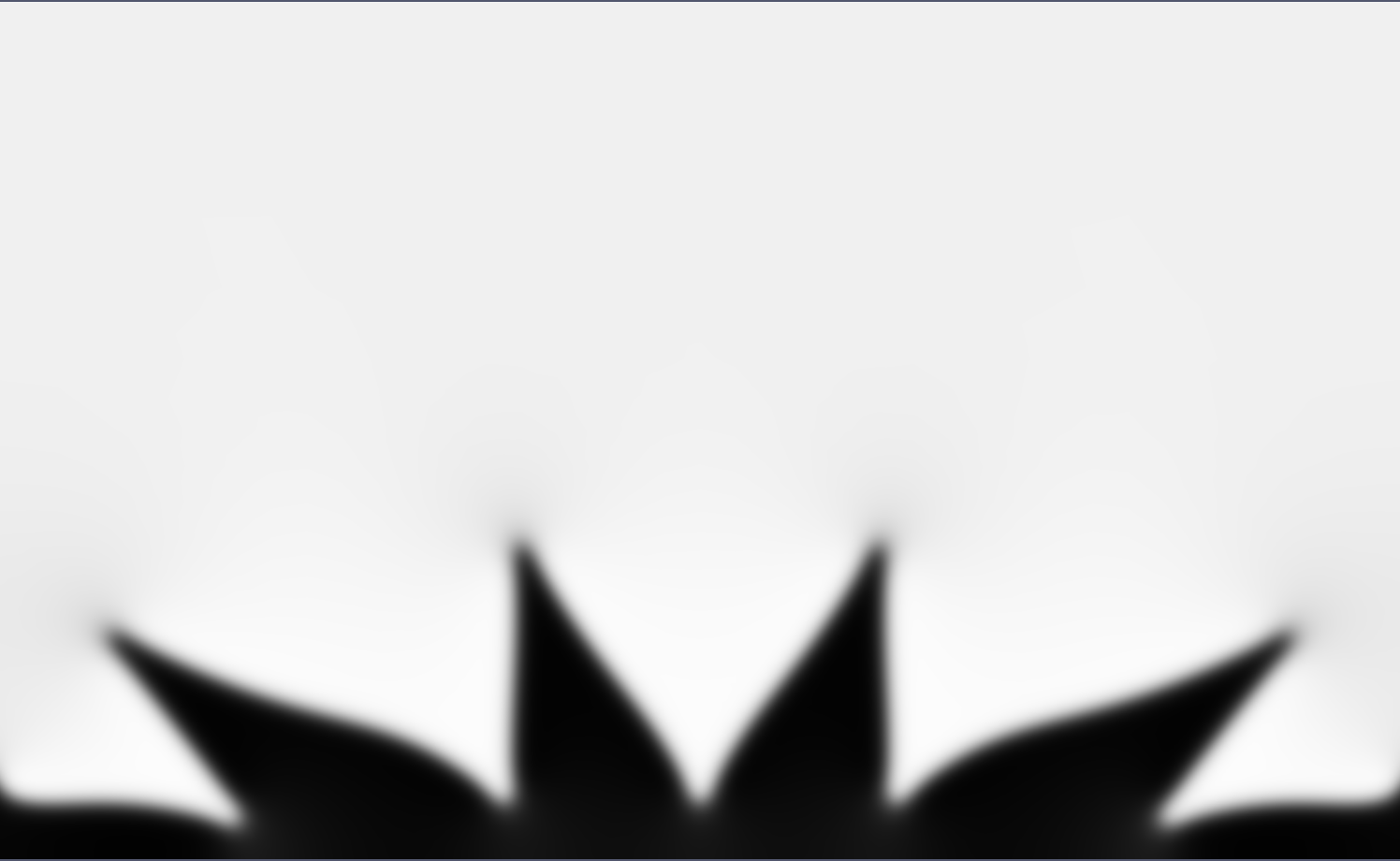}}&
    \fbox{\includegraphics[width=32mm]{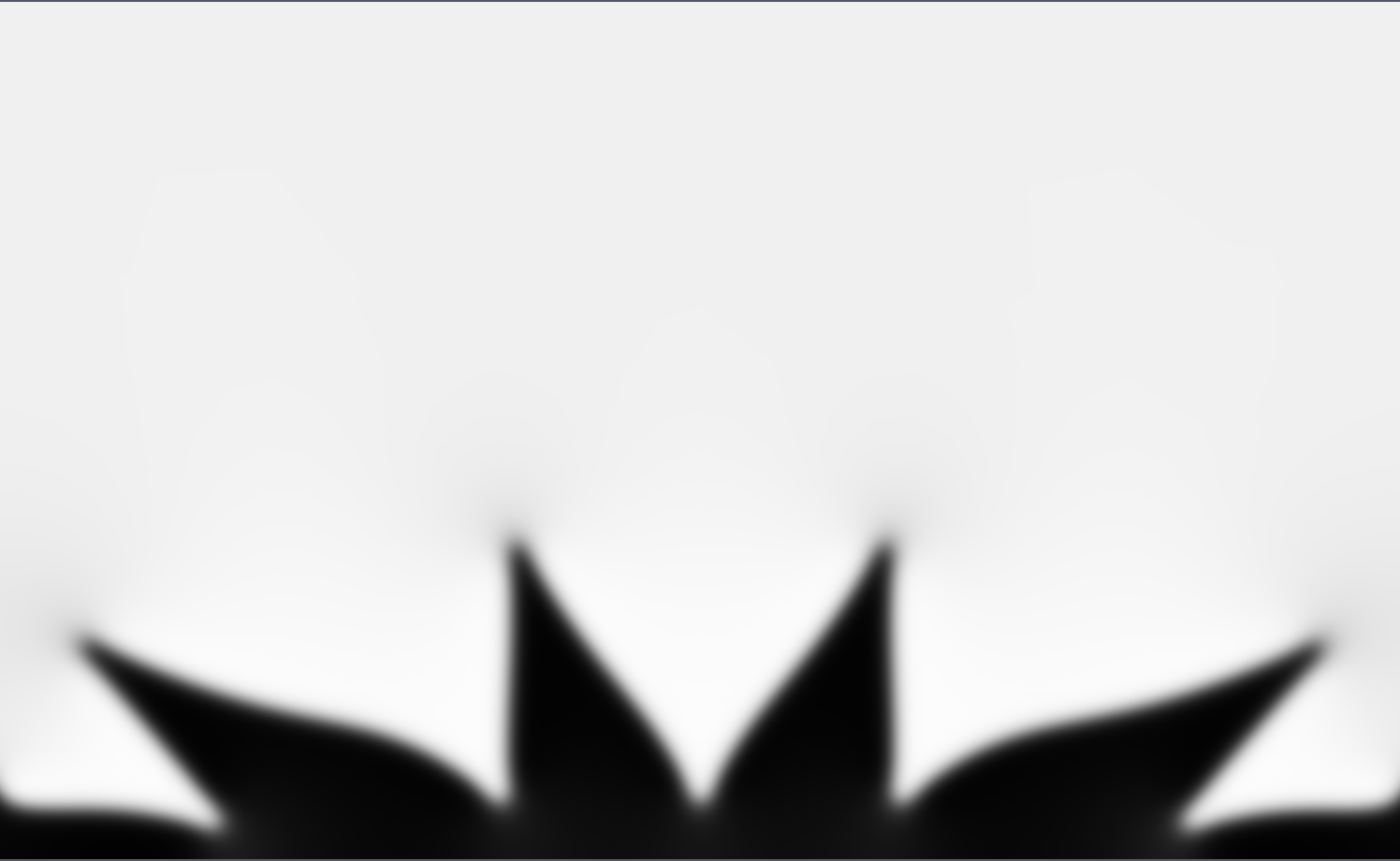}}&
    \fbox{\includegraphics[width=32mm]{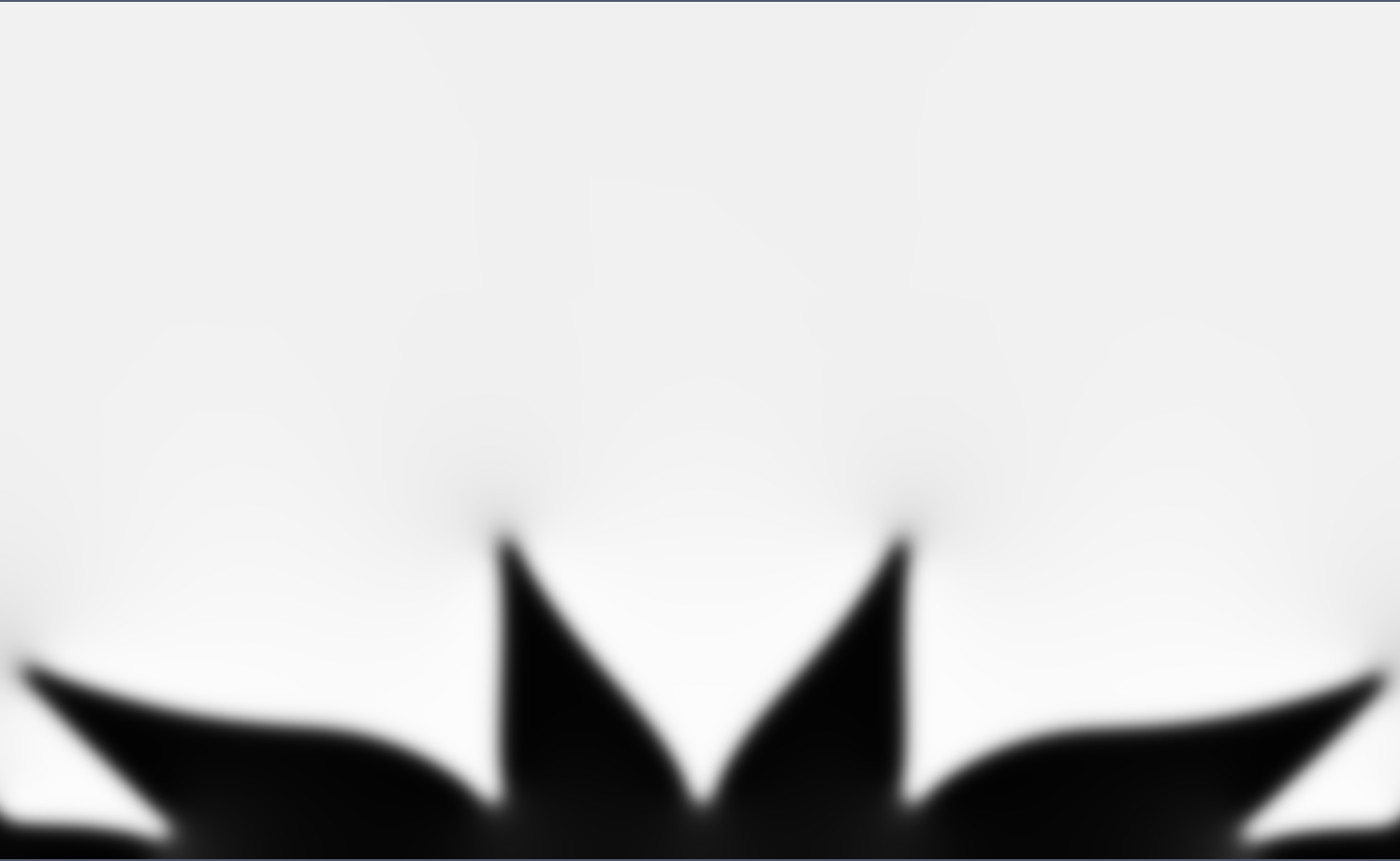}}&
    \fbox{\includegraphics[width=32mm]{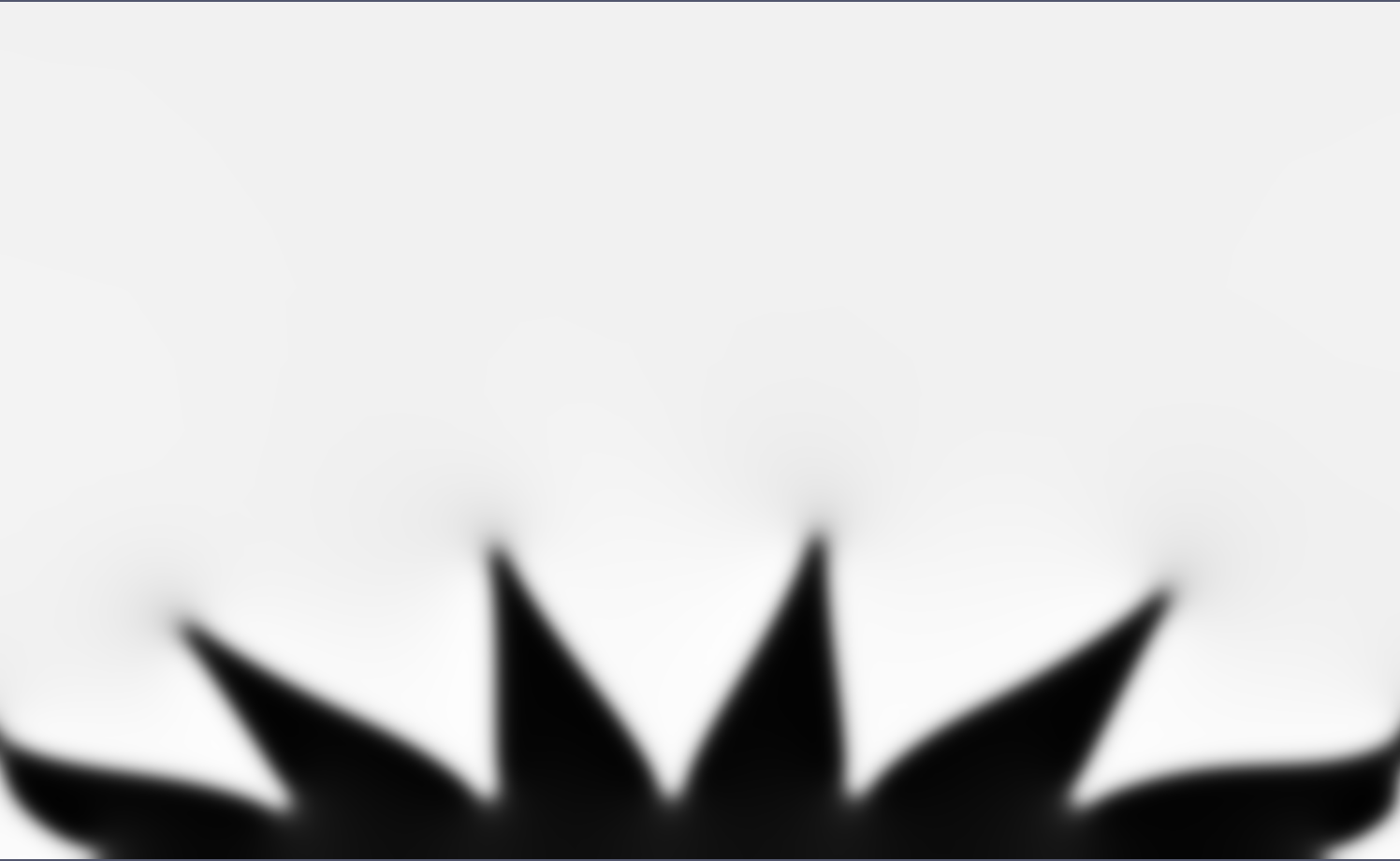}}


  \end{tabular}
\end{center}
\caption[The ferrofluid hedgehog: complete magnetizing field $\heff = \ha + \hd$]{\textbf{The ferrofluid hedgehog: complete magnetizing field $\heff = \ha + \hd$}. This figure shows another instance of the Rosensweig instability, this time with a non-uniform magnetic field (see Figure \ref{FigDipsHedge} for details regarding the source of magnetic field). This computation was carried out using the definition \eqref{totalhTwPh}--\eqref{phiNeuIIItwo} for the effective magnetizing field $\heff$. Here the frames correspond to $t=0$ for the first frame, to time $t=6$ for the last frame, at regular intervals of time. In contrast to Figure \ref{figureResolved} of experiment \ref{exp1}, the ferrofluid spikes exhibit an open pattern, just like the magnetic field $\ha$ driving the system. In addition, the interface starts flat, and develops four spikes in the middle region (where there is a narrow band of quasi-uniform magnetic field pointing upwards), a configuration that remains quite stable throughout 
most of the simulation time, but finally it has a second (much faster) transition from four to six spikes which can be seen in the last frames. This numerical experiment clearly exhibits noticeable resemblance with the physical experiment shown in Figure \ref{experiment1}.\label{hedgefig}}
\end{figure}

The intensity $\alpha_s$ is the same for each dipole, but it evolves in time. More precisely, $\alpha_s$ is increased linearly in time from $\alpha_s = 0$ at time $t=0$, to its maximum value $\alpha_s = 4.3$ at time $t = 4.2$, and from time $t = 4.2$ to $t = 6.0$ the intensity of the dipoles is kept constant in order to let the system rest and develop a stable configuration. The motivation behind a longer simulation time ($\tf=6$ units) and the maximum intensity $\alpha_s = 4.3$ is to push the system to a barely stable configuration at the brink of a second transition, so that allowing more simulation time, the system could evolve further and capture a second bifurcation.

Numerical results using the complete (effective) magnetizing field $\heff = \ha + \hd$ defined in \eqref{totalhTwPh}--\eqref{phiNeuIIItwo} are displayed in Figure \ref{hedgefig}. Numerical results using the reduced definition $\heff := \ha$ can be found in Figure \ref{hedgefigha}. These figures are strikingly different, thereby highlighting the importance of using a physically reasonable definition for effective magnetizing field $\heff$, and the critical influence of the demagnetizing field $\hd$ in the overall behavior of the system. Even though \eqref{totalhTwPh}--\eqref{phiNeuIIItwo}, used in Figures \ref{figureResolved}--\ref{hedgefig}, is a questionable approach to compute an approximation of $\heff = \ha + \hd$, it retains the presence of $\hd$, and is able to reproduce the classical Rosensweig instability. It is even able to respect the scaling \eqref{gravity} reasonably well in the context of uniform magnetic fields (see Figure \ref{figureResolved}), and the more common version of the Rosensweig instability in the context of non-uniform magnetic fields, as shown for instance in Figure \ref{hedgefig}.

Many attempts to model and explain the Rosensweig instability (and ferrofluid behavior in general) found in the literature (\cf \cite{Abou2000,Engel2001,AFK2008,AFK2010,Trung2011,Oden2002}), pay special attention to the modeling of non-linear susceptibilities and saturation effects (usually carried out with the Langevin function). However, they rarely ever elaborate on the effective magnetizing field $\heff$, the demagnetizing field $\hd$ (also called stray field), and their approximation/computation. They are sometimes not even mentioned in the entire text.

Saturation is indeed an important component in the physical behavior of magnetic materials, specially if the material is past the saturation limit. However, it is our opinion that current emphasis on the modeling of saturation effects is somehow not commensurate with its actual impact in the physical behavior of ferrofluids. Preliminary computational experiments, carried out by us (not reported), seem to indicate that the modeling of saturation effects add almost imperceptible nuances to the overall behavior of the system, while proper computation of the effective field (using \eqref{totalhTwPh}--\eqref{phiNeuIIItwo} or a better approximation if possible) has much more striking consequences in the global behavior of the system (particularly relevant for the study of the Rosensweig instability). Those consequences can be as noticeable as the difference between Figures \ref{hedgefig} and \ref{hedgefigha}.

\begin{figure}
\begin{center}
    \setlength\fboxsep{0pt}
    \setlength\fboxrule{1pt}
  \begin{tabular}{cccc}

    \fbox{\includegraphics[width=32mm]{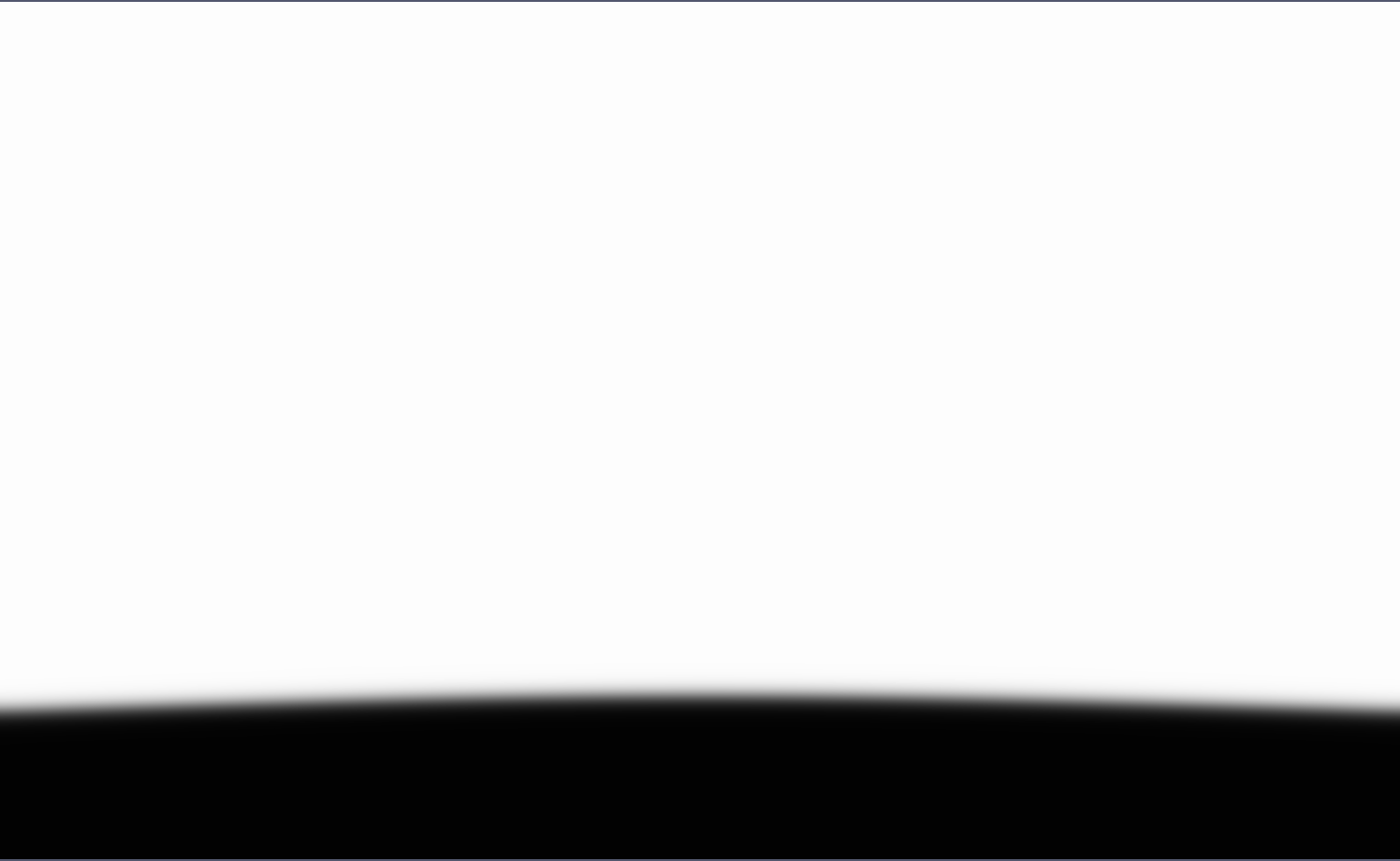}}&
    \fbox{\includegraphics[width=32mm]{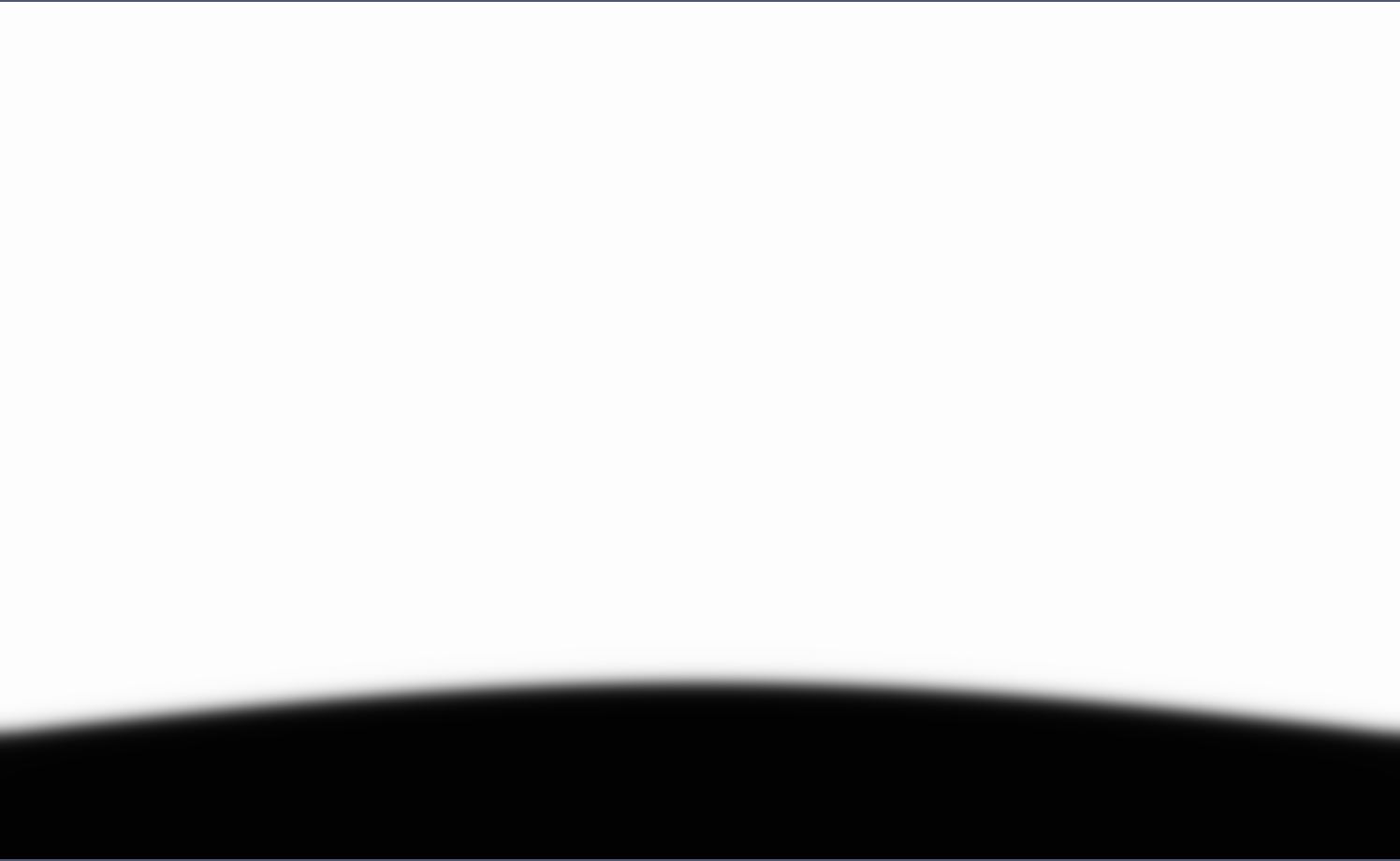}}&
    \fbox{\includegraphics[width=32mm]{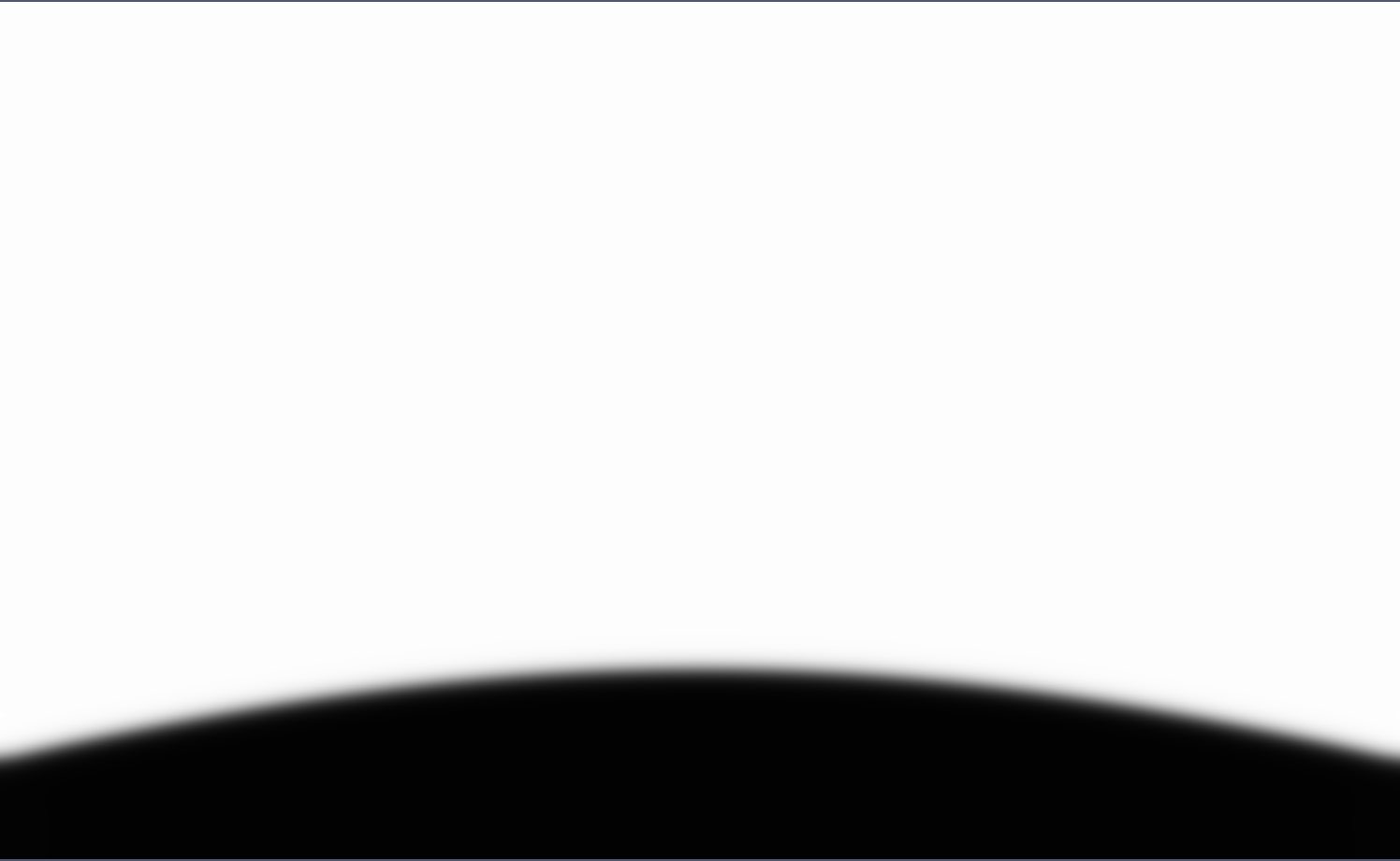}}&
    \fbox{\includegraphics[width=32mm]{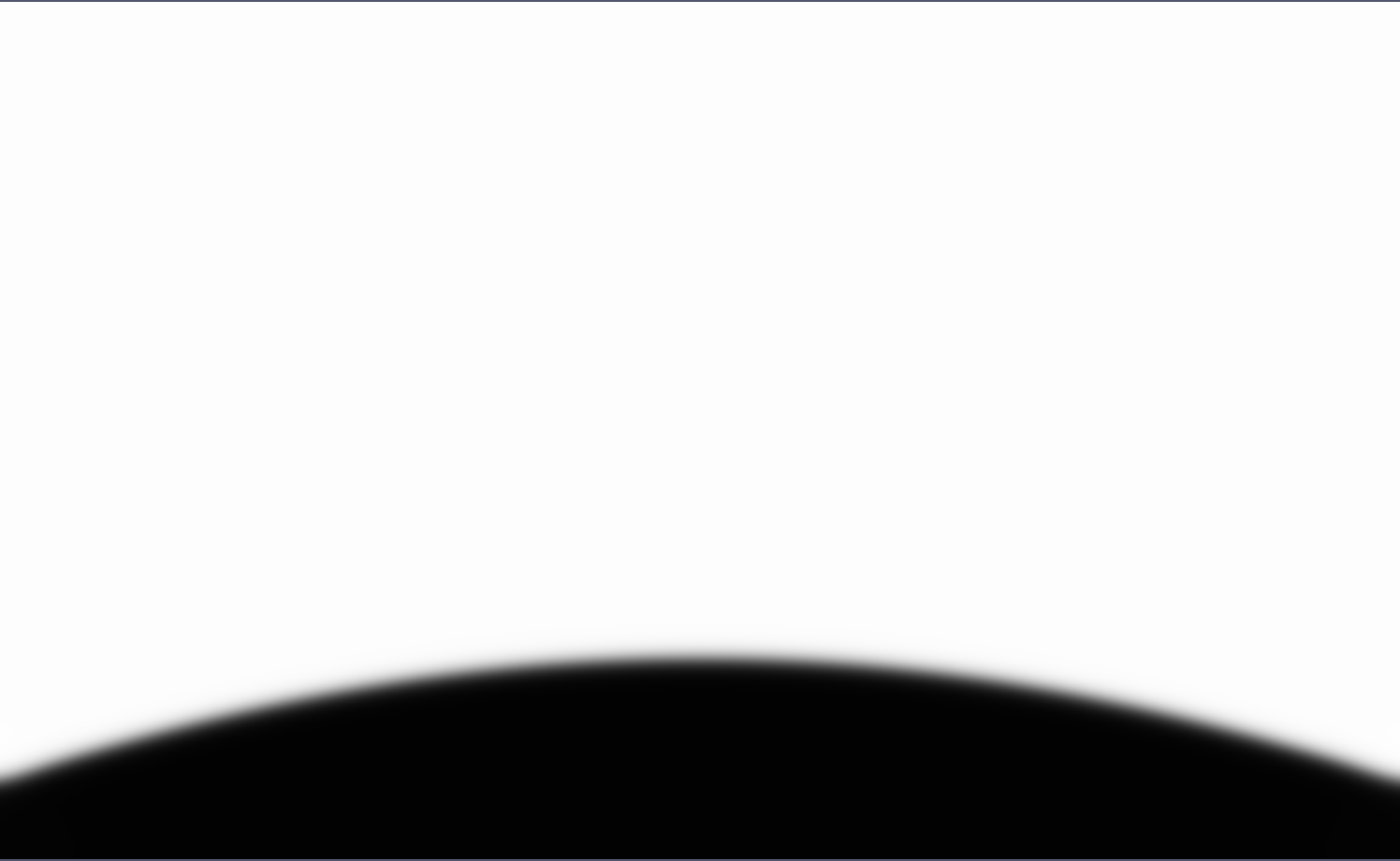}}\\

    \fbox{\includegraphics[width=32mm]{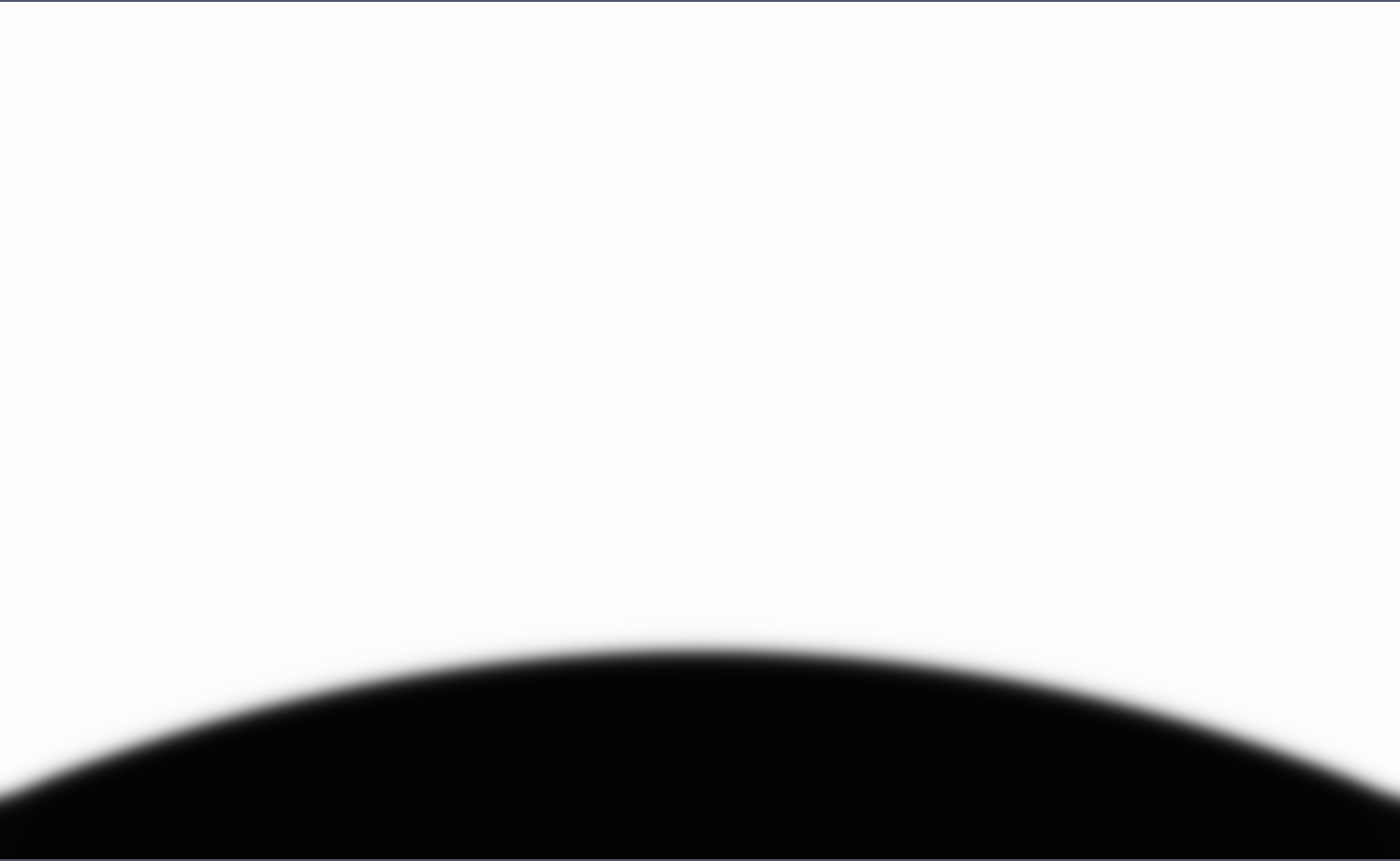}}&
    \fbox{\includegraphics[width=32mm]{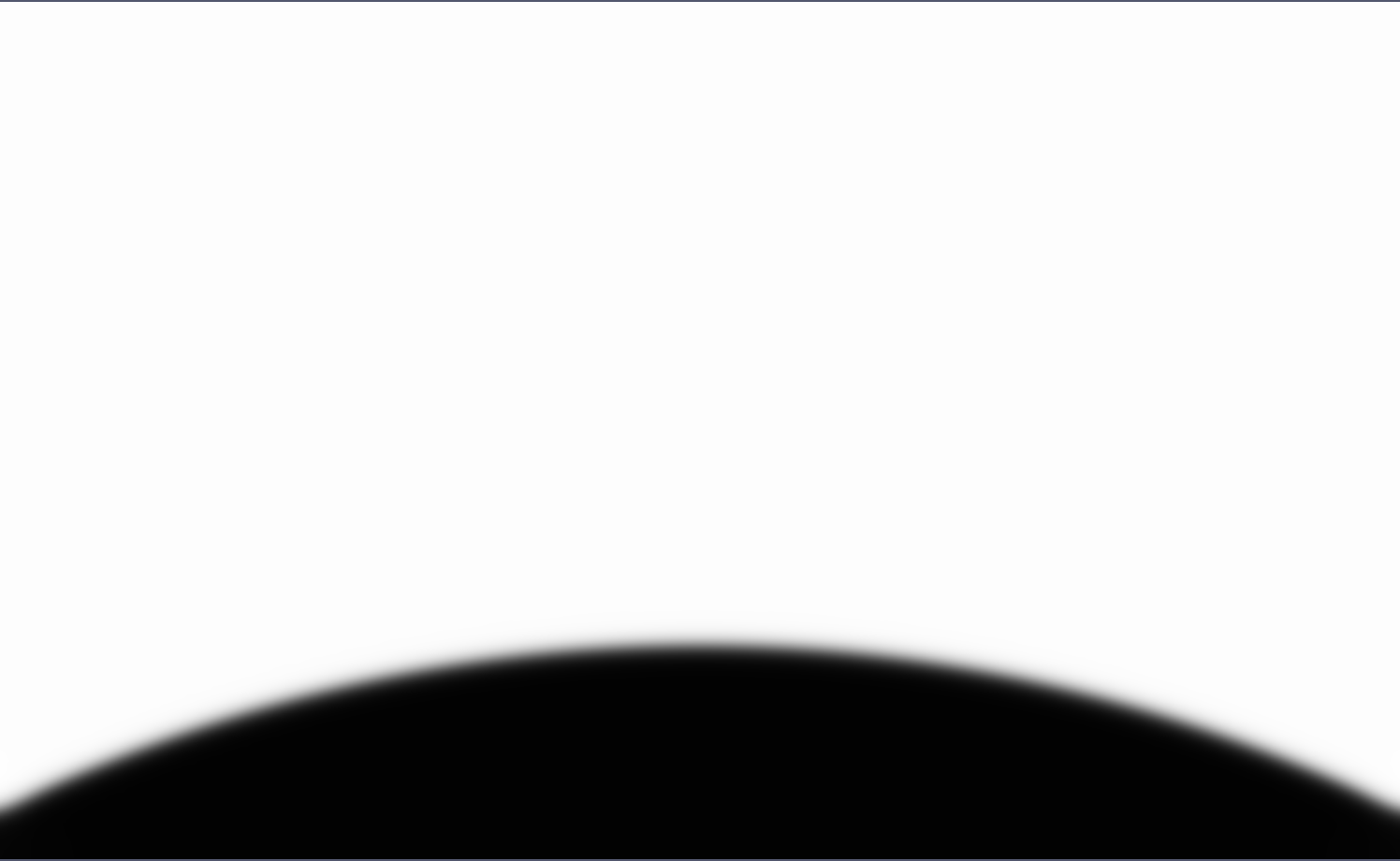}}&
    \fbox{\includegraphics[width=32mm]{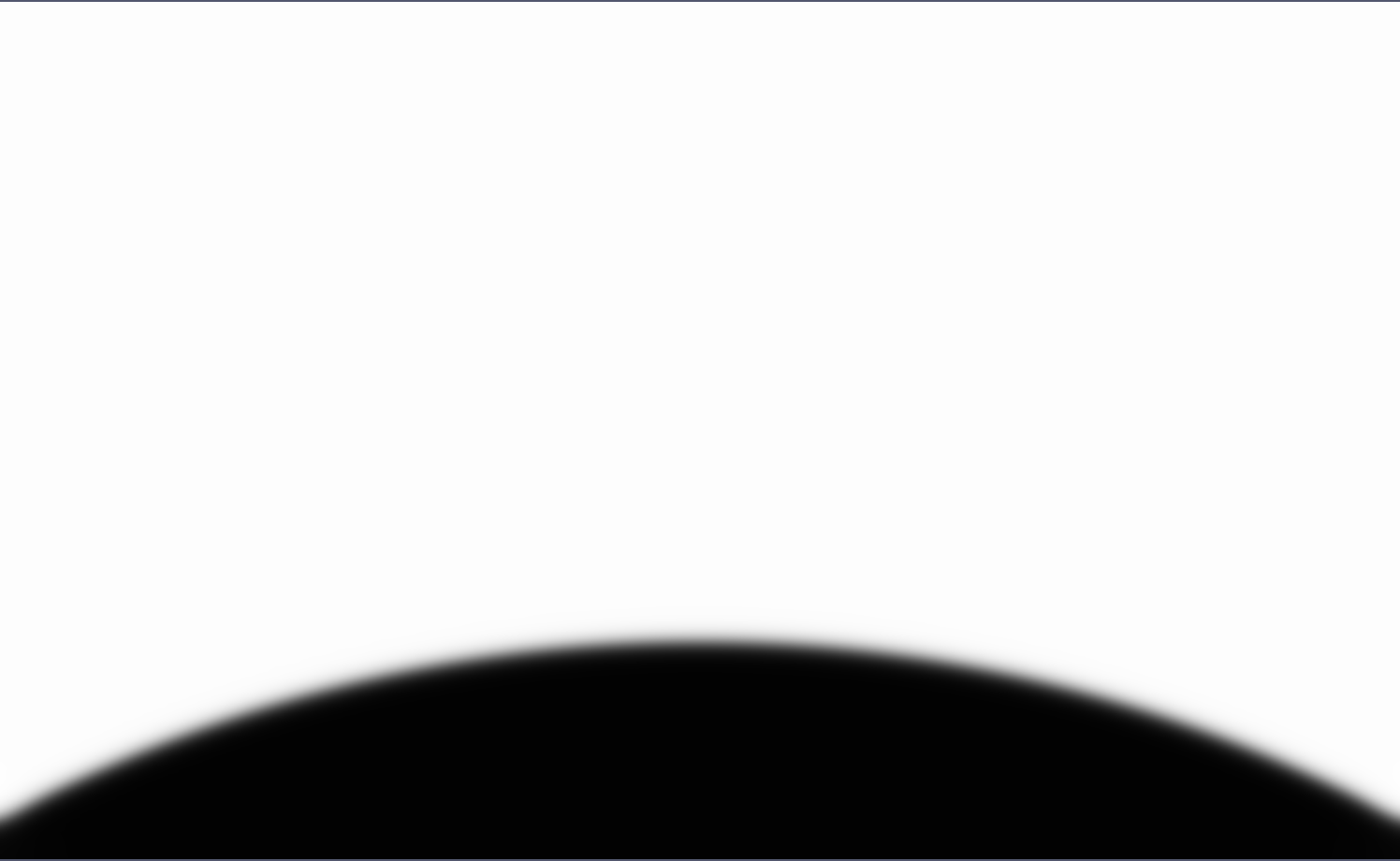}}&
    \fbox{\includegraphics[width=32mm]{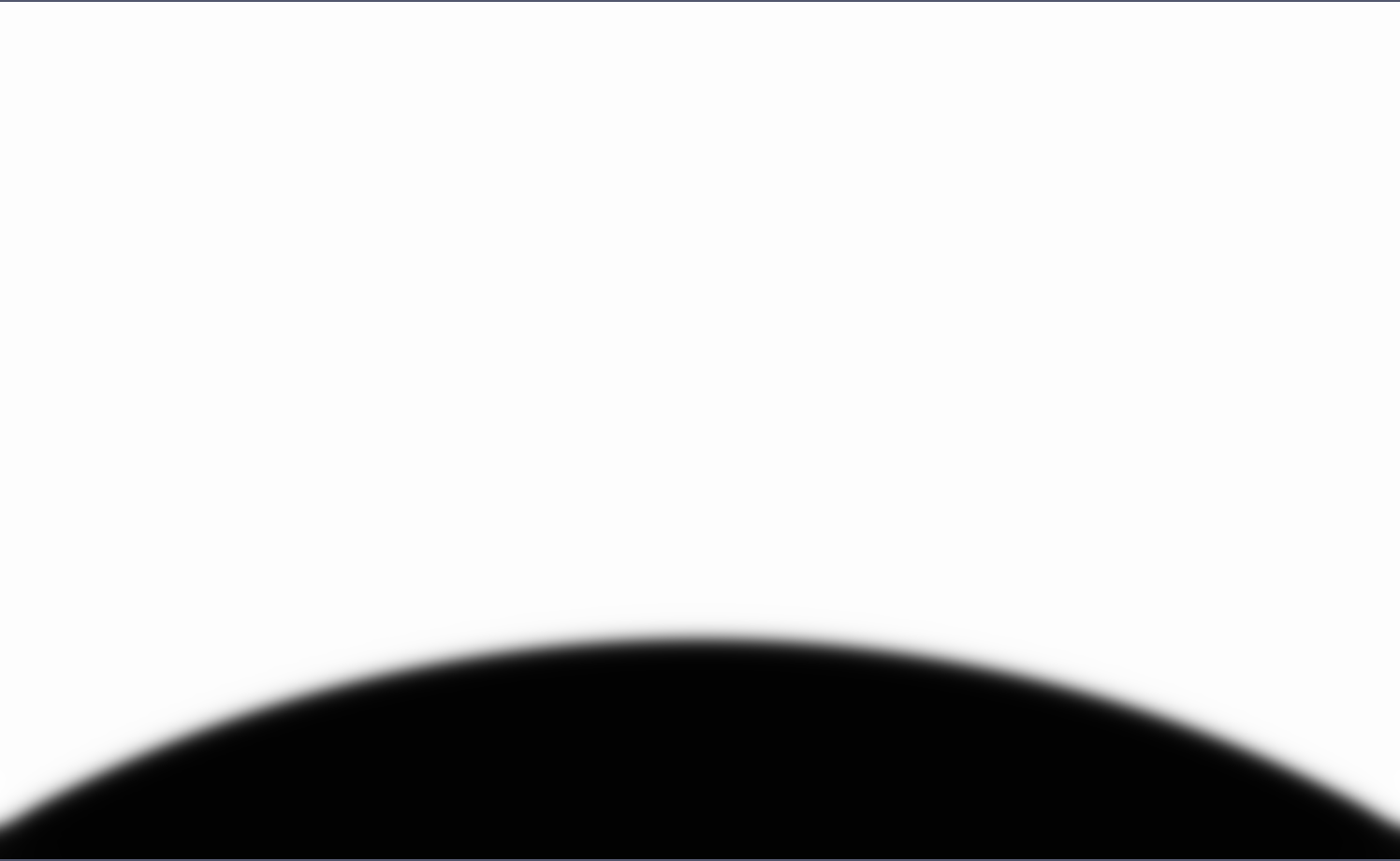}}\\

    \fbox{\includegraphics[width=32mm]{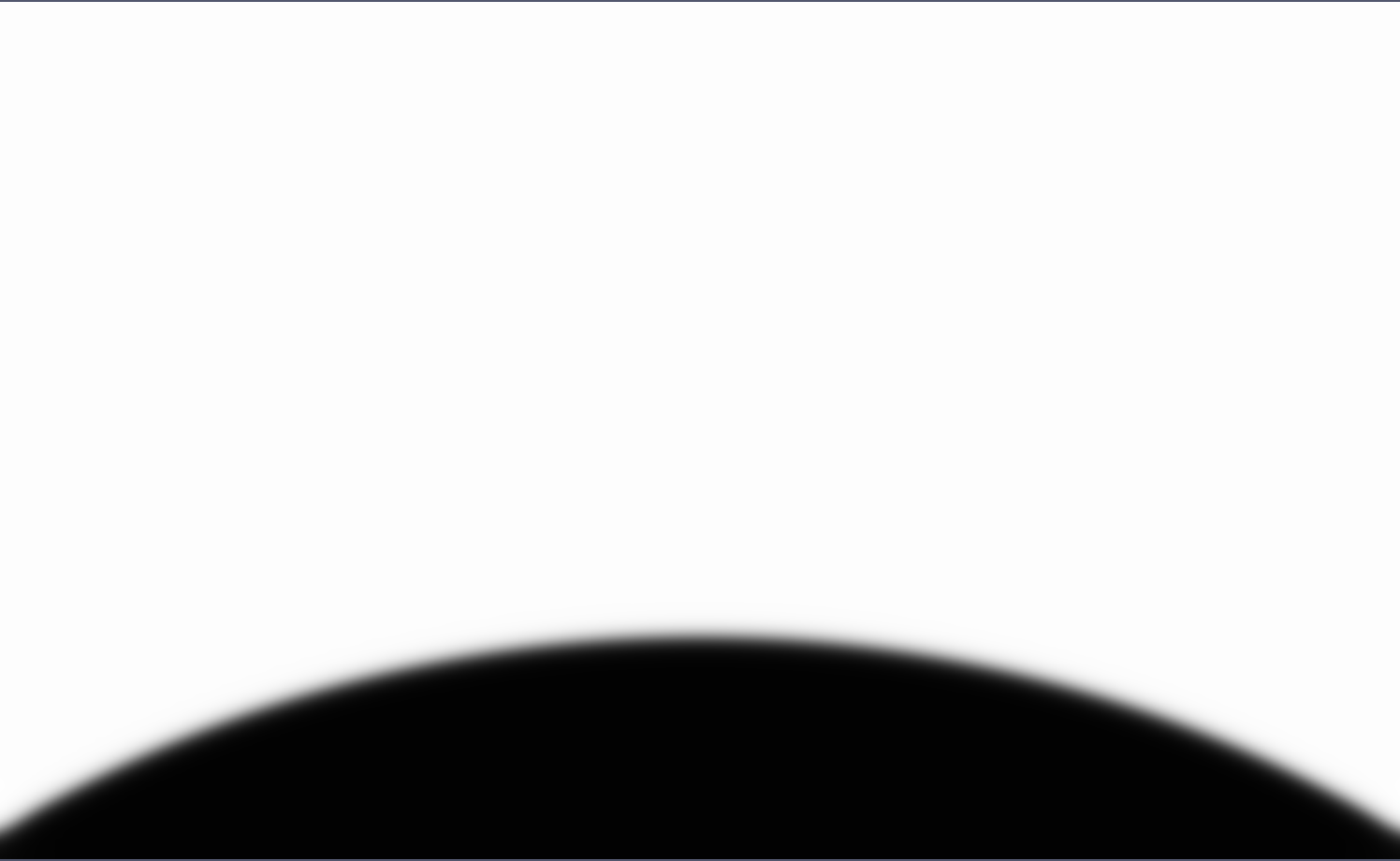}}&
    \fbox{\includegraphics[width=32mm]{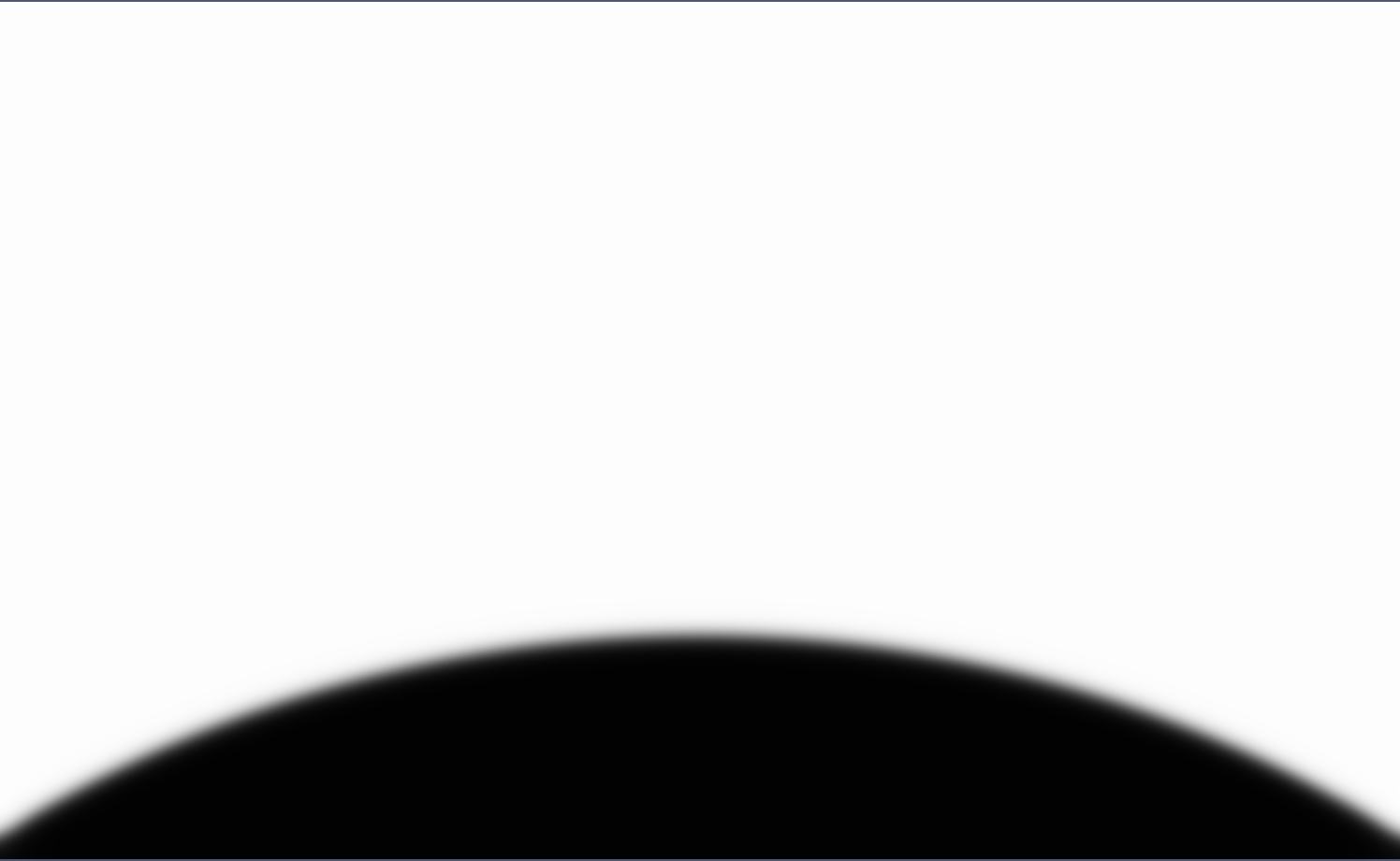}}&
    \fbox{\includegraphics[width=32mm]{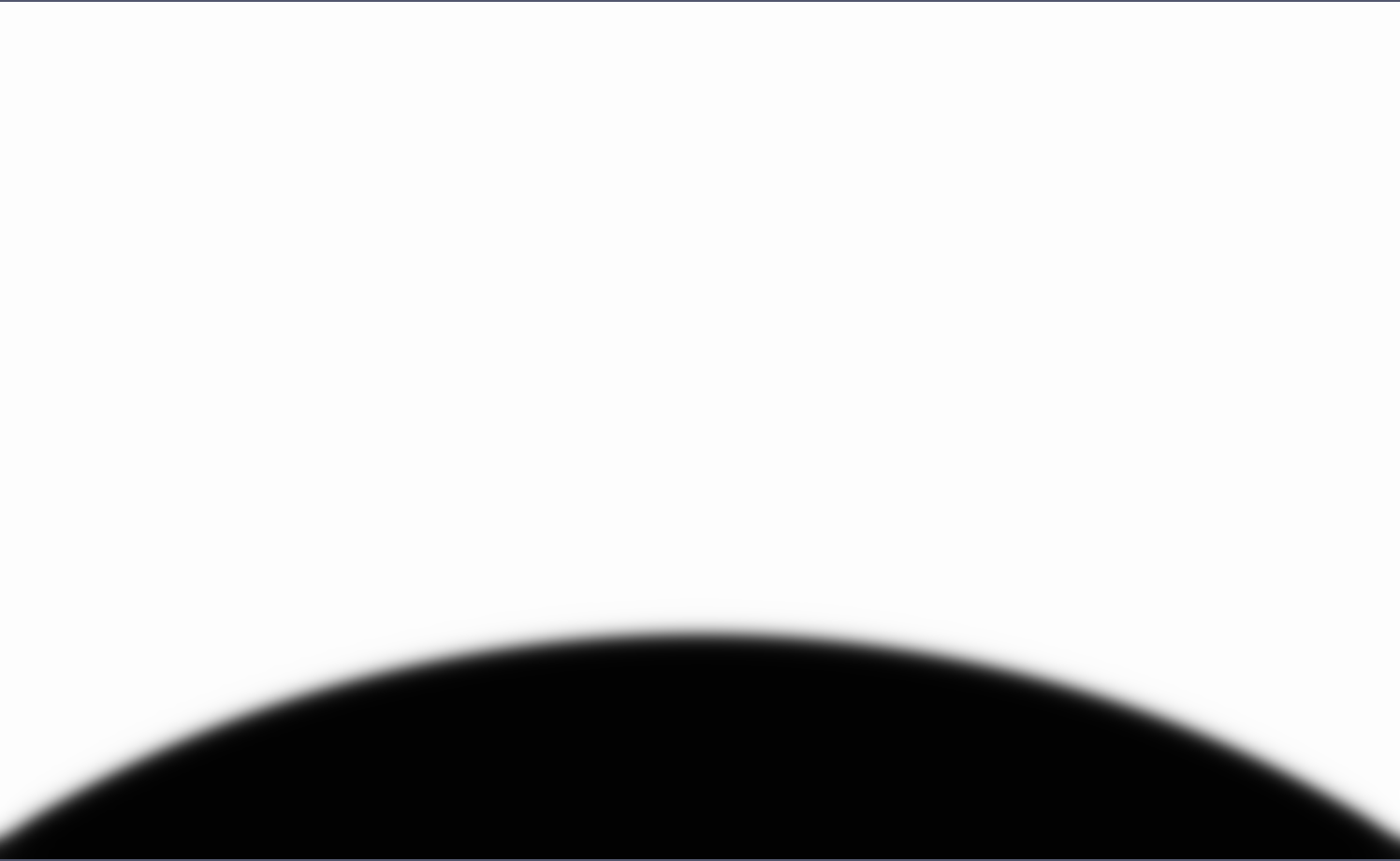}}&
    \fbox{\includegraphics[width=32mm]{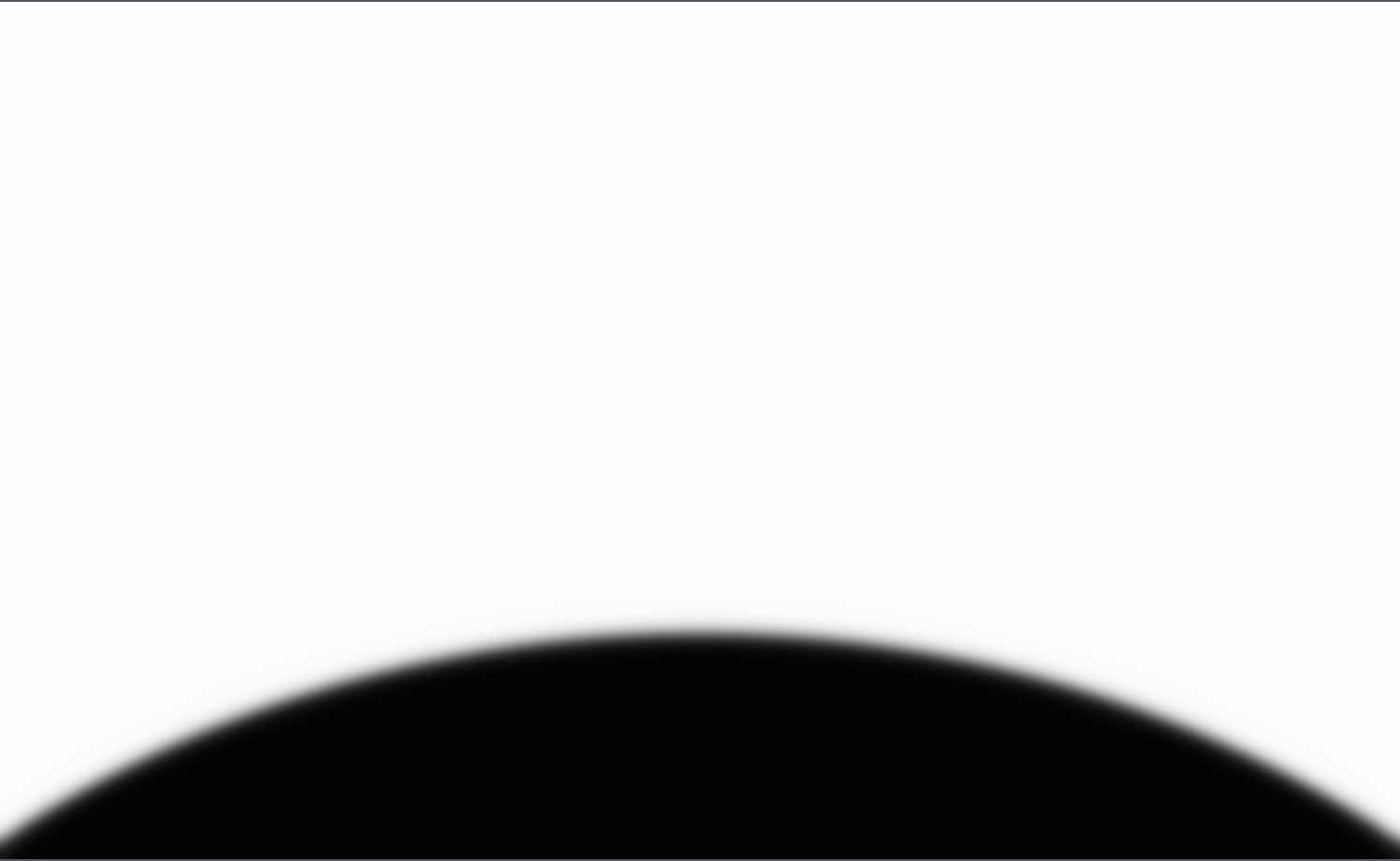}}\\

    \fbox{\includegraphics[width=32mm]{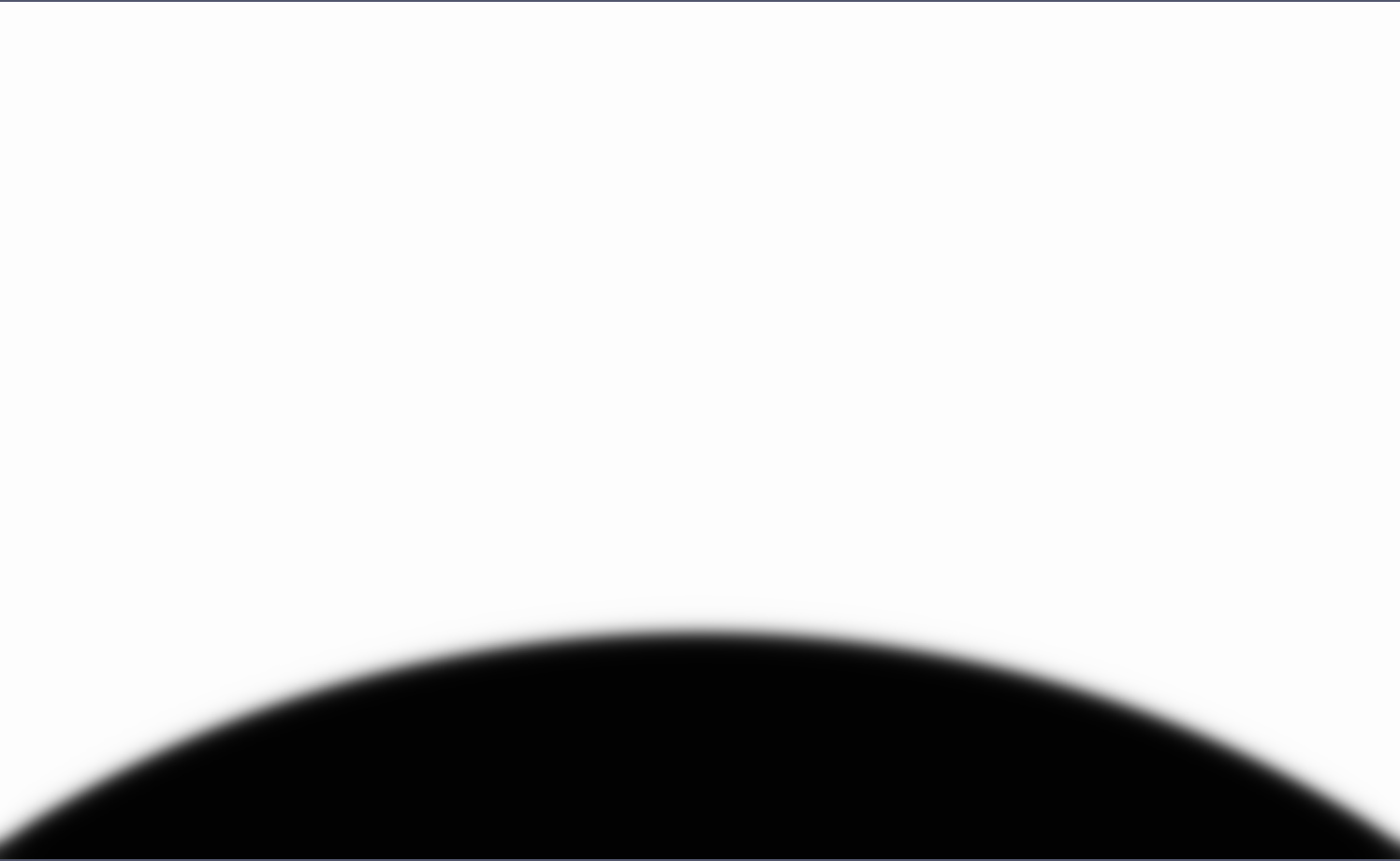}}&
    \fbox{\includegraphics[width=32mm]{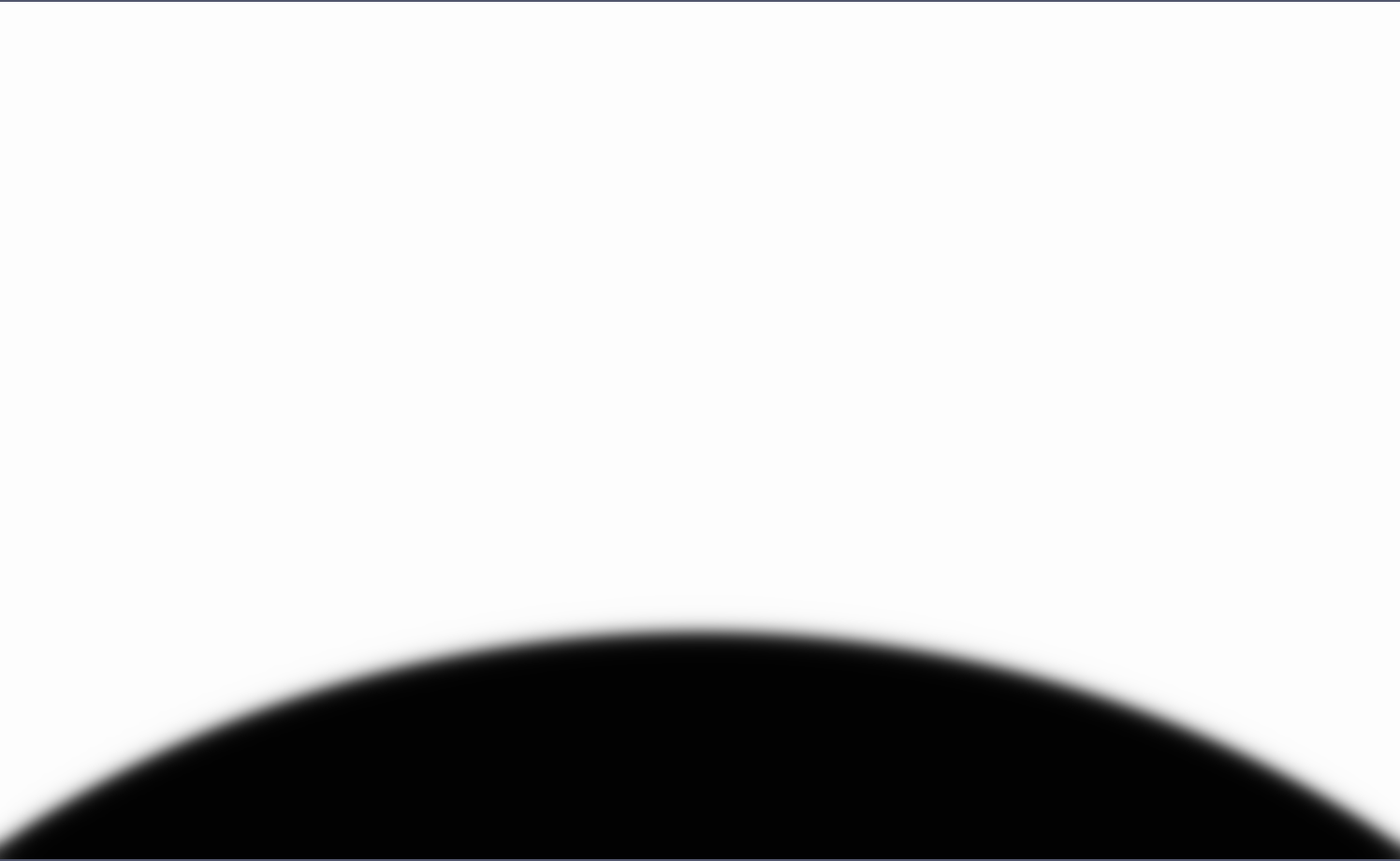}}&
    \fbox{\includegraphics[width=32mm]{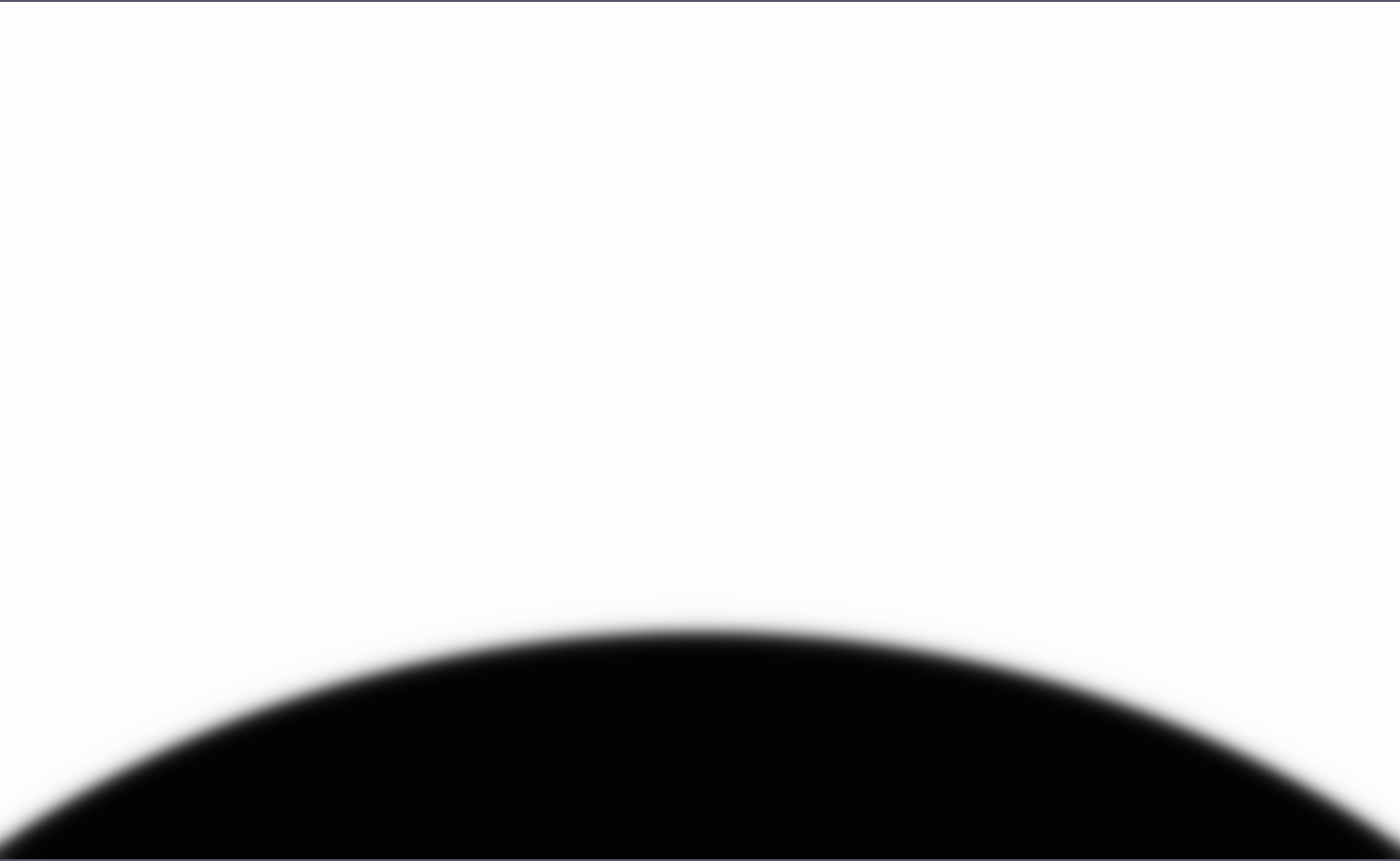}}&
    \fbox{\includegraphics[width=32mm]{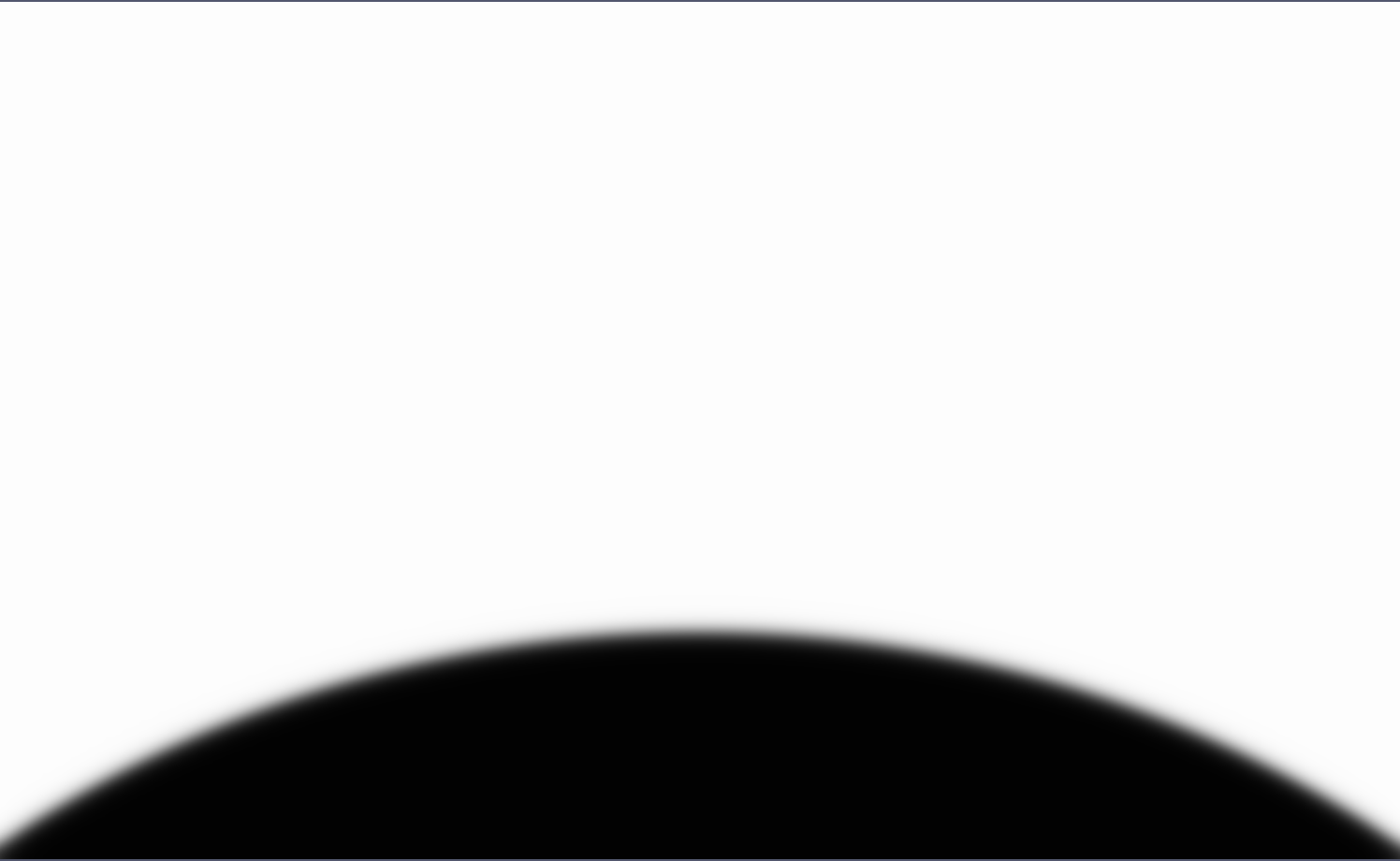}}\\

    \fbox{\includegraphics[width=32mm]{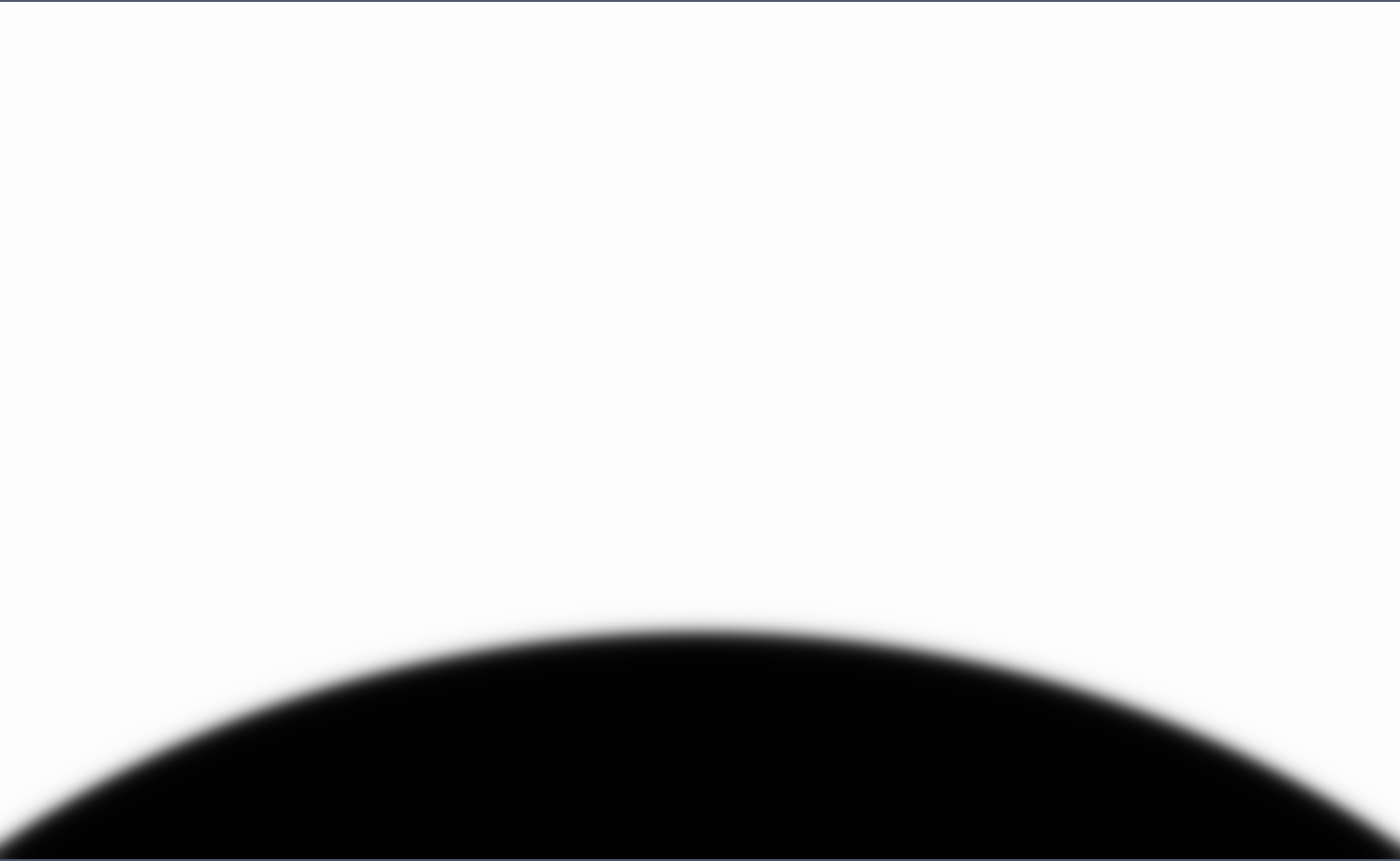}}&
    \fbox{\includegraphics[width=32mm]{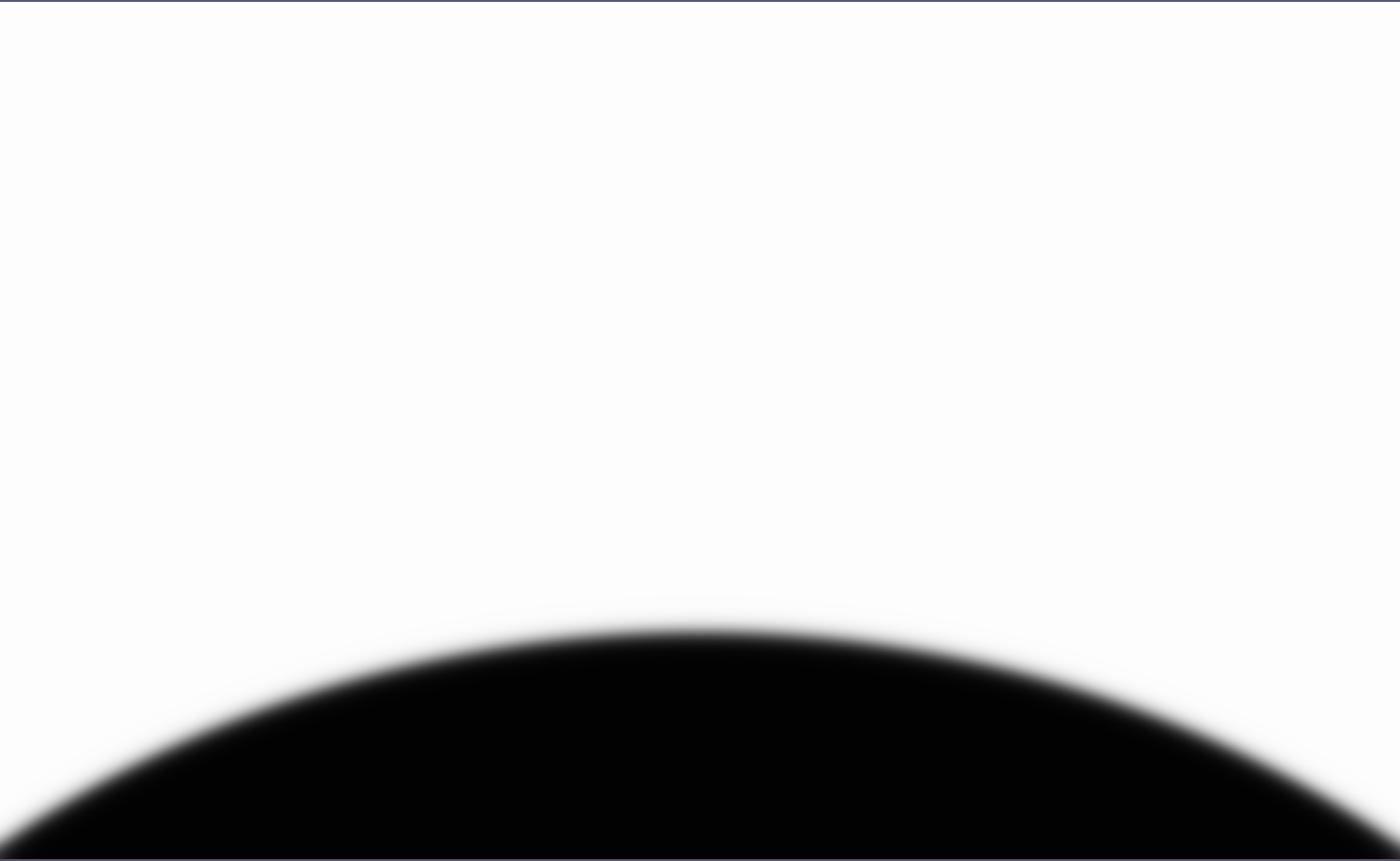}}&
    \fbox{\includegraphics[width=32mm]{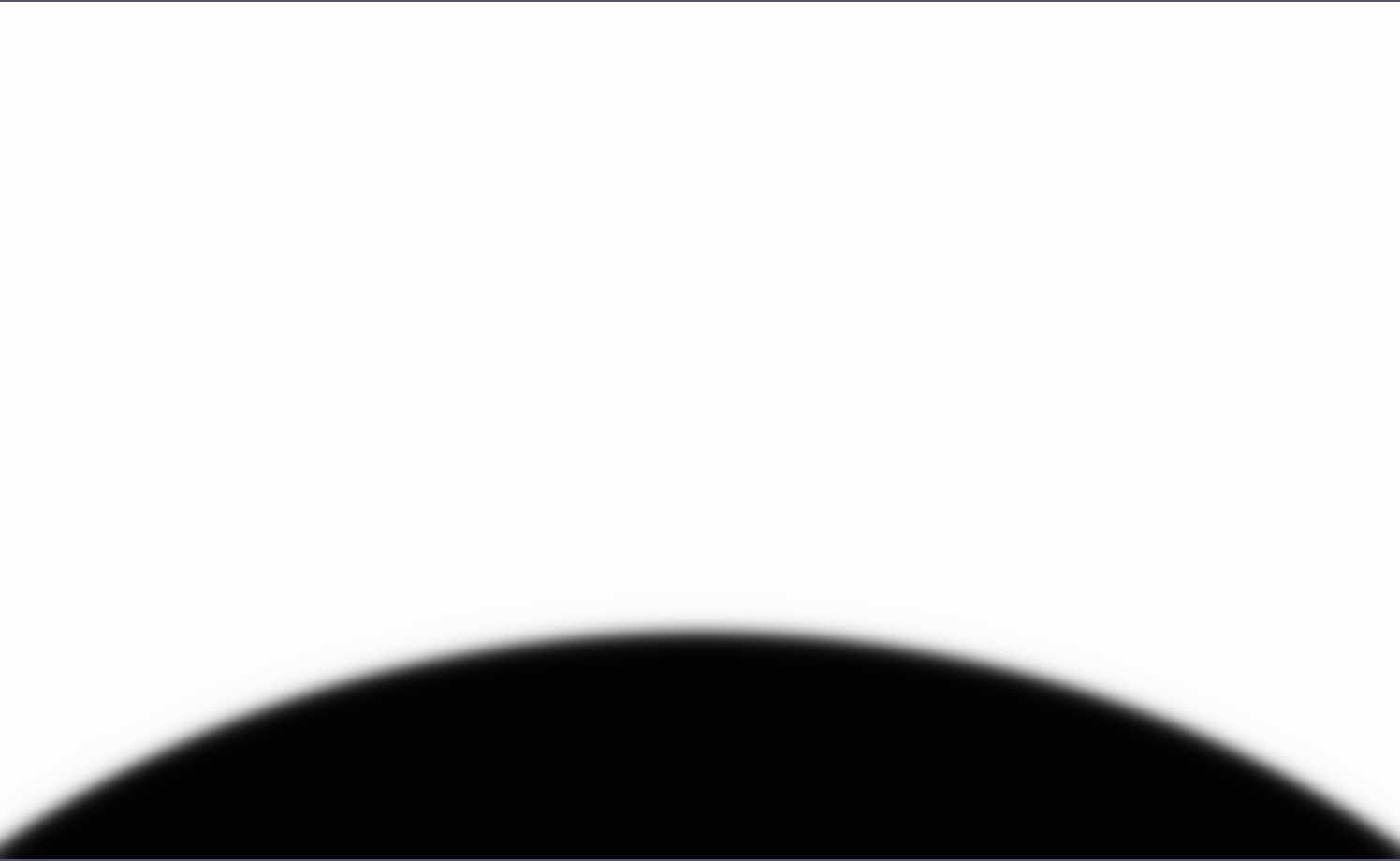}}&
    \fbox{\includegraphics[width=32mm]{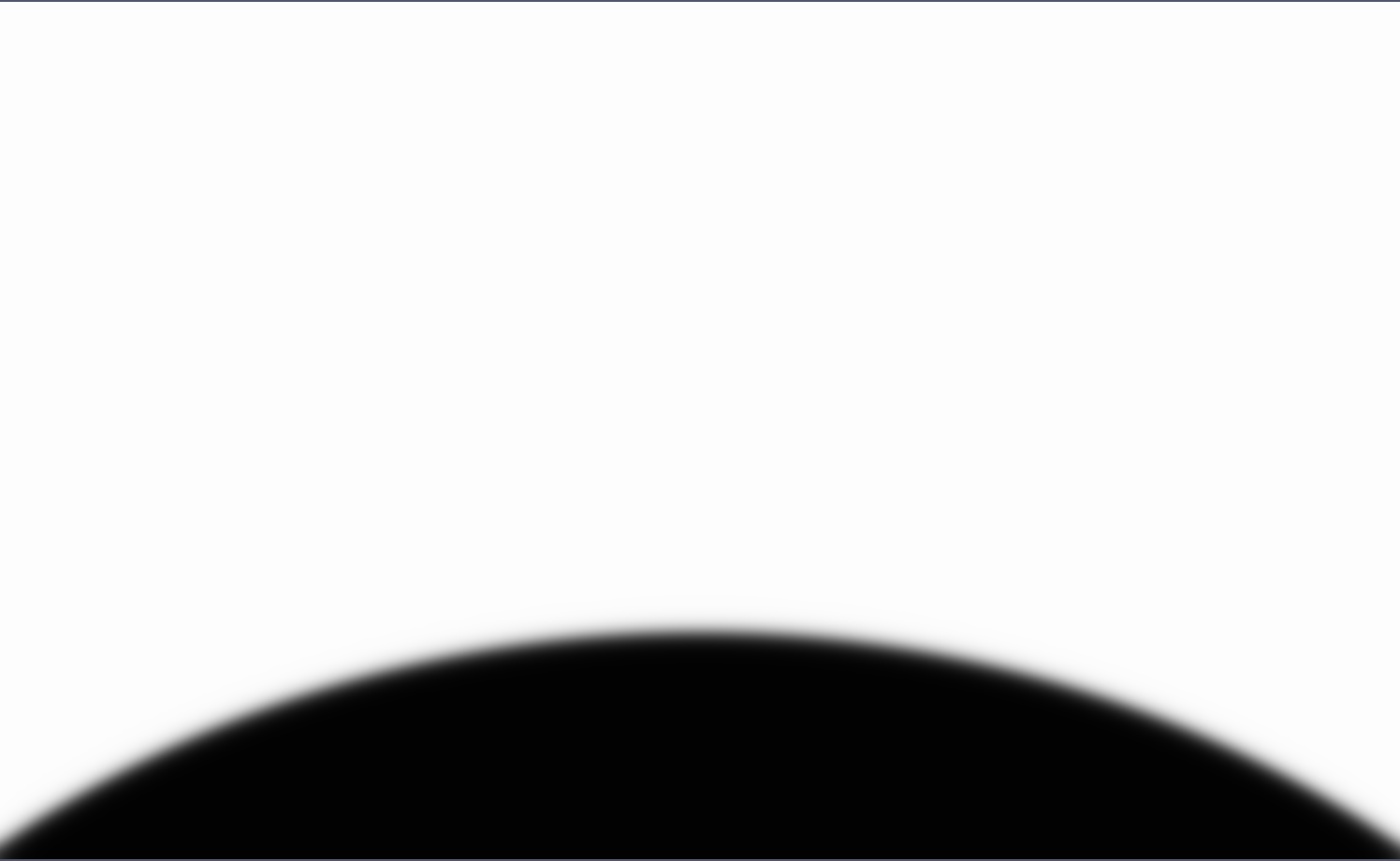}}

  \end{tabular}
\end{center}
\caption[The Hedgehog (attempt): reduced magnetizing field $\heff =
  \ha$]{\textbf{The ferrofluid hedgehog (attempt): reduced magnetizing field
    $\heff = \ha$}. This figure shows a computation carried out using
  the same setup as in Figure \ref{hedgefig}, but this time, we
  neglect the demagnetizing field $\hd$ and employ the reduced field
  $\heff = \ha$ for the effective magnetizing field (at discrete level
  we use \eqref{newHdef}). The evolution of the interface is totally
  different to that one of Figure \ref{hedgefig}, the Rosensweig instability does not manifest. The final configuration adopted by the system does not even resemble what would happen in a real life experiment (see for instance Figure \ref{experiment1}). \label{hedgefigha}}
\end{figure}

\begin{figure}
\begin{center}
  \setlength\fboxsep{0pt}
  \setlength\fboxrule{1pt}
  \fbox{\includegraphics[scale=0.30]{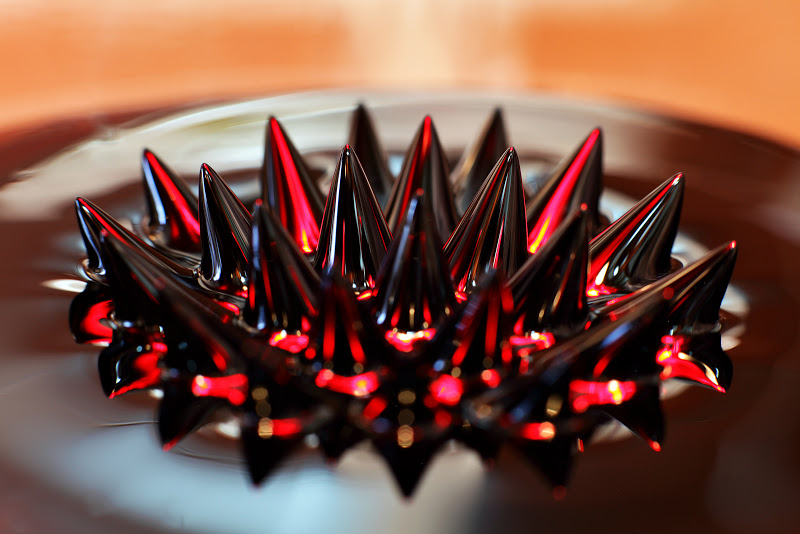}}
\end{center}
\caption[Experiment of a ferrofluid subject to a non-uniform magnetic field]{\textbf{Real experiment of a ferrofluid subject to a non-uniform magnetic field}. Courtesy (reproduced with permission) of \cite{Mied2011}. \label{experiment1}}
\end{figure}

\section{Conclusions}
In this work we propose a simple PDE model describing the behavior of two-phase ferrofluid flows. The model is assembled by choosing components from the one-phase Shliomis model of ferrofluids, magnetostatics, and well-known assumptions and simplifications from phase field techniques. The model satisfies a formal energy law and we are able to devise a numerical scheme that mimics it. The use of a discontinuous finite element space $\FEspaceM$ for the magnetization seems to be mandatory if we want to have a discrete energy law, hence, numerical stability. We are also able to prove that the scheme always has a solution.

We also present a simplified version of this model, which has a somehow more restrictive scope of physical validity, its use being primarily oriented to ferrofluids with small magnetic susceptibilities. For this simplified model we are able to develop a convergent numerical scheme. Convergence of the scheme relies on classical compactness arguments. On the other hand, the fact that the limits for $h , \dt \longrightarrow 0$ are ultra-weak solutions of \eqref{weakfor} demands a special construction; in particular, it requires the use of discontinuous pressures, which is an idea originally advanced by Walkington \cite{Walk2005}. 

We show a series of numerical experiments which illustrate the potential of these models and their ability to capture basic phenomenological features of ferrofluids. One such feature is the classical Rosensweig instability (uniform magnetic field from below; see Figures \ref{figureResolved}-\ref{parafigtime}), the more common case of the Rosensweig instability (non-uniform magnetic field; see Figures \ref{FigDipsHedge}-\ref{hedgefig}) when the interface forms and open pattern of spikes, and a final experiment which neglects the contribution of the demagnetizing field $\hd$ on the effective magnetizing field $\heff$ (see Figure \ref{hedgefigha}). This last experiment reveals the critical importance of the demagnetizing field $\hd$ in Rosensweig instability.

Finally, we must comment that many important issues are not
discussed. Among them we have to mention that regularizing the
model \eqref{Themodel} is very much an open problem, as is actually
solving the system posed by the numerical schemes proposed in this
work, modeling of saturation effects, modeling of magneto-rheological
effects, and deriving ferrofluid models via energy-variational techniques.

\section{Acknowledgements}
We want to thank the developers of \texttt{deal.II} for the significant amount of help received while developing the code used for the numerical experiments. In particular, the visit of Timo Heister to UMCP in Spring of 2014 was particularly fruitful. We also want to thank Martin J{\"o}nsson-Niedzi\'olka for giving us permission to include his pictures in this work.

\bibliographystyle{siam}
\bibliography{biblio}

\begin{thebibliography}{100}

\bibitem{Garcke2012}
{\sc H.~Abels, H.~Garcke, and G.~Gr{\"u}n}, {\em Thermodynamically consistent,
  frame indifferent diffuse interface models for incompressible two-phase flows
  with different densities}, Math. Models Methods Appl. Sci., 22 (2012),
  pp.~1150013, 40.

\bibitem{Abou2000}
{\sc B.~Abou, J.-E. Wesfreid, and S.~Roux}, {\em The normal field instability
  in ferrofluids: hexagon–square transition mechanism and wavenumber
  selection}, Journal of Fluid Mechanics, 416 (2000), pp.~217--237.

\bibitem{AFK2008}
{\sc S.~Afkhami, Y.~Renardy, M.~Renardy, J.~S. Riffle, and T.~St~Pierre}, {\em
  Field-induced motion of ferrofluid droplets through immiscible viscous
  media}, Journal of Fluid Mechanics, 610 (2008), pp.~363--380.

\bibitem{AFK2010}
{\sc S.~Afkhami, A.~J. Tyler, Y.~Renardy, M.~Renardy, T.~G. St.~Pierre, R.~C.
  Woodward, and J.~S. Riffle}, {\em Deformation of a hydrophobic ferrofluid
  droplet suspended in a viscous medium under uniform magnetic fields}, Journal
  of Fluid Mechanics, 663 (2010), pp.~358--384.

\bibitem{AmiShliomis2008}
{\sc Y.~Amirat and K.~Hamdache}, {\em Global weak solutions to a ferrofluid
  flow model}, Math. Methods Appl. Sci., 31 (2008), pp.~123--151.

\bibitem{Ami2009}
\leavevmode\vrule height 2pt depth -1.6pt width 23pt, {\em Strong solutions to
  the equations of a ferrofluid flow model}, J. Math. Anal. Appl., 353 (2009),
  pp.~271--294.

\bibitem{Ami2010}
\leavevmode\vrule height 2pt depth -1.6pt width 23pt, {\em Unique solvability
  of equations of motion for ferrofluids}, Nonlinear Anal., 73 (2010),
  pp.~471--494.

\bibitem{Ami2008}
{\sc Y.~Amirat, K.~Hamdache, and F.~Murat}, {\em Global weak solutions to
  equations of motion for magnetic fluids}, J. Math. Fluid Mech., 10 (2008),
  pp.~326--351.

\bibitem{Ander1998}
{\sc D.~M. Anderson, G.~B. McFadden, and A.~A. Wheeler}, {\em Diffuse-interface
  methods in fluid mechanics}, in Annual review of fluid mechanics, {V}ol. 30,
  vol.~30 of Annu. Rev. Fluid Mech., Annual Reviews, Palo Alto, CA, 1998,
  pp.~139--165.

\bibitem{Heister2011}
{\sc W.~Bangerth, C.~Burstedde, T.~Heister, and M.~Kronbichler}, {\em
  Algorithms and data structures for massively parallel generic adaptive finite
  element codes}, ACM Trans. Math. Software, 38 (2011), pp.~Art. 14, 28.

\bibitem{BHK2007}
{\sc W.~Bangerth, R.~Hartmann, and G.~Kanschat}, {\em {deal.II} -- a general
  purpose object oriented finite element library}, ACM Trans. Math. Softw., 33
  (2007), pp.~24/1--24/27.

\bibitem{DealIIReference}
{\sc W.~Bangerth, T.~Heister, and G.~Kanschat}, {\em {\tt deal.{I}{I}}
  Differential Equations Analysis Library, Technical Reference}.
\newblock \texttt{http://www.dealii.org}.

\bibitem{MR3150226}
{\sc R.~Bank and H.~Yserentant}, {\em On the {$H^1$}-stability of the
  {$L_2$}-projection onto finite element spaces}, Numer. Math., 126 (2014),
  pp.~361--381.

\bibitem{Pak2006}
{\sc Y.~Bao, A.~B. Pakhomov, and K.~M. Krishnan}, {\em Brownian magnetic
  relaxation of water-based cobalt nanoparticle ferrofluids}, Journal of
  Applied Physics, 99 (2006).

\bibitem{Bart2010}
{\sc S.~Bartels and R.~M{\"u}ller}, {\em A posteriori error controlled local
  resolution of evolving interfaces for generalized {C}ahn-{H}illiard
  equations}, Interfaces and Free Boundaries, 12 (2010), pp.~45--74.

\bibitem{Behrens}
{\sc S.~Behrens and H.~Bonemann}, {\em Synthesis and characterization}, in
  Colloidal Magnetic Fluids. Basics, Development and Application of
  Ferrofluids., S.~Odenbach, ed., Lecture Notes in Physics, Springer-Verlag,
  2009, pp.~1--69.

\bibitem{VPN1990}
{\sc G.~Bellettini, M.~Paolini, and C.~Verdi}, {\em {$\Gamma$}-convergence of
  discrete approximations to interfaces with prescribed mean curvature}, Atti
  Accad. Naz. Lincei Cl. Sci. Fis. Mat. Natur. Rend. Lincei (9) Mat. Appl., 1
  (1990), pp.~317--328.

\bibitem{Boff2013}
{\sc D.~Boffi, F.~Brezzi, and M.~Fortin}, {\em Mixed finite element methods and
  applications}, vol.~44 of Springer Series in Computational Mathematics,
  Springer, Heidelberg, 2013.

\bibitem{Braides2012}
{\sc A.~Braides and N.~K. Yip}, {\em A quantitative description of mesh
  dependence for the discretization of singularly perturbed nonconvex
  problems}, SIAM J. Numer. Anal., 50 (2012), pp.~1883--1898.

\bibitem{Stein2002}
{\sc J.~H. Bramble, J.~E. Pasciak, and O.~Steinbach}, {\em On the stability of
  the {$L^2$} projection in {$H^1(\Omega)$}}, Math. Comp., 71 (2002),
  pp.~147--156 (electronic).

\bibitem{Brou07}
{\sc D.~Brousseau, E.~F. Borra, and S.~Thibault}, {\em Wavefront correction
  with a 37-actuator ferrofluid deformable mirror}, Opt. Express, 15 (2007),
  pp.~18190--18199.

\bibitem{Chav2014}
{\sc A.~Chaves and C.~Rinaldi}, {\em Interfacial stress balances in structured
  continua and free surface flows in ferrofluids}, Physics of Fluids
  (1994-present), 26 (2014).

\bibitem{Chav2008}
{\sc A.~Chaves, M.~Zahn, and C.~Rinaldi}, {\em Spin-up flow of ferrofluids:
  Asymptotic theory and experimental measurements}, Physics of Fluids, 20
  (2008).

\bibitem{Ciar78}
{\sc P.~G. Ciarlet}, {\em The finite element method for elliptic problems},
  North-Holland Publishing Co., Amsterdam-New York-Oxford, 1978.
\newblock Studies in Mathematics and its Applications, Vol. 4.

\bibitem{CowRos1967}
{\sc M.~D. Cowley and R.~E. Rosensweig}, {\em The interfacial stability of a
  ferromagnetic fluid}, Journal of Fluid Mechanics, 30 (1967), pp.~671--688.

\bibitem{Crou1987}
{\sc M.~Crouzeix and V.~Thom{\'e}e}, {\em The stability in {$L_p$} and
  {$W^1_p$} of the {$L_2$}-projection onto finite element function spaces},
  Math. Comp., 48 (1987), pp.~521--532.

\bibitem{Demlow2013}
{\sc A.~Demlow and S.~Larsson}, {\em Local pointwise a posteriori gradient
  error bounds for the {S}tokes equations}, Math. Comp., 82 (2013),
  pp.~625--649.

\bibitem{DiErn2012}
{\sc D.~Di~Pietro and A.~Ern}, {\em Mathematical Aspects of Discontinuous
  Galerkin Methods}, Mathematiques Et Applications, Springer, 2012.

\bibitem{DiPi10}
{\sc D.~A. Di~Pietro and A.~Ern}, {\em Discrete functional analysis tools for
  discontinuous {G}alerkin methods with application to the incompressible
  {N}avier-{S}tokes equations}, Math. Comp., 79 (2010), pp.~1303--1330.

\bibitem{DurNoch1990}
{\sc R.~G. Dur{\'a}n and R.~H. Nochetto}, {\em Weighted inf-sup condition and
  pointwise error estimates for the {S}tokes problem}, Math. Comp., 54 (1990),
  pp.~63--79.

\bibitem{Elli1993}
{\sc C.~M. Elliott and A.~M. Stuart}, {\em The global dynamics of discrete
  semilinear parabolic equations}, SIAM J. Numer. Anal., 30 (1993),
  pp.~1622--1663.

\bibitem{ErnGuermond}
{\sc A.~Ern and J.-L. Guermond}, {\em Theory and practice of finite elements},
  vol.~159 of Applied Mathematical Sciences, Springer-Verlag, New York, 2004.

\bibitem{Feng2006}
{\sc X.~Feng}, {\em Fully discrete finite element approximations of the
  {N}avier-{S}tokes-{C}ahn-{H}illiard diffuse interface model for two-phase
  fluid flows}, SIAM J. Numer. Anal., 44 (2006), pp.~1049--1072 (electronic).

\bibitem{FerrotecWebpage}
{\sc Ferrotec}, {\em https://www.ferrotec.com/products/ferrofluid/emg/water/},
  2014.

\bibitem{Engel2001}
{\sc R.~Friedrichs and A.~Engel}, {\em Pattern and wave number selection in
  magnetic fluids}, Phys. Rev. E, 64 (2001), p.~021406.

\bibitem{Gaili1977}
{\sc A.~Gailitis}, {\em Formation of the hexagonal pattern on the surface of a
  ferromagnetic fluid in an applied magnetic field}, Journal of Fluid
  Mechanics, 82 (1977), pp.~401--413.

\bibitem{GT2001}
{\sc D.~Gilbarg and N.~S. Trudinger}, {\em Elliptic partial differential
  equations of second order}, Classics in Mathematics, Springer-Verlag, Berlin,
  2001.
\newblock Reprint of the 1998 edition.

\bibitem{MR3422453}
{\sc V.~Girault, R.~H. Nochetto, and L.~R. Scott}, {\em Max-norm estimates for
  {S}tokes and {N}avier--{S}tokes approximations in convex polyhedra}, Numer.
  Math., 131 (2015), pp.~771--822.

\bibitem{GirNochScott}
{\sc V.~Girault, R.~H. Nochetto, and R.~Scott}, {\em Maximum-norm stability of
  the finite element {S}tokes projection}, J. Math. Pures Appl. (9), 84 (2005),
  pp.~279--330.

\bibitem{Girault}
{\sc V.~Girault and P.-A. Raviart}, {\em Finite element methods for
  {N}avier-{S}tokes equations}, vol.~5 of Springer Series in Computational
  Mathematics, Springer-Verlag, Berlin, 1986.
\newblock Theory and algorithms.

\bibitem{Gir2005}
{\sc V.~Girault, B.~Rivi{\`e}re, and M.~F. Wheeler}, {\em A discontinuous
  {G}alerkin method with nonoverlapping domain decomposition for the {S}tokes
  and {N}avier-{S}tokes problems}, Math. Comp., 74 (2005), pp.~53--84
  (electronic).

\bibitem{Tob2007}
{\sc C.~Gollwitzer, G.~Matthies, R.~Richter, I.~Rehberg, and L.~Tobiska}, {\em
  The surface topography of a magnetic fluid: a quantitative comparison between
  experiment and numerical simulation}, Journal of Fluid Mechanics, 571 (2007),
  pp.~455--474.

\bibitem{grun2013}
{\sc G.~Gr\"un}, {\em On convergent schemes for diffuse interface models for
  two-phase flow of incompressible fluids with general mass densities}, SIAM
  Journal on Numerical Analysis, 51 (2013), pp.~3036--3061.

\bibitem{GuzLeyk2012}
{\sc J.~Guzm{\'a}n and D.~Leykekhman}, {\em Pointwise error estimates of finite
  element approximations to the {S}tokes problem on convex polyhedra}, Math.
  Comp., 81 (2012), pp.~1879--1902.

\bibitem{salga2013}
{\sc J.~Guzm\'an, A.~Salgado, and F.-J. Sayas}, {\em A note on the {L}ady{\v
  z}enskaja-{B}abu{\v s}ka-{B}rezzi condition}, Journal of Scientific
  Computing, 56 (2013), pp.~219--229.

\bibitem{Manuel2015}
{\sc J.~Guzm\'an and M.~A. S\'anchez}, {\em Max-norm stability of low order
  taylor--hood elements in three dimensions}, Journal of Scientific Computing,
  (2015), pp.~1--24.

\bibitem{Hart2004}
{\sc H.~Hartshorne, C.~J. Backhouse, and W.~E. Lee}, {\em Ferrofluid-based
  microchip pump and valve}, Sensors and Actuators B: Chemical, 99 (2004),
  pp.~592 -- 600.

\bibitem{TimoThesis}
{\sc T.~Heister}, {\em A Massively Parallel Finite Element Framework with
  Application to Incompressible Flows}, PhD thesis, Georg-August-Universit\"at
  at G\"ottingen, Mathematisch-Naturwissenschaftlichen Fakult\"aten, June 2011.

\bibitem{HeyRann}
{\sc J.~G. Heywood and R.~Rannacher}, {\em Finite-element approximation of the
  nonstationary {N}avier-{S}tokes problem. {IV}. {E}rror analysis for
  second-order time discretization}, SIAM J. Numer. Anal., 27 (1990),
  pp.~353--384.

\bibitem{Liu2010}
{\sc Y.~Hyon, D.~Y. Kwak, and C.~Liu}, {\em Energetic variational approach in
  complex fluids: maximum dissipation principle}, Discrete Contin. Dyn. Syst.,
  26 (2010), pp.~1291--1304.

\bibitem{Jack1998}
{\sc J.~Jackson}, {\em Classical Electrodynamics}, Wiley, 1998.

\bibitem{Jacq1996}
{\sc D.~Jacqmin}, {\em Calculation of two-phase {N}avier-{S}tokes flows using
  phase-field modeling}, Journal of Computational Physics, 155 (1999), pp.~96
  -- 127.

\bibitem{Mied2011}
{\sc M.~J{\"o}nsson-Niedzi{\'o}lka}, {\em
  https://plus.google.com/u/1/photos/+martinj\%c3\%b6nssonniedziolka/albums/5574178843082633953,
  http://ekorrbilden.blogspot.com/2011/02/ferrofluid.html}, 2011.

\bibitem{Kelly1983}
{\sc D.~W. Kelly, J.~P. De~S. R.~Gago, O.~C. Zienkiewicz, and I.~Babuska}, {\em
  A posteriori error analysis and adaptive processes in the finite element
  method: Part i—error analysis}, International Journal for Numerical Methods
  in Engineering, 19 (1983), pp.~1593--1619.

\bibitem{Kess2004}
{\sc D.~Kessler, R.~H. Nochetto, and A.~Schmidt}, {\em A posteriori error
  control for the {A}llen-{C}ahn problem: circumventing {G}ronwall's
  inequality}, ESAIM: Mathematical Modelling and Numerical Analysis, 38 (2004),
  pp.~129--142.

\bibitem{Laird06}
{\sc P.~Laird, N.~Caron, M.~Rioux, E.~F. Borra, and A.~Ritcey}, {\em
  Ferrofluidic adaptive mirrors}, Appl. Opt., 45 (2006), pp.~3495--3500.

\bibitem{Laird04}
{\sc P.~R. Laird, E.~F. Borra, R.~Bergamasco, J.~Gingras, L.~Truong, and
  A.~Ritcey}, {\em Deformable mirrors based on magnetic liquids}, Proc. SPIE,
  5490 (2004), pp.~1493--1501.

\bibitem{Les1974}
{\sc P.~Lasaint and P.-A. Raviart}, {\em On a finite element method for solving
  the neutron transport equation}, in Mathematical aspects of finite elements
  in partial differential equations ({P}roc. {S}ympos., {M}ath. {R}es.
  {C}enter, {U}niv. {W}isconsin, {M}adison, {W}is., 1974), Math. Res. Center,
  Univ. of Wisconsin-Madison, Academic Press, New York, 1974, pp.~89--123.
  Publication No. 33.

\bibitem{Rin2009}
{\sc M.~Latorre and C.~Rinaldi}, {\em Applications of magnetic nanoparticles in
  medicine: magnetic fluid hyperthermia}, PR Health Sciences Journal, 28
  (2009).

\bibitem{Tob2006}
{\sc O.~Lavrova, G.~Matthies, T.~Mitkova, V.~Polevikov, and L.~Tobiska}, {\em
  Numerical treatment of free surface problems in ferrohydrodynamics}, Journal
  of Physics: Condensed Matter, 18 (2006), p.~S2657.

\bibitem{Tob2008}
{\sc O.~Lavrova, G.~Matthies, and L.~Tobiska}, {\em Numerical study of
  soliton-like surface configurations on a magnetic fluid layer in the
  {R}osensweig instability}, Communications in Nonlinear Science and Numerical
  Simulation, 13 (2008), pp.~1302 -- 1310.

\bibitem{LL1995}
{\sc F.-H. Lin and C.~Liu}, {\em Nonparabolic dissipative systems modeling the
  flow of liquid crystals}, Comm. Pure Appl. Math., 48 (1995), pp.~501--537.

\bibitem{LinLiu2007}
{\sc P.~Lin, C.~Liu, and H.~Zhang}, {\em An energy law preserving {$C^0$}
  finite element scheme for simulating the kinematic effects in liquid crystal
  dynamics}, J. Comput. Phys., 227 (2007), pp.~1411--1427.

\bibitem{Lions1969}
{\sc J.-L. Lions}, {\em Quelques m\'ethodes de r\'esolution des probl\`emes aux
  limites non lin\'eaires}, Dunod; Gauthier-Villars, Paris, 1969.

\bibitem{Liu2009}
{\sc C.~Liu}, {\em An introduction of elastic complex fluids: an energetic
  variational approach}, in Multi-scale phenomena in complex fluids, vol.~12 of
  Ser. Contemp. Appl. Math. CAM, World Sci. Publishing, Singapore, 2009,
  pp.~286--337.

\bibitem{LiuShen2003}
{\sc C.~Liu and J.~Shen}, {\em A phase field model for the mixture of two
  incompressible fluids and its approximation by a {F}ourier-spectral method},
  Phys. D, 179 (2003), pp.~211--228.

\bibitem{WaLiu01}
{\sc C.~Liu and N.~J. Walkington}, {\em An {E}ulerian description of fluids
  containing visco-elastic particles}, Arch. Ration. Mech. Anal., 159 (2001),
  pp.~229--252.

\bibitem{LiuWalk2007}
\leavevmode\vrule height 2pt depth -1.6pt width 23pt, {\em Convergence of
  numerical approximations of the incompressible {N}avier-{S}tokes equations
  with variable density and viscosity}, SIAM J. Numer. Anal., 45 (2007),
  pp.~1287--1304 (electronic).

\bibitem{Trung2011}
{\sc J.~Liu, S.-H. Tan, Y.~Yap, M.~Ng, and N.-T. Nguyen}, {\em Numerical and
  experimental investigations of the formation process of ferrofluid droplets},
  Microfluidics and Nanofluidics, 11 (2011), pp.~177--187.

\bibitem{Lowen1998}
{\sc J.~Lowengrub and L.~Truskinovsky}, {\em Quasi-incompressible
  {C}ahn-{H}illiard fluids and topological transitions}, R. Soc. Lond. Proc.
  Ser. A Math. Phys. Eng. Sci., 454 (1998), pp.~2617--2654.

\bibitem{MarTem}
{\sc M.~Marion and R.~Temam}, {\em Navier-{S}tokes equations: theory and
  approximation}, in Handbook of numerical analysis, {V}ol. {VI}, Handb. Numer.
  Anal., VI, North-Holland, Amsterdam, 1998, pp.~503--688.

\bibitem{Miwa2003}
{\sc M.~Miwa, H.~Harita, T.~Nishigami, R.~Kaneko, and H.~Unozawa}, {\em
  Frequency characteristics of stiffness and damping effect of a ferrofluid
  bearing}, Tribology Letters, 15 (2003), pp.~97--105.

\bibitem{Tom14a}
{\sc R.~Nochetto, A.~Salgado, and I.~Tomas}, {\em The equations of
  ferrohydrodynamics: modeling and numerical methods}.
\newblock Submitted (arXiv:1511.04381), 2014.

\bibitem{Tom13}
\leavevmode\vrule height 2pt depth -1.6pt width 23pt, {\em The micropolar
  {N}avier-{S}tokes equations: {\it a priori} error analysis}, Math. Models
  Methods Appl. Sci., 24 (2014), pp.~1237--1264.

\bibitem{ElectroAbner}
{\sc R.~H. Nochetto, A.~J. Salgado, and S.~W. Walker}, {\em A diffuse interface
  model for electrowetting with moving contact lines}, Mathematical Models and
  Methods in Applied Sciences, 24 (2014), pp.~67--111.

\bibitem{Oden2002}
{\sc S.~Odenbach}, ed., {\em Ferrofluids: Magnetically Controllable Fluids and
  Their Applications (Lecture Notes in Physics)}, Springer, 2002.

\bibitem{Oden2009}
\leavevmode\vrule height 2pt depth -1.6pt width 23pt, ed., {\em Colloidal
  Magnetic Fluids: Basics, Development and Application of Ferrofluids (Lecture
  Notes in Physics)}, Springer, 2009.

\bibitem{Pank03}
{\sc Q.~Q.~A. Pankhurst, J.~Connolly, S.~K. Jones, and J.~Dobson}, {\em
  Applications of magnetic nanoparticles in biomedicine},  (2003),
  pp.~R167--R181.

\bibitem{Raj1995}
{\sc K.~Raj, B.~Moskowitz, and R.~Casciari}, {\em Advances in ferrofluid
  technology}, Journal of Magnetism and Magnetic Materials, 149 (1995), pp.~174
  -- 180.
\newblock Proceedings of the Seventh International Conference on Magnetic
  Fluids.

\bibitem{rinaldi2002}
{\sc C.~Rinaldi}, {\em Continuum Modeling of Polarizable Systems}, PhD thesis,
  Massachusetts Institute of Technology, Department of Chemical Engineering,
  June 2002.

\bibitem{Rinal02}
{\sc C.~Rinaldi and M.~Zahn}, {\em Effects of spin viscosity on ferrofluid flow
  profiles in alternating and rotating magnetic fields}, Physics of Fluids, 14
  (2002), pp.~2847--2870.

\bibitem{Ros97}
{\sc R.~E. Rosensweig}, {\em Ferrohydrodynamics}, Dover Publications, 1997.

\bibitem{Ros2002}
\leavevmode\vrule height 2pt depth -1.6pt width 23pt, {\em Basic equations for
  magnetic fluids with internal rotations}, in Ferrofluids: Magnetically
  controllable fluids and their applications, S.~Odenbach, ed., Lecture Notes
  in Physics, Springer-Verlag, 2002, pp.~61--84.

\bibitem{ros2007}
\leavevmode\vrule height 2pt depth -1.6pt width 23pt, {\em Stress
  boundary-conditions in ferrohydrodynamics}, Industrial \& Engineering
  Chemistry Research, 46 (2007), pp.~6113--6117.

\bibitem{Shap2013}
{\sc A.~Sarwar, R.~Lee, D.~Depireux, and B.~Shapiro}, {\em Magnetic injection
  of nanoparticles into rat inner ears at a human head working distance},
  Magnetics, IEEE Transactions on, 49 (2013), pp.~440--452.

\bibitem{ShenReview}
{\sc J.~Shen}, {\em Modeling and numerical approximation of two-phase
  incompressible flows by a phase-field approach}, in Multiscale modeling and
  analysis for materials simulation, vol.~22 of Lect. Notes Ser. Inst. Math.
  Sci. Natl. Univ. Singap., World Sci. Publ., Hackensack, NJ, 2012,
  pp.~147--195.

\bibitem{shenyang2010}
{\sc J.~Shen and X.~Yang}, {\em Energy stable schemes for {C}ahn-{H}illiard
  phase-field model of two-phase incompressible flows}, Chinese Annals of
  Mathematics, Series B, 31 (2010), pp.~743--758.

\bibitem{Shib2011}
{\sc Y.~Shibata, T.~Takamine, and M.~Kawaji}, {\em Emission of liquid droplets
  from an interface of bidrops pulled by a ferrofluid in a microchannel},
  International Journal of Thermal Sciences, 50 (2011), pp.~233 -- 238.

\bibitem{Shli2002}
{\sc M.~I. Shliomis}, {\em Ferrohydrodynamics: Retrospective and issues}, in
  Ferrofluids: Magnetically controllable fluids and their applications,
  S.~Odenbach, ed., Lecture Notes in Physics, Springer-Verlag, 2002,
  pp.~85--111.

\bibitem{Virga2012}
{\sc A.~M. Sonnet and E.~G. Virga}, {\em Dissipative ordered fluids}, Springer,
  New York, 2012.
\newblock Theories for liquid crystals.

\bibitem{Stephen1995}
{\sc P.~S. Stephen}, {\em Low viscosity magnetic fluid obtained by the
  colloidal suspension of magnetic particles}, Nov.~2 1965.
\newblock US Patent 3,215,572.

\bibitem{Huan2009}
{\sc H.~Sun and C.~Liu}, {\em On energetic variational approaches in modeling
  the nematic liquid crystal flows}, Discrete Contin. Dyn. Syst., 23 (2009),
  pp.~455--475.

\bibitem{Sunil2007}
{\sc Sunil, P.~Chand, and P.~K. Bharti}, {\em Double-diffusive convection in a
  micropolar ferromagnetic fluid}, Applied Mathematics and Computation, 189
  (2007), pp.~1648 -- 1661.

\bibitem{Temam}
{\sc R.~Temam}, {\em Navier-{S}tokes equations}, vol.~2 of Studies in
  Mathematics and its Applications, North-Holland Publishing Co., Amsterdam,
  third~ed., 1984.
\newblock Theory and numerical analysis, With an appendix by F. Thomasset.

\bibitem{MR2249024}
{\sc V.~Thom{\'e}e}, {\em Galerkin finite element methods for parabolic
  problems}, vol.~25 of Springer Series in Computational Mathematics,
  Springer-Verlag, Berlin, second~ed., 2006.

\bibitem{Vinoy2011}
{\sc K.~J. Vinoy and R.~M. Jha}, {\em Radar Absorbing Materials: From Theory to
  Design and Characterization}, Springer, Kluwer Academic Publishers, 2011.

\bibitem{Walk2005}
{\sc N.~J. Walkington}, {\em Convergence of the discontinuous {G}alerkin method
  for discontinuous solutions}, SIAM J. Numer. Anal., 42 (2005), pp.~1801--1817
  (electronic).

\bibitem{Wang2010}
{\sc Y.~Wang and Z.~Tan}, {\em Global existence and asymptotic analysis of weak
  solutions to the equations of ferrohydrodynamics}, Nonlinear Anal. Real World
  Appl., 11 (2010), pp.~4254--4268.

\bibitem{Yama2005}
{\sc C.~Yamahata, M.~Chastellain, V.~Parashar, A.~Petri, H.~Hofmann, and
  M.~A.~M. Gijs}, {\em Plastic micropump with ferrofluidic actuation},
  Microelectromechanical Systems, Journal of, 14 (2005), pp.~96--102.

\bibitem{Zahn01}
{\sc M.~Zahn}, {\em Magnetic fluid and nanoparticle applications to
  nanotechnology}, Journal of Nanoparticle Research,  (2001), pp.~73 -- 78.

\bibitem{Zahn95}
{\sc M.~Zahn and D.~R. Greer}, {\em Ferrohydrodynamic pumping in spatially
  uniform sinusoidally time-varying magnetic fields}, Journal of Magnetism and
  Magnetic Materials, 149 (1995), pp.~165 -- 173.

\bibitem{Zeng2013}
{\sc J.~Zeng, Y.~Deng, P.~Vedantam, T.-R. Tzeng, and X.~Xuan}, {\em Magnetic
  separation of particles and cells in ferrofluid flow through a straight
  microchannel using two offset magnets}, Journal of Magnetism and Magnetic
  Materials, 346 (2013), pp.~118 -- 123.

\end{thebibliography}

\end{document}